\documentclass[12pt,a4paper]{amsart}
\usepackage[leqno]{amsmath}
\usepackage[normalem]{ulem}
\usepackage{a4}
\usepackage{graphicx,amssymb,amsfonts,epsfig,amsthm,verbatim,url,enumerate}
\usepackage[latin1]{inputenc}
\usepackage[all]{xy}

\usepackage[usenames,dvipsnames]{color}

\makeatletter
\newcommand{\leqnomode}{\tagsleft@true}
\newcommand{\reqnomode}{\tagsleft@false}
\makeatother

\newtheorem*{thmm}{Theorem}

\newtheorem{ithm}{Theorem}

\newtheorem{icor}[ithm]{Corollary}

\newtheorem{iithm}{Theorem}

\newtheorem{iccor}{Corollary}

\newtheorem{iiithm}{Theorem}[section]

\theoremstyle{definition}
\newtheorem{iiidefn}[iiithm]{Definition}
\newtheorem{iiiex}[iiithm]{Example}

\numberwithin{equation}{subsection}
\theoremstyle{theorem}
\newtheorem{thm}[equation]{Theorem}
\newtheorem{cor}[equation]{Corollary}
\newtheorem{lem}[equation]{Lemma}

\newtheorem{prop}[equation]{Proposition}
\theoremstyle{definition}
\newtheorem{defn}[equation]{Definition}

\newtheorem{nota}[equation]{Notation}

\newtheorem{rem}[equation]{Remark}
\newtheorem{ex}[equation]{Example}

\newcommand{\Conv}{\mathop{\scalebox{2}{\raisebox{-0.2ex}{$\ast$}}}}%
\newcommand{\eps}{\varepsilon}
\newcommand{\Z}{\mathbf{Z}}
\newcommand{\N}{\mathbf{N}}
\newcommand{\R}{\mathbf{R}}
\newcommand{\Q}{\mathbf{Q}}

\newcommand{\K}{\mathbf{K}}
\newcommand{\g}{\mathfrak{g}}
\newcommand{\n}{\mathfrak{n}}
\newcommand{\A}{\mathsf{A}}
\newcommand{\B}{\mathsf{B}}

\newcommand{\RR}{\mathsf{R}}
\newcommand{\mk}{\mathfrak}
\newcommand{\td}{\triangledown}

\newcommand{\ot}{\otimes}
\newcommand{\we}{\wedge}

\newcommand{\Hom}{\textnormal{Hom}}
\newcommand{\Aut}{\textnormal{Aut}}
\newcommand{\Ker}{\textnormal{Ker}}

\newcommand{\GL}{\textnormal{GL}}
\newcommand{\SL}{\textnormal{SL}}

\newcommand{\SOL}{\textnormal{SOL}}
\newcommand{\BS}{\textnormal{BS}}
\newcommand{\Isom}{\text{Isom}}
\newcommand{\Kill}{\textnormal{Kill}}

\newcommand{\HC}{\textnormal{HC}}

\newcommand{\Cone}{\textnormal{Cone}}
\newcommand{\are}{\textnormal{area}}

\newcommand{\cc}{\circledcirc}

\newcommand{\cd}{\circledast}

\newcommand{\lp}{(\!(}
\newcommand{\rp}{)\!)}

\newcommand{\tw}{\twoheadrightarrow}
\newcommand{\rs}{\rightsquigarrow}

\newcommand{\epi}{\twoheadrightarrow}

\newcommand{\lB}{[\![}
\newcommand{\rB}{]\!]}

\newcounter{saveenum}

\begin{document}

\title{Geometric presentations of Lie groups and their Dehn functions}


\author[Cornulier]{Yves Cornulier}
\author[Tessera]{Romain Tessera}
\address{Laboratoire de Math\'ematiques\\
B\^atiment 425, Universit\'e Paris-Sud 11\\
91405 Orsay\\FRANCE}
\email{yves.cornulier@math.u-psud.fr}
\email{romain.tessera@math.u-psud.fr}

\date{November 25, 2016}
\thanks{The authors are supported by ANR Project GAMME (ANR-14-CE25-0004)}







\maketitle

\begin{abstract} We study the Dehn function of connected Lie groups. We show that this function is always exponential or polynomially bounded, according to the geometry of weights and of the 2-cohomology of their Lie algebras. Our work, which also addresses algebraic groups over local fields, uses and extends Abels' theory of multiamalgams of graded Lie algebras, in order to provide workable presentations of these groups.\end{abstract}

\tableofcontents



\section{Introduction}

\subsection{Dehn function of Lie groups}\label{introintro}

The object of study in this paper is the Dehn function of connected Lie groups. For a simply connected Lie group $G$ endowed with a left-invariant Riemannian metric, this can be defined as follows: the area of a loop $\gamma$ is the infimum of areas of filling discs, and the Dehn function $\delta_G(r)$ is the supremum of areas of loops of length at most $r$. The asymptotic behaviour of $\delta_G$ (when $r\to +\infty$) actually does not depend on the choice of a left-invariant Riemannian metric. If $G$ is an arbitrary connected Lie group and $K$ a compact subgroup such that $G/K$ is simply connected (e.g.\ $K$ is a maximal compact subgroup, in which case $G/K$ is diffeomorphic to a Euclidean space), we can endow $G/K$ with a $G$-invariant Riemannian metric and thus define the Dehn function $\delta_{G/K}(r)$ in the same way; its asymptotic behaviour depends only on $G$, neither on $K$ nor on the choice of the invariant Riemannian metric, and is called the Dehn function of $G$.

Let us provide some classical illustrating examples.
If for some maximal compact subgroup $K$, the space $G/K$ has a negatively curved $G$-invariant Riemannian metric, then the Dehn function of $G$ has exactly linear growth. Otherwise, $G$ is not Gromov-hyperbolic, and by a very general argument due to Bowditch \cite{Bowditch} (not specific to Lie groups), the Dehn function is known to be {\it at least} quadratic.  
On the other hand, the Dehn function is {\it at most} quadratic whenever $G/K$ can be endowed with a non-positively curved invariant Riemannian metric, notably when $G$ is reductive. It is worth emphasizing that many simply connected Lie groups $G$ fail to have a non-positively curved homogeneous space $G/K$ and nevertheless have a quadratic Dehn function. Characterizing Lie groups with a quadratic Dehn function is a very challenging problem, even in the setting of nilpotent Lie groups. Indeed, although connected nilpotent Lie groups have an at most polynomial Dehn function, there are examples with Dehn function of polynomial growth with arbitrary integer degree.
Besides, the Dehn function of a connected Lie group is at most exponential, the prototypical example of a Lie group with an exponential Dehn function being the three-dimensional SOL group.

A simple main consequence of the results we describe below is the    
following theorem.

\begin{ithm}\label{t_alt}
Let $G$ be a connected Lie group. Then the Dehn function of $G$ is either exponential or polynomially bounded. 
\end{ithm}

Using the well-known fact that polycyclic groups are virtually cocompact lattices in connected Lie groups, we deduce
\begin{icor}\label{c_alt}
The Dehn function of a polycyclic group is either exponential or polynomially bounded. 
\end{icor}

Let us mention that in both results, ``polynomially bounded" cannot be improved to ``of polynomial growth", since S.~Wenger \cite{Wen} has exhibited some simply connected nilpotent Lie groups (with lattices) with a Dehn function satisfying $n^2\prec\delta(n)\preccurlyeq n^2\log n$.

Our results are more precise than Theorem \ref{t_alt}: we characterize {\it algebraically} which ones have a polynomially bounded or exponential Dehn function. In order to state the result, we describe below two ``obstructions" implying exponential Dehn function; the first being related to SOL, and the second to homology in degree~2; these two obstructions appeared in a related context in Abels' seminal work on $p$-adic algebraic groups \cite{A}. We prove that each of these obstructions indeed implies that the Dehn function has an exponential growth, and conversely
that if none of these obstructions is fulfilled, then the group has an at most polynomial Dehn function, proving in a large number of cases that the Dehn function is at most quadratic or cubic.

These results can appear as unexpected. Indeed, it was suggested by Gromov \cite[5.$\textnormal{A}_9$]{Gro} that the only obstruction should be related to SOL. This has been proved in several important cases \cite{Gro,Dru,LP} but turns out to be false in general.

Our approach relies on the dynamical structure arising from the action of $G$ on itself by conjugation.  A crucial role is played by some naturally defined subgroups that are contracted by suitable elements. The presence of these subgroups is a non-discrete feature, which is invisible in any cocompact lattice (when such lattices exist). In particular, it sounds unlikely that Corollary \ref{c_alt} can be proved directly with no reference to ambient Lie groups, and the obstructions themselves are not convenient to state directly in terms of the structure of those discrete polycyclic groups.

The remainder of this introduction is organized as follows: in \S\ref{riveco}, we define a combinatorial Dehn function for compactly presented locally compact groups, and use it to state a version of Theorem \ref{t_alt} for algebraic $p$-adic groups. Then \S\ref{subsectionIntro:main} is dedicated to our main results. 
The most difficult part of the main theorem is the fact that in the absence of the two obstructions, the Dehn function is polynomially bounded. In \S\ref{subsectionIntro:computable}, we describe
a useful characterization
of these obstructions in terms of {\em Lie algebra gradings}. We apply these results to various concrete examples in~\S\ref{exe_wei}.
We outline the general strategy in \S\ref{introsection:proof}. We introduce an additional condition under which we prove that the Dehn function is quadratic (Theorem \ref{imainss}). In \S\ref{s_skpr}, we sketch the proof of this theorem; this allows us emphasize the key ideas of our approach.
A substantial part of our work is purely algebraic and possibly of interest for other purposes, we introduce it independently in \S\ref{sectionIntro:Liealgebra}.

\subsection{Riemannian versus Combinatorial Dehn function of Lie groups}\label{riveco}

 The previous approaches consisted of either working with groups admitting a cocompact lattice and using combinatorial methods, or using the Riemannian definition. Our approach, initiated in \cite{BQ}, consists of extending the combinatorial language and methods to general locally compact compactly generated groups. In particular, Lie groups are treated as combinatorial objects, i.e.\ groups endowed with a compact generating set and the corresponding Cayley graph. The object of study is the combinatorial Dehn function, usually defined for discrete groups, which turns out to be asymptotically equivalent to its Riemannian counterpart.
 This unifying approach allows to treat $p$-adic algebraic groups and connected Lie groups on the same footing.

We now give the combinatorial definition of Dehn function (rechristening the above definition of Dehn function as {\bf Riemannian Dehn function}); see \S\ref{dede} for more details. Let $G$ be a locally compact group, generated by a compact subset $S$. Let $F_S$ be the free group over the (abstract) set $S$ and $F_S\to G$ the natural epimorphism, and $K$ its kernel (its elements are called {\bf relations}). We say that $G$ is {\bf compactly presented} if for some $\ell$, $K$ is generated, as a normal subgroup of $F_S$, by the set $K_\ell$ of elements with length at most $\ell$ with respect to $S$, or equivalently if $K$ is generated, as a group, by the 
union $\bigcup_{g\in F_S}gK_\ell g^{-1}$ of conjugates of $K_\ell$ in $F_S$; this does not depend on the choice of $S$; the subset $K_\ell$ is called a set of {\bf relators}. Assuming this, if $x\in K$, the {\bf area} of $x$ is by definition the number $\are(x)$ defined as its length with respect to $\bigcup_{g\in F_S}gK_\ell g^{-1}$. Finally, the Dehn function of $G$ is defined as
$$\delta(n)=\sup\{\are(x):\;x\in K,\,|x|\le n\}.$$
In the discrete setting ($S$ finite), this function takes finite values, and this remains true in the locally compact setting. If $G$ is not compactly presented, a good convention is to set $\delta(n)=+\infty$ for all $n$.
The Dehn function of a compactly presented group $G$ depends on the choices of $S$ and $\ell$, but its asymptotic behavior does not.
In addition, for a simply connected Lie group, the combinatorial and Riemannian Dehn functions have the same asymptotic behavior, see Proposition \ref{eqde}.

With this definition at hand, we can now state a version of Theorem \ref{t_alt} in a non-Archimedean setting.

\begin{ithm}\label{t_altp}
Let $G$ be an algebraic group over some $p$-adic field. Then the Dehn function of $G$ is at most cubic, or $G$ is not compactly presented. 
\end{ithm}
Before providing more detailed statements, let us compare Theorems \ref{t_alt} and \ref{t_altp}. 
It is helpful to have in mind a certain analogy between Archimedean and non-Archimedean groups, where exponential Dehn function corresponds to not compactly presented. On the other hand, a striking difference between these two theorems is the absence for $p$-adic groups of polynomial Dehn functions of arbitrary degree. The explanation of this fact can be summarized as follows. 
In the connected Lie group setting, Dehn functions of ``high polynomial degree" witness to the presence of simply connected non-abelian nilpotent quotients, see Theorem \ref{generalized} for a precise statement. By way of contrast, any totally disconnected, compactly generated locally compact nilpotent group is compact-by-discrete; 
and if a compactly generated locally compact group $G$ is isomorphic to the group of $\Q_p$-points of some $p$-adic algebraic group, then any nilpotent quotient of $G$ is compact-by-abelian.

\subsection{Main results}\label{subsectionIntro:main}

We now turn to more comprehensive statements. Let us first introduce the two main classes of groups we will be considering in the sequel.

\begin{iiidefn}
A {\bf real triangulable} group is a Lie group isomorphic to a closed connected group of real triangular matrices. Equivalently, it is a simply connected solvable group in which for every $g$, the adjoint operator $\textnormal{Ad}(g)$ has only real eigenvalues.
\end{iiidefn}

It can be shown that every connected Lie group $G$ is quasi-isometric to a real triangulable Lie group. Namely, there exist compactly generated Lie groups $G_1,G_2,G_3$ and maps $$G\leftarrow G_1 \to G_2\leftarrow G_3,$$ where $G_3$ is real triangulable, and each arrow is a proper continuous homomorphism with cocompact image and thus is a quasi-isometry, see Lemma \ref{comtri}.

Let $A$ be an abelian group and consider a representation of $A$ on a $\K$-vector space $V$, where $\K$ is a finite product of complete normed fields. Let $V_0$ be the largest $A$-equivariant quotient of $V$ on which $A$ acts with only eigenvalues of modulus one.

\begin{iiidefn}\label{d_ssg}A locally compact group is a {\bf standard solvable group} if it is topologically isomorphic to a semidirect product $U\rtimes A$ so that
\begin{enumerate}
\item\label{ssg0} $A$ is a compactly generated locally compact abelian group
\item\label{ssg1} $U$ decomposes as a finite direct product $\prod U_i$, where each $U_i$ is normalized by $A$ and can be written as $U_i=\mathbb{U}_i(\K_i)$, where $\mathbb{U}_i$ is a unipotent group over some nondiscrete locally compact field of characteristic zero $\K_i$;
\item\label{ssg3} $(U/[U,U])_0=\{0\}$.
\end{enumerate}
\end{iiidefn}

For a group $G$ satisfying (\ref{ssg0}) and (\ref{ssg1}), condition (\ref{ssg3}) implies that $G$ is compactly generated, and conversely if $U$ is totally disconnected, the failure of condition (\ref{ssg3}) implies that $G$ is not compactly generated. If $G$ is a compactly generated $p$-adic group as in Theorem \ref{t_altp}, then it has a Zariski closed cocompact subgroup which is a standard solvable group (with a single $i$ and $\K_i=\Q_p$). Many real Lie groups have a closed cocompact standard solvable group; however, for instance, a simply connected nilpotent Lie group is not standard solvable unless it is abelian. We now introduce a very special but important class of standard solvable groups.

\begin{iiidefn}\label{d_sol}
A {\bf group of SOL type} is group $U\rtimes A$, where $U=\K_1\times\K_2$, where $\K_1$, $\K_2$ are nondiscrete locally compact fields of characteristic zero, and $A\subset\K_1^*\times\K_2^*$ is a closed subgroup of $\K_1^*\times\K_2^*$ containing, as a cocompact subgroup, the cyclic group generated by some element $(t_1,t_2)$ with $|t_1|>1>|t_2|$. Note that this is a standard solvable group. We call it a {\bf non-Archimedean group of SOL type} if {\em both} $\K_1$ and $\K_2$ are non-Archimedean.
\end{iiidefn}

\begin{iiiex}\label{e_sol}If $\K_1=\K_2=\K$ and $A$ is the set of pairs $(t,t^{-1})$, then $G$ is the usual group $\textnormal{SOL}(\K)$. More generally, $A$ is the set of pairs $(t^k,t^{-\ell})$ where $(k,\ell)$ is a fixed pair of positive integers, then this provides another group, which is unimodular if and only $k=\ell$. Another example is $(\R\times\Q_p)\rtimes\Z$, where $\Z$ acts as the cyclic subgroup generated by $(p,p)$ (note that $|p|_\R>1>|p|_{\Q_p}$); the latter contains the Baumslag-Solitar group $\Z[1/p]\rtimes\Z$ as a cocompact lattice.

Also, define, for $\lambda>0$, the group $\SOL_\lambda$ as the semidirect product $\R^2\rtimes\R$
where $\R$ is identified with the subgroup $\{(t,t^\lambda):t>0\}$ of $(\R^*)^2$. Note that $\SOL_1$ has index 2 in $\SOL(\R)$; there are obvious isomorphisms $\SOL_\lambda\simeq\SOL_{\lambda^{-1}}$, and the $\SOL_\lambda$, for $\lambda\ge 1$, are pairwise non-isomorphic. These are the only real triangulable groups of SOL type.
\end{iiiex}

\begin{iiidefn}{\bf (SOL obstruction)}
A locally compact group has the SOL obstruction (resp.\ non-Archimedean SOL obstruction) if it admits a homomorphism with dense image to a group of SOL type (resp.\ non-Archimedean SOL type). 
\end{iiidefn}

That the SOL obstruction implies that the Dehn function grows at least exponentially is the contents of the forthcoming Theorems \ref{thmi_sol} and \ref{thmi_hom}. 

In the positive direction, we start with following result, which generalizes \cite[5.$\textnormal{A}_9$]{Gro}, \cite[Theorem 1.1 (2)]{Dru}, \cite{LP} and \cite{BQ}.

\begin{ithm}\label{ith_veryt}
Let $G=U\rtimes A$ be a standard solvable group. Suppose that every closed subgroup of $G$ containing $A$ (thus of the form $V\rtimes A$ with $V$ a closed $A$-invariant subgroup of $U$) is a standard solvable group and does not satisfy the SOL obstruction. Then $G$ has an at most quadratic Dehn function.
\end{ithm}

The main result of \cite{BQ} is essentially the case when the normal subgroup $U$ is abelian (but on the other hand works in arbitrary characteristic), which itself generalizes results of Gromov \cite{Gro}, Dru\c tu \cite{Dru}, Leuzinger-Pittet \cite{LP}.

 Theorem \ref{ith_veryt} turns out to be a particular instance of the much more general Theorem \ref{imainss} below, but is considerably easier: the material is the length estimates of the beginning of Section \ref{s_spec} and Gromov's trick described in \S\ref{grotrick}. A direct proof of Theorem \ref{ith_veryt} is given in \S\ref{s2t}.

In \cite[5.$\textnormal{B}'_4$]{Gro}, quoth Gromov, {\it ``We conclude our discussion on lower and upper bounds for filling area by a somewhat pessimistic note. The present methods lead to satisfactory results only in a few special cases even in the friendly geometric surroundings of solvable and nilpotent groups."}

Indeed, for standard solvable groups without the SOL obstruction, it seems that Theorem \ref{ith_veryt} is the best result that can be gotten without bringing forward new ideas. The first example of a standard solvable group without the SOL obstruction but not covered by Theorem \ref{ith_veryt} is Abels' group $A_4(\K)$ (see the previous subsection). In this particular example, the authors obtain a quadratic upper bound for the Dehn function in \cite{CTAbels}. In this case, the group is tractable enough to work with explicit matrices, but such a pedestrian approach becomes hopelessly intricate in an arbitrary group as in Theorem \ref{int_p}.

In a more general context, we have to deal with groups without the SOL obstruction but not satisfying the assumptions of Theorem \ref{ith_veryt}. In this context, we have to introduce the 2-homological obstruction. For this, we need to recall a fundamental notion introduced and studied by Guivarc'h \cite{Gui80} and later by Osin \cite{Osin}. Let $G$ be a real triangulable group. Its {\bf exponential radical} $G^\infty$ is defined as the intersection of its lower central series and actually consists of the exponentially distorted elements in $G$. Let $\mathfrak{g}^\infty$ be its Lie algebra. In the case of a standard solvable group, the role of exponential radical is played by $U$ itself (it can be checked to be equal to the derived subgroup of $G$, so is a characteristic subgroup).

\begin{iiidefn}({\bf 2-homological obstruction})
\begin{itemize}
\item
The real triangulable group $G$ is said to satisfy the {\bf 2-homological obstruction} if $H_2(\mk{g}^\infty)_0\neq\{0\}$, or equivalently if the $A$-action on $H_2(\mk{g}^\infty)$ has some nonzero invariant vector.
\item
The standard solvable group $G=U\rtimes A$ is said to satisfy the {\bf 2-homological obstruction} if $H_2(\mathfrak{u})_0\neq\{0\}$, that is to say, $H_2(\mathfrak{u}_j)_0\neq\{0\}$ for some $j$. If moreover $j$ can be chosen so that $\K_j$ is non-Archimedean, we call it the {\bf non-Archimedean 2-homological obstruction}.
\end{itemize}
\end{iiidefn}

In most cases, including standard solvable groups, the 2-homological obstructions can be characterized by the existence of suitable central extensions. For instance, if a real triangulable group $G$ has a central extension               
$\tilde{G}$, also real triangulable, with nontrivial kernel $Z$ such that       
$Z\subset (\tilde{G})^\infty$ then it satisfies the 2-homological               
obstruction.
 
The converse is true when $G$ admits a semidirect decomposition    
$G^\infty\rtimes N$, but nevertheless does not in general, see \S\ref{s_sce}.

We are now able to state our main        
theorem, which immediately entails Theorems \ref{t_alt} and \ref{t_altp}.

\begin{ithm}\label{main}
Let $G$ be a real triangulable group, or a standard solvable group.
\begin{itemize}
\item if $G$ satisfies one of the two non-Archimedean (SOL or 2-homological) obstructions, then $G$ is not compactly presented;
\item otherwise $G$ is compactly presented and has an at most exponential Dehn function. Moreover, in this case
\begin{itemize}
\item if $G$ satisfies one of the two (SOL or 2-homological) obstructions, then $G$ has an exponential Dehn function;
\item if $G$ satisfies none of the obstructions, then it has a polynomially bounded Dehn function; in the case of a standard solvable group, the Dehn function is at most cubic.
\end{itemize}
\end{itemize}
\end{ithm}

This result can be seen as both a generalization and a strengthening of the following seminal result of Abels \cite{A}.

\begin{thmm}[Abels]\label{t_abels}
Let $G$ be a standard solvable group over a $p$-adic field.
Then $G$ is compactly presented if and only if it satisfies none of the non-Archimedean obstructions.
\end{thmm}

Let us split Theorem \ref{main} into several independent statements. The first two provide lower bounds and the last two provide upper bounds on the Dehn function.

\begin{iithm}\label{thmi_sol}
Let $G$ be a standard solvable or real triangulable group. If $G$ satisfies the SOL (resp.\ non-Archimedean SOL) obstruction, then $G$ has an at least exponential Dehn function (resp.\ is not compactly presented).
\end{iithm}
We provide a unified proof of both assertions of Theorem \ref{thmi_sol} in Section \ref{s_not}. Note that the non-Archimedean case is essentially contained in the ``only if" (easier) part of Abels' theorem above, itself inspired by previous work of Bieri-Strebel, notably \cite[Theorem A]{BiS}. Part of the proof consists in estimating the size of loops in the groups of SOL type, where our proof  is inspired by the original case of the real 3-dimensional SOL group, due to Thurston \cite{ECHLPT}, which uses integration of a well-chosen differential form. Our method in Section \ref{s_not} is based on a discretization of this argument, leading to both a simplification and a generalization of the argument.

\begin{iithm}\label{thmi_hom}
Let $G$ be a standard solvable or real triangulable group. If $G$ satisfies the 2-homological (resp.\ non-Archimedean 2-homological) obstruction, then $G$ has an at least exponential Dehn function (resp.\ is not compactly presented).
\end{iithm}
The case of standard solvable groups reduces, after a minor reduction, to a simple and classical central extension argument, see \S\ref{s_ce}. The case of real triangulable groups is considerably more difficult; in the absence of splitting of the exponential radical, we construct an ``exponentially distorted hypercentral extension". This is done in Section \ref{s_cent}.

\begin{iithm}\label{dee}
Let $G$ be a real triangulable group. Then $G$ has an at most exponential Dehn function.

Let $G=U\rtimes A$ be a standard solvable group and $U^\circ$ the identity component in $U$. If $G/U^\circ$ is compactly presented, then $G$ is compactly presented with an at most exponential Dehn function.
\end{iithm}

Since compact presentability is stable under taking extensions \cite{Ab72}, for an arbitrary locally compact group $G$ with a closed connected normal subgroup $C$, it is true that $G$ is compactly presented if and only $G/C$ is compactly presented. We do not know if this can be generalized to the statement that if $G/C$ has Dehn function $\preccurlyeq f(n)$, then $G$ has Dehn function $\preccurlyeq \max(f(n),\exp(n))$. Theorem \ref{dee}, which follows from Theorem \ref{t_lieexp} and Corollary \ref{corex}, contains two particular instances where the latter assertion holds, which are enough for our purposes. The first instance, namely that every connected Lie group has an at most exponential Dehn function, was asserted by Gromov, with a sketch of proof \cite[Corollary 3.$\textnormal{F}'_5$]{Gro}.

\begin{iiiex}
Fix $n\in\Z$ with $|n|\ge 2$. Consider the group $G_n=(\R\times\Q_n)\rtimes_n\Z$, where $\Q_n$ is the product of $\Q_p$ where $p$ ranges over distinct primes divisors of~$n$. Here $U^\circ\simeq\R$ and $G/U^\circ\simeq\Q_n\rtimes_n\Z$, which, as a hyperbolic group (it is an HNN ascending extension of the compact group $\Z_n$), has a linear Dehn function. So, by Theorem \ref{dee}, $G_n$ has an at most exponential Dehn function. Since $|n|\ge 2$, it admits a prime factor $p$, so $G_n$ admits the group of SOL type $(\R\times\Q_p)\rtimes_n\Z$ as a quotient, and therefore $G_n$ satisfies the SOL obstruction and thus has an at least exponential Dehn function by Theorem \ref{thmi_sol}. We conclude that $G_n$ has an exponential Dehn function. This provides a new proof that its lattice, the Baumslag-Solitar group
\[\BS(1,n)=\langle t,x\mid txt^{-1}=x^n\rangle,\]
has an exponential Dehn function. This is actually true for arbitrary Baumslag-Solitar groups $\BS(m,n)$, $|m|\neq |n|$, for which the exponential upper bound was first established in \cite[Theorem 7.3.4 and Example 7.4.1]{ECHLPT} and independently in \cite{BGSS}, and the exponential lower bound, attributed to Thurston, was obtained in \cite[Example 7.4.1]{ECHLPT}.
\end{iiiex}

Let us provide a useful corollary of Theorems \ref{thmi_sol} and \ref{dee}.

\begin{iccor}
Let $G=U\rtimes A$ be a standard solvable group in which $A$ has rank 1 (i.e., has a closed infinite cyclic cocompact subgroup). Then exactly one of the following occurs
\begin{itemize}
\item $G$ satisfies the non-Archimedean SOL obstruction and thus is not compactly presented;
\item $G$ satisfies the SOL obstruction but not the non-Archimedean one; it is compactly presented with an exponential Dehn function;
\item $G$ does not satisfy the SOL obstruction; it has a linear Dehn function and is Gromov-hyperbolic.
\end{itemize}
\end{iccor}

It is indeed an observation that if $A$ has rank 1 and $G$ does not satisfy the SOL obstruction, then some element of $A$ acts on $G$ as a compaction (see Definition \ref{d_cpon}) and it follows from \cite{CCMT} that $G$ is Gromov-hyperbolic, or equivalently has a linear Dehn function. In this special case where $A$ has rank~1, the 2-homological obstruction, which may hold or not hold, implies the SOL obstruction and is accordingly unnecessary to consider; see also Theorem \ref{ith_veryt}.

Turning back to Theorem \ref{main}, the fourth and most involved of all the steps is the following.

\begin{iithm}\label{int_p}
Let $G$ be a standard solvable (resp.\ real triangulable) group not satisfying neither the SOL nor the 2-homological obstructions. Then $G$ has an at most cubic (resp.\ at most polynomial) Dehn function.
\end{iithm}
The proof of Theorem \ref{int_p} for standard solvable groups is done in Section \ref{s_spec}, relying on algebraic preliminaries, occupying Sections \ref{s_abels} and \ref{s:am}. We actually obtain, with some additional work, a similar statement for ``generalized standard solvable groups", where $A$ is replaced by some nilpotent compactly generated group $N$ (see Theorem \ref{generalized}). The case of real triangulable groups requires an additional step, namely a reduction to the case where the exponential radical is split, in which case the group is generalized standard solvable. This reduction is performed in \S\ref{rssg}, and relies on results from \cite{CorIll}.

\subsection{The obstructions as ``computable" invariants of the Lie algebra}\label{subsectionIntro:computable}

The obstructions were introduced above in a convenient way for expository reasons, but the natural framework to deal with them uses the language of graded Lie algebras, which we now describe.

Let $G$ be either a standard solvable group $U\rtimes A$ or a real triangulable group with exponential radical also denoted, for convenience, by $U$.
 Let $\mk{u}$ be the Lie algebra of $U$; it is a Lie algebra (over a finite product of nondiscrete locally compact fields of characteristic zero). It is, in a natural way, a graded Lie algebra. In both cases, the grading takes values into a finite-dimensional real vector space, namely $\Hom(G/U,\R)$. It is introduced in \S\ref{gssg}, relying on Theorem \ref{gradnor} (see \S\ref{exe_wei} for a representative particular case.) In this setting, there are useful restatements (see Propositions \ref{2tsol} and \ref{2tsoltriang}) of the obstructions. We say that $\alpha\in\Hom(G/U,\R)$ is a {\bf weight} of $G$ (or of $U$, when $\mk{u}$ is endowed with the grading) if $\mk{u}_\alpha\neq 0$ and is a {\bf principal weight} of $G$ if $\alpha$ is a weight of $U/[U,U]$. The definitions imply that 0 is not a principal weight (although it can be a weight). We say that two nonzero weights $\alpha,\beta$ are {\em quasi-opposite} if $0\in [\alpha,\beta]$, i.e., $\beta=-t\alpha$ for some $t>0$. We write $U=U_{\textnormal{a}}\times U_{\textnormal{na}}$ as the product of its Archimedean and non-Archimedean parts. Then we have the following restatements:

\begin{itemize}
\item $G$ satisfies the SOL obstruction $\Leftrightarrow$ $U$ admits two quasi-opposite principal weights (Propositions \ref{2tsol} and \ref{2tsoltriang});
\item $G$ satisfies the non-Archimedean SOL obstruction $\Leftrightarrow$ $U_{\textnormal{na}}$ admits two quasi-opposite principal weights (Proposition \ref{2tsol} applied to $G/G_0$);
\item $G$ satisfies the 2-homological obstruction $\Leftrightarrow$ $H_2(\mk{u})_0\neq \{0\}$;
\item $G$ satisfies the non-Archimedean 2-homological obstruction $\Leftrightarrow$ $H_2(\mk{u}_{\textnormal{na}})_0\neq \{0\}$;
\item $G$ fails to satisfy the assumption of Theorem \ref{ith_veryt} $\Leftrightarrow$ 0 belongs to the interval joining some pair of weights.
\end{itemize}

Here, $H_2(\mk{u})$ denotes the homology of the Lie algebra $\mk{u}$; the grading on $\mk{u}$ in the real vector space $\Hom(G/U,\R)$ canonically induces a grading of $H_2(\mk{u})$ in the same space (see \S\ref{basc}), and $H_2(\mk{u})_0$ is its component in degree zero.

Another important module associated to $\mk{u}$ is $\Kill(\mk{u})$, the quotient of the second symmetric power $\mk{u}\cc\mk{u}$ by the submodule generated by elements of the form $[x,y]\cc z-x\cc [y,z]$; thus, in case of a single field (that is, $U=U_1$ in Definition \ref{d_ssg}), the invariant quadratic forms on $\mk{u}$ are elements in the dual of $\Kill(\mk{u})$. 

\begin{ithm}\label{imainss}
Let $G=U\rtimes A$ be a standard solvable group not satisfying any of the SOL or 2-homological obstructions. Suppose in addition that $\Kill(\mk{u})_0=\{0\}$. Then $G$ has an at most quadratic Dehn function (thus exactly quadratic if $A$ has rank at least two). 
\end{ithm}

 Theorem \ref{imainss} is proved in \S\ref{concl0}. Part of it is proved 
along with Theorem \ref{int_p} and involves the same difficulties, except that the welding relations do not appear.

 The condition $\Kill(\mk{u})_0\neq\{0\}$ corresponds to the existence of certain central extensions of $G$ as a discrete group. Using asymptotic cones, it will be shown in a subsequent paper that it implies, in many cases (without the SOL and 2-homological obstructions), that the Dehn function grows strictly faster than a quadratic function.

\subsection{Examples}\label{exe_wei}

Let us give a few examples. All are standard solvable connected Lie groups $G=U\rtimes A$ so that the action of $A$ on the Lie algebra $\mk{u}$ of $U$ is $\R$-diagonalizable. In this context, we call $\Hom(A,\R)$ the weight space. The grading of the Lie algebra $\mk{u}$ in $\Hom(A,\R)$ is given by
$$\mk{u}_\alpha=\{u\in\mk{u}\mid\forall v\in A,\;v^{-1}uv=e^{\alpha(v)}u\}.$$

When we write the set of weights, we use boldface for the set of principal weights. We underline the zero weight (or just denote $\centerdot$ to mark zero if zero is not a weight). 

\subsubsection{Groups of SOL type}

For a group of SOL type (Definition \ref{d_sol}), the weight space is a line and the weights lie on opposite sides of zero:

$$\begin{array}[c]{ccccc} 
 \mathbf{1} & \centerdot  &&  \mathbf{2}
\end{array}$$

By definition it satisfies the SOL obstruction. On the other hand, it satisfies the 2-homological obstruction only in a few special cases. For instance, $(\R\times\Q_p)\rtimes_p\Z$ does not satisfy the 2-homological obstruction, and the real group $\SOL_\lambda$ (see Example \ref{e_sol}) satisfies the 2-homological obstruction only for $\lambda=1$.

\subsubsection{Gromov's higher SOL groups}

For the group $\R^3\rtimes \R^2$  where $\R^2$ acts on $\R^3$ as the group of diagonal matrices with positive diagonal entries and determinant one (often called higher-dimensional SOL group, but not of SOL type nor even satisfying the SOL obstruction according to our conventions), the weight space is a plane in which the weights form a triangle whose center of gravity is zero:

$$\begin{array}[c]{ccccc} 
 && \mathbf{2} &&\\
  &&&& \\
 && \centerdot &&\\
 \mathbf{1} &&&& \mathbf{3}
\end{array}$$

Since there are no opposite weights, we have $(\mk{u}\otimes\mk{u})_0=0$ and therefore $H_2(\mk{u})_0$ and $\Kill(\mk{u})_0$ (which are subquotients of $(\mk{u}\otimes\mk{u})_0$) are also zero. It was stated with a sketch of proof by Gromov that this group has a quadratic Dehn function \cite[5.$\textnormal{A}_9$]{Gro}. Dru\c tu obtained in \cite[Corollary 4.18]{Dru1} that it has a Dehn function $\preccurlyeq n^{3+\eps}$, and then obtained a quadratic upper bound in \cite[Theorem 1.1]{Dru}, a result also obtained by Leuzinger and Pittet in \cite{LP}. The quadratic upper bound can also be viewed as an illustration of Theorem \ref{ith_veryt}. (The assumption that 0 is the center of gravity is unessential: the important fact is that 0 belongs to the convex hull of the three weights but does not lie in the segment joining any two weights.)

\subsubsection{Abels' first group}
The group $G=A_4(\K)$ consists of matrices of the form 
$$\begin{pmatrix} 1 & u_{12} & u_{13} & u_{14}\\ 0 & s_{2} & u_{23} & u_{24}\\ 0 & 0 & s_{3} & u_{34} \\ 0 & 0 & 0 & 1\end{pmatrix};\quad s_{i}\in \K^\times,\;u_{ij}\in \K.$$
Its weight configuration is given by

$$\begin{array}[c]{ccccc} 
 && \mathbf{23} &&\\
 13 &&&& 24\\
 && \underline{14} &&\\
 \mathbf{12} &&&& \mathbf{34}
\end{array}$$

This example is interesting because it does not satisfy the SOL obstruction but admits opposite weights. A computation shows that $H_2(\mk{u})_0=0$ and $\Kill(\mk{u})_0=0$ (see Abels \cite[Example 5.7.1]{A}).

If $\bar{G}=\bar{U}\rtimes A$ is the quotient of $G$ by its one-dimensional center, then $G$ does not satisfy the SOL obstruction but satisfies $H_2(\bar{\mk{u}})_0\neq 0$, i.e.\ satisfies the 2-homological obstruction.

This example was studied specifically by the authors in the paper \cite{CTAbels}, where it is proved that it has a quadratic Dehn function, which also follows from Theorem~\ref{imainss}.

\subsubsection{Abels' second group}\label{abelssecond}

This group was introduced in \cite[Example 5.7.4]{A}. 
Consider the group $U\rtimes A$, where $A\simeq\R^2$ and $U$ is the group corresponding to the quotient of the free $3$-nilpotent Lie algebra generated by the $3$-dimensional $\K$-vector space of basis $(X_1,X_2,X_3)$ by the ideal generated by $[X_i,[X_i,Y_j]]$ for all $i,j\in \{1,2,3\}$. 
Its weight structure is as follows
$$\begin{array}[c]{ccccc} 
 && \mathbf{2} &&\\
 12 &&&& 23\\
 && \underline{**} &&\\
 \mathbf{1} &&&& \mathbf{3}\\
 && 31 && 
\end{array}$$

(The sign $\underline{**}$ indicates that the degree 0 subspace $\mk{u}_0$ is 2-dimensional.) A computation (see Remark \ref{A2si}) shows that $H_2(\mk{u})_0=0$ and $\Kill(\mk{u})_0$ is 1-dimensional.  

This example was introduced by Abels as typically difficult because although $H_2(\mk{u})_0$ vanishes, $U$ is not the multiamalgam of its tame subgroups (see \S\ref{s_skpr}); as we show in this paper, this is reflected in the fact that $\Kill(\mk{u})_0$ is nonzero. This results in a significant additional difficulty in order to estimate the Dehn function, which is at most cubic by Theorem \ref{int_p}.

\subsubsection{Semidirect products with $\SL_3$}

We consider the groups $V(\K)\rtimes\SL_3(\K)$, where $V$ is one of the following three irreducible modules: $V=V_{10}$, the standard 3-dimensional module; $V_{20}=\textnormal{Sym}^2(V_{10})$, the 6-dimensional second power of $V_{10}$ and $V_{11}$, the 8-dimensional adjoint representation. (The notation is borrowed from \cite[Lecture 13]{FuH}, writing $V_{ij}$ instead of $\Gamma_{ij}$.)
These groups are not solvable but have a cocompact $\K$-triangulable subgroup, namely $V(\K)\rtimes T_3(\K)$, where $T_3(\K)$ is the group of lower triangular matrices. It is a simple verification that for an arbitrary nontrivial irreducible representation, the group $V(\K)\rtimes T_3(\K)$ admits exactly three principal weights, namely the two principal negative roots $r_{21}$, $r_{32}$ of $\SL_3$ itself, and the highest weight of the representation $V_{ab}$, which can be written as $aL_1-bL_3$, 
where $\pm L_1,L_2,L_3$ are the minimal vectors of the weight lattice, arranged as in the following picture.

$$\begin{array}[c]{cccccc} 
 &&& -L_3 &&\\
 & L_2 &&&& L_1 \\
 r_{21} &&& \centerdot &&\\
  & -L_1 &&&& -L_2\\
 &&& L_3 &&\\
 &  r_{31} &&&& r_{32}
\end{array}$$
The three principal weights always form a triangle with zero contained in its interior. In particular, $V(\K)\rtimes T_3(\K)$ does not satisfy the SOL obstruction. Let us write the weight diagram for each of the three examples (we mark some other points in the weight lattice as $\cdot$ for the sake of readability). 

More specifically, for $V_{10}$, the weights configuration looks like

$$\begin{array}[c]{cccccc} 
 & L_2 &&&& \mathbf{L_1} \\
 \mathbf{r_{21}} &&& \centerdot &&\\
  & . &&&& .\\
 . &&& L_3 &&\\
 &  r_{31} &&&& \mathbf{r_{32}}
\end{array}$$
We see that there are no quasi-opposite weights at all. This is accordingly a case for which Theorem \ref{ith_veryt} applies directly.
Thus $\K^3\rtimes\SL_3(\K)$ has a quadratic Dehn function (it can be checked to also hold for $\K^d\rtimes\SL_d(\K)$ for $d\ge 3$).

For $V_{20}$, writing $L_{ij}=L_i+L_j$, the weights are as follows

$$\begin{array}[c]{ccccccc} 
 2L_2 &&& L_{12} && & \mathbf{2L_1}\\
 & . &&&& . & \\
 \mathbf{r_{21}} &&& \centerdot &&& .\\
  & L_{23} &&&& L_{13} &\\
 . &&& . && & .\\
 &  r_{31} &&&& \mathbf{r_{32}} &\\
 . &&& 2L_3 && & .\\ 
\end{array}$$
Thus there are quasi-opposite weights but no opposite weights. Theorem \ref{mainss0}  implies that $V_{20}(\K)\rtimes \SL_3(\K)$ has a quadratic Dehn function.

For $V_{11}$, writing $L_{ij}=L_i-L_j$, the weights are as follows

$$\begin{array}[c]{ccccccc} 
 & L_{23} &&&& \mathbf{L_{13}} &\\
 .&&& . &&&.\\
 & . &&&& . &\\
 \mathbf{r_{21}},L_{21} &&& \underline{**} &&& L_{12}\\
  & . &&&& .&\\
 .&&& . &&&.\\
 &  r_{31},L_{31} &&&& \mathbf{r_{32}},L_{32}&
\end{array}$$
In this case, there are opposite weights, there is an invariant quadratic form in degree zero (akin to the Killing form), defined by $\phi(r_{ji},L_{ij})=1$ for all $i<j$ and all other products being zero, so $\Kill(\mk{u})_0\neq 0$. However, a simple computation shows that $H_2(\mk{u})_0=0$. So Theorem \ref{int_p} implies that $\mk{sl}_3(\K)\rtimes\SL_3(\K)$ has an at most cubic Dehn function.

\subsection{Description of the strategy}\label{introsection:proof}

Let us now describe the main ideas that underlie the proof of Theorems \ref{int_p} and \ref{imainss}. We then proceed to sketch, with more details, the proof of Theorem \ref{imainss}.

\noindent{\bf General picture.}
A central idea, already essential in Gromov's approach is to use non-positively curved subgroups (called {\em tame subgroups} in the sequel). It was previously used in Abels' work on compact presentability of $p$-adic groups \cite{A}. Abels considers a certain abstract group, obtained by amalgamating the tame subgroups over their intersections (more details are given in \ref{s_skpr}). Our combinatorial approach to the Dehn function allows us to take advantage of the consideration of this ``multiamalgam" $\hat{G}$. One similarly defines a multiamalgam of the tame Lie subalgebras, denoted by $\hat{\mk{g}}$. 

\noindent{\bf Strategy.} To simplify the discussion, let us assume that $G$ is standard solvable over a single nondiscrete locally compact field of characteristic zero $\K$, satisfying none of the SOL and 2-homological assumptions.  
Roughly speaking, the strategy is as follows. Ideally, one would like to
obtain $\hat{G}=G$; this turns out to be true under the additional assumption that $\Kill(\mk{u})_0$ vanishes. In general we obtain that $\hat{G}$ is a central extension of $G$, and describe generators of the central kernel.

Second, we need to be able to decompose any combinatorial loop into {\it boundedly many} loops corresponding to relations in the tame subgroups.
 It turns out that both steps are quite challenging. While Abels' work provides substantial material to tackle the first step, we had to introduce completely new ideas to solve the second one. 

\noindent{\bf The first step: giving a compact presentation for $G$.} 
Under the assumption that the group $G$ does not satisfy the SOL obstruction, we show, improving a theorem of Abels, that the multiamalgam is a central extension of $G$.
Thus we have $\hat{G}=\hat{U}\rtimes A$, where $\hat{U}$ is a central extension of $U$, whose central kernel is centralized by $A$. 
At first sight, it seems that the condition $H_2(\mk{u})_0=0$ should be enough to ensure that $\hat{G}=G$. However, it turns out that in general, $\hat{U}$ is a ``wild" central extension, in the sense that it does not carry any locally compact topology such that the projection onto $U$ is continuous. 
This strange phenomenon is easier to describe at the level of the Lie algebras. There, we have that $\hat{\mk{g}}=\hat{\mk{u}}\rtimes \mk{a}$, where $\hat{\mk{u}}$ is a central extension in degree $0$ of $\mk{u}$, seen as {\it Lie algebras over $\Q$}. Now, if in the last statement, we could replace Lie algebras over $\Q$ by Lie algebras over $\K$, then clearly $H_2(\mk{u})_0=0$ would imply that $\hat{\mk{u}}=\mk{u}$. This happens if and only if the natural morphism $H^{\Q}_2(\mk{u})_0\to H_2(\mk{u})_0$ is an isomorphism; we show in Section \ref{s_abels} that this happens if and only if the module $\Kill(\mk{u})_0$ (see \S\ref{subsectionIntro:computable}) vanishes. In fact, there are relatively simple examples, already pointed out by Abels where $\Kill(\mk{u})_0$ {\it does not} vanish (see the previous subsection). 
 As a consequence, even when none of the obstructions hold, we need to complete the presentation of $G$ with a family of so-called {\em welding relations}, which encode generators of the kernel of $\hat{U}\to U$. At the Lie algebra level, these relations encode $\K$-bilinearity of the Lie bracket. 

\noindent{\bf The second step: reduction to special relations.} 
The second main step is to reduce the estimation of area of arbitrary relations to that of relations of a special form (e.g., relations inside a tame subgroup). This idea amounts to Gromov and is instrumental in Young's approach for nilpotent groups \cite{Yo1} and for $\SL_{d\ge 5}(\Z)$ \cite{Yo2}; it is also used in \cite{BQ,CTAbels}. The main difference in this paper is that we have to perform such an approach without going into explicit calculations (which would be extremely complicated for an arbitrary standard solvable group, since the unipotent group $U$ is essentially arbitrary). Our trick to avoid calculations is to use a presentation of $U$ that is stable under ``extensions of scalars". Let us be more explicit, and write $U=\mathbb{U}(\K)$, so that $\mathbb{U}(\A)$ makes sense for any commutative $\K$-algebra $\A$. We actually provide a presentation, based on Abels' multiamalgam and welding relations, of $\mathbb{U}(\A)$ for any $\K$-algebra $\A$. When applying it to a suitable algebra $\A$ of functions of at most polynomial growth, we obtain area estimates (in the sketch of \S\ref{s_skpr}, we do this with the algebra $\A$ of sequences of elements of at most exponential growth). This is the core of our argument; it is performed in \S\ref{concl0} (in the setting of Theorem \ref{imainss}) and in \S\ref{concl} in general, relying on the general approach developed in \S\ref{s_awbcl}.
The presentation itself is established in Sections \ref{s_abels} and \ref{s:am}.

\noindent{\bf Last step: computation of the area of special relations.} For a standard solvable group not satisfying the SOL and 2-homological obstructions, these relations are of two types: those that are contained (as loops) in a tame subgroup and thus have an at most quadratic area; and the more mysterious welding relations. We show that welding relations have an at most cubic area. When the welding relations are superfluous, namely when $\Kill(\mk{u})_0$ vanishes, Theorem \ref{mainss0} asserts that $G$ then has an at most quadratic Dehn function. A study based on the asymptotic cone, in a paper in preparation by the authors, will show that conversely, in some cases where $\Kill(\mk{u})_0$ does not vanish, the Dehn function of $G$ grows strictly faster than quadratic. We mention this to enhance the important role played by the welding relations in the geometry of these groups. We suspect that they might be relevant as well in the study of the Dehn function of nilpotent Lie groups.

\subsection{Sketch of the proof of Theorem \ref{imainss}}\label{s_skpr}

In the following, we emphasize the main ideas, omitting some technical issues occurring in the detailed, rigorous proof completed in \S\ref{concl0}.

To begin with, let us explain in detail the notion of multiamalgam. Let $G$ be a group and $(H_i)$ a family of subgroups. The multiamalgam $\hat{G}$ of the $H_i$ is the quotient of the free product of the $H_i$ by the obvious amalgamation relations (identifying $H_i\cap H_j$ to its image in both $H_i$ and $H_j$ for all $i,j$). We say that $G$ is the multiamalgam of the $H_i$ if the canonical homomorphism $\hat{G}\to G$ is an isomorphism.

Basic example: suppose that $G=\bigoplus G_i$. Then $G$ is the multiamalgam of the family of subgroups $(G_i\oplus G_j)_{i\neq j}$. This is because $G$ is the quotient of the free product by the commutation relations, and each of these commutation relations lives within one pair of summands.

Consider now a standard solvable group $G=U\rtimes A$ satisfying the hypotheses of Theorem \ref{imainss}. We call tame subgroup of $G$ a subgroup of the form $V\rtimes A$, where $V$ is a closed $A$-invariant subgroup such that some element of $A$ acts on $V$ as a contraction.

Let $(U_i\rtimes A)_{1\le i\le\nu}$ be the family of maximal tame subgroups. There are finitely many. 

We now define a presentation of the multiamalgam $\hat{G}$.
Let $S_i$ be a compact symmetric neighborhood of 1 in $U_i$, $T$ a compact symmetric generating subset of $A$ with nonempty interior, and $S=\bigsqcup S_i\sqcup T$.
Let $R$ be the set of bounded length relators of the following four types:
\begin{itemize} 
\item bounded length relators of $(A,T)$, 
\item  relators of length 3 inside $S$, 
\item  relators of the form 
$ts't^{-1}s^{-1}$ for $s,s'\in S$, $t\in T$, 
\item (amalgamation relators) 
relators of the form $s^{-1}s'$ when $s\in S_i,s'\in S_j$ have the same image in $U_i\cap U_j$.
\end{itemize}

Fix $k\ge 1$ and let $\wp\in\{1\dots,\nu\}^k$ be a $k$-tuple of indices. Let $W_\wp(n)$ be the set of group words in the free group $F_S$ of the form $\prod_{j=1}^kt_js_{j}t_j^{-1}$ where $t_j\in F_T$ with $|t_j|\le n$ and $s_j\in S_{\wp_j}$. Let $\delta_\wp(n)$ be the supremum of areas (for the presentation $\langle S\mid R\rangle$) of those null-homotopic words in $W_\wp(n)$.

Gromov's trick (see \S\ref{grotrick}) says that there exists $k$ and $\wp$ such that $\delta_\wp(n)\preceq n^2$ implies that the Dehn function of $G$ is $\preceq n^2$. Showing that Gromov's trick can be performed requires some length estimates which take some work but we do not insist on this here. So we fix $k$ and are reduced to showing $\delta_\wp(n)\preceq n^2$.

For each $n$, we fix some $w(n)\in W_\wp(n)$; we decompose it as \[w(n)=\prod_{i=1}^kt_j(n)s_j(n)t_j(n)^{-1}.\] Let us proceed to show that the area of $w(n)$ is finite and quadratically bounded in terms of $n$. 
For each $n$, define $\sigma_j(n)$ as the element of $U_{\wp_j}$ represented by $t_j(n)s_j(n)t_j(n)^{-1}$. Then $\prod_{i=1}^k\sigma_j(n)$, viewed as an element of the free product $U_1\ast\dots\ast U_\nu$, has a trivial image in $U$.

Let $\K^{\exp}$ be the $\K$-algebra of sequences in $\K$ with at most exponential growth. Viewing $\sigma_j$ as a function of $n$, simple estimates show that $\sigma_j\in U_{\wp_j}(\K^{\exp})$.

Using the universal property of the multiamalgam, there is a canonical well-defined homomorphism from the multiamalgam $\widehat{U(\K^{\exp})}$ of the $U_i(\K^{\exp})$ to $U(\K^{\exp})$.
It is at this point that we use in an essential way the assumptions of Theorem \ref{imainss}, namely that $G$ does not satisfy the SOL and 2-homological obstructions, and $\Kill(\mk{u})_0=\{0\}$. These assumptions, thanks to the preliminary algebraic work of Sections \ref{s_abels} and \ref{s:am}, imply that this homomorphism $\widehat{U(\K^{\exp})}\to U(\K^{\exp})$ is an isomorphism (Corollary \ref{presam}).
For $q,r\in\{1,\dots,\nu\}$, denote by $\eta_{qr}$ the inclusion of $U_q\cap U_r$ into $U_q$. So, the multiamalgam of the $U_i(\K^{\exp})$ is the quotient of the free product by the normal subgroup generated by the relators $\eta_{qr}(x)\eta_{rq}(x)^{-1}$, when $(q,r)$ ranges over pairs and $x$ ranges over $U_q(\K^{\exp})\cap U_r(\K^{\exp})$.

Therefore, there exists a finite family $(\mu_m(n))_{1\le m\le p}$ of elements of $U(\K^{\exp})$, and elements $q_m,r_m\in\{1,\dots,\nu\}$, such that $\mu_m(n)\in U_{q_m}\cap U_{r_m}$, elements $(g_m(n))$ in the free product of the $U_i(\K^{\exp})$ such that 
\[\prod_{j=1}^k\sigma_j=\prod_{m=1}^pg_m\eta_{q_mr_m}(\mu_m)\eta_{r_mq_m}(\mu_m)^{-1}g_m^{-1}.\]
Thus for every $n$ we have, in $U_1\ast\dots\ast U_\nu$
\[\prod_{j=1}^k\sigma_j(n)=\prod_{m=1}^pg_m(n)\eta_{q_mr_m}(\mu_m)(n)\eta_{r_mq_m}(\mu_m)(n)^{-1}g_m(n)^{-1}.\]

Denote by $\widehat{g_m(n)}$ a choice of word of length $\le Cn$ representing $g_m(n)$ in $U_{q_m}\rtimes A$, in the generators $S_{q_m}\cup T$; this can be done because $g_m(n)$ has exponential size, so can be written as $t^{-cn}st^{cn}$ with $s\in S_{q_m}$ and $t\in T$, and $c$ independent of $n$. We also define simultaneously $\widehat{\eta_{q_mr_m}(\mu_m)(n)}$ and $\widehat{\eta_{r_mq_m}(\mu_m)(n)}$ to be of the form $t^{-cn}st^{cn}$ and $t^{-cn}s't^{cn}$, with $s\in S_{q_m}$, $s'\in S_{r_m}$ having the same image in $U_{q_m}\cap U_{r_m}$.

So the word
\[w'(n)=w(n)^{-1}\prod_{m=1}^p\widehat{g_m(n)}\,\widehat{\eta_{q_mr_m}(\mu_m)(n)}\,\widehat{\eta_{r_mq_m}(\mu_m)(n)^{-1}}\,\widehat{g_m(n)}^{-1}.\]
is null-homotopic in $(U_1\ast\dots\ast U_\nu)\rtimes A$ and has an at most linear length. 

Let us obtain a 
uniform bound on the area of $\widehat{\eta_{q_mr_m}(\mu_m)(n)}\widehat{\eta_{r_mq_m}(\mu_m)(n)^{-1}}$ for each $m$. We have \[\widehat{\eta_{q_mr_m}(\mu_m)(n)}\widehat{\eta_{r_mq_m}(\mu_m)(n)^{-1}}=t^{-cn}ss'^{-1}t^{cn};\]
since $s{s'}^{-1}$ is a relator, its area is equal to 1. Thus the area of $w(n)w'(n)$ is $\le p$ (independently of $n$).

This shows that we are reduced to proving that $w'(n)$ has an at most quadratic area. 
For convenience, we rewrite 
\[w'(n)=\prod_{m=1}^{\nu'}h_m(n),\]
where $h_m$ is a word in $T\cup S_{\wp'_m}$ for some $\nu'$ and $\nu'$-tuple $\wp'$, independent of $n$. Let us show that the area of a word $w'=\prod_{m=1}^{\nu'}h_m$ that is null-homotopic in $(U_1\ast\dots\ast U_{\nu'})\rtimes A$ is at most quadratic with respect to the total length $\sum |h_m|$ (with a constant depending only on $\wp'$). This is shown by induction on $\nu'$. Since $w'$ is null-homotopic, there exists either $m$ such that $h_m$ is null-homotopic, or there exists $m$ such that $\wp'_m=\wp'_{m+1}$. In the second case, this reduces to the case of $\nu'-1$. In the first case, it also reduces to the case of $\nu'-1$, at the cost of replacing $h_m$ with 1, which has at most quadratic cost since $U_{\wp'_m}\rtimes A$ has an at most quadratic Dehn function. So $w'(n)$ has an at most quadratic area. It follows that $w(n)$ has an at most quadratic area. Since this holds for every choice of $w$, this shows that $\delta_\wp(n)\preceq n^2$. Thus $G$ has an at most quadratic Dehn function.

This completes the sketch of proof of Theorem \ref{imainss}.
The first two conditions (absence of SOL and 2-homological obstructions) were necessary as otherwise, the Dehn function is at least exponential.
On the other hand, in the proof of Theorem \ref{int_p} we also have to proceed when the Killing module is not assumed zero in degree 0. In this case, $U$ can be described as a quotient of the multiamalgam by some {\em welding relators}, in a way that is inherited to $U(\K^{\exp})$ (see \ref{su_2t}). So, in the above proof, some further relations appear (in addition to the $\widehat{\eta_{q_mr_m}(\mu_m)(n)}\widehat{\eta_{r_mq_m}(\mu_m)(n)^{-1}}$), and we establish (\S\ref{awe}) that these further relations have at most cubic area with respect to some compact presentation (where the additional relators are some welding relations of bounded length).

\subsection{Introduction to the Lie algebra chapters}\label{sectionIntro:Liealgebra}

Although Sections \ref{s_abels} and \ref{s:am} can be viewed as technical sections when primarily interested in the results about Dehn functions, they should also be considered as self-contained contributions to the theory of graded Lie algebras. Recall that graded Lie algebras form a very rich theory of its own interest, see for instance \cite{Fu,Kac}. Let us therefore introduce these chapters independently.
In this context, the Lie algebras are over a given commutative ring $\A$, with no finiteness assumption. This generality is essential in our applications, since we have to consider (finite-dimensional) Lie algebras over an infinite product of fields. We actually consider Lie algebras graded in a given abelian group $\mathcal{W}$, written additively. 

\subsubsection{Universal central extensions}
Recall that a Lie algebra is perfect if $g=[g,g]$. It is classical that every perfect Lie algebra admits a universal central extension. We provide a graded version of this fact. Say that a graded Lie algebra is relatively perfect in degree zero if $\g_0\subset [\g,\g]$ (in other words, $H_1(\g)_0=\{0\}$).
In \S\ref{s_bu}, to any graded Lie algebra $\g$, we canonically associate another graded Lie algebra $\tilde{\g}$ along with a graded Lie algebra homomorphism $\tau:\tilde{\g}\to\g$, which has a central kernel, naturally isomorphic to the 0-component of the 2-homology module $H_2(\g)_0$. 

\begin{ithm}(Theorem \ref{uc0})
Let $\g$ be a graded Lie algebra. If $\g$ is relatively perfect in degree zero, then the morphism $\tilde{\g}\stackrel{\tau}\to\g$ is a graded central extension with kernel in degree zero, and is universal among such central extensions. 
\end{ithm}
An important feature of this result is  that it applies to graded Lie algebras that are far from perfect: indeed, in our case, the Lie algebras are even nilpotent.

\subsubsection{Restriction of scalars}

In \S\ref{sec_T}, we study the behavior of $H_2(\g)_0$ under restriction of scalars. We therefore consider a homomorphism $\A\to \B$ of commutative rings. If $\g$ is a Lie algebra over $B$, then it is also a Lie algebra over $A$ and therefore to avoid ambiguity we denote its 2-homology by $H_2^{\A}(\g)$ and $H_2^{\B}(\g)$ according to the choice of the ground ring. There is a canonical surjective $\A$-module homomorphism $H_2^{\A}(\g)\to H_2^{\B}(\g)$; we call its kernel the {\bf welding module} and denote it by $W_2^{\A,\B}(\g)$. It is a graded $\A$-module, and $W_2^{\A,\B}(\g)_0$ is also the kernel of the induced map $H_2^{\A}(\g)_0\to H_2^{\B}(\g)_0$.
 
These considerations led us to introduce the {\em Killing module}. Let $\g$ be a Lie algebra over the commutative ring $\B$. 
Consider the homomorphism $\mathcal{T}$ from $\g^{\ot 3}$ to the symmetric square $\g\cc\g$, defined by $$\mathcal{T}(u\ot v\ot w)=u\cc [v,w]-[u,w]\cc v$$  (all tensor products are over $\B$ here). The {\bf Killing module} is by definition the cokernel of $\mathcal{T}$. The terminology is motivated by the observation that for every $\B$-module $\mk{m}$, $\Hom_{\B}(\Kill^{\B}(\g),\mk{m})$ is in natural bijection with the module of invariant $\B$-bilinear forms $\g\times\g\to\mk{m}$. If $\g$ is graded, then $\Kill(\g)$ is canonically graded as well.

\begin{ithm}
Let $Q\subset K$ be fields of characteristic zero, such that $K$ has infinite transcendence degree over $Q$. Let $\g$ be a finite-dimensional graded Lie algebra over $K$, relatively perfect in degree zero (i.e., $H_1(\g)_0=\{0\}$). Then the following are equivalent:
\begin{enumerate}[(i)]
\item $W_2^{Q,K}(\g)_0=\{0\}$;
\item $W_2^{Q,\mathsf{R}}(\g\otimes_K \mathsf{R})_0=\{0\}$ for every commutative $K$-algebra $\mathsf{R}$;
\item\label{kill0} $\Kill^K(\g)_0=\{0\}$.
\setcounter{saveenum}{\value{enumi}}
\end{enumerate}

In particular, assuming moreover that $H_2(\g)_0=\{0\}$, these are also equivalent to:
\begin{enumerate}[(i)]
  \setcounter{enumi}{\value{saveenum}}
\item $H_2^{Q}(\g)_0=\{0\}$;
\item\label{h2r0} $H_2^{Q}(\g\otimes_K \mathsf{R})_0=\{0\}$ for every commutative $K$-algebra $\mathsf{R}$.
\end{enumerate}
\end{ithm} 
 
The interest of such a result is that (\ref{kill0}) appears as a checkable criterion for the vanishing of a complicated and typically infinite-dimensional object (the welding module). Let us point out that this result follows, in case $\g$ is defined over $Q$, from the results of Neeb and Wagemann \cite{NW}. In our application to Dehn functions (specifically, in the proof of Theorems \ref{mainss0} and \ref{mainss}), we make an essential use of the implication (\ref{kill0})$\Rightarrow$(\ref{h2r0}), where $Q=\Q$, $\K$ is a nondiscrete locally compact field, and $\mathsf{R}$ is a certain algebra of functions on $\K$. That (\ref{kill0}) implies the other properties actually does not rely on the specific hypotheses (restriction to fields, finite dimension), see Corollary \ref{noceb}. The converse, namely that the negation of (\ref{kill0}) implies the negation of the other properties, follows from Theorem \ref{cano}.

\subsubsection{Abels' multiamalgam}
Section \ref{s:am} is devoted to the study of Abels' multiamalgam $\hat{\g}$ and to its connection with the universal central extension $\tilde{\g}\to \g$. Here, we again consider arbitrary Lie algebras over a commutative ring $\mathsf{R}$, but we now assume that the abelian group $\mathcal{W}$ is a {\em real vector space}. Given a graded Lie algebra, a Lie subalgebra is called {\bf tame} if 0 does not belong to the convex hull of its weights. Abels' multiamalgam $\hat{\g}$ is the (graded) Lie algebra obtained by amalgamating all tame subalgebras of $\g$ along their intersections (see \ref{multi} for details); it comes with a natural graded Lie algebra homomorphism $\hat{\g}\to \g$. 
Abels defines a 2-tame nilpotent graded Lie algebra to be such that $0$ does not belong to the convex hull of any pair of principal weights (this condition is related to the condition that the SOL-obstruction is not satisfied). A more general notion of 2-tameness, for arbitrary $\mathcal{W}$-graded Lie algebras, is introduced in \ref{gla}. Although Abels works in a specific framework ($p$-adic fields, finite-dimensional nilpotent Lie algebras), his methods imply with minor changes the following result.

\begin{thmm}[essentially due to Abels, see Theorem \ref{kappaisa}]
If $\g$ is 2-tame, then $\hat{\g}\to\g$ is the universal central extension in degree 0. In other words, $\hat{\g}$ is canonically isomorphic to $\tilde{\g}$.
\end{thmm}

This means that in this case, $\hat{\g}$ is an excellent approximation of $\g$, the discrepancy being encoded by the central kernel $H_2(\g)_0$.

We actually need the translation of this result in the group-theoretic setting, which involves significant difficulties. Assume now that the ground ring $\mathsf{R}$ is a commutative algebra over the field $\Q$ of rationals. Recall that the Baker-Campbell-Hausdorff formula defines an equivalence of categories between nilpotent Lie algebras over $\Q$ and uniquely divisible nilpotent groups. Then $\g$ is the Lie algebra of a certain uniquely divisible nilpotent group $G$, and we can define the multiamalgam $\hat{G}$ of its tame subgroups, i.e.\ those subgroups corresponding to tame subalgebras of $\g$. Note that it does {\em not} follow from abstract nonsense that $\hat{G}$ is controlled in any way by $\g$, because $\hat{G}$ is defined in the category of groups and not of uniquely divisible nilpotent groups. In a technical tour de force, Abels \cite[\S 4.4]{A} managed to prove that $\hat{G}$ is nilpotent and asked whether the extension $\hat{G}\to G$ is central. The following theorem, which we need for our estimates of Dehn function, answers the latter question positively.

\begin{ithm}[Theorem \ref{oabels}]
Let $\g$ be a 2-tame nilpotent graded Lie algebra over a commutative $\Q$-algebra $\mathsf{R}$. If $\g$ is 2-tame, then $\hat{G}$ is nilpotent and uniquely divisible, and $\hat{G}\to G$ corresponds to $\hat{\g}\to\g$ under the equivalence of categories. In particular, the kernel of $\hat{G}\to G$ is central and canonically isomorphic to $H_2^{\Q}(\g)_0$.
\end{ithm}

\subsubsection{On the ``computational" aspect.} 
 Let $\g$ be a Lie algebra (over some commutative ring) graded in some abelian group. Theorems \ref{main} and \ref{imainss} involve the conditions $H_2(\mk{g})_0=\{0\}$ and $\Kill(\mk{g})_0=\{0\}$. Typically, in the cases of interest, $\mk{g}$ is a finite-dimensional nilpotent Lie algebra over a field, and these conditions mean the surjectivity of some explicit matrices. There are simple families of examples for which $\mk{g}$ has dimension $\simeq n^3$, while $\mk{g}_\td=\bigoplus_{\alpha\neq 0}\mk{g}_\alpha$ has dimension $\simeq n^2$, so that most of the contribution to $\dim(\mk{g})$ is in degree zero. Hence, when computing $H_2(\mk{g})_0$ and $\Kill(\mk{g})_0$, the main contribution comes from those terms involving $\mk{g}_0$. It turns out that under mild assumptions fulfilled in both theorems (implied by the failure of the SOL obstruction), we can get rid of these terms.

Namely, recalling that the homology module $H_2(\mk{g})$ is the cokernel of the map $d_3:\Lambda^3\mk{g}\to \Lambda^2\mk{g}$, we observe (see \S\ref{t2hm} for details)
that $d_3$ maps $(\Lambda^3\mk{g}_\td)_0\to \Lambda^2\mk{g}_\td$; let $H_2^\td(\g)_0$ be the cokernel of the latter module homomorphism; it is endowed with a canonical map $H_2^\td(\g)_0\to H_2(\g)_0$.
 
Similarly, $\Kill(\mk{u})$ is the cokernel of the map $\mathcal{T}:\mk{g}^{\ot 3}\to\mathcal{S}^2\mk{g}$ (the second symmetric power) mapping 
$x\ot y\ot z$ to $[x,y]\cc z-x\cc [y,z]$, which maps $((\mk{g}_\td)^{\ot 3})_0$ into $(\mathcal{S}^2\mk{g}_\td)_0$; let $\Kill^\td(\mk{g})_0$ be the cokernel of the latter; it admits a canonical module homomorphism into $\Kill(\mk{g})_0$.

The following theorem is contained in Theorem \ref{thm:h2tame} for $H_2$ and in Theorem \ref{killt} for the Killing module.

\begin{ithm}\label{h2killtame} Let $\g$ be a nilpotent Lie algebra (over any commutative ring) graded in some real vector space. Suppose that for every $\alpha$ we have $\g_0=\sum_{\beta\notin\{0,\alpha,-\alpha\}}[\g_\beta,\g_{-\beta}]$.
 Then the natural 
module homomorphisms $H_2^\td(\g)_0\to H_2(\g)_0$ and $\Kill^\td(\mk{g})_0\to\Kill(\mk{g})_0$ are isomorphisms.
\end{ithm}

The assumption is satisfied if $\g$ is nilpotent and 2-tame, that is, $\g/[\g,\g]$ has no pair $\{\alpha,\beta\}$ of weights such that $0$ belongs to the segment $[\alpha,\beta]$ \cite[Lemma 4.3.1(c)]{A}.

\subsection{Guidelines}

\begin{itemize}
\item The reader interested in the proof of Theorem \ref{thmi_sol} can directly go to Section \ref{s_not}, which is essentially independent.
\item The reader interested in the proof of Theorem \ref{thmi_hom} can directly go to Section \ref{s_cent}, and refer when necessary to the preliminaries of Sections \ref{s_p}, \ref{adumn}, and, more scarcely, Section \ref{s_abels}.
\item The reader interested in the proof of the polynomial upper bounds will find the first important steps in Section \ref{s_metrd} and the weight definitions in Section \ref{s_p}, and then can proceed to Sections \ref{s_spec}, \ref{s_ea}, and \ref{s_main}, and refer when necessary to all the preceding sections.
\item The reader interested specifically in the algebraic results, as described in \S\ref{sectionIntro:Liealgebra}, can directly go to Sections \ref{s_abels} and \ref{s:am}, and refer when necessary to the preliminary Section \ref{adumn}.
\end{itemize}

Although being of preliminary nature here, Sections \ref{s_deh} (introducing the Dehn function) and \ref{s_p} (weights and tameness conditions, Cartan grading) can be read for their own interest.

\subsection{Acknowledgements}

We are grateful to Christophe Pittet for several fruitful discussions. We thank Cornelia Dru\c tu for discussions and pointing out useful references. We thank Pierre de la Harpe for many useful corrections. We are grateful to the referee for a thorough work and various suggestions improving the presentation of the paper.


\renewcommand{\thesubsection}{{\thesection.\Alph{subsection}}}

\section{Dehn function: preliminaries and first reductions}\label{s_deh}

This section contains general preliminaries on the Dehn function. In particular, we state Gromov's trick in \S\ref{grotrick}, essentially following arguments from \cite{BQ}.

\subsection{Asymptotic comparison}\label{asco}

\begin{defn}\label{defas}If $f,g$ are real functions defined on any set where $\infty$ makes sense (e.g., in a locally compact group, it usually refers to convergence to the point at infinity in the 1-point compactification), we say that $f$ is {\bf asymptotically bounded} by $g$ and write $f\preceq g$ if there exists a constant $C\ge 1$ such that for all $x$ close enough to $\infty$ we have
$$f(x)\le Cg(x)+C;$$
if $f\preceq g\preceq f$ we write $f\simeq g$ and say that $f$ and $g$ have the same {\bf asymptotic behavior}, or the same $\simeq$-asymptotic behavior.
\end{defn}

\begin{defn}\label{defbas}If $f,g$ are non-decreasing non-negative real functions defined on $\R_+$ or $\N$, we say that $f$ is bi-asymptotically bounded by $g$ and write $f\preccurlyeq g$ if there exists a constant $C\ge 1$ such that for all $x$
$$f(x)\le Cg(Cx)+C;$$
if $f\preccurlyeq g\preccurlyeq f$ we write $f\approx g$, and say that $f$ and $g$ have the same bi-asymptotic behavior, or the same $\approx$-{\bf asymptotic behavior}.
\end{defn}

Besides the setting, the essential difference between $\simeq$-asymptotic and $\approx$-asymptotic behavior is the constant at the level of the source set. For instance, $2^n\approx 3^n$ but $2^n\simeq\!\!\!\!\!/\;\, 3^n$.

\subsection{Definition of Dehn function}\label{dede}

If $G$ is a group and $S$ any subset, we denote by $|g|_S$ the (possibly infinite) word length of $g\in G$ with respect to $S$; the length $|\cdot|_S$ takes finite values on the subgroup generated by $S$.

If $H$ is a group and $R\subset H$, we define the {\bf area} of an element $w\in H$ as the (possibly infinite) word length of $w$ with respect to the union $\bigcup_{h\in H}hRh^{-1}$; we denote it by $\are_R(w)$. It takes finite values on the normal subgroup generated by $R$.

Now let $S$ be an abstract set and $F_S$ the free group over $S$, and $\pi:F_S\to G$ a surjective homomorphism. Its kernel is denoted by $K$; elements in $K$ are called {\bf relations}. If a subset $R$ of $K$ is given, the elements in the normal subgroup generated by $R$ are called {\bf null-homotopic}. Thus any null-homotopic word is a relation, and conversely, if $R$ generates $K$ as a normal subgroup (then $R$ is called a {\bf system of relators}), then relations are null-homotopic. By a common abuse of terminology, elements in a system of relators are called relators.

We define the Dehn function
$$\delta_{S,\pi,R}(n)=\max(\lfloor n/2\rfloor,\sup\{\are_R(w)\;:\;w\in K,\;|w|_S\le n\}).$$ 
(The term $n/2$ is not serious, and only avoids some pathologies. Intuitively, it corresponds to the idea that the area of a word $s^ns^{-n}$, which is zero by definition, should account for at least $n$.)
We like to think of $\delta_{S,\pi,R}$ as the Dehn function of $G$, but in this general setting, this function as well as its asymptotic behavior can depend on the choice of $S$, $\pi$ and $R$. It can also take infinite values. Note that often $S$ is a given subset of $G$ and since $\pi$ is obviously chosen as the unique homomorphism extending the identity of $S$, we write the Dehn function as $\delta_{S,R}$.

Let now $G$ be a compactly generated LC-group (LC-group means locally compact group), and $S$ a compact generating symmetric subset whose interior contains 1. View $S$ as an abstract set and consider the surjective homomorphism $F_S\to G$ which is the identity on $S$, and $K$ its kernel. Let $K(d)$ be the intersection of $K$ with the $d$-ball in $(F_S
,|\cdot|_S)$. We say that $G$ is {\bf compactly presented} if for some $d\ge 0$, the subset $K(d)$ generates $K$ as a normal subgroup. This does not depend on the choice of $S$ (but the minimal value of $d$ can depend). If so, the function $\delta_{S,K(d)}$ takes finite values (see Lemma \ref{tafi} below), and its $\approx$-asymptotic behavior does not depend on the choice of $S$ and $d$.

 The latter is then called the {\bf Dehn function} of $G$. If $G$ is not compactly presented, we say by convention that the Dehn function is infinite. 
For instance, when we say that a compactly generated LC-group $G$ has a Dehn function $\succcurlyeq f(n)$, we allow the possibility that $G$ is not compactly presented, i.e. its Dehn function is (eventually) infinite.

\begin{prop}\label{qiinv}
The $\approx$-behavior of the Dehn function is a quasi-isometry invariant of the compactly generated locally compact group.\end{prop}
\begin{proof}[On the proof]
This follows from a more general statement for arbitrary connected graphs. In this case, the Dehn function can be defined as soon that the graph is large-scale simply connected, i.e., can be made simply connected by adding 2-cells with a bounded number of edges, thus defining a notion of area for arbitrary loops.
The quasi-isometry invariance of the asymptotic behavior of the Dehn function for arbitrary connected graph is similar to the usual one showing that the Dehn function is a quasi-isometry invariant among finitely generated groups, see \cite{Al} or \cite[Theorem 26]{BaMS}. 
\end{proof}

\begin{lem}\label{tafi}
As above, let $G$ be an LC-group with a compact symmetric generating subset $S$, and $K$ the kernel of the canonical homomorphism $F_S\to G$. Assume that, for some $d\ge 4$, the intersection $K(d)$ of the $d$-ball in $F_S$ with $K$ generates $K$ as a normal subgroup. Then $\delta_{S,K(d)}(n)$ is finite for all $n$.
\end{lem}
\begin{proof}
Fix $n$. Denote by $S^n$ the product of $n$ copies of $S$, with the product topology (rather than the $n$-ball in $F_S$). Consider the product map $p:S^n\to F_S$, and let $(S^n)_0$ be the inverse image of $K(n)$ in $S^n$; thus $(S^n)_0$ is a closed subset of $S^n$, hence is compact. By assumption, the area function $\are_{K(d)}\circ p$ takes finite values on $(S^n)_0$. We have to show that it is bounded on $(S^n)_0$; thus it is enough to show that this function is locally bounded. Indeed, given $s=(s_1,\dots,s_n)\in (S^n)_0$, there exists a neighborhood $V$ of $s$ in $(S^n)_0$ such that for all $s'=(s'_1,\dots,s'_n)\in V$ and all $1\le k\le n$, we have the element $(s_1\dots s_k)^{-1}(s'_1\dots s'_k)$ represents an element $t_k\in S$. Note that $t_n=1$. 

Fix $s'\in V$, and define $t_1,\dots,t_n\in S$ as above. Set $t_0=1$, and define, for $0\le k\le n$ 
\[s_{(k)}=(s_1,\dots,s_k,t_k,s'_{k+1},\dots,s'_n)\in (S^{n+1})_0\]
Then $p(s_{(0)})=p(s')$, $p(s_{(n)})=p(s)$, and for all $0\le k\le n-1$ 
\begin{align*}\are_{K(d)}(p(s_{(k-1)}^{-1}s_{(k)}))= & \are_{K(d)}\big((s'_{k+2}\dots s'_n)^{-1}{s'_{k+1}}^{-1}t_k^{-1}s_{k+1}t_{k+1}(s'_{k+2}\dots s'_n)\big)
\\ =& \are_{K(d)}({s'_{k+1}}^{-1}t_k^{-1}s_{k+1}t_{k+1})\le 1;\end{align*}
it follows that $\are_{K(d)}(p(s')^{-1}p(s))\le n$. Therefore \[\are_{K(d)}(p(s'))\le \are_{K(d)}(p(s))+n,\qquad\forall s'\in V,\] proving the local boundedness.
\end{proof}

\subsection{Dehn vs Riemannian Dehn}

Let $X$ is a Riemannian manifold. If $\gamma$ is a Lipschitz loop in $X$, define the area of $\gamma$ as the infimum of areas of Lipschitz disc fillings of $\gamma$. Define the {\em filling function} of $X$ as the function mapping $r$ to the supremum $F(r)$ of areas of all Lipschitz disc fillings of Lipschitz loops of length $\le r$. 

\begin{prop}\label{eqde}Let $G$ be a locally compact group with a proper cocompact isometric action on a simply connected Riemannian manifold $X$. Then $G$ is compactly presented and the Dehn function of $G$ satisfies
$$\delta(n)\approx \max(F(n),n).$$
\end{prop}

(The $\max(\cdot,n)$ is essentially technical: unless $X$ has dimension $\le 1$ or is compact, it can be shown that $F(r)$ grows at least linearly.)

The proof is given, for $G$ discrete, by Bridson \cite[Section 5]{Bri}. Here we only provide a complete proof of the easier inequality $F(n)\preccurlyeq\delta(n)$, because the proof in \cite{Bri} makes a serious use of the assumption that $G$ is finitely presented. For the converse inequality $\delta(n)\preccurlyeq\max(F(n),n)$, the (highly technical) proof given in \cite{Bri} uses general arguments of filling in Riemannian manifolds and a general cellulation lemma, and the remainder of the proof carries over our broader context, so we shall only provide a brief sketch of the converse.

\begin{lem}\label{bdlo}Let $X$ be a simply connected Riemannian manifold with $\Isom(X)$ acting cocompactly on $X$. Then the Riemannian area of loops of bounded length is bounded.
\end{lem}
\begin{proof}
Write $G=\Isom(X)$. By cocompactness, there exists $r_0$ such that for every $x\in X$, the exponential is $(1/2,2)$-bilipschitz from the $r_0$-ball in $T_xX$ to $X$. In particular, given a loop of length $\le r_0$, it passes through some point $x$; its inverse image by the exponential at $x$ has length $\le 2r_0$ and can be filled by a disc of area $\le \pi r_0^2$ in $T_xX$, and its image by the exponential is a filling of area $\le 4\pi r_0^2$ in $X$.

Now fix a positive integer $m_0\ge 6/r_0$. Let $D$ be an $m_0^{-1}$-metric lattice in $X$, that is, a maximal subset in $X$ in which any two distinct points have distance $\ge m_0^{-1}$. Note that every point in $X$ is at distance $\le m_0^{-1}$ to some element in $D$. Also note that, by properness of $X$, the intersection of $D$ with any bounded subset is finite. For any $x,y\in D$ with $d(x,y)\le 3m_0^{-1}$, fix a geodesic path $S(x,y)$ between $x$ and $y$. Consider a compact subset $\Omega$ such that $G\Omega=X$.

Consider a 1-Lipschitz loop $f:[0,k]\to X$ and let us bound the Riemannian area of $f$ by some constant depending only of $k$. By homogeneity, we can suppose that $f(0)\in\Omega$. For all $0\le n\le km_0$, let $x_n$ be a point in $D$ that is $m_0^{-1}$-close to $f(nm_0^{-1})$. Fix geodesic paths joining $x_n$ and $f(nm_0^{-1})$. So there is a homotopy from $f$ to the concatenation of the $S(x_n,x_{n+1})$ through $km_0$ ``squares" of perimeter at most $6m_0^{-1}\le r_0$. By the above, each of these $km_0$ squares can be filled with area $\le\pi 4r_0^2$. The remaining loop is a concatenation of $k$ segments of the form $S(x_n,x_{n+1})$ with $x_n\in D$ (with $n$ taken modulo $km_0$). For given $k$, there are only finitely many such loops, since all the points $x_n$ are at distance $\le k/2+m_0^{-1}$ to $\Omega$. Since $X$ is simply connected, each of these loops has finite area. So the remaining loop has area $\le a_k$ for some $a_k<\infty$. So we found a Lipschitz filling of the original loop, of area $\le a_k+km_0\pi r_0^2$. 
\end{proof}

\begin{proof}[Partial proof of Proposition \ref{eqde}]

Fix a compact symmetric generating set $S$ in $G$ and a set $R$ of relators. Fix $x_0\in X$. Set $r=\sup_{s\in S}d(x_0,sx_0)$. If $s\in S$, fix a $r$-Lipschitz map $j_s:[0,1]\to X$ mapping 0 to $x_0$ and $1$ to $sx_0$. If $w=s_1\dots s_k$, define $j_w:[0,k]\to X$ as follows: if $0\le\ell\le k-1$ and $0\le t\le 1$, $j_w(\ell+t)=s_1\dots s_\ell j_{s_{\ell+1}}(t)$. It is doubly defined for an integer, but both definitions coincide. So $j_w$ is $r$-Lipschitz. If $w$ represents 1 in $G$, then $j_w(0)=j_w(k)$. 

If $w$ is a word in the letters in $S$ and represents the identity, let $A(w)$ be the area of the loop $j_w$.

By Lemma \ref{bdlo}, $A(w)$ is bounded when $w$ is bounded. Also, it is clear that $A(gwg^{-1})=A(w)$ for all group words $g$. This shows that there exists a constant $C>0$, namely $C=\sup_{r\in R}A(r)$, such that $A(w)\le C\are(w)$ for some constant $C$.

For some $r_0$, every point in $X$ is at distance $\le r_0$ of a point in $Gx_0$. 
Consider a loop of length $k$ in $X$, given by a 1-Lipschitz function $u:[0,k]\to X$. For every $n$ (modulo $k$), let $g_nx_0$ be a point in $Gx_0$ with $d(g_nx_0,u(n))\le r_0$. We have $d(g_nx_0,g_{n+1}x_0)\le 2r_0+1$. By properness, there exists $N$ (depending only on $r_0$) such that $g_n^{-1}g_{n+1}\in S^N$. If $\sigma_n$ is a word of length $N$ representing $g_n^{-1}g_{n+1}$, and $\sigma=\sigma_0\dots\sigma_{k-1}$, then we pass from $u$ to $j_\sigma$ by a homotopy consisting of $k$ squares of perimeter $\le 4r_0+2$. By Lemma \ref{bdlo}, there is a bound $M_0$ on the Riemannian area of such squares. So the Riemannian area of $u$ is bounded by $kM_0+A(j_\sigma)\le kM_0+C\are(\sigma)\le kM_0+\delta_{S,R}(Nk)$. This shows that $F(k)\le kM_0+\delta_{S,R}(Nk)$.

For the (more involved) converse inequality, we only give the following sketch: let $\rho>0$ be such that each point in $X$ is at distance $<\rho/8$ to $Gx_0$, and assume in addition that $\rho>\rho_\kappa=\frac{\pi}{2\sqrt{\kappa}}$, where $\kappa$ is an upper bound on the sectional curvature of $X$. Let $S$ be the set of elements in $G$ such that $d(gx_0,x_0)\le\rho$ and $R$ the set of words in $F_S$, of length at most 12 and representing 1 in $G$. Then the proof in \cite[\S 5.2]{Bri} shows that $\langle S\mid R\rangle$ is a presentation of $G$ with Dehn function $\delta(n)$ bounded above by $4\lambda_\kappa(F(\rho n)+\rho n+1)$, where $\lambda_\kappa=1/\min(4\sqrt{\kappa}/\pi,\alpha(r,\kappa))$ and $\alpha(r,\kappa)$ is the area of a disc of radius $r$ in the standard plane or sphere of constant curvature $\kappa$.
\end{proof}

\subsection{Two combinatorial lemmas on the Dehn function}

This subsection contains several general lemmas about the Dehn function, which will be used at some precise parts of the paper.

\subsubsection{Free products}

Recall that a function $\R_{\ge 0}\to\R$ is {\em superadditive} if it satisfies $f(x+y)\ge f(x)+f(y)$ for all $x,y$. For instance, the function $r\mapsto r^\alpha$ is superadditive for every $\alpha\ge 1$.

\begin{lem}\label{dehnfree}
Let $f$ be a superadditive function. Let $(G_i)$ be a finite family of (abstract) groups, each with a presentation $\langle S_i\mid R_i\rangle$, with Dehn function $\le f$. Then the free product $H$ of the $G_i$ has Dehn function $\delta\le f$ with respect to the presentation $\langle\bigsqcup S_i\mid\bigsqcup R_i\rangle$.
\end{lem}

\begin{proof}
Let $w=s_1\dots s_n$ be a null-homotopic word with $n\ge 1$. Because $H$ is a free product, there exists $i$ and $1\le j\le j+k-1\le n$ such that every letter $s_\ell$ for $j\le\ell\le j+k-1$ is in $S_i$ and $s_j\dots s_{j+k-1}$ represents the identity in $G_i$. So, with the corresponding cost, which is $\le f(k)$, we can simplify $w$ to the null-homotopic word $s_1\dots s_{j-1}s_k\dots s_n$. Thus
$\delta(n)\le f(k)+\delta(n-k)$ (with $1\le k\le n$). Using the property that $f$ is superadditive, we can thus prove by induction on $n$ that $\delta(n)\le f(n)$ for all $n$. (This argument is used in \cite{GS} for finitely generated groups.)
\end{proof}

\subsubsection{Conjugating elements}

\begin{lem}\label{vk}
Let $\langle S\mid R\rangle$ be a group presentation, and $r$ a bound on the length of the words in $R$. Then for every null-homotopic $w\in F_S$, with length $n$ and area $\alpha$, we can write, in $F_S$, $w=\prod_{i=1}^\alpha g_ir_ig_i^{-1}$ with $r_i\in R^{\pm 1}\cup\{1\}$ and $g_i\in F_S$, with the additional condition $|g_i|_S\le n+r\alpha$.
\end{lem}
\begin{proof}
We start with the following claim: consider a connected polygonal planar complex, with 
$n$ vertices on the boundary (including multiplicities); suppose that the number of polygons is $\alpha$ and each has at most $r$ edges. Fix a base-vertex. Then the distance in the 
one-skeleton of the base-vertex to any other vertex is $\le n+r\alpha$. 
Indeed, pick an injective path: it meets at most $n$ boundary vertices. 
Other vertices belong to some face, but each face can be met at most $r$ 
times. So the claim is proved.

Now a van Kampen diagram for a null-homotopic word of size $n$ 
and area $\alpha$ with relators of length $\le r$ satisfies these 
assumptions, the distance from the identity to some vertex corresponds 
to the length to the conjugating element that comes into play. Thus the $g_i$ can be chosen with $|g_i|_S\le n+r\alpha$.
\end{proof}

\subsection{Gromov's trick}\label{grotrick}

\begin{defn}\label{d_comb}
Let $F_S$ be the free group over an abstract set $S$, let $G$ be an arbitrary group, and let $\pi:F_S\to G$ be a surjective homomorphism. We call (linear) {\bf combing} of $(G,\pi)$ (or, informally, of $G$ if $\pi$ is implicit), a subset $\mathcal{F}\subset F_S$ such that $1\in\mathcal{F}$ and some constant $C\ge 1$, we have the property that $\pi(S)^n\subset \pi(S^{Cn}\cap\mathcal{F})$ for all $n\ge 0$.
\end{defn}

\begin{ex}\label{combab}
Let $T$ be any generating subset of a finitely generated abelian group $A$. Suppose that $\{t_1,\dots,t_\ell\}\subset T$ is also a generating subset. Then the set of words $\{\prod_{j=1}^\ell t_j^{m_j}\}$ where $(m_j)$ ranges over $\Z^\ell$, is a combing of $F_T\to A$. If every element of $T$ is equal or inverse to some $t_i$, the constant $C$ can be taken equal to 1. 
\end{ex}

The following is, up to minor changes explained below, established in \cite[Proposition 4.3]{BQ}:

\begin{thm}[Gromov's trick]\label{gromtri}
Let $f:\R_{>0}\to\R_{>0}$ be a function such that for some $\alpha>1$, the function $r\mapsto f(r)/r^\alpha$ is non-decreasing.

Let $S$ be an abstract set, let $G$ be an arbitrary group, and let $\pi:F_S\to G$ be a surjective homomorphism. Let $R\subset F_S$ be a subset contained in $\Ker(\pi)$. Consider a combing $\mathcal{F}\subset F_S$. 

Assume that for all $n$, the area with respect to $R$ of any $w\in\Ker(\pi)$ of the form $w=w_1w_2w_3$ with $\max_i|w_i|\le n$ with $w_i\in\mathcal{F}$ is $\le f(n)$. Then $\langle S\mid R\rangle$ is a presentation of $G$ (i.e., $R$ generates $\Ker(\pi)$ as a normal subgroup) and for some constants $C',C''>0$, the Dehn function of $G$ with respect to $R$ is $\le C'f(C''n)$.
\end{thm}

\begin{rem}
The statement in \cite[Proposition 4.3]{BQ} is awkward because it purportedly considers a locally compact group without specifying a presentation and gives a conclusion on its Dehn function as a function (and not an asymptotic type of function). Actually the local compactness assumption is irrelevant and the correct statement is the one given here, the proof given in \cite{BQ} applying without modification. The setting is just that of a group presentation; it is even not necessary to assume that the relators have bounded length.

Second, the assumption was for words $w$ of length $\le n$, while here we assume that all $w_i$ have length $\le n$; this is more natural and handy, and requires no change in the proof.

Third, \cite[Proposition 4.3]{BQ} is stated with the function $n\mapsto n^\alpha$ for $\alpha>1$, but it relies on \cite[Lemma 4.1]{BQ}, which allows much more general functions. At the time \cite{BQ} it was not obvious to the author how to state a general result without a complicated technical assumption, but it turns out that the non-decreasing assumption of $f$ is very general: we are not even aware of any non-hyperbolic compactly presented locally compact group whose Dehn function is not $\approx$-equivalent to a function $f$ such that, say, $f(r)/r^{3/2}$ is non-decreasing.
\end{rem}

\begin{proof}[Proof of Theorem \ref{gromtri}]
We first claim that for any $k\ge 3$, any $n$ and any $w=w_1\dots w_k\in\Ker(\pi)$ with $w_i\in\mathcal{F}$ and $\sup_i|w_i|\le n$, we have $\are_R(w)\le f_k(n)=kf(Ckn/2)$. Indeed, for every $i$ the length of $\pi(w_1\dots w_i)$ is $\le n/2$, so we can find $u_i\in\mathcal{F}$ with $\pi(u_i)=\pi(w_1\dots w_i)$ and $|u_i|\le Cn/2$, with $u_0=u_k=1$. Then 
\[w=\prod_{i=1}^{k}w_i=\prod_{i=1}^{k}u_{i-1}w_iu_i^{-1};\]
since $\pi(u_{i-1}w_iu_i^{-1})=1$, by assumption we have $\are_R(u_{i-1}w_iu_i^{-1})\le f(Ckn/2)$.  

Then, denoting $\delta=\delta_{S,R}$, by \cite[Lemma 4.2]{BQ}, we have, for all $n\ge k$,
\[\delta(n)\le k\delta(2(C+1)n/k)+f_k(C\lceil n/k\rceil),\qquad \forall n\ge k\]
(where we have $f_k(C\lceil n/k\rceil)$ rather than $f_k(Ck\lceil n/k\rceil)$ because we consider here $\sup_i|w_i|$ rather than $\sum_i|w_i|$).
So, using that $\lceil n/k\rceil\le 2n/k$, we have
\[\delta(n)\le k\delta(2(C+1)n/k)+kf(C^2n),\qquad \forall n\ge k.\]

If we define $u_k(r)=kf(C^2r)$, then we claim that $zu_k(y)/u_k(yz)$ tends to zero uniformly in $y\ge 1$, $k\ge 1$, when $z\to\infty$. Indeed, writing $f(r)=r^\alpha v(r)$ with $v$ non-decreasing, we have
\[\frac{zu_k(y)}{u_k(yz)}=\frac{z(C^2y)^\alpha v(C^2y)}{(C^2yz)^\alpha v(C^2yz)}=\frac{v(C^2y)}{z^{\alpha-1} v(C^2yz)}\le \frac{1}{z^{\alpha-1}}.\]
This uniform convergence to zero allows to use \cite[Lemma 4.1]{BQ}, which yields that for some $r_0$, some $k$, and some constant $c>0$, we have $f(r)\le cu_k(r)$ for all $r\ge r_0$. Thus $\delta(r)\le ckf(C^2r)$ for all $r\ge r_0$.

(Here the value of the constant $C'=ck$ is hidden behind the estimates from the proof \cite[Lemma 4.1]{BQ} but the constant $C''=C^2$ is explicit, $C$ being given in Definition \ref{d_comb}.)
\end{proof}

\section{Geometry of Lie groups: Metric reductions and exponential upper bounds}\label{s_metrd}

In the study of upper bounds for the Dehn function of Lie groups, there are two important reduction steps, described in this section. The first (\S\ref{redtri}) allows to reduce the problem to triangulable groups, with no cost on the Dehn function. The second (\S\ref{rssg}) passes to a generalized standard solvable group, with a polynomial cost on the Dehn function. 
In between (\S\ref{exde}), we describe a general retraction argument in Lie groups, proving in particular that connected Lie groups have an at most exponential Dehn function.

\subsection{Reduction to triangulable groups}\label{redtri}

The following lemma is essentially borrowed from \cite{CorTop}.

\begin{lem}\label{comtri}
For every connected Lie group $G$, there exists a series of proper homomorphisms with cocompact images $G\leftarrow G_1\to G_2\leftarrow G_3$, with $G_3$ triangulable. In particular, $G_3$ is quasi-isometric to $G$ and the Dehn functions of $G_3$ and $G$ are $\approx$-equivalent.
\end{lem}
\begin{proof}
Let $N$ be the nilradical of $G$. By \cite[Lemma 6.7]{CorTop}, there exists a closed cocompact solvable subgroup $G_1$ of $G$ containing $N$, and a cocompact embedding $G_1\subset H_2$ with $H_2$ a connected solvable Lie group, such that $H_2$ is generated by $G_1$ and its center $Z(H_2)$. In particular, every normal subgroup of $G_1$ is normal in $H_2$. Let $W$ be the largest compact normal subgroup of $H_2$ and define $L_2=H_2/W$, so $L_2$ is a connected solvable Lie group whose derived subgroup has a simply connected closure. By \cite[Lemma 2.4]{CorTop}, there are cocompact embeddings $L_2\subset G_2\supset G_3$, with $G_2$ and $G_3$ connected Lie groups, and $G_3$ triangulable.

The latter statement follows from the fact that continuous proper homomorphisms with cocompact images are quasi-isometries, along with Proposition \ref{qiinv}, which says that the $\approx$-behavior of the Dehn function is a quasi-isometry invariant.
\end{proof}

\subsection{At most exponential Dehn function}\label{exde}

\begin{thm}[Gromov]\label{t_lieexp}
Every connected Lie group has an at most exponential Dehn function.
\end{thm}

As usual, Gromov gives a rough sketch of proof \cite[Corollary 3.$\textnormal{F}'_5$]{Gro}, but we are not aware of a complete written proof.

The basic idea is that the law can be described using polynomials and exponentials, but this is not enough because iterated exponentials would be an issue. A workable lemma is the following.

Recall that an {\bf exponential polynomial} $\R^{n}\to\R^{n'}$ is a real-valued function each of whose $n'$ coordinates is a polynomial in the coordinates $x_1,\dots,x_{n}$ and in exponentials $e^{\lambda_jx_i}$ for suitable complex scalar numbers $\lambda_j$. 

\begin{lem}\label{lawscslie}
Every simply connected solvable Lie group is isomorphic to a Lie group described as $(\R^m\times\R^\ell,*)$ with the law of the form
$$(u_1,v_1)*(u_2,v_2)=(P(v_1,v_2,u_1,u_2),v_1+v_2); \; (u,v)^{-1}=(T(v,u),-v),$$ 
where $P$ is an exponential polynomial depending polynomially on the variable $(u_1,u_2)$ (that is, not involving exponentials in the coordinates of $u_1$ and $u_2$), and where $T$ is an exponential polynomial depending polynomially on the variable $u$.
\end{lem}
For instance, the law of $\SOL_\lambda$ can be described as
$$(x_1,y_1,v_1)\ast (x_2,y_2,v_2)=(e^{v_2}x_1+x_2,e^{-\lambda v_2}y_1+y_2,v_1+v_2)$$
(here $(\ell,m)=(2,1)$).

\begin{proof}
Let $N$ be the derived subgroup of $G$ and $H$ a Cartan subgroup; both are a simply connected nilpotent Lie group. By \cite[Chap.~7,\S 1,2,3]{Bou} (see also Proposition \ref{cartan}), we have $NH=G$. Consider a supplement subspace of the Lie algebra of $M=N\cap H$ in the Lie algebra of $H$, and define $V$ as its exponential (it is usually not a subgroup). Set $m=\dim(V)$ and $\ell=\dim(N)$. We identify $V$ and $N$ to $\R^m$ and $\R^\ell$ through the exponential and a choice of bases of $\log(V)$ and $\log(N)$.

The product map $M\times V\to H$ is an algebraic isomorphism of varieties, which yields a decomposition $H=M\times V$, for which the law is given by polynomials; for the moment we just need to use the fact that the product $(0,v)(0,v')$ can be described as $(Q(v,v'),v+v')$ for some polynomial $Q$ on $2k$ variables.
 
The adjoint action of $\mk{h}$ on $\mk{n}$ is given by some Lie algebra homomorphism $\phi:\mk{h}\to\mathrm{End}(\mk{n})$. 

If $h\in H$ and $n\in N$, we have
\[hnh^{-1}=\exp(\mathrm{Ad}(h)\log(n))=\exp(\exp(\phi(\log(h)))\log(n)).\]
Beware that in this expression, the first $\exp$ is the exponential of the Lie algebra $N$, with inverse $\log$ (both being polynomials), while the second is the operator exponentiation on $\mk{n}$. The operator $\exp(\phi(\log(h)))$ is a polynomial in the entries and exponentials of the coefficients of $\log(h)$ (the exponentials appear because $\phi(\log(h))$ is not necessary nilpotent), possibly with complex coefficients.

Hence after identification of $H$ and $N$ to real vector spaces using the Lie algebras, we can write $hnh^{-1}=R(h,n)$, where $R$ is a polynomial in the coefficients of $h$ and $n$, and in some (possibly complex) exponentials of the coefficients of $h$.

Also, describe the product of 3 elements in $N$ by a polynomial: $nn'n''=S(n,n',n'')$.

The product map yields a decomposition $G=VN$, and for $v,v'\in V$, $n,n'\in N$, we have $(nv)(n'v')=nvn'v^{-1}v(vv')$, so in coordonates, we have
\begin{align*}(n,v)(n',v')= & \Big(S\big(n,vn'v^{-1},Q(v,v')\big),v+v'\Big)\\ = & \Big(S\big(n,R(v,n'),Q(v,v')\big),v+v'\Big).\end{align*} 

That the inverse map has the given form is similar, since $(nv)^{-1}=(v^{-1}n^{-1}v)v^{-1}$.
\end{proof}

\begin{lem}\label{eretra}
If $G$ is a simply connected solvable Lie group with a left-invariant Riemannian metric, there is an exponentially Lipschitz strong deformation retraction of $G$ to the trivial subgroup, i.e.\ a map $F:G\times [0,1]\to G$ such that for all $g$, $F(g,0)=g$ and $F(g,1)=1$, and such that for some constant $C$, if $B(n)$ is the $n$-ball in $G$ then $F$ is $\exp(Cn)$-Lipschitz in restriction to $B(n)\times [0,1]$.

\end{lem}
\begin{proof}
We identify $G$ with $\R^m\times\R^\ell$ and write its law as
$$(u_1,v_1)*(u_2,v_2)=(u_1+u_2,P(u_1,u_2,v_1,v_2)); \; (u,v)^{-1}=(-u,-v),$$
with $P$ as in Lemma \ref{lawscslie}.

Define, for $((u,v),t,\tau)\in G\times [0,1]^2$, $$s((u,v),t,\tau)=(tu,\tau v).$$
We need an upper bound on the differential of $s$ at a given point $(u_0,v_0,t_0,\tau_0)$. 

Let $L_g$ denote the left translation by $g$, given by $h\mapsto g*h$.

Then
\begin{align*} & \left(L_{(t_0u_0,\tau_0 v_0)}^{-1}\circ s\right) \left(L_{(u_0,v_0)}(u,v),t,\tau\right)\\
= & \big(T(\tau_0 v_0,t_0 u_0),-\tau_0 v_0\big)*\big(tP(v_0,v,u_0,u),\tau v_0+\tau v\big) \\
= & \Big( P\big(-\tau_0 v_0,\tau v_0+\tau v,T(\tau_0 v_0,t_0 u_0),tP(v_0,v,u_0,u)\big)\,,-\tau_0v_0+\tau v_0+\tau v\Big)
\end{align*}
We can write it as $Q(\tau_0v_0,\tau v_0,\tau v,v_0,v;t_0,t,u_0,u)$, where $Q=(Q_1,Q_2)$, $Q_1$ is an exponential polynomial depending polynomially on the last four variables and $Q_2$ is a polynomial.

It follows that each partial derivative of $Q_1$ with respect to the variables $(u,v,t,\tau)$ is an exponential polynomial $R$ in the same variables, polynomial on the last four variables (with coefficients depending only the group law). Each such partial derivative, when $(u,v,t,\tau)=(0,0,t_0,\tau_0)$, is equal to $R(\tau_0v_0,\tau_0 v_0,\penalty-10000 0,v_0,0;t_0,t_0,u_0,0)$, which can be rewritten as $R'(\tau_0v_0,v_0;t_0,u_0)$, with $R'$ an exponential polynomial that is polynomial in the last two variables.
Hence, keeping in mind that $\tau_0,t_0$ belong to $[0,1]$, the partial derivatives of $Q_1$ at $(0,0,t_0,\tau_0)$ at any $(0,0,t_0,\tau_0)$ are bounded above by $c_1e^{c_2\|v_0\|}(1+\|u_0\|)^{c_3}$, where $c_1,c_2$ are positive constants depending only on the group law and the choice of left-invariant Riemannian metric $\mu$.

Therefore, the differential of $s$ at any $(u_0,v_0,t_0,\tau_0)$, for the left-invariant Riemannian metric $\mu$, has the same bound. 

If $(u_0,v_0)\in B(2n)$ (the $2n$-ball around 1 in $G$), then $\|v_0\|\le n$ (up to rescaling $\mu$) and $\|u_0\|\le e^{c_4n}$ (for some fixed constant $c_4$). So, for every $(u_0,v_0,t_0,\tau_0)\in B(2n)\times [0,1]^2$, the differential of $s$ at $(u_0,v_0,t_0,\tau_0)$ is bounded by $e^{Cn}$, for some positive constant $C$ only depending on $(G,\mu)$. In particular, since any two points in $B(n)$ can be joined by a geodesic within $B(2n)$, we deduce that the restriction of $s$ to $B(n)\times [0,1]^2$ is $e^{Cn}$-Lipschitz.

The function $(g,t)\mapsto s(g,t,t)$ is the desired retraction. (We used an extra variable $\tau$ by anticipation, in order to reuse the argument in the proof of Proposition \ref{p_ssle}.)
\end{proof}

\begin{proof}[Proof of Theorem \ref{t_lieexp}]
By Lemma \ref{comtri}, we can restrict to the case of a simply connected solvable (actually triangulable) group $G$. 
Given a loop of size $n$ in $G$ based at the unit element, Lemma \ref{eretra} provides a Lipschitz homotopy with exponential area to the trivial loop.
\end{proof}

\begin{rem}
Using Guivarc'h's estimates on the word length in simply connected solvable Lie groups \cite{Gui,Gui80}, we see that there exists a constant $C'$ such that if $B(n)$ is the $n$-ball in $G$, then $F(B(n),[0,1])$ is contained in the ball $B(C'n)$ (here $F$ is the function constructed in the proof of Lemma \ref{eretra}). Thus in particular, $F$ provides a filling of every loop of linear size, with exponential area and inside a ball of linear size. In particular, any virtually connected Lie group (and cocompact lattice therein) has a linear isodiametric function.
\end{rem}

Let $G=(U_{\textnormal{a}}\times U_{\textnormal{na}})\rtimes\Z^d$ be a standard solvable group. The group $G_1=U_{\textnormal{a}}\rtimes\Z^d$ can be embedded as a closed cocompact subgroup into a virtually connected Lie group $G_2$ with maximal compact subgroup $K$. Consider a left-invariant Riemannian metric on the connected manifold $G_2/K$. The composite map $G_1\to G_2/K$ is a $G$-equivariant quasi-isometric injective embedding; endow $G_1$ with the induced metric. Endow $U_{\textnormal{na}}\rtimes\Z^d$ with a word metric with respect to a compact generating subset, and endow $G$ with a metric induced by the natural quasi-isometric embedding $G\to (U_{\textnormal{na}}\rtimes\Z^d)\times G_2$.

\begin{prop}\label{p_ssle}
Let $G$ be a standard solvable group of the form $(U_{\textnormal{a}}\times U_{\textnormal{na}})\rtimes \Z^d$, with the above metric. Then there is an exponentially Lipschitz homotopy between the identity map of $G$ and its natural projection $\pi$ to $U_{\textnormal{na}}\rtimes\Z^d$. Namely, for some constant $C$, there is a map 
\[\sigma:\left((U_{\textnormal{a}}\times U_{\textnormal{na}})\rtimes\Z^d\right)\times [0,1]\to (U_{\textnormal{a}}\times U_{\textnormal{na}})\rtimes\Z^d,\]
such that for all $g\in G$, $\sigma(g,0)=g$ and $\sigma(g,1)=\pi(g)$, and $\sigma(g,t)=g$ if $g\in U_{\textnormal{na}}\rtimes\Z^d$, and $\sigma$ is $e^{Cn}$-Lipschitz in restriction to $B(n)\times [0,1]$.
\end{prop}
\begin{proof}
As in the definition of standard solvable group, write $G=U\rtimes\Z^d$ and $U=U_{\textnormal{a}}\times U_{\textnormal{na}}$. Since $U_{\textnormal{a}}$ is a simply connected nilpotent Lie group, we can identify it to its Lie algebra thorough the exponential map. Define, for $(v,w,u)\in U_{\textnormal{a}}\times U_{\textnormal{na}}\rtimes\Z^d$ and $t\in [0,1]$, $\sigma(v,w,u,t)=(tv,w,u)$. By the computation in the proof of Lemma \ref{eretra}, $\sigma$ is $e^{Cn}$-Lipschitz in restriction to $B(n)\times [0,1]$; the presence of $w$ does not affect this computation.
\end{proof}

\begin{cor}\label{corex}
Under the assumptions of the proposition, if $G/U^0\simeq U_{\textnormal{na}}\rtimes A$ is compactly presented with at most exponential Dehn function, then $G$ has an at most exponential Dehn function.
\end{cor}
\begin{proof}
Given a loop $\gamma$ of size $n$ in $G$, the retraction of Proposition \ref{p_ssle} interpolates between $\gamma$ and its projection $\gamma'$ to $G/G^0$. The interpolation has an at most exponential area because the retraction is exponentially Lipschitz; $\gamma'$ has linear length and hence has at most exponential area by the assumption. So $\gamma$ has an at most exponential area.
\end{proof}

\subsection{Reduction to split triangulable groups}\label{rssg}

\begin{prop}[\cite{CorIll}]
Let $G$ be a triangulable real group and $E=G^\infty$ its exponential radical. 
There exists a triangulable group $\breve{G}=E\rtimes V$ and a homeomorphism $\phi:G\to\breve{G}$, so that, denoting by $d_G$ and $d_{\breve{G}}$ left-invariant word distances on $G$ and $\breve{G}$  
\begin{itemize}
\item $\phi$ restricts to the identity $E\to E$,
\item $E$ is the exponential radical of $\breve{G}$
\item $V$ is isomorphic to the simply connected nilpotent Lie group $G/E$, 
\item the map $\phi$ quasi-preserves the length: for some constant $C>0$,
$$C^{-1}|g|\le|\phi(g)|\le C|g|;\quad\forall g\in G$$
and is logarithmically bilipschitz
$$D(|g|+|h|)^{-1}d_G(g,h)\le d_{\breve{G}}(g,h)\le D(|g|+|h|)d_G(g,h);\quad\forall g,h\in G,$$
where $C'>0$ is a constant and where $D$ is an increasing function satisfying $D(n)\le C'\log(n)$ for large $n$.
\end{itemize}
\end{prop}

\begin{cor}\label{ghats}
Set $\{H,L\}=\{G,\breve{G}\}$. Suppose that the Dehn function $\delta_L$ of $L$ satisfies $\delta_L(n)\preccurlyeq n^\alpha$.

Then for any $\eps>0$, the Dehn function $\delta_H$ of $H$ satisfies 
$$\delta_H(n)\preccurlyeq \log(n)^{\alpha+\eps}\delta_L(n\log(n))\preccurlyeq \log(n)^{2\alpha+\eps}n^\alpha.$$
\end{cor}

\begin{proof}
Suppose, more precisely, that every loop of length $n$ in $L$ can be filled with area $\delta(n)$ in a ball of radius $s(n)$; note that $s$ can be chosen to be asymptotically equal to $\delta$ (by Lemma \ref{vk}).

Start with a combinatorial loop $\gamma$ of length $n$ in $H$. It maps (by $\phi$ or $\phi^{-1}$) to a ``loop" in $L$, in the $Cn$-ball, in which every pair of consecutive vertices are at distance $\le C\log(n)$. Join those pairs by geodesic segments and fill the resulting loop $\gamma'$ of length $\le Cn\log(n)$ by a disc consisting of $\delta(Cn\log(n))$ triangles of bounded radius (say, $\le C$), inside the $s(Cn\log(n))$-ball. Map this filling back to $H$. We obtain a ``loop" $\gamma''$ consisting of $Cn\log(n)$ points, each two consecutive being at distance $\le C\log(s(Cn\log(n)))$,
with a filling by $\delta(Cn\log(n))$ triangles of diameter at most\penalty10000 $C\log(s(Cn\log(n)))$. Interpolate $\gamma''$ by geodesic segments, so as to obtain a genuine loop $\gamma_1$. So $\gamma''$ is filled by $\gamma$ and $Cn\log(n)$ ``small" loops of size
$$\le C\log(s(Cn\log(n)))+1\le C\log(C'(Cn\log(n))^\alpha)+1\le C_1\log(n).$$

The loop $\gamma''$ itself is filled by $\delta(Cn\log(n))$ triangles of diameter at most \penalty-10000 $C\log(s(Cn\log(n)))$ and thus of size $\le 3C_1\log(n)$.

We know that $H$ has its Dehn function bounded above by $C_2e^{cn}$. So each of these small loops has area $\le C_2\exp(3c(C_1\log(n)))=C'n^{c'}$. 
We deduce that $\gamma$ can be filled by
$$(Cn\log(n)+\delta(Cn\log(n)))C'n^{c'}\preceq n^{1+c'+\max(1,\alpha)}$$
triangles of bounded diameter.

We deduce that $H$ has a Dehn function of polynomial growth, albeit with an outrageous degree, the constants $c'$ being out of control. Anyway, this provides a proof that $H$ has a Dehn function $\le C_3n^q$ for some $q$, and we now repeat the above argument with this additional information.

The small loops of size $C_1\log(n)$ therefore have area $\le C_3(C_1\log(n))^q$ and the $\delta(Cn\log(n))$ triangles filling $\gamma''$ can now be filled by $\le C_3  (3C\log(s(Cn\log(n))))^q$ triangles of bounded diameter. We deduce this time that $\gamma$ can be filled by at most
$$ Cn\log(n)C_3(C_1\log(n))^q+3CC_3\log(s(Cn\log(n)))^q \delta(Cn\log(n))$$ $$\approx \log(s(n\log(n)))^q\delta(n\log(n))$$
triangles of bounded diameter.

We have $$\log(s(n\log n))\le \log(s(n^2))\preceq \log(n^{2\alpha})\preceq\log(n),$$ 
so we deduce that the Dehn function of $\gamma$ is
$$\preccurlyeq \log(n)^q\delta(n\log(n)).$$ 

Since $\delta(n)\preccurlyeq n^\alpha$, the previous reasoning can be held with $q$ of the form $\alpha+\eps$ for any $\eps>0$. This proves the desired result.
\end{proof}

\begin{rem}
A variant of the proof of Corollary \ref{ghats} shows that if the Dehn function of $\breve{G}$ is exponential, then the Dehn function of $G$ is $\succcurlyeq \exp(n/\log(n)^2)$, but is not strong enough to show that the Dehn function of $G$ is exponential, nor even $\succcurlyeq\exp(n/\log(n)^\alpha)$ for small $\alpha\ge 0$.
\end{rem}


\section{Gradings and tameness conditions}\label{s_p}

This preliminary section describes basics on the weight decomposition of standard solvable groups and real triangulable groups. 

In \S\ref{granf}, we state a general weight decomposition theorem for finite-dimensional representations of nilpotent groups over complete normed fields. We use it to introduce weights in standard solvable groups in \S\ref{gssg}, where we provide some useful characterizations. We introduce the tameness conditions in \S\ref{suts}, which allow to reinterpret the SOL obstructions in a more conceptual way. In \S\ref{su_ca}, we deal with the Cartan grading in real triangulable groups, which is also needed in~\S\ref{s_cent}.

\subsection{Grading in a representation}\label{granf}

\begin{lem}\label{extni}
Let $H\subset G$ be an inclusion of finite index between nilpotent groups. Then any homomorphism $f:H\to\R$ has a unique extension $\tilde{f}:G\to\R$.
\end{lem}
\begin{proof}
If $g\in G$ and $g^k\in H$, the element $f(g^k)/k$ does not depend on $k$, we define it as $\tilde{f}(g)$; note that this is the only possible choice for $\tilde{f}$ and already proves uniqueness. To prove the existence, we need to check that $\tilde{f}$ is a homomorphism. Since checking $\tilde{f}(xy)=\tilde{f}(x)+\tilde{f}(y)$ only involves two elements, we can suppose that $G$ and $H$ are finitely generated. We can also suppose that they are torsion-free, as the problem is not modified if we mod out by the finite torsion subgroups. So $G$ and $H$ have the same rational Malcev closure, and every homomorphism $H\to\R$ extends to the rational Malcev closure. Necessarily, the extension is equal to $\tilde{f}$ in restriction to $G$, so $\tilde{f}$ is a homomorphism. (Note that $\R$ could be replaced in the lemma by any torsion-free divisible nilpotent group.)
\end{proof}

Recall that for every complete normed field $\K$,
the norm on $\K$ extends to every finite extension field in a unique way \cite[Theorem~5.1, p.~17]{DwGS}.

\begin{thm}\label{gradnor}
Let $\K$ be a non-discrete complete normed field. Let $N$ be a nilpotent topological group and $V$ a finite-dimensional vector space with a continuous linear $N$-action $\rho:N\to\GL(V)$. Then there is a canonical decomposition 
$$V=\bigoplus_{\alpha\in\Hom(N,\R)}V_\alpha,$$
where, for $\alpha\in\Hom(N,\R)$, the subspace $V_\alpha$
is the sum of characteristic subspaces associated to irreducible polynomials whose roots have modulus $e^{\alpha(\omega)}$ for all $\omega\in N$; moreover we have, for $\alpha\in\Hom(N,\R)$
\begin{align*}V_\alpha= & \{0\}\cup\left\{v\in V\smallsetminus\{0\}: \forall \omega\in N,\;\lim_{n\to +\infty}\|\rho(\omega)^n\cdot v\|^{1/n}=e^{\alpha(\omega)}\right\}\\
=& \left\{v\in V: \forall \omega\in N,\;\underset{n\to +\infty}{\overline{\lim}}\|\rho(\omega)^n\cdot v\|^{1/n}\le e^{\alpha(\omega)}\right\}
\end{align*}
\end{thm}
(Note that we do not assume that $\rho(N)$ has a Zariski-connected closure.)

\begin{proof}
Let us begin with the case when $\K$ is algebraically closed. Let $\mathbb{N}$ be the Zariski closure of $\rho(N)$; decompose its identity component $\mathbb{N}^0=\mathbb{D}\times\mathbb{U}$ into diagonalizable and unipotent parts. Consider the corresponding projections $d$ and $u$ into $\mathbb{D}$ and $\mathbb{U}$. Define $N^0=\rho^{-1}(\mathbb{N}^0)$; it is an open subgroup of finite index in $N$. Let $D$ be the (ordinary) closure of the projection $d(\rho(N^0))$. We can decompose, with respect to $D$, the space $V$ into weight subspaces: $V=\bigoplus_{\gamma\in\Hom(D,\K^*)}V_\gamma$, where $V_\gamma=\{v\in V:\forall d\in D, d\cdot v=\gamma(d)v\}$. Write $(\log|\gamma|)(v)=\log(|\gamma(v)|)$, so $\log|\gamma|\in\Hom(D,\R)$. For $\delta\in\Hom(D,\R)$, define $V_\delta=\bigoplus_{\{\gamma:\;\log|\gamma|=\delta\}}V_\gamma$, so that $V=\bigoplus_{\delta\in\Hom(D,\R)}V_\delta$. If $\delta\in\Hom(D,\R)$, then $\delta\circ d\circ\rho\in\Hom(N^0,\R)$. So by Lemma \ref{extni}, it uniquely extends to a homomorphism $\hat{\delta}:N\to\R$, which is continuous because its restriction $\delta\circ d\circ\rho$ to $N^0$ is continuous. Note that $\delta\mapsto\hat{\delta}$ is obviously injective.

If $v\in V_\delta\smallsetminus\{0\}$ and $\omega\in N^0$, write it as a sum $v=\sum_{\gamma\in I} v_\gamma$ where $0\neq v_\gamma\in V_\gamma$, where $I$ is a non-empty finite subset of $\Hom(D,\K^*)$, actually consisting of elements $\gamma$ for which $\log|\gamma|=\delta$. Then, changing the norm if necessary so that the norm is the supremum norm with respect to the norms on the $V_\gamma$ (which does not affect the limits because of the exponent $1/n$), we have
\begin{align*}\|\rho(\omega)^nv\|^{1/n}= & \sup_{\gamma\in I}\|u(\rho(\omega))^nd(\rho(\omega))^nv_\gamma\|^{1/n}\\
=& |\gamma(d(\rho(\omega)))|\sup_{\gamma\in I}\|u(\rho(\omega))^nv_\gamma\|^{1/n}\\
=& \exp(\hat{\delta}(\omega))\sup_{\gamma\in I}\|u(\rho(\omega))^nv_\gamma\|^{1/n};\\
\end{align*}
the spectral radii of both $u(\rho(\omega))$ and its inverse being equal to 1 and $v_\gamma\neq 0$, we deduce that $\lim\|u(\rho(\omega))^nv_\gamma\|^{1/n}=1$. It follows that $\lim\|\rho(\omega)^nv\|^{1/n}=\exp(\hat{\delta}(\omega))$ for all $\omega\in N^0$ and $v\in V_\delta\smallsetminus\{0\}$.
If $\omega\in N$, there exists $k$ such that $\omega^k\in N^0$. Writing $f_\omega(n)=\|\rho(\omega)^nv\|^{1/n}$, we therefore have
$$\lim_{n\to\infty} f_\omega(kn)=\lim_{n\to\infty}f_{\omega^k}(n)^{1/k}=\exp(\hat{\delta}(\omega^k))^{1/k}=\exp(\hat{\delta}(\omega));$$
on the other hand a simple verification shows that $\lim_{n\to\infty}f(n+1)/f(n)=1$, and it follows that $\lim_{n\to\infty} f_\omega(n)=\exp(\hat{\delta}(\omega))$.

It follows in particular that the spectral radius of $\rho(\omega)^{\pm 1}$ on $V_\delta$ is $\exp(\pm\hat{\delta}(\omega))$ and since the $\hat{\delta}$ are distinct, it follows that $V_\delta$ is the sum of common characteristic subspaces associated to eigenvalues of modulus $\exp(\hat{\delta}(\omega))$ for all $\omega$.

Conversely, suppose that $v\in V$ and that there exists $\alpha\in\Hom(N,\R)$ such that for all $\omega\in N$ we have $\underset{n\to +\infty}{\overline{\lim}}\|\rho(\omega)^n\cdot v\|^{1/n}\le e^{\alpha(\omega)}$, and let us check that $v\in\bigcup V_\delta$. Observe that $\underset{n\to +\infty}{\overline{\lim}}(\|\rho(\omega)^n\cdot v\|\,\|\rho(\omega^{-1})^{n}\cdot v\|)^{1/n}\le 1$. Write $v=\sum_{\delta\in J}v_\delta$ with $v_\delta\in V_\delta$ and suppose by contradiction that $J$ contains two distinct elements $\delta_1,\delta_2$. So $\hat{\delta}_1\neq\hat{\delta_2}$ and there exist $\omega_0\in N$ 
such that $\hat{\delta_1}(\omega)>\hat{\delta_2}(\omega)$.
Then 
\begin{align*}1\ge &\underset{n\to +\infty}{\overline{\lim}}(\|\rho(\omega_0)^n\cdot v\|\,\|\rho(\omega_0^{-1})^{n}\cdot v\|)^{1/n}\\
 \ge & \underset{n\to +\infty}{\overline{\lim}}\|\rho(\omega_0)^n\cdot v_{\delta_1}\|^{1/n}\|\rho(\omega_0^{-1})^{n}\cdot v_{\delta_2}\|^{1/n}\\
 & = \exp(\hat{\delta_1}(\omega_0)-\hat{\delta_2}(\omega_0))>1,\end{align*}
a contradiction; thus $v\in\bigcup V_\delta$. This concludes the proof in the algebraically closed case.

Now let $\K$ be arbitrary. The above decomposition can be done in an algebraic closure of $\K$, and is defined on a finite extension $\mathbf{L}$ of $\K$, so $W=V\otimes_\K\mathbf{L}=\bigoplus_{\alpha\in\Hom(N,\R)} W_\alpha$, where $W_\alpha$ satisfies all the characterizations.
Consider a $\K$-linear projection $\pi$ of $\mathbf{L}$ onto $\K$; it extends to a $\K$-linear projection $\pi'=\mathrm{id}_V\otimes_\K\pi$ of $W=V\otimes_\K\mathbf{L}$ onto $V$, which commutes with the action of $\rho(N)$. Clearly $V=\sum_\alpha\pi'(W_\alpha)$. Since $\K$ is a complete normed field, $\pi'$ is Lipschitz. It then follows from the definition that $\pi'$ maps $W_\alpha$ into itself (where we naturally consider the inclusion $V\subset W$): this follows from the characterization of those elements of $W_\alpha$ as those $v\in W$ such that for all $\omega\in N$ we have $\underset{n\to +\infty}{\overline{\lim}}\|\rho(\omega)^n\cdot v\|^{1/n}\le e^{\alpha(\omega)}$. It follows that $V=\bigoplus_\alpha V_\alpha$, where $V_\alpha=\pi'(W_\alpha)=\bigoplus_\alpha (W_\alpha\cap V)$. So the proof is complete.

Note that the proof, as a byproduct, characterizes the elements $v$ of $\bigcup_{\alpha\in\Hom(N,\R)} V_\alpha$ as those in $V$ for which, for all $\omega\in N$, we have $\underset{n\to +\infty}{\overline{\lim}}(\|\rho(\omega)^n\cdot v\|\,\|\rho(\omega)^{-n}\cdot v\|)^{1/n}\le 1$ (which is actually a limit, and equal to 1, if $v\neq 0$).
\end{proof}

\subsection{Grading in a standard solvable group}\label{gssg}

Let $G=U\rtimes A$ be a standard solvable group in the sense of Definition \ref{d_ssg}. It will be convenient to consider the product ring $\K=\prod_{j=1}^\tau\K_j$; it is endowed with the supremum norm, and view $U$ as $\mathbb{U}(\K)$. We thus call $G$ a standard solvable group over $\K$. In a first reading, the reader can assume there is a single field $\K=\K_1$.

Since $\K$ is a finite product of fields, a finite length $\K$-module is the same as a direct sum $V=\bigoplus V_j$, where each $V_j$ is a finite-dimensional $\K_j$-vector space. The length of $V$ as a $\K$-module, is equal to $\sum_j\dim_{\K_j}V_j$.

Let $\mk{u}_j$ be the Lie algebra of $U_j$. So $\mk{u}=\prod_j\mk{u}_j$ is a Lie algebra over $\K$ and the exponential map, which is truncated by nilpotency, is a homeomorphism $\mk{u}\to U$.
This conjugates the action of $D$ on $U$ to a linear action on $\mk{u}$, preserving the Lie algebra structure; for convenience we denote it as an action by conjugation.

We endow $\mk{u}_j$ with the action of $A$, and with the grading in $\Hom(A,\R)$, as introduced in Theorem \ref{gradnor}. Thus $\mk{u}$ itself is graded by $\mk{u}_\alpha=\bigoplus_j\mk{u}_{j,\alpha}$. The finite-dimensional vector space $\mathcal{W}=\Hom(A,\R)$ is called the {\bf weight space}. This is a Lie algebra grading:
$$[\mk{u}_\alpha,\mk{u}_\beta]\subset\mk{u}_{\alpha+\beta},\quad\forall\alpha,\beta\in\Hom(A,\R).$$

Note that since $A$ is a compactly generated locally compact abelian group, it is isomorphic to $\R^{d_1}\times\Z^{d_2}\times K$ for some integers $d_1,d_2$, and $K$ a compact abelian group. In particular, if $d=d_1+d_2$, then the weight space $\mathcal{W}=\Hom(A,\R)$ is a $d$-dimensional real vector space.

If we split $\mk{u}$ as the direct sum $\mk{u}=\mk{u}_{\textnormal{a}}\oplus\mk{u}_{\textnormal{na}}$ of its Archimedean and non-Archimedean parts, the weights of $\mk{u}_{\textnormal{a}}$ and $\mk{u}_{\textnormal{na}}$, respectively, are called {\bf Archimedean weights} and {\bf non-Archimedean weights}.

\begin{ex}
Assume that the action of $A$ on $\mk{u}$ is diagonalizable. 
If $\K_j=\R$ and the diagonal entries are positive, we have
$$(\mk{u}_j)_\alpha=\{x\in\mk{u}_j:\forall v\in A,\;v^{-1}xv=e^{\alpha(v)}x\}$$
and if $\K_j=\Q_p$ and the diagonal entries are powers of $p$, we have
$$(\mk{u}_j)_\alpha=\{x\in\mk{u}_j:\forall v\in A,\;v^{-1}xv=p^{-\alpha(v)/\log(p)}x\}.$$
\end{ex}

\begin{ex}[Weights in groups of SOL type]\label{wei_sol}
Let $G=(\K_1\times\K_2)\rtimes A$ be a group of SOL type as in Definition \ref{d_sol}, where $A$ contains as a cocompact subgroup the cyclic subgroup generated by some element $(t_1,t_2)$ with $|t_1|>1>|t_2|$. Then the weight space is a one-dimensional real vector space, and with a suitable normalization, the weights are $\alpha_1=\log(|t_1|)>0$ and $\alpha_2=\log(|t_2|)<0$, and $U_{\alpha_i}=\K_i$. It is useful to think of the weight $\alpha_i$ with multiplicity $q_i$, namely the dimension of $\K_i$ over the closure of $\Q$ in $\K_i$ (which is isomorphic to $\R$ or $\Q_p$ for some $p$). In particular, $G$ is unimodular if and only if $q_1\alpha_1+q_2\alpha_2=0$.

For instance, if $G=(\R\times\Q_p)\rtimes_p\Z$, then $\mk{u}=\R\times\Q_p$, $\mk{u}_{\log(p)}=\R\times\{0\}$, $\mk{u}_{-\log(p)}=\{0\}\times\Q_p$.
\end{ex}

For an arbitrary standard solvable group, we define the set of {\bf weights}
$$\mathcal{W}_{\mk{u}}=\{\alpha:\mk{u}_\alpha\neq\{0\}\}\subset\Hom(A,\R).$$
It is finite. Weights of the abelianization $\mk{u}/[\mk{u},\mk{u}]$ are called {\bf principal weights} of $\mk{u}$. 

\begin{lem}\label{weighp}
Every weight of $\mk{u}$ is a sum of $\ge 1$ principal weights; 0 is not a principal weight.
\end{lem}
\begin{proof}
The first statement is a generality about nilpotent graded Lie algebras. If $P$ is the set of principal weights and $\mk{v}=\bigoplus_{\alpha\in P}\mk{u}_\alpha$, then $\mk{u}=\mk{v}+[\mk{u},\mk{u}]$, i.e., $\mk{v}$ generates $\mk{u}$ modulo the derived subalgebra. A general fact about nilpotent Lie algebras (see Lemma \ref{estimni} for a refinement of this) then implies that $\mk{v}$ generates $\mk{u}$. The first assertion follows.

The condition that 0 is not a principal weight is a restatement of Definition \ref{d_ssg}(\ref{ssg3}). 
\end{proof}

\begin{defn}\label{d_cpon}If $U$ is a locally compact group and $v$ a topological automorphism, $v$ is called a {\bf compaction} if there exists a compact subset $\Omega\subset U$ that is a {\bf vacuum subset} for $v$, in the sense that for every compact subset $K\subset U$ there exists $n\ge 0$ such that $v^n(K)\subset\Omega$. If every neighborhood of 1 is a vacuum subset, we say that $v$ is a {\bf contraction}.
\end{defn}

\begin{prop}\label{cocoh}
Let $U\rtimes A$ be a standard solvable group. Equivalences:
\begin{enumerate}[(i)]
\item\label{sch1} some element of $A$ acts as a compaction of $U$;
\item\label{sch2} some element of $A$ acts as a contraction of $U$;
\item\label{sch3} 0 is not in the convex hull in $\mathcal{W}$ of the set of weights;
\item\label{sch4} 0 is not in the convex hull in $\mathcal{W}$ of the set of principal weights.
\end{enumerate}
\end{prop}
\begin{proof}
Automorphisms of $U$ are conjugate, through the exponential, to linear automorphisms; in particular, contractions and compactions coincide, being characterized by the condition that all eigenvalues have modulus $<1$. Thus (\ref{sch1})$\Leftrightarrow$(\ref{sch2}). 

By Lemma \ref{weighp}, weights are sums of principal weights, and thus (\ref{sch3})$\Leftrightarrow$(\ref{sch4}). 

If $v\in A$ acts as a contraction of $U$, then $\alpha(v)<0$ for every weight $\alpha$. Thus $\alpha\mapsto \alpha(v)$ is a linear form on $\mathcal{W}$, which is positive on all weights. Thus (\ref{sch2})$\Rightarrow$(\ref{sch3}). 

Conversely, let $L$ be the set of linear forms $\ell$ of $\mathcal{W}=\Hom(A,\R)$ such that $\ell(\alpha)>0$ for every weight $\alpha\in\mathcal{W}_{\mk{u}}$. Suppose that 0 is not in the convex hull of $\mathcal{W}_{\mk{u}}$, or equivalently that $L\neq\emptyset$. Since the image of $A$ in the bidual $\mathcal{W}^*$ linearly spans $\mathcal{W}^*$, it has a cocompact closure in bidual $\mathcal{W}^*$; since $L$ is a nonempty open convex cone, it follows that it has a non-empty intersection with the image of $A$ in the bidual $\mathcal{W}^*$. Thus there exists $v\in A$ such that $\alpha(v)>0$ for every weight $\alpha$. So (\ref{sch3})$\Rightarrow$(\ref{sch2}) holds.
\end{proof}
 
\subsection{Tameness conditions in standard solvable groups}\label{suts}

Motivated by Proposition \ref{cocoh}, we introduce the following terminology:
 
\begin{defn}\label{d_2ta}
We say that the standard solvable group $G=U\rtimes A$ (or the graded Lie algebra $\mk{u}$) is 
\begin{itemize}
\item {\bf tame} if 0 is not in the convex hull of the set of weights;
\item {\bf 2-tame} if 0 is not in the segment joining any pair of {\em principal} weights;
\item {\bf stably 2-tame} if 0 is not in the segment joining any pair of weights.
\end{itemize}
\end{defn} 

To motivate the adjective ``stably", observe that $\mk{u}$ is stably 2-tame if and only if every graded subalgebra of $\mk{u}$ is 2-tame.

The importance of 2-tameness, first brought forward by Abels, will appear gradually in the paper, starting with Proposition \ref{2tsol}.
 
\begin{rem}\label{r_triv2tame} 
 Clearly $$\textnormal{tame }\Rightarrow \textnormal{ stably 2-tame } \Rightarrow \textnormal{ 2-tame}\,;$$ the converse implications do not hold in general; however they hold when $\mathcal{W}$ is 1-dimensional, i.e.\ when $A$ has a discrete cocompact infinite cyclic subgroup.

Also, when $\mk{u}$ is abelian, then 2-tame and stably 2-tame are obviously equivalent.
\end{rem}

The following characterization of 2-tameness will be needed to show in \S\ref{2tsol} that if $G$ is not 2-tame then it has an at least exponential Dehn function. 

\begin{prop}\label{2tsol}
Let $G=U\rtimes A$ be a standard solvable group. Then $G$ is not 2-tame if and only if there exists a group $V\rtimes E$ of type SOL (see Definition \ref{d_sol}) and a homomorphism $f$ into $V\rtimes E$ whose image contains $V$ and is dense.
\end{prop}

\begin{proof}
In general, let $V\rtimes E$ be a standard solvable group, and consider a continuous homomorphism $f:U\rtimes A\to V\rtimes E$ with dense image (assuming that the closure of the image is cocompact would be enough).
Since $U$ and $V$ are the derived subgroups of the two groups, we have $f(U)\subset V$. Since $f(U)\subset V$, passing to the quotients, $f$ induces a continuous homomorphism $A\to E$ with dense image, which induces an injective linear map $f^*:\Hom(E,\R)\to\Hom(A,\R)$.

Denote by $\underline{f}$ the Lie algebra map it induces $\mk{u}\to\mk{v}$. If we consider the weight decomposition $\mk{u}=\bigoplus_{\alpha\in\Hom(A,\R)}\mk{u}_\alpha$, we see that either $\underline{f}(\mk{u}_\alpha)$ is zero, or $\alpha$ has the form $f^*(\beta)$ and $\underline{f}(\mk{u}_\alpha)\subset\mk{v}_\beta$. In particular, $\underline{f}(\mk{u})$ is contained in the sum of $\mk{v}_\beta$, where $\beta$ ranges over elements in $\Hom(E,\R)$ such that $f^*(\beta)$ is a weight of $U\rtimes A$. In particular, if we assume that $f(U)=V$, we deduce that $f^*$ maps weights to weights.

Assume now that $V\rtimes E$ is a group of type SOL; since $E$ is abelian, $f$ factors through a homomorphism $U/[U,U]\rtimes A\to V\rtimes E$. Recall from Example \ref{wei_sol} that for a group of SOL type, the set of weights consists of a {\bf quasi-opposite pair}, i.e. a pair of nonzero elements such that the segment joining them contains zero. Since $f^*$ is injective $\R$-linear and maps weights to weights, we deduce that $U/[U,U]\rtimes A$ admits a quasi-opposite pair of weights.

Conversely, suppose that $U\rtimes A$, where $U/[U,U]$ has two nonzero weights $\alpha,\beta$ with $\beta=t\alpha$ for some $t<0$. 

To construct the map $f$, first mod out by $[U,U]$, so we are reduced to the case when $U$ is abelian. 
Modding out if necessary by all other nonzero weights, we can suppose that $\mk{u}=\mk{u}_\alpha\oplus\mk{u}_\beta$. Write $\mk{u}_\alpha=\bigoplus_j\mk{u}_{j,\alpha}$, and mod out all $\mk{u}_{j,\alpha}$ except one, so that $\mk{u}_\alpha$ is a Lie algebra over a single field $\K_j$; again modding out by a maximal invariant subspace and possibly replacing $\K_j$ with a finite extension, we can assume that $\mk{u}_\alpha$ is 1-dimensional over $\K_j$, with a scalar action. Similarly we can assume that $\mk{u}_\beta$ is 1-dimensional over $\K_{j'}$, with a scalar action. We can mod out by the kernel of the homomorphism $A\to\K_j^*\times\K_{j'}^*$; let $B$ be the closure of its image. Since $\alpha$ and $\beta$ are proportional, $E$ contains a cyclic cocompact subgroup. The resulting group $V\rtimes E$, where $V=\K_j\times\K_{j'}$, is of type SOL and the resulting homomorphism $U\rtimes A\to V\rtimes E$ has a dense image containing $V$.
\end{proof}

\begin{rem}
The homomorphism to a group of type SOL cannot always be chosen to have a closed image. For instance, 
let $\Z^2$ act on $\R^2$ by $$(m,n)\cdot (x,y,z)=(2^{-m}3^{-n}x,2^m3^ny).$$
Let $G=\R^2\rtimes\Z^2$ be the corresponding standard solvable group. Then every nontrivial normal subgroup of $G$ contains either $\R\times\{0\}$ or $\{0\}\times\R$; thus every proper quotient of $G$ is tame and it follows that no quotient of $G$ is of SOL type.
\end{rem}

Let $G=U\rtimes A$ is a standard solvable group in the sense of Definition \ref{d_ssg}. The Lie algebra $\mk{u}$ of $U$ admits a grading in the weight space $\mathcal{W}=\Hom(A,\R)$, introduced in \S\ref{gssg}. Define the set of weights of $\mk{u}$ as the finite subset $\mathcal{W}_{\mk{u}}=\{\alpha\in\mathcal{W}:\mk{u}_\alpha\neq\{0\}\}$.

 We say that a subset of $\mathcal{W}_{\mk{u}}$ is {\bf conic} if it is of the form $\mathcal{W}_{\mk{u}}\cap C$, with $C$ an open convex cone not containing 0. Let $\mathcal{C}$ be the set of conic subsets of $\mathcal{W}_{\mk{u}}$. If $C\in\mathcal{C}$, define $$\mk{u}_C=\bigoplus_{\alpha\in C}\mk{u}_\alpha;$$
this is a graded Lie subalgebra of $\mk{u}$; clearly it is nilpotent. Let $U_C$ be the closed subgroup of $U$ corresponding to $\mk{u}_C$ under the exponential map and $G_C=U_C\rtimes A$. In particular, if $v\in A$, define $H(v)=\{\alpha\in\mathcal{W}_{\mk{u}}:\alpha(v)>0\}$.

\begin{defn}\label{estam}
The $G_{H(v)}$ are called the {\bf essential tame subgroups} of $G$. The $G_C$ are called the {\bf standard tame subgroups} of $G$.
\end{defn}

Note that each essential tame subgroup is standard tame, and there finitely many standard tame subgroups; moreover every standard tame subgroup is a finite intersection of essential tame subgroups.

\begin{rem}
These notions can be developed in a broader context, avoiding references to gradings and Lie algebras.
Let $G$ be a locally compact group with a fixed semidirect product decomposition $G=U\rtimes A$.

A {\bf tame subgroup} of $G$ (relative to the given semidirect decomposition) is defined as a subgroup of the form $V\rtimes A$, which is tame, i.e., in which some element of $A$ acts on $V$ as a compaction (in the sense of Definition \ref{d_cpon}).

For $v\in A$, if we define $U_v$ as the contraction subgroup
\begin{equation}\left\{x\in U:\;\lim_{n\to+\infty}v^{-n}xv^n= 1\right\},\label{ubeta}\end{equation}
then $v$ acts as a compaction on $\overline{U_v}$ (\cite[Prop.~6.17]{CCMT})
and therefore $\overline{U_v}\rtimes A$ is tame. We call this a {\bf essential tame subgroup}. Then {\bf standard tame subgroups} are defined as finite intersections of essential tame subgroups; this can be shown to match the previous definition in the setting of standard solvable groups.
\end{rem}


\subsection{Cartan grading and weights}\label{su_ca}
Here we introduce the notion of Cartan grading, especially in the context of real triangulable Lie algebras; this will especially be needed in \S\ref{s_cent}. Triangulable Lie groups are not standard solvable in general and although the treatment is analogous to that in \S\ref{gssg}, we have to face some specific difficulties.

All Lie algebras in this \S\ref{su_ca} are finite-dimensional over a fixed field $K$ of characteristic zero. 

\begin{defn}\label{d_gi}If $\g$ is a Lie algebra, let $\g^\infty=\bigcap_{k\ge 1}\g^k$ be the intersection of its lower central series, so that $\g/\g^\infty$ is the largest nilpotent quotient of $\g$.
\end{defn}

\begin{defn}\label{d_er}
If $G$ is a triangulable Lie group with Lie algebra $\g$, define its {\bf exponential radical} $G^\infty$ as the intersection of its lower central series (so that its Lie algebra is equal to $\g^\infty$).
\end{defn}

We need to recall the notion of Cartan grading of a Lie algebra, which is used in \S\ref{prex} and \S\ref{s_sce}.
Let $\mk{n}$ is a nilpotent Lie algebra; denote by $\mk{n}^\vee$ the space of homomorphisms from $\mk{n}$ to $K$ (that is, the linear dual of $\mk{n}/[\mk{n},\mk{n}]$).

Let $\mk{v}$ be an $\mk{n}$-module (finite-dimensional) with structural map $\rho:\mk{n}\to\mk{gl}(\mk{v})$. If $\alpha\in\mk{n}^\vee$, define the characteristic subspace
\[\mk{v}_\alpha=\bigcup_{k\ge 1}\{v\in\mk{v}:\;\forall g\in\mk{n},\;(\rho(g)-\alpha(g))^kv=0\}.\]
The subspaces $\mk{v}_\alpha$ generate their direct sum; we say that $\mk{v}$ is $K$-{\bf triangulable} if $\mk{v}=\bigoplus_{\alpha\in\mk{n}^\vee}\mk{v}_\alpha$; this is automatic if $K$ is algebraically closed. If $\mk{v}$ is a $K$-triangulable $\mk{n}$-module, the above decomposition is called the {\bf natural grading} (in $\mk{n}^\vee$) of $\mk{v}$ as an $\mk{n}$-module. If $\mk{v}=\mk{v}_0$, we call $\mk{v}$ a {\bf nilpotent} $\mk{n}$-module.

\begin{defn}[\cite{Bou}]
A {\bf Cartan subalgebra} of $\g$ is a nilpotent subalgebra that is equal to its normalizer.
\end{defn}

We use the following proposition, proved in \cite[Chap.~7,\S 1,2,3]{Bou} (see Definition \ref{d_gi} for the meaning of $\g^\infty$).

\begin{prop}\label{cartan}
Every Lie algebra admits a Cartan subalgebra. If $\mk{n}$ is a Cartan subalgebra, then $\g=\mk{n}+\g^\infty$ and $\mk{n}$ contains the hypercenter of $\g$ (the union of the ascending central series).

If $\g$ is solvable, any two Cartan subalgebras of $\g$ are conjugate by some elementary automorphism $\exp(\mk{ad}(x))$ with $x\in\g^\infty$ (here $\mk{ad}(x)$ is a nilpotent endomorphism so its exponential makes sense).\qed 
\end{prop}

Let us provide a simple way to recognize a Cartan subalgebra.

\begin{prop}\label{cartcri}
Let $\g$ be a Lie algebra graded in some abelian group $\mathcal{W}$ (see \S\ref{basc} if necessary). Assume that $\g_0$ is nilpotent and that for every $\alpha\neq 0$, we have $[\g_0,\g_\alpha]=\g_\alpha$. Then $\g_0$ is a Cartan subalgebra in $\g$.
\end{prop}
\begin{proof}
We have to show that the normalizer $\mk{h}$ of $\g_0$ in $\g$ is $\g_0$ itself. This is a graded subalgebra; hence we have to show that $\mk{h}_\alpha=0$ for all $\alpha\neq 0$. Since $[\mk{g}_0,\mk{h}_\alpha]$ is contained in both $\g_0$ and $\g_\alpha$, we deduce that $[\mk{g}_0,\mk{h}_\alpha]=0$. 

To show that $\mk{h}=\g_0$, it is enough to do it in the case $K$ is algebraically closed. For $\alpha\neq 0$, we can decompose the module $\g_\alpha$ with respect to the action of the nilpotent Lie algebra $\g_0$ as a sum of characteristic subspaces $V_\lambda$, for homomorphisms $\lambda:\g_0\to K$. Then $V_0=0$ since the contrary would contradict $[\g_0,\g_\alpha]=\g_\alpha$. On the other hand, since $\mk{h}_\alpha$ is contained in the centralizer of $\g_0$, it is contained in $V_0$. We deduce that it is zero.
\end{proof}

\begin{ex}\label{excart}
For $\lambda\in K$, let $\g(\lambda)$ be the 4-dimensional Lie algebra with basis $(u,x,y,z)$ and nonzero brackets $[x,y]=z$, $[u,x]=x$, $[u,y]=\lambda y$, $[u,z]=(1+\lambda)z$. Then a Cartan subalgebra is given as follows:
\begin{itemize}
\item if $\lambda\notin\{-1,0\}$, it can be taken to be the 1-dimensional subalgebra generated by $u$;
\item if $\lambda=-1$, as the abelian subalgebra generated by $u,z$;
\item if $\lambda=0$, as the abelian subalgebra generated by $u,y$.
\end{itemize}
Indeed, we can consider the grading in $\Z$ for which $u$ has weight 0, $x$ has weight 1, $y$ has weight $\lambda$, and $z$ has weight $1+\lambda$. Then it satisfies the condition of Proposition \ref{cartcri}.
Other illustrating examples of Cartan subalgebras in triangulable Lie algebras can be found in \cite[Ex.~4.1 and 4.2]{CorTop}.
\end{ex}

We now turn to a partial converse for Proposition \ref{cartcri}.

Let $\g$ be a Lie algebra and $\mk{n}$ a Cartan subalgebra. For the adjoint representation, $\g$ is an $\mk{n}$-module. 
Assume that $\g$ is $K$-triangulable as an $\mk{n}$-module (e.g., this holds if $K$ is algebraically closed).
The corresponding natural grading is called the {\bf Cartan grading} of $\g$ (relative to the Cartan subalgebra $\mk{n}$); moreover the Cartan grading determines $\mk{n}$, namely $\mk{n}=\g_0$. We call {\bf weights} the set of $\alpha$ such that $\g_\alpha\neq 0$.
The Cartan grading is a Lie algebra grading, i.e.\ satisfies $[\g_\alpha,\g_\beta]\subset\g_{\alpha+\beta}$ for all $\alpha,\beta\in\mk{n}^\vee$. In addition it satisfies $[\g_0,\g_\alpha]=\g_\alpha$ for all $\alpha\neq 0$; hence it fulfills the conditions of Proposition \ref{cartcri}.

We have $\mk{n}+\g^\infty=\g$, and $\g^\infty\subset [\g,\g]$, so the projection $\mk{n}\to\g/[\g,\g]$ is surjective, inducing an injection $(\g/[\g,\g])^\vee\subset\mk{n}^\vee$.

Now assume that $\g$ is $K$-triangulable. Then all the weights of the Cartan grading lie inside the subspace $(\g/[\g,\g])^\vee$ of $\mk{n}$; the advantage is that this space does not depend on $\mk{n}$, allowing to refer to a weight $\alpha$ without reference to a the choice of a Cartan subalgebra (although the weight space $\g_\alpha$ still depends on this choice).
In view of Proposition \ref{cartan}, any two Cartan gradings of $\g$ are conjugate. In particular, the set of weights, viewed as a finite subset of $(\g/[\g,\g])^\vee$, does not depend on the Cartan grading. Actually, the subspace linearly spanned by weights is exactly $(\g/\mk{r})^\vee$, where $\mk{r}\supset [\g,\g]$ is the nilpotent radical of $\g$. The {\bf principal weights} of $\g$ are by definition the weights of the Lie algebra $\g^\infty/[\g^\infty,\g^\infty]$ (that is, the nonzero weights of the Lie algebra $\g/[\g^\infty,\g^\infty]$); note that every nonzero weight is a sum of principal weights, as a consequence of the following lemma, which relies on the (basic) results of \S\ref{def1t}.

\begin{lem}
Define $\g_\td=\bigoplus_{\alpha\neq 0}\g_\alpha$. Then $\g^\infty$ is the subalgebra generated by $\g_\td$.
\end{lem}
\begin{proof}
It is clear from the definition that $[\mk{n},\g_\alpha]=\g_\alpha$ for all $\alpha\neq 0$ and therefore $\g_\td\subset\g^\infty$.
Conversely, since $\g_0=\mk{n}$ is nilpotent, Lemma \ref{g0ni} implies that $\g^\infty$ is contained in the subalgebra generated by $\g_\td$.\end{proof}

If $\mk{m}$ is a $\g$-module, define $\mk{m}^\g=\{x\in\mk{m}:\forall g\in\g,\,gx=0\}$.

\begin{lem}\label{eqh2}
Let $\g$ be a $K$-triangulable Lie algebra with a Cartan grading, and $i\ge 1$. Then $H_i(\g^\infty)^\g\neq \{0\}$ if and only if $H_i(\g^\infty)_0\neq \{0\}$.
\end{lem}
\begin{proof}
By definition of Cartan grading, we have $H_i(\g^\infty)^\g\subset H_i(\g^\infty)_0$;
this provides one implication. 

Note, that $H_i(\g^\infty)$ is a trivial $\g^\infty$-module, hence can be viewed as a $(\g/\g^\infty)$-module, and for any $\alpha$, $H_i(\g^\infty)_\alpha$ is a nonzero $\g_0$-submodule; since it is stable under both $\g_0$ and $\g^\infty$; since $\g=\g^\infty+\g_0$ we deduce that each $H_i(\g^\infty)_\alpha$ is a $\g$-submodule of $H_i(\g^\infty)$ and $H_i(\g^\infty)^{\g_0}=H_i(\g^\infty)^{\g}$. 

On the other hand, it follows from the definition of Cartan grading that $(\g^{\ot i})_0$ is a nilpotent $\g_0$-module; hence its subquotient $H_i(\g^\infty)_0$ is also a nilpotent $\g_0$-module and hence is a nilpotent $\g$-module by the previous remark. Hence if $H_i(\g^\infty)_0\neq 0$, then $H_i(\g^\infty)_0^{\g}$ is nonzero as well.
\end{proof}

Now consider a real triangulable Lie algebra $\g$, or the corresponding real triangulable Lie group. We call {\bf weights} of $\g$ the weights of graded Lie subalgebra $\g^\infty$ endowed with the Cartan grading from $\g$ (note that since $\g=\g_0+\g^\infty$, this is precisely the set of weights of the graded Lie algebra $\g$, except possibly the 0 weight). 

We define a real triangulable Lie algebra $\g$, or the corresponding real triangulable Lie group $G$, to be {\bf tame}, {\bf 2-tame}, or {\bf stably 2-tame} exactly in the same fashion as in Definition \ref{d_2ta}. In the case $G$ is also a standard solvable group $U\rtimes D$, we necessarily have $\mathfrak{u}=\g^\infty$ and its grading is also the Cartan grading; in particular whether $G$ is tame (resp.\ 2-tame, stably 2-tame) does not depend on whether $G$ is viewed as a real triangulable Lie group or a standard solvable Lie group.

For $\alpha>0$, define $\SOL_\alpha$ as the semidirect product $\R^2\rtimes\R$, where the action is given by $t\cdot (x,y)=(e^{t}x,e^{-\alpha t}y)$. Note that apart from the obvious isomorphisms $\SOL_\alpha\simeq\SOL_{1/\alpha}$, they are pairwise non-isomorphic.

In a way analogous to Proposition \ref{2tsol}, we have

\begin{prop}\label{2tsoltriang}
Let $G$ be a real triangulable group. Then $G$ is not 2-tame if and only if it admits $\SOL_\alpha$ as a quotient for some $\alpha>0$.
\end{prop}
\begin{proof}
The ``if" part can be proved in the same lines as the corresponding statement of Proposition \ref{2tsol} and is left to the reader.

Conversely, assume that $G$ is not 2-tame. Clearly, $G/[G^\infty,G^\infty]$ is not 2-tame, so we can assume that $G^\infty$ is abelian. Endow $\g$ with a Cartan grading; then since $\g^\infty$ is abelian, 0 is not a weight of $\g^\infty$ and hence $\g=\g^\infty\rtimes\g_0$. Let $G=G^\infty\rtimes G_0$ be the corresponding decomposition of $G$. 

In the same way as in the proof of Proposition \ref{2tsol}, we can pass to a quotient that is still not 2-tame, for which the new $G^\infty$ is 2-dimensional. 
Since the action of $G_0$ on $G^\infty$ is given by two proportional weights, its kernel has codimension 1 in $N$; in particular this kernel is normal in $G$; we see that the quotient is necessarily isomorphic to $\SOL_\alpha$ for some $\alpha>0$.
\end{proof}



\section{Algebraic preliminaries about nilpotent groups and Lie algebras}\label{adumn}

\subsection{Divisibility and Malcev's theorem}

Recall that a group $G$ is {\em divisible} (resp.\ {\em uniquely divisible}) if for every $n\ge 1$, the power map $G\to G$ mapping $x$ to $x^n$, is surjective (resp.\ bijective). The following lemma is very standard.

\begin{lem}\label{tfdud}
Every torsion-free divisible nilpotent group is uniquely divisible.
\end{lem}
\begin{proof}
We have to check that $x^k=y^k\Rightarrow x=y$ holds in any torsion-free 
nilpotent group. Assume that $x^k=y^k$ and embed the finitely generated torsion-free nilpotent group $\Gamma=\langle 
x,y\rangle$ into the group of upper unipotent matrices over the reals. Since the latter is uniquely divisible (the $k$-th extraction of root being defined by some explicit polynomial), we get the result.
\end{proof}

Let $\mk{N}_\Q$ be the category of nilpotent Lie algebras over $\Q$ (of possibly infinite dimension) with Lie algebras homomorphisms, and $\mathcal{N}$ the category of nilpotent groups with group homomorphisms, and $\mathcal{N}_\Q$ its subcategory consisting of uniquely divisible (i.e.\ divisible and torsion-free, by Lemma \ref{tfdud}) groups.
If $\g\in\mathfrak{N}_\Q$, consider the law $\cd_\g$ on $\g$ defined by the Campbell-Baker-Hausdorff formula. 

\begin{thm}[Malcev \cite{M,St}]\label{malcev}
For any nilpotent Lie algebra $\g$ over $\Q$, $(\g,\cd_\g)$ is a group and if $f:\g\to\mk{h}$ is a Lie algebra homomorphism, then $f$ is also a group homomorphism $(\g,\cd_\g)\to (\mk{h},\cd_{\mk{k}})$. In other words, $\g\mapsto(\g,\cd_\g)$, $f\mapsto f$ is a functor from $\mk{N}_\Q$ to $\mathcal{N}$. Moreover, this functor induces an equivalence of categories $\mk{N}_\Q\to\mathcal{N_\Q}$. 
\end{thm}

The contents of the last statement is that 
\begin{itemize}
\item any group homomorphism $(\g,\cd_\g)\to (\mk{h},\cd_{\mk{h}})$ is a Lie algebra homomorphism;
\item any uniquely divisible nilpotent group $(G,\bullet)$ has a unique $\Q$-Lie algebra structure $\g=(G,+,[\cdot,\cdot])$ such that $\cd_\g=\bullet$.
\end{itemize}

Recall that an element in a group is {\em divisible} if it admits $n$-roots for all $n\ge 1$.

\begin{lem}[see Lemma 3 in \cite{Howe}]\label{divsub}
In a nilpotent group, divisible elements form a subgroup.
\end{lem}

\begin{lem}[see Theorem 14.5 in \cite{Ba}]\label{dsp}
Let $G$ be nilpotent and uniquely divisible, with lower central series $(G^n)$. Then $G/G^n$ is torsion-free (hence uniquely divisible) for all $n$. \qed
\end{lem}

The following lemma is needed in the proof of Theorem \ref{oabels}.

\begin{lem}\label{freenil}
In the category of $s$-nilpotent groups, any free product of uniquely divisible groups is uniquely divisible.
\end{lem}
\begin{proof}
First, the free groups in this category are torsion-free: to see this, it is enough to consider the case of a free group of finite rank in this category; such a group is of the form $F/F^i$ with $F$ a free group; it is indeed torsion free for all $i$: this is a result about the lower central series of a non-abelian free group and is due to Magnus \cite{Mag} (see also \cite[IV.6.2]{Ser}).

Let $G_1,G_2$ be torsion-free uniquely divisible $s$-nilpotent groups and $G_1\ast_sG_2$ their $s$-nilpotent free product, which is divisible by Lemma \ref{divsub}. Denote by $(N^i)$ the lower central series of $G_1\ast G_2$. 
Let $g$ be a non-trivial element in $G_1\ast_s G_2=(G_1\ast G_2)/N^{s+1}$ and let us show that $g$ is not torsion in this group. By \cite{Ma2}, a free product of torsion-free nilpotent groups is residually torsion-free nilpotent, and therefore there exists $t\ge s+1$ such that the image of $g$ is not torsion in $(G_1\ast G_2)/N^t$. Applying Lemma \ref{dsp} to $(G_1\ast G_2)/N^t$, we see that $(G_1\ast G_2)/N^{s+1}$ is torsion-free, so since $g$ is non-trivial, it is not torsion. So $G_1\ast_sG_2$ is uniquely divisible by Lemma \ref{tfdud}.
\end{proof}

\subsection{Powers and commutators in nilpotent groups}\label{powcom}
This subsection includes several lemmas, which will be used at various places of the paper.

Denote by $\lp\cdot,\cdot\rp$ group commutators, namely
$$\lp x,y\rp=x^{-1}y^{-1}xy,$$
and $n$-fold group commutators
\begin{equation}\lp x_1,\dots,x_n\rp=\lp x_1,\lp x_2,\dots,x_n\rp\rp.\label{itcom}\end{equation}
Define similarly $n$-fold Lie algebra brackets. When $n\ge 2$ is not specified, we just call them iterated group commutators or Lie algebra brackets.

If $G$ is a group, its {\em lower central series} is defined by $G^1=G$ and $G^{i+1}=\lp G,G^i\rp$ (the group generated by commutators $\lp x,y\rp$ when $(x,y)$ ranges over $G\times G^i$). The group $G$ is $s$-{\em nilpotent} if $G^{s+1}=\{1\}$. In particular, 0-nilpotent means trivial, 1-nilpotent means abelian, and more generally, $s$-nilpotent means that $(s+1)$-fold group commutators vanish in $G$. Similarly, if $\g$ is a Lie algebra, its lower central series is defined in the same way (and does not depend on the ground commutative ring), and $s$-nilpotency has the same meaning.

The following lemma will be used in the proof of Lemma \ref{swt}.

\begin{lem}\label{l_xyi}
Let $N$ be an $s$-nilpotent group, and let $i$ be an integer. Then there exists an integer $m=m(i,s)$ such that for all $x,y\in N$, we can write
\[\lp x,y^i\rp=w_1\ldots w_m,\]
where each $w_j$ $(1\le j\le m)$ is an iterated commutator (or its inverse) whose letters are $x^{\pm 1}$ or $y^{\pm 1}$, that is, $w_j=\lp t_{j,1},\dots,t_{j,k_j}\rp$ for some $k_j\ge 2$ and $t_{j,i}\in\{x^{\pm 1},y^{\pm 1}\}$.
\end{lem}

\begin{ex}In a 3-nilpotent group $N$, we have, for all $x,y\in N$ and $i\in\Z$
\[\lp x,y^i\rp=\lp x,y\rp^i\lp y,x,y\rp^{-i(i-1)/2}.\]
\end{ex}

\begin{proof}[Proof of Lemma \ref{l_xyi}]
We shall prove the lemma by induction on $s$. The statement is obvious for $s=0$ (i.e.\ when $N$ is the trivial group), so let us suppose $s\geq 1$. Applying the induction hypothesis modulo the $s$-th term of the descending series of $N$, one can write $\lp x,y^i\rp=wz$, where $w$ has the form $w_1\dots w_{m'}$ where $m'=m(i,s-1)$, and where $z$ lies in the $s$-th term of the lower central series of the subgroup generated by $x$ and $y$, which will be denoted by $H$. Since the word length according to $S= \{x^{\pm 1}, y^{\pm 1}\}$ of both $\lp x,y^i\rp$ and $w$ is bounded by a function of $i$ and $s$, this is also the case for $z$. Now $H$ is generated by the set of iterated commutators $T=\{\lp x_1, x_2\dots,x_s\rp \mid  x_1,\dots,x_s\in S\}$. Therefore, $z$ can be written as a word in $T^{\pm 1}$, whose length only depends on $s$ and $i$. So the lemma follows.
\end{proof}

The following lemma will be used in the proof of Lemma \ref{estimni}.

\begin{lem}\label{multicom}
In any uniquely divisible $s$-nilpotent group, if $x_i=\exp v_i$
\[\lp x_1,\dots,x_s\rp=\exp[v_1,v_2,\dots,v_s].\]
\end{lem}
\begin{proof}
(Refer to (\ref{itcom}) for the conventions defining iterated Lie brackets or group commutators.)
Use the convenient convention to identify the group and the Lie algebra through the exponential and with this convention, the lemma simply states that 
$$\lp x_1,\dots,x_s\rp=[x_1,x_2,\dots,x_s]\qquad\forall x_1,\dots,x_s\in G.$$
We prove the result by induction on $s$; it is trivial for $s=1$ (abelian groups); assume it holds for $s-1$.
By induction, $\lp x_2,\dots,x_s\rp=[x_2,\dots,x_s]+O(s)$, where $O(s)$ means some combination of $s$-fold Lie algebra brackets.

It follows from the Baker-Campbell-Hausdorff formula that if $[x,y]$ is central then $\lp x,y\rp=[x,y]$. We can apply this to $x=x_1$ and $y=\lp x_2,\dots,x_s\rp$. This yields
\[\lp x_1,\lp x_2,\dots,x_s\rp\rp=[x_1,[x_2,\dots,x_s]+O(s)]=[x_1,\dots,x_s].\qedhere\]
\end{proof}

The following proposition is the key new feature in the proof of Theorem \ref{oabels}.

\begin{prop}\label{amaldiv}
Let $x,y$ be elements of a uniquely divisible nilpotent group $G$. Let $N$ be the normal subgroup generated by the elements of the form $x^ry^{-r}$, where $r$ ranges over $\Q$. Then $N$ is divisible. Equivalently, $G/N$ is torsion-free.
\end{prop}

We need the following lemma.

\begin{lem}\label{normrac}
Let $x,y$ be elements of a uniquely divisible nilpotent group $G$, and $n$ an integer. Then $(xy^{-1})^{1/n}$ is contained in the normal subgroup generated by $\{x^{1/k}y^{-1/k}:k\in\Z\}$.
\end{lem}
\begin{proof}
Let $G$ be $s$-nilpotent. By \cite[Lemma~5.1]{BG} (see however Remark \ref{roundabout}), there exists a sequence of rational numbers $a_1,b_1,\dots,a_k,b_k$ such that in any uniquely divisible $s$-nilpotent group $H$ and any $u,v\in H$, we have
$$(uv^{-1})^{1/n}=u^{1/n}v^{-1/n}\prod_{i=1}^k u^{a_i}v^{b_i}.$$ 
In particular, picking $(H,u,v)=(\R,1,\sqrt{2})$, we see that $\sum a_i=\sum b_i=0$. Therefore, if $d$ is a common denominator to $a_1,\dots,b_k$ and $n$, $\prod_{i=1}^k u^{a_i}v^{b_i}$ as well as $u^{1/n}v^{-1/n}$ belong to the normal subgroup generated by $u^{1/d}v^{-1/d}$, hence $(uv^{-1})^{1/n}$ as well. In particular, this applies to $(H,u,v)=(G,x,y)$.
\end{proof}

\begin{rem}\label{roundabout}
In the above proof, we used \cite[Lemma~5.1]{BG} to be concise. However, this is not very natural, because the latter is proved using the Hall-Petrescu formula; the problem is that in this formula, exponents are put outside the commutators. The proof of the Hall-Petrescu formula can easily be modified to prove by induction on the degree of nilpotency a similar formula with exponents inside the commutators. In \cite{BG}, in order to shorten the argument (as we also do), instead of processing this induction, they work with a much simpler induction based on the Hall-Petrescu formula; this is very unnatural, because if we do not allow ourselves to use the Hall-Petrescu formula, to go through the latter is very roundabout; moreover the exponents $a_k,b_k$ obtained in \cite{BG} depend on $s$, while in a direct induction, we pass from the $s$-nilpotent case to the $(s+1)$-nilpotent case by multiplying on the right by some suitable iterated commutator of powers.
\end{rem}

\begin{proof}[Proof of Proposition \ref{amaldiv}]
By Lemma \ref{normrac}, $N$ contains elements $(x^ry^{-r})^\rho$ for any $\rho\in\Q$. So $N$ is generated as a normal subgroup by the divisible subgroups $N_r=\{(x^ry^{-r})^\rho:\rho\in\Q\}$, and therefore $N$ is divisible by Lemma \ref{divsub}.
\end{proof}

The following lemma is used in the proof of Lemma \ref{431}.

 \begin{lem}\label{glll}Let $\g$ be a Lie algebra over the commutative ring $R$. Let $\mk{m}$ be a generating $R$-submodule of $\g$. Define 
$\g^{[1]}=\mk{m}$, and by induction 
the submodule $\g^{[i]}=[\g^{[1]},\g^{[i-1]}]$ for $i\ge 2$ (namely, the submodule generated by the brackets of the given form). Let $(\g^i)$ be the lower central series of $\g$. Then for all $i$ we have $\g^i=\sum_{m\ge i}\g^{[m]}$.
\end{lem}
\begin{proof}

Let us check that
\begin{equation}\label{crocro}\big[\g^{[i]},\g^{[j]}\big]\subset \g^{[i+j]}\qquad\forall i,j\ge 1.\end{equation}

Note that (\ref{crocro}) holds when either $i=1$ or $j=1$. We prove (\ref{crocro}) in general by induction on $k=i+j\ge 2$, the case $k=2$ being already settled. So suppose that $k\ge 3$ and that the result is proved for all lesser $k$. We argue again by induction, on $\min(i,j)$, the case $\min(i,j)=1$ being settled. Let us suppose that $i,j\ge 2$ and $i+j=k$ and let us check that (\ref{crocro}) holds. We can suppose that $j\le i$. By the Jacobi identity and then the induction hypothesis
\begin{align*}[\g^{[i]},\g^{[j]}]= &[\g^{[i]},[\g^{[1]},\g^{[j-1]}]]\\\subset & [\g^{[1]},[\g^{[i]},\g^{[j-1]}]]+[\g^{[j-1]},[\g^{[1]},\g^{[i]}]]\\
\subset &[\g^{[1]},\g^{[i+j-1]}]+[\g^{[j-1]},\g^{[1+i]}]\\
\subset &\g^{[i+j]}+[\g^{[j-1]},\g^{[1+i]}]
\end{align*}
Since $j\le i$, we have $\min(j-1,i+1)<\min(i,j)$, the induction hypothesis yields 
$[\g^{[j-1]},\g^{[1+i]}]\subset\g^{[i+j]}$ and (\ref{crocro}) is proved.

It follows from (\ref{crocro}) that, defining for all $i\ge 1$
$$\g^{\{i\}}=\sum_{j\ge i}\g^{[j]},$$
the graded submodule $\g^{\{i\}}$ is actually a Lie subalgebra of $\g$. Let us check by induction on $i\ge 1$ that $\g^i=\g^{\{i\}}$. Since $\g^{[1]}$ generates $\g$, it follows that $\g^{\{1\}}=\g=\g^1$. Now, suppose $i\ge 2$ and the equality holds for $i-1$. Then $\g^i$ is, in our notation, the Lie subalgebra generated by
$$[\g,\g^{i-1}] =[\g,\g^{\{i-1\}}]=\left[\g,\sum_{j\ge i-1}\g^{[j]}\right]$$ $$=\sum_{j\ge i-1}\left[\g,\g^{[j]}\right]=\sum_{j\ge i-1}\g^{[j+1]}=\g^{\{i\}};$$
since $\g^{\{i\}}$ is a Lie subalgebra we deduce that $\g^i=\g^{\{i\}}$.
\end{proof}

\subsection{Lazard's formulas}\label{su_lf}

Lazard's formulas (which will be used in \S\ref{su_2t} and \S\ref{awe}) are, in a sense, inversion formulas for the Baker-Campbell-Hausdorff formula: roughly speaking, they express the Lie algebras laws in terms of the group law; actually it involves taking roots and we find it convenient to state them by ``canceling denominators" as below. Here, $F_2$ denotes the free group on 2 generators).

\begin{thm}[Lazard \cite{Laz}]\label{t_laz}
For every $s\ge 1$, there exist group words $A_s,B_s\in F_2$ and positive integers $q_s,q'_s$, such that for every simply connected $s$-nilpotent Lie group $G$ with Lie algebra $\g$, and all $x,y\in G$, writing $X=\log(x)\in\g$, $Y=\log(y)\in\g$, 
we have 
$$\log(A_s(x,y))=q_s(X+Y)\textnormal{ and }\log(B_s(x,y))=q'_s[X,Y]$$
\end{thm}

Here A stands for Add, and B for Bracket. 

\begin{rem}\label{formalab}
We see that we have the formal equality of group words $A_s(x,1)=x^{q_s}$. Indeed, $A_s(x,1)$ has the form $x^\ell$ for some $\ell\in\Z$, and by evaluation in $\R$ we obtain $\ell=q_s$. Similarly, we have the formal equality $B_s(x,1)=1$.
\end{rem}

The usual Lie ring axioms are reflected in the following proposition:

\begin{prop}\label{lazeq}
In every $s$-nilpotent group $G$, abbreviating $A=A_s$, $B=B_s$, $q=q_s$, we have identities $\forall x,y,z\in G$,
$$A(x,y)=A(y,x);\; B(x,y)=B(y,x)^{-1}\,;$$
$$B(A(x,y),z)=A(B(x,z),B(y,z))\,;$$
$$A(A(x,y),z^q)=A(x^q,A(y,z))\,;$$
$$B(x^k,y)=B(x,y^k)=B(x,y)^k\quad\forall k\in\Z.$$
\end{prop}
\begin{proof}
If $G$ is a simply connected nilpotent Lie group, this follows from the corresponding identities in the Lie algebra: commutativity of addition, anti-commutativity of the bracket, distributivity, associativity of addition; in the last equality if follows from the fact that $\log(x^k)=k\log(x)$. Therefore this holds in every subgroup of such a group $G$, and in particular in any finitely generated free $s$-nilpotent group, and therefore in any $s$-nilpotent group, by substitution.
\end{proof}



\section{Dehn function of tame and stably 2-tame groups}\label{s_spec}

In \S\ref{ss_tg}, we introduce the class of tame groups, showing in particular that their Dehn function is at most quadratic. In \S\ref{lengthe}, we formulate and prove a version of Gromov's trick for standard solvable groups involving its tame subgroups. This result relies on some lengths estimates on nilpotent algebraic groups over local fields, which are obtained in \S\ref{lenig}.
Then \S\ref{s2t} is dedicated to the proof of Theorem \ref{ith_veryt} from the introduction (which is reformulated in terms of 2-tameness). Finally we extend this result to generalized 2-tame groups in \S\ref{s_gtame}.

\subsection{Tame groups}\label{ss_tg}
We introduce here a class of locally compact groups, called {\bf tame groups}, including tame standard solvable groups from Definition \ref{d_2ta}. 

Tame groups are, for our purposes, the best-behaved. Using a ``large-scale Lipschitz homotopy", these groups are shown to be compactly presented with an at most quadratic Dehn function and a have simple presentation (Theorem \ref{retratameb} and Corollary \ref{prestame}).

While standard solvable groups are ``usually" not tame, their geometrical study will be based on the family of their tame subgroups.

\begin{defn}We call a {\bf tame} group any locally compact group with a semidirect product decomposition $G=U\rtimes A$, where $A$ is a compactly generated abelian group, such that some element of $A$ acts as a compaction of $U$ (in the sense of Definition \ref{d_cpon}), in which case we call $U\rtimes A$ a {\bf tame decomposition}.
\end{defn}

Recall that this means that there exist $v\in A$ and a compact subset $\Omega$ (called vacuum subset) of $U$ such that $\bigcup_{n\ge 0}v^n\Omega v^{-n}=U$ uniformly on compact subsets, in the sense that for every compact subset $K$ of $U$ there exists $n$ such that $K\subset v^n\Omega v^{-n}$. If we apply this to $\Omega$ itself, so that $v^{-n}\Omega v^n\subset\Omega$, we see, by replacing $\Omega$ with the larger $\bigcup_{0\le k\le n-1}v^{-n}\Omega v^n=\bigcup_{k\ge 0}v^{-n}\Omega v^n$, that we can suppose $v^{-1}\Omega v\subset\Omega$, in which case we call $\Omega$ a {\em stable vacuum subset}.

\begin{rem}\label{r_ta}
If $G=U\rtimes A$ is a tame decomposition, and if $W$ is the largest compact subgroup in $A$, then $W$ admits a direct factor $A'$ in $A$ (isomorphic to $\R^k\times\Z^\ell$ for some $k,\ell\ge 0$), so that $G=UW\rtimes A'$; if $x\in A$ acts as a compaction on $U$ then it also acts as a compaction on $UW$, and so does the projection of $x$ on $A'$ (because the set of compactions of a given locally compact group is stable by multiplication by inner automorphisms \cite[Lemma~6.16]{CCMT}). This shows that assuming that $A=\R^k\times\Z^\ell$ is not really a restriction.
\end{rem}

We are going to prove that tame groups are compactly presented with an at most quadratic Dehn function; in order to make quantitative statements, we use the following language.

Let $v\in A$ be an element acting (by conjugation on the right) as a compaction of $U$. Let $m$ be an integer. Let $S_U$ be a compact symmetric subset of $U$ with unit. Let $T$ be a compact symmetric generating subset of $A$ with unit.

\begin{defn}\label{d_adapted}
We say that $(m,S_U,T)$ is {\bf adapted} to $U\rtimes A$ and $v$ if $v\in T$, the subset $S_U$ is a stable vacuum subset for the right conjugation $u\mapsto v^{-1}uv$ by $v$, and there exist non-negative integers $k,\ell\ge 1$ with $k+\ell\le m$ such that $v^{-k}S_U^2v^k\subset S_U$ and $(v^\ell w)^{-1}S_U(v^\ell w)\subset S_U$ for all $w\in T$.
\end{defn}

There always exists such an adapted triple; more precisely, for every symmetric generating subset $T$ of $A$ containing $\{1,v\}$ and every symmetric stable vacuum subset $S_U\subset U$ of $v$, $(m,S_U,T)$ is adapted for all $m$ large enough. Also, if $S_U$ is a stable vacuum subset for some $v\in A$ and $T$ is arbitrary, then $(2,S_U,T\cup\{v^{\pm k\}})$ is adapted to $v^k$ for large enough $k$.

\begin{thm}\label{retratameb}
Let $G$ be a tame group with a tame decomposition $U\rtimes A$, with an element $v$ acting as a compaction of $U$. Denote by $\pi$ the canonical projection 
$G\to A$.

Let $(m,S_U,T)$ be adapted to $U\rtimes A$ and $v$ in the above sense. 
Then $S=S_U\cup T$ is a compact generating subset of $G$, and $G$ being endowed with the corresponding word metric, the function
$$\gamma:G\times\N\to G;\quad (g,n)\mapsto gv^n$$
is 1-Lipschitz in each variable and satisfies: $\gamma(g,0)=g$, and $d(\gamma(g,n),v^n\pi(g))\le 1$ whenever $n\ge m|g|$.\qed
\end{thm}

The function $\gamma$ should be understood as ``locally" a large-scale homotopy between the identity of $G$ and its projection onto $A$. This is only ``local" since the time needed to reach $A$ depends on the size of the element. See however Remark \ref{asconetame}. 

\begin{proof}[Proof of Theorem \ref{retratameb}]
Since $v^{-1}S_Uv\subset S_U$, we have $v^{-1}Sv\subset S$. So, for $n\ge 0$, the automorphism $g\mapsto v^{-n}gv^n$ is 1-Lipschitz, and since the left multiplication by $v^n$ is an isometry, we deduce that $\gamma(\cdot,n)$ is 1-Lipschitz. Also, assuming that $v\in T$, it is immediate that $\gamma(g,\cdot)$ is 1-Lipschitz.

It remains to prove the last statement. 
Consider an element in $U$, of size at most $n$. We can write it as
$$g=\prod_{i=1}^{n}s_iu_i$$
with $s_i\in T$ and $u_i\in S_U$. Defining $t_i=s_i\dots s_n$, we deduce 
$$g=t_n\left(\prod_{i=1}^{n}t_iu_it_i^{-1}\right).$$
So $t_n=\pi(g)$ and
$$v^{-\ell n}gv^{\ell n}=\pi(g)\left(\prod_{i=1}^{n}v^{-\ell n}t_iu_it_i^{-1}v^{\ell n}\right).$$
By definition of $\ell$, the element $v_i=v^{-\ell n}t_iu_it_i^{-1}v^{\ell n}$ belongs to $S_U$. So 
$$v^{-\ell n}gv^{\ell n}=\pi(g)\left(\prod_{i=1}^{n}v_i\right).$$

Since $v^{-k}(S_U^2)v^k\subset S_U$, setting $j=\lceil\log_2(n)\rceil$, we have $v^{-kj}(S_U^n)v^{kj}\subset S_U$. Thus we obtain
$$v^{-\ell n-k\lceil\log_2(n)\rceil}gv^{\ell n+k\lceil \log_2(n)\rceil}\in \pi(g)S_U,$$
since for all $n$ we have $\lceil\log_2(n)\rceil\le n$ and $k+\ell\le n$, we obtain $d(gv^{mn},v^{mn}\pi(g))\le 1$.
\end{proof}

A first consequence is an efficient way of writing elements in a tame group:

\begin{cor}\label{comb}
Under the assumptions of Theorem \ref{retratameb}, for every $x\in G$ of with $|x|_S=n$, we can write $x=\pi(x)v^{mn}sv^{-mn}$ and $s\in S_U$.
\end{cor}
\begin{proof}
Define $s=(v^{mn}\pi(x))^{-1}\gamma(x,mn)$. By the theorem, $s\in S_U$, while $x=\pi(x)v^{mn}sv^{-mn}$. 
\end{proof}

The corollary allows us to introduce the following definition, which will be used in the sequel. Let $G=U\rtimes A$ be tame and $v\in A$ act as a compaction on $U$. Let $(m,S_U,T)$ be adapted to $U\rtimes A$ and $v$ in the sense of Definition \ref{d_adapted}, and $S=S_U\cup T$. By Corollary \ref{comb}, if $x\in U$ and $n=|x|_S$, the element $s=v^{-mn}xv^{mn}$ belongs to $S_U$. 

\begin{defn}\label{barx}
For $x\in U$ with $|x|_S=n$ as above, we define $\overline{x}\in F_S$ as the word $v^{mn}sv^{-mn}$, which has length $2mn+1$ and represents $x$.
\end{defn}

We now write the consequence of Theorem \ref{retratameb}, that tame groups have a ``nice" presentation and a good control on the Dehn function.

\begin{cor}\label{prestame}
Under the assumptions of Theorem \ref{retratameb}, $G$ admits a presentation with generating set $S$ in which the relators are the following
\begin{enumerate}
\item\label{re1} All relators of the form $s_1s_2s_3$ for $s_1,s_2,s_3\in S_U$, whenever $s_1s_2s_3=1$ in $U$;
\item\label{re2} all relators of the form $w s_1w^{-1} s_2$ whenever $w\in T$, $s_1,s_2\in A$ and $w s_1w^{-1}s_2=1$ in $G$;
\item\label{re3} a finite set $\mathcal{R}_3$ of defining relators of $A$ (with respect to the generating subset $T$) including all commutation relators.  
\end{enumerate}
The Dehn function of this presentation is bounded above by $mn^2+\delta_{T,\mathcal{R}_3}(n)$, where $\delta_{T,\mathcal{R}_3}$ is the Dehn function of the presentation $\langle T\mid \mathcal{R}_3\rangle$ of the abelian group $A$.
\end{cor}

\begin{cor}\label{tameq}
If $G$ is a tame locally compact group, then it has an at most quadratic Dehn function.\qed
\end{cor}

This follows since $A$ has an at most quadratic Dehn function.

\begin{rem}\label{D1}
Since $G$ has a 1-Lipschitz retraction onto $A$ (namely the canonical projection), we deduce that the Dehn function of $G$ is exactly quadratic when $A$ has rank at least 2. On the other hand, if $A$ has rank one (i.e., contains $\Z$ as a cocompact lattice), then $G$ is hyperbolic and thus its Dehn function is linear \cite{CCMT}. 
\end{rem}

\begin{proof}[Proof of Corollary \ref{prestame}]
Start from a combinatorial loop $\gamma_0=(u_i)_i$ of length $n$. This is a combinatorial loop in the Cayley graph of $(G,S)$, i.e., $d(u_i,u_{i+1})\le 1$. Since $\gamma:(g,n)\mapsto gv^n$ is 1-Lipschitz by Theorem \ref{retratameb}, each $\gamma_j=(u_iv^j)_i$ is also a loop of length $n$. Define $\mu_j$ as the loop $(v^j\pi(u_i))_i$ in $A$. We know that for $j\ge mn$, we have $d(\mu_j(i),\gamma_j(i))\le 1$ for all $i$.
So the ``homotopy" consists in
$$\gamma_0\rightsquigarrow\gamma_1\rightsquigarrow\dots\rightsquigarrow \gamma_{mn}\rightsquigarrow \mu_{mn}$$
and is finished by a homotopy from $\mu_{mn}$ to a trivial loop inside $A$. So we have to describe each step of this homotopy.

To go from $\gamma_j$ to $\gamma_{j+1}$, we use $n$ squares, each with vertices of the form $$(u_iv^j, u_{i+1}v^j,u_{i+1}v^{j+1},u_{i}v^{j+1})$$ for $i=1,\dots n$ (modulo $n$). To describe this square, we discuss whether $u_i^{-1}u_{i+1}$ belongs to $S_U$ or $T$: in the first case this is a relator of the form (\ref{re2}) and in the second case it is a relator of the form (\ref{re3}). 

To go from $\gamma_{mn}$ to $\mu_{mn}$, we use $n$ squares vith vertices $$(u_iv^{mn}, u_{i+1}v^{mn},\pi(u_{i+1})v^{mn},\pi(u_i)v^{mn}).$$ We discuss again: if $u_i^{-1}u_{i+1}\in T$, this is a relator of the form (\ref{re3}). If $u_i^{-1}u_{i+1}\in S_U$, actually $\pi(u_i)=\pi(u_{i+1})$ and this square actually degenerates to a triangle of the form (\ref{re1}).

Finally, the loop $\mu_{mn}$, of length $n$ can be homotoped within $A$ to a constant loop with area $\le \delta_{T,\mathcal{R}_3}(n)$.
\end{proof}

\begin{rem}[Asymptotic cones of tame groups are contractible]\label{asconetame}
Let $X$ be a metric space and $\omega$ a nonprincipal ultrafilter on the set of positive integers. If $(x_n)$, $(y_n)$ are sequences in $X$, define $d_\omega((x_n),(y_n))=\lim_\omega d(x_n,y_n)/n\in [0,\infty]$. The {\em asymptotic cone} $(\Cone_\omega(X),d_\omega)$ is defined as the metric space consisting of those sequences $(x_n)$ with $d_\omega((x_n),(x_0))<\infty$, modulo identification of $(x_n)$ and $(y_n)$ whenever $d_\omega((x_n),(y_n))=0$.

A straightforward corollary of Theorem \ref{retratameb} is that if $G$ is a tame locally compact group, then all its asymptotic cones are contractible. Indeed, the large-scale Lipschitz mapping $\gamma$ induces a Lipschitz map
\begin{eqnarray*}
\tilde{\gamma}:\Cone_\omega(G)\times \R_{\ge 0} & \to & \Cone_\omega(G)\\
((x_n),t) & \mapsto & \gamma(x_n,\lfloor tn\rfloor)
\end{eqnarray*}
such that $\tilde{\gamma}(x,0)=x$ and $\tilde{\gamma}(x,t)\in\Cone_\omega(A)$ if $t\ge C|x|$. Defining $h(x,t)=\tilde{\gamma}(x,t|x|)$, then $h$ is continuous, $h(x,0)=x$ and $h(x,C)\in\Cone_\omega(A)$ for all $x\in\Cone_\omega(G)$. In other words, $h$ is a homotopy between the identity of $\Cone_\omega(G)$ and a map whose image is contained in $\Cone_\omega(A)$. Since $\Cone_\omega(A)$ is (bilipschitz) homeomorphic to a Euclidean space, this shows that $\Cone_\omega(G)$ is contractible.
\end{rem}




\subsection{Length in nilpotent groups}\label{lenig}

As in Definition \ref{d_ssg}, let $G=U\rtimes A$ be a standard solvable group. We fix a decomposition $U=\prod_{j=1}^d U_{(j)}$, where $U_{(j)}=\mathbb{U}_{(j)}(\K_j)$, where $\K_j$ is some nondiscrete locally compact field. We write $\K=\bigoplus_{j=1}^d\K_j$. We say that a subgroup $V$ of $U$ is {\em $\K$-closed} if it is a direct product $\prod_{j=1}^d V_{(j)}$, where $V_{(j)}=V\cap U_{(j)}$ is Zariski-closed in $U_{(j)}$. 

The condition that $V$ is $\K$-closed is equivalent to the requirement that $\log(V)$ is a Lie $\K$-subalgebra of the Lie $\K$-algebra $\mk{u}$ of $U$. When $\K$ is $\R$ or $\Q_p$ for some prime $p$, it is just equivalent to the condition that $V$ is closed (for the ordinary topology) and divisible.

We fix on $G$ an {\bf admissible} finite family $U_1,\dots,U_\nu$ of tame subgroups of $U$, where by admissible we mean it fulfills all the following conditions:

\begin{enumerate}
\item\label{iskclo} each $U_i$ is $\K$-closed in $U$
\item $U$ is generated by $\bigcup_{i=1}^\nu U_i$.
\item each $U_i$ is normalized by $A$.
\end{enumerate}
 For instance, we can choose them to be the standard tame subgroups (there are finitely many, see Definition \ref{estam}), but in the proof of Theorem \ref{quas2t} we will make another choice.
 
We can find a compact symmetric subset $T\subset A$ containing $1$ and compact symmetric subsets $S_i$ of $U_i$ containing 1, such that for every $i$, $S_i$ is a stable vacuum subset for some element of $T$. Define $S=T\cup\bigcup_i S_i$. We assume that $S_i=S\cap U_i$ for all $i$ (we can always replace $S_i$ with $S\cap U_i$ so that it holds).

\begin{nota}\label{not:S(r)}
For $r\in[-\infty,+\infty]$ and $n\in\N$, let $T^n\subset F_S$ be the $n$-ball of $F_T$ and define $S_i^{(r)}$ as the set elements of $F_S$ of the form $tst^{-1}$, where $s\in S_i$ and $t\in T^{\max(0,\lfloor r\rfloor)}$. Let $\pi$ be the canonical projection $F_S\to G$.
\end{nota}

\begin{thm}\label{comlest}
Fix a standard solvable group $G$, an admissible family of tame subgroups and generating subset as above. There exist constants $\lambda>0$, $\mu_1,\mu_2>1$, a positive integer $\kappa$ and a $\kappa$-tuple $(q_1,\dots,q_\kappa)$ of elements in $\{1,\dots,\nu\}$ satisfying\begin{itemize}
\item for every small enough norm $\|\cdot\|$ on $\mk{u}$, 
denoting by $U[r]$ the exponential of the $r$-ball in $(\mk{u},\|\cdot\|)$, we have, for all $n$, the inclusions
\[S^n\subset U[\mu_2^n]T^n,\quad U[\mu_1^n]\subset S^{n},\quad   U[\mu_1^n]T^n\subset S^{2n}.\]
\item
the $n$-ball $S^n$ of $G$ is contained in the projection $\pi\big(S_{q_1}^{(\lambda n)}\dots S_{q_\kappa}^{(\lambda n)}T^n\big)$ (which is itself contained in the $(\kappa\lambda+1)n$-ball).
\end{itemize}
\end{thm}

We first need a general lemma.

\begin{lem}\label{estimni}
Fix an integer $s\ge 1$. Let $\K=\prod\K_\ell$ be a finite product of nondiscrete locally compact fields of characteristic zero (or of characteristic $p>s$). Let $\mk{u}$ be a $s$-nilpotent finite length Lie algebra over $\K$ (finite length just means that under the canonical decomposition $\mk{u}=\bigoplus\mk{u}_\ell$, each $\mk{u}_\ell$ is finite-dimensional over $\K_\ell$) and fix a submultiplicative norm $\|\cdot\|$ on $\mk{u}$. Let $U$ be the corresponding nilpotent topological group; if $x\in U$, write $\|x\|=\|\log(x)\|$. 

Let $(\mk{u}_i)_{1\le i\le c}$ be $\K$-subalgebras of $\mk{u}$; suppose that the $\mk{u}_i$ generate $\mk{u}$ as a $\K$-subalgebra modulo $[\mk{u},\mk{u}]$. Let $U_i=\exp(\mk{u}_i)\subset U$ be the corresponding subgroups.

Then there exists an integer $\kappa$ and a constant $K$ such that every element $x\in U$ can be written $x_1\dots x_\kappa$ with $x_k\in\bigcup_i\mk{u}_i$, and $\sup_k\|x_k\|\le K \max(\|x\|,1)$.
\end{lem}
\begin{proof}
We argue by induction on $s$. If $s=1$, the assumption is that $\mk{u}$ is abelian and generated by the $\mk{u}_i$. So there exist subspaces $\mk{h}_i\subset\mk{u}_i$ such that $\mk{u}=\bigoplus\mk{h}_i$, so if $x\in\mk{u}$ and $p_i$ is the projection to $\mk{h}_i$, then $x=\sum_{i=1}^cp_i(x)$, and if $K_0=\sup\|p_i\|$ (operator norm) then $\|p_i(x)\|\le K_0\|x\|$.

Suppose now $s>1$ and that the result is proved for $s-1$. Denoting by $(\mk{u}^i)$ the lower central series, we choose the norm on $\mk{u}/\mk{u}^{i}$ to be the quotient norm.
We use the induction hypothesis modulo $U^{s}$, so that there exist $d'$ and $K'\ge 1$ such that every $x\in U$ can be written as $y_1\dots y_{d'}z$, with $\|y_i\|\le K'\max(\|x\|,1)$ and $z\in U^{s}$. By the Baker-Campbell-Hausdorff formula, there exists a constant $C>0$ (depending on $s$ and $d'$) such that for all $x_1,\dots x_{d'+1}$ in $\mk{u}$, we have
\[\|x_1\cdots x_{d'+1}\|\le C\sup_i\max(\|x_i\|^s,\|x_i\|).\]
Since $z=y_{d'}^{-1}\dots y_1^{-1}x$, we deduce 

\begin{align*}
\|z\|\le & C\max\big(\sup_i\max(\|y_i\|^s,\|y_i\|),\max(\|x\|^s,\|x\|)\big)\\ 
\le & C\max\big(\max(K'^s,K^s\|x\|^s),\max(\|x\|^s,\|x\|)\big)=C'K'^s\max(\|x\|^s,1).
\end{align*}

Consider the $\K$-multilinear map $(\mk{u}/[\mk{u},\mk{u}])^s\to \mk{u}^{s}$ given by the $s$-fold bracket. Since its image generates $\mk{u}^{s}$ as a $\K$-submodule and since the $\mk{u}_i$ generate $\mk{u}$ modulo $[\mk{u},\mk{u}]$, we can find a finite sequence $(i_j)_{1\le j\le j_0}$ and a fixed finite family $(\kappa_{jk})_{1\le j\le j_0,1\le k\le s}$ with $\kappa_{jk}\in\mk{u}_{i_j}$ such that, setting $\zeta_j=[\kappa_{j1},\dots,\kappa_{js}]$ we have $\mk{u}^{s}=\bigoplus_j\K\zeta_j$ (we can normalize so that $\|\zeta_j\|=1$). Denote $C'=\sup_{j,k}\|\kappa_{jk}\|$.
 Let $K''$ be the supremum of the norm of projections onto the submodules $\K\zeta_j$.
 
Then we can write $z=\sum_{j}\lambda_j\zeta_j$ with $\lambda_j\in\K$, satisfying $\sup_j|\lambda_j|\le K''\|z\|$.

If $\mathbf{L}$ is a normed field, denote $\beta(\mathbf{L})=\inf\{|x|:|x|\in\mathbf{L},|x|>1\}$. Define $\alpha_0=\sup_\ell\beta(\K_\ell)\ge 1$ (it only depends on $\K$).

Hence for any element $\lambda\in\K$, there exist $\mu_1,\dots,\mu_s\in\K$ with $\sup_k|\mu_k|\le |\lambda|^{1/s}\alpha_0$, such that $\lambda=\prod_{k=1}^s\mu_k$. We apply this to write $\lambda_j=\prod_{k=1}^s\mu_{jk}$, with $|\mu_{jk}|\le |\lambda_j|^{1/s}\alpha_0$.

Now identify the Lie algebra and the group through the exponential map.
By Lemma \ref{multicom}, we have
$$\lambda_j\zeta_j=\lp\mu_{j1}\kappa_{j1},\dots,\mu_{js}\kappa_{js}\rp.$$
We have 
\[|\mu_{jk}\kappa_{jk}|\le \alpha_0C'{K''}^{1/s}\|z\|^{1/s}\le  \alpha_0C'{(CK'')}^{1/s}\max(K'\|x\|,{K'}^{1/s}\|x\|^{1/s})\] and
\[x=y_1\dots y_d\prod_j\lp\mu_{j1}\kappa_{j1},\dots,\mu_{js}\kappa_{js}\rp,\]
which is a bounded number of terms in $\bigcup U_i$, each with norm \[\le 
\max(A\|x\|,B\|x\|^{1/s}),\]
where $A=K'\max(\alpha_0C'{(CK'')}^{1/s},1)$ and $B=\alpha_0C'{(CK''K')}^{1/s}$.
\end{proof}

\begin{proof}[Proof of Theorem \ref{comlest}]
We can identify $U$ with its Lie algebra $\mk{u}$ through the exponential map. Use an embedding of $U$ into matrices over $\K$ to define a matrix norm $\|\cdot\|'$ on $U$, which is submutiplicative. Since both the embedding of $U$ into its image and its inverse are polynomial maps, there is a polynomial control of the two norms $\|\cdot\|$ and $\|\cdot\|'$  with respect to each other, say $\|u\|\le \max(A{\|u\|'}^k,B)$ and $\|u\|'\le \max(A'{\|u\|}^{k'},B')$, with $A,A'\ge 1$, $B,B'\ge 0$.

Let us now prove the existence of $\mu_2>1$ such that the inclusion $S^n\subset U[\mu_2^n]T^n$ holds for all $n$. Any element $x$ in $S^n$ can be rewritten as a product $ua$, where $a\in T^n$ and $u$ is a product of at most $n$ elements, each of which has the form $tat^{-1}$, where $t\in T^n$ and $a\in U\cap S$. 
If $C_1\ge 1$ is an upper bound for $\|a\|$ when $a\in U\cap S$ and $\lambda_1\ge 2$ is an upper bound for the Lipschitz constant of elements of $T$ viewed as operators on $\mk{u}$, we obtain $\|tat^{-1}\|\le C_1\lambda_1^n$. Hence $\|tat^{-1}\|'\le \max(A'C_1^{k'}\lambda_1^{nk'},B')$. Since $\|\cdot\|'$ is submultiplicative, it follows that $\|u\|'\le\max(nA'C_1^{k'}\lambda_1^{nk'},n)$. Defining $\lambda_2=2A'C_1^{k'}\lambda_1^{k'}$, we obtain $\|u\|'\le\lambda_2^n$ for all $n$. In turn, we deduce $\|x\|\le\max(An^k{A'}^kC^{kk'}c^{nkk'},An^k,B)$. Thus there exists $\mu_2$ and $n_0$ such that for all $n\ge n_0$ and all $x=ua\in S^n$ we have $\|u\|\le \mu_2^n$; enlarging $\mu_2$ if necessary we can ensure $n_0=0$. Hence $S^n\subset U[\mu_2^n]T^n$ holds for all $n$.

The $U_i$ generate $U$ modulo $[U,U]$ (because $\mk{u}_0\subset [\mk{u},\mk{u}]$ and $\mk{u}_\alpha$ is contained in one of the $\mk{u}_i$ for every nonzero weight $\alpha$). Hence we can apply Lemma \ref{estimni}: there exists $c_3\ge 1$ and a $\kappa$-tuple $(q_1,\dots,q_\kappa)\in\{1,\dots,\nu\}^\kappa$ such that every $x\in U$ can be written as $\prod_{\ell=1}^\kappa x_\ell$ with $x_\ell\in U_{q_\ell}$ with $\sup_\ell\|x_\ell\|\le e^{c_3}\max(1,\|x\|)$. Then there exists $c_4$ such that for every $i$ and every $y\in U_i$, there exists $t\in T$ and $m\in\N$ such that $m\le c_4\max(1,\log\|y\|)$ and $t^myt^{-m}\in S_i$. In other words, for every $i$ and $y\in U_i$, we have $y\in S_i^{(c_4\max(1,\log\|y\|))}$.

In particular, setting $c_5=c_3c_4$
\[x_\ell\in S_{q_\ell}^{(c_4\max(1,\log\|x_\ell\|))}\subset S_{q_\ell}^{(c_4[c_3+\max(0,\log\|x\|)])}\subset S_{q_\ell}^{(c_5\max(1,\log\|x\|))};\]
thus if we set $c_5=c_3c_4$ we have 
\[x_\ell\in S_{q_\ell}^{(c_5\max(1,\log\|x\|))}\]
Hence, for every $r>0$ we have
\[U[e^{r}]\subset \prod_{\ell=1}^\kappa S_{q_\ell}^{(c_5\max(1,r))}\subset S^{\kappa c_5\max(1,r)};\]
in particular, if $C=\kappa c_5$ and $\mu=e^{C^{-1}}$,
\[U[\mu^n]\subset S^{\max(n,C)}\]
Thus $U[\mu^n]\subset S^n$ for all $n\ge C$. In particular, $U[\mu^{n-C}]\subset S^{n-C}\subset S^n$ for all $n\ge 0$. Multiplying the norm by $\mu^{-C}$ redefines $U[\mu^{n-C}]$ to be $U[\mu^n]$ and hence with this new norm we have $U[\mu^n]\subset S^n$ for all $n$.

For the second inclusion, let us go back to the inclusion of $U[e^r]$, which yields (for all $n\ge 0$), with $\lambda=c_5\log\mu_2>0$
\[U[\mu_2^n]\subset \prod_{\ell=1}^\kappa S_{q_\ell}^{(\max(c_5,\lambda n))};\]
as in the previous case, multiplying the norm allows to assume $c_5=0$.
\end{proof}



\subsection{Application of Gromov's trick to standard solvable groups}\label{lengthe}

As above, let $U_1,\dots,U_\nu$ be an admissible family of tame subgroups in $U$.
If $c$ is a positive integer and $\wp$ is a $c$-tuple of elements in $\{1,\dots,\nu\}$, we can consider the product $U_\wp=\prod_{\ell=1}^cU_{\wp_\ell}$. 

Then we can define, for $r\le\infty$, a set of null-homotopic words
\[\mathcal{U}_\wp(r)=\{w=w_1\dots w_c\mid\; \pi(w)=1,\;\forall \ell, w_\ell\in S_{\wp_\ell}^{(r)}\}\subset F_S.\]

If $R$ is a subset of the kernel of $F_S\to G$ consisting of words of bounded length, define
\[\delta_{\wp,S,R}(n)=\sup_{w\in \mathcal{U}_\wp(n)}\mathrm{area}_{S,R}(w)\in [0,\infty].\]
(The area can a priori take infinite values, since we do not assume that $R$ normally generates the kernel.)

The following theorem is an essential reduction in the study of the Dehn function of standard solvable groups.

\begin{thm}\label{grossg}
Let $G=U\rtimes A$ be a standard solvable group. Fix an admissible family $U_1,\dots,U_\nu$ of tame subgroups and generating subsets as in \S\ref{lenig}; also fix a subset $R$ of the kernel of $\pi:F_S\to G$. Let $f$ be a function such that $r\mapsto f(r)/r^\alpha$ is non-decreasing for some $\alpha>1$.
Suppose that for
every $c$ and every $c$-tuple $\wp\in\{1,\dots,\nu\}^c$ we have $\delta_{\wp,S,R}(n)\preccurlyeq f(n)$. Then $G$ is compactly presented by $\langle S\mid R\rangle$ and has Dehn function $\preccurlyeq f(n)$.
\end{thm}

\begin{rem}The proof actually even shows that the following holds, $G$, $S$ and $R$ being fixed as well $\alpha>1$: {\em There exists $c$ and a $c$-tuple $\wp$ satisfying: if for some $f$ we have $\delta_{\wp,S,R}(n)\preccurlyeq f(n)$, then $G$ is compactly presented by $\langle S\mid R\rangle$ and has Dehn function $\preccurlyeq f(n)$.} However, we will only use the result in the slightly weaker form given in Theorem \ref{grossg}.
\end{rem}

\begin{proof}
Denote $W_n=S_{q_1}^{(\lambda n)}\dots S_{q_\kappa}^{(\lambda n)}T^n$, with $\lambda,k$ as given by Theorem \ref{comlest} (we use the notation of  \ref{not:S(r)}). Hence $\pi(W_n)$ contains the $n$-ball in $G$; note that $W_n$ is contained in the $[(2\kappa\lambda+1)n+\kappa]$-ball of $F_S$, which itself is contained in the $\kappa'n$-ball, where $\kappa'=2\kappa\lambda+\kappa+1$. A simple observation is that
\[W_nW_nW_n\subset S_{q_1}^{(\lambda n)}\dots S_{q_\kappa}^{(\lambda n)}S_{q_1}^{(\lambda n+n)}\dots S_{q_\kappa}^{(\lambda n+n)}S_{q_1}^{(\lambda n+2n)}\dots S_{q_\kappa}^{(\lambda n+2n)}T^{3n};\]
in particular, setting $\lambda'=\lambda+2$, and defining $\wp_{\ell+p\kappa}=q_\ell$ for $1\le\ell\le\kappa$, $p\in\{0,1,2\}$, we have
\[W_nW_nW_n\cap\Ker(\pi)\subset S_{\wp_1}^{(\lambda' n)}\dots S_{\wp_{3\kappa}}^{(\lambda' n)}\cap\Ker(\pi)=\mathcal{U}_\wp(\lambda' n).\]

Thus the assumption implies that words in $W_nW_nW_n\cap\Ker(\pi)$ have an area $\preccurlyeq f(n)$. By Theorem \ref{gromtri}, it follows that $\langle S\mid R\rangle$ is a presentation of $G$ and has a Dehn function $\preccurlyeq f(n)$.
\end{proof}


\subsection{Dehn function of stably 2-tame groups}\label{s2t}

The following is a reformulation of Theorem \ref{ith_veryt} from the introduction.

\begin{thm}\label{quas2t}
Let $G=U\rtimes A$ be a stably 2-tame standard solvable group (see Definition \ref{d_2ta}). Then $G$ has a linear or quadratic Dehn function (linear precisely when $A$ has rank one).
\end{thm}

\begin{rem}
Theorem \ref{quas2t} is actually a corollary of Theorem \ref{mainss0}, because for a stably 2-tame group, we have $H_2(\mk{u})_0=0$ and $\Kill(\mk{u})_0=0$ (because both are subquotients of $(\mk{u}\otimes\mk{u})_0$ which is itself trivial if $\mk{u}$ is stably 2-tame). The point is that  Theorem \ref{mainss0} is considerably more difficult, since it relies on the work of \S\ref{s_awbcl} and the algebraic work of \S\ref{s_abels} and \S\ref{s:am}.
\end{rem}

\begin{proof}[Proof of Theorem \ref{quas2t}]
If $A$ has rank one, then $G$ is tame and thus hyperbolic, see Remark \ref{D1}, and thus has a linear Dehn function. Otherwise, $A$ being a retract of $G$, the Dehn function has a quadratic lower bound. 

Now let us prove the quadratic upper bound.

 We prove the quadratic upper bound by arguing by induction on the length $\ell$ of $\mk{u}$ as a $\K$-module.
If $\ell\le 2$, then $G$ is tame, hence has an at most quadratic Dehn function.

Assume now that $\ell\ge 2$ and the result is proved for lesser $\ell$.

In order to apply Theorem \ref{grossg}, we first need to choose a suitable admissible family of tame subgroups. 
Let $\mathcal{P}$ be the (finite) set of principal weights, that is, the set of weights of $\mk{u}/[\mk{u},\mk{u}]$. Given a principal weight $\alpha$; define $\mk{u}_{[\alpha]}=\bigoplus_{t>0}\mk{u}_{t\alpha}$; this is a graded Lie algebra (beware that $\mk{u}_\alpha$ need not be a Lie subalgebra). Let $\alpha_1,\dots,\alpha_\mu$ be representatives of the principal weights modulo positive collinearity; write $U_{[\alpha_i]}=\exp(\mk{u}_{[\alpha_i]})$. Then we choose $U_{[\alpha_1]},\dots,U_{[\alpha_\mu]}$ as admissible family of tame subgroups, and we let $S_{1},\ldots S_{\nu}$ be finite generating subsets of the $U_{[\alpha_i]}$'s satisfying the conditions of \S \ref{lenig}.  We let $\beta\geq 1$ be an integer such that for all $n\geq 1$, $S_1^{(n)}S_1^{(n)}\subset S_1^{(\beta n)}$ (the existence of $\beta$ follows from the second part of Theorem \ref{comlest} applied to $U_{[\alpha_1]}\rtimes A$, for the admissible family  $\{U_{[\alpha_1]}\}$).

Let $\mk{v}$ be the sum of $[\mk{u},\mk{u}]$ and all $\mk{u}_\beta$ when $\beta$ ranges over all weights not positively collinear to $\alpha_1$; then $\mk{v}$ is a graded ideal. Note that $\mk{u}=\mk{u}_{[\alpha_1]}+\mk{v}$. Define $V=\exp(\mk{v})$. 

Let $V_1,\dots,V_\nu$ be the standard tame subgroups of $V$ and denote by $\Sigma_i$ a compact subset of $V_i$ (with all the previous requirements we made for $S_i$).

Fix $k\ge 1$, $\wp\in\{1,\dots,\mu\}^k$. For positive integers $n_1,\dots,n_k$, consider an element $w$ in $F_S$ belonging to $S_{\wp_1}^{(n_1)}\dots S_{\wp_k}^{(n_k)}$, we do the following: if there are any two consecutive occurrences of $S_1$, we glue them using the inclusion $S_{1}^{(n)}S_{1}^{(m)}\subset S_{1}^{(\beta\max(m,n))}$ in $G$. The cost of this is at most quadratic since it lies in the tame subgroup $U_{[\alpha_1]}\rtimes A$. If not, we consider $i$ maximal such that $\wp_i=1$. If $i=1$ or $i$ does not exist, say we are done. Otherwise $j=\wp_{i-1}\neq 1$. We first observe that $S_{j}^{(n_j)}S_{1}^{(n_i)}$ is contained in $S_{j}^{(n_i)}S_{1}^{(n_j)}\lp S_{j}^{(n_j)},S_{1}^{(n_i)}\rp$ (inclusion in $F_S$).
Then $\alpha_1$ and $\alpha_j$ belong to a single standard tame subgroup, and hence the commutator belongs to a single standard tame subgroup of $V\rtimes A$. So the commutator can be rewritten in $V\rtimes A$, thanks to 
the second part of Theorem \ref{comlest}, as an element of $\Sigma_{q_1}^{(c(n_1+n_2))}\dots \Sigma_{q_\kappa}^{(c(n_1+n_2))}$, where $c$ and $\kappa$ only depend on $G$.  We note that doing this, we make some generators in $\bigcup\Sigma_i$ appear, but they stay at the right of any occurrence of $S_{1}$. Therefore we can reiterate some bounded number of times (say $\le k^2$) until we obtain a word $w_1w_2$ with $|w_2|\le c'n$ ($n=\sum n_i$), $w_1\in S_{1}^{(c'n)}$, and $w_2$ a word in $\bigcup_{i\ge 2}S_i\cup\bigcup_i\Sigma_i\cup T$. In particular, assuming that $w$ represents the trivial element in $G$, $w_1$ represents an element of $V$. There exists a standard tame subgroup of $U$ containing $U_{[\alpha_1]}$; then its intersection with $V$ is a standard tame subgroup, say $V_1$. So $w_1$ can be replaced with some element $w'_1$ of $\Sigma_1^{c''n}$ with quadratic cost since the word $w_1^{-1}w'_1$ lies in a tame subgroup of $G$. So from $w$ we passed with quadratic cost to a word in $V$ of length $\le C|w|$. Since by induction $V$ has an at most quadratic Dehn function, we deduce that $w$ has an at most quadratic area with respect to $n$.

Since this works for every $\wp$ (with constants possibly depending on $\wp$), we can apply Theorem \ref{grossg} to conclude that $G$ has an at most quadratic Dehn function.
\end{proof}

\subsection{Generalized tame groups}\label{s_gtame}

\begin{defn}A locally compact group is {\em generalized tame} if it has a semidirect product decomposition $U\rtimes N$ where some element $c$ of $N$ acts on $U$ as a compaction, and $N$ is nilpotent and compactly generated.
\end{defn}
Thus, the assumption that $N$ is abelian in tame groups is relaxed to nilpotent.

If $G=U\rtimes N$ is a tame generalized standard solvable group it is tempting to believe that, in a way analogous to \S\ref{ss_tg}, there is a large-scale Lipschitz deformation retraction of $G$ onto $N$. However, the proof only carries over when the element $c$ of $N$ acting as a compaction of $N$ can be chosen to be central in $N$. Unfortunately, this can not always be assumed and, in order to get an upper bound on the Dehn function, we need a more complicated approach.

\begin{thm}[The generalized tame case]\label{thm:tameN} Consider a generalized standard solvable group $G=U\rtimes N$ such that there exists $c\in N$ acting on $U$ as a compaction. Let $\delta_N$ be the Dehn function of $N$, and let $f$ be a function such that $\delta_N\preccurlyeq f$ and $r\mapsto f(r)/r^\alpha$ is non-decreasing for some $\alpha>1$. 

Then the Dehn function $\delta_G$ of $G$ satisfies $\delta_N\preccurlyeq \delta_G\preccurlyeq f$. In particular, if $f$ can be chosen to be $\approx$-equivalent to the Dehn function of $N$, then $\delta_G\approx\delta_N$.
\end{thm}

Note that in all examples we are aware of, we can indeed choose $f$ to be $\approx$-equivalent to $\delta_N$, such that $r\mapsto f(r)/r^2$ is non-decreasing (unless $N$ admits $\Z$ as a lattice, in which case $G$ is tame and has a linear Dehn function). In general, we can always fix $\alpha>1$ and consider the function $r\mapsto f(r)=r^\alpha\sup_{1\le s\le r}s^{-\alpha}\delta_N(s)$.

\begin{rem}
Theorem \ref{thm:tameN} has some similarity with a theorem of Varopoulos \cite[Main theorem, p.~57]{Var} concerning connected Lie groups. Namely, for a simply connected Lie group of the form $U\rtimes N$ with $U,N$ simply connected nilpotent Lie groups such that $N$ contains an element acting as a contraction on $N$, he proves that there exists a ``polynomially Lipschitz" homotopy from the identity of $G$ to its projection on $N$.
\end{rem}

To prove the theorem, we need the following lemma, whose proof is much more complicated than we could expect at first sight (see Remark \ref{notastrite}).

\begin{lem}\label{swt}
Relations of the form $s^wt$, where $s$ and $t$ belong to $S_U$ and where $w$ is a word of length $\simeq n$ in $S_N$ have area $\preccurlyeq \delta_N(n)$.
\end{lem}                                                                                                                                                                                                                

\begin{proof}
For the sake of readability, we first consider the (easier) case when $N$ is 2-nilpotent.

Denote $j=in$, where $i$ is an integer to be determined latter in the proof, but that will only depend on $G$ and $S$ (hence is to be considered as bounded).

Up to conjugating by a power of $c$, it is enough to evaluate the area of the relation 
$s^{wc^j}t^{c^j}$. Since conjugation by $c$ is a contraction (see Definition \ref{d_cpon}), it turns out that  $t^{c^j}=u^{-1} \in S_U$. It is straightforward to check that the relation $t^{c^j}u$ has area $\preceq j$. Hence we are left to consider the relation $s^{wc^j}u$.

Denoting $y=\lp w,c^n\rp=w^{-1}c^{-n}wc^n$, we have
\begin{equation}\label{eq:commutateur}
wc^j=c^jw \lp w,c^j\rp=c^jy^i.
\end{equation}
Moreover the area of the relation $\lp w,c^j\rp y^{-i}$ is controlled by the Dehn function of $N$, so we are reduced to compute the area of 
$$(s^{c^j})^{wy^i}u.$$

Denote $N_a$ the Zariski closure of the range of $N$ in $\Aut(U)$. The algebraic group $N_a$ decomposes as a direct product $AV$ where $A$ (resp. $V$) is semisimple (resp. unipotent).
Let us write $c=c_d c_v$ and                                                                                                                                                                                                       
$w=w_d w_v$ according to this decomposition. Endow the Lie algebra of $U$ with some norm.
The crucial observation is that $y=\lp w,c^n\rp=\lp w_v,c_v^n\rp$. It follows that the matrix norm of $y$ (acting on the Lie algebra of $U$) is at most $Cn^D$ for some $C,D$ depending only on $G$ and $S$. 

Let $K>1$ be a constant such that the matrix norm of every subword of $wy$ is at most $K^n$. Let $z$ be a prefix of $wy^i$: it is of the form $ry^k$, where $r$ is a subword of $wy$, and $k\leq i$. The matrix norm of $z$ is therefore at most $Cin^D K^{n}$.

Since $c$ acts as a contraction, one can choose $i$ be such that the matrix norm of $c^i$ is less than $K^{-2}$.  Hence the matrix norm of $c^jz$ is less than $Cin^DK^{-n}$ which is bounded by some function of $i, C, D$ and $K$. It follows that for any prefix $r$ of $c^jwy^i$ has bounded matrix norm. Let $a_q$ for $q=1,2,\ldots$ be the sequence of letters of the word $c^jwy^i$, and let  $z_q=a_1\ldots a_q$. It follows that the elements $s^{z_q}$ are bounded in $U$. Now let $t_q^{-1}$ be words in $S_U$ of bounded length representing the elements $s^{z_q}$.  We conclude by reducing successively the relations $(t_{q-1})^{a_q}t_q$ whose area are bounded. This solves the case where $N$ is 2-nilpotent. 

If $N$ is not assumed to be 2-nilpotent, then the relation $\lp a,b^i\rp=\lp a,b\rp^i$ does not hold anymore. Therefore we cannot simply replace $\lp w,c^j\rp$ by $\lp w,c^n\rp^{i}$, as we did above. We shall use instead a more complicated formula,
namely the one given by Lemma \ref{l_xyi}.
According to that lemma, one can write $\lp w,c^j\rp$ as a product of $m$ iterated commutators (or their inverses) in the letters $w^{\pm1}$ and $c^{\pm n}$. The rest of the proof is then identical to the 2-nilpotent case, replacing in the previous proof, the power of commutators $y^i=\lp w,c^n\rp^i$ by this product of (iterated) commutators. 
\end{proof}

\begin{rem}\label{notastrite}
Lemma \ref{swt} is not as trite as it may look at first sight, and fails when $N$ is not nilpotent. For instance, consider the free group $F_2=\langle s,t\rangle$, and the semidirect product $U\rtimes F_2$, where $U\simeq\R$ is written as $N=\{(u_x)_{x\in\R}\}$ is written multiplicatively, $sxs^{-1}=x^2$ for all $x\in N$ and $[t,N]=1$. Write $m_n=s^{-n}ts^n$. Then $m_nv_1m_n^{-1}v_1^{-1}$ can be shown to have an exponential area. 
\end{rem}

\begin{lem}\label{gentamees}
Consider a semidirect product of groups $G=U\rtimes N$; let $c$ be an element of $N$ and $\Omega$ a symmetric subset of $U$ such that $\Omega\Omega\subset c^{-1}\Omega c$ and $\bigcup_{n\ge 0}c^{-n}\Omega c^n=U$. Let $T$ be a symmetric generating subset of $N$ containing $\{1,c\}$, such that $\bigcup_{t\in T}t\Omega t^{-1}\subset c^{-1}\Omega c$. Then $S=\Omega\cup T$ generates $G$ and $S^n\subset (c^{-2n}\Omega c^{2n})T^n\subset S^{5n+1}$.
\end{lem}
\begin{proof}
It is immediate that $S$ generates $G$.
By an immediate induction, we have $\Omega^{2^k}\subset c^{-k}\Omega c^k$. In particular, $\Omega^n\subset c^{-\lceil\log_2n\rceil}\Omega c^{\lceil\log_2n\rceil}$. 

Consider a word $w$ of length $n$ in $S$. Then it can be rewritten as $\left(\prod_{i=1}^nt_i\omega_it_i^{-1}\right)t$, where each $t_i$ and $t$ belong to $T^n$ and each $\omega_i$ belong to $\Omega$. Thus, in $G$, we have
\[S^n\subset (c^{-n}\Omega c^{n})^nT^n= c^{-n}\Omega^n c^{n}T^n\] \[\subset c^{-n-\lceil\log_2n\rceil}\Omega c^{n+\lceil\log_2n\rceil}T^n\subset c^{-2n}\Omega c^{2n}T^n\subset S^{5n+1}\qedhere\]
\end{proof}

\begin{proof}[Proof of Theorem \ref{thm:tameN}]
Since $N$ is a Lipschitz retract of $G$, we have $\delta_N\preccurlyeq\delta_G$; let us prove that $\delta_G\preccurlyeq f$. 

We pick the 1-element family $(U_1)$, where $U_1$, as an admissible family of subgroups; we let $S_1$ be a stable vacuum subset for some element $c\in T$ acting as a compaction on $U$.

We first bound (assuming $c\in T$) the area of relations of length $\le n$, of the form $w=\left(\prod_{i=1}^3c^{-n_i}\omega_ic^{n_i}\right)w'$ with $\omega_i\in\Omega$ and $w'$ a word in $T$ (necessarily $n_i\le n$). Since $w'$ is a relation in $N$, with cost $\le f(n)$, we can replace $w'$ with the trivial word. Then after conjugation by $c^n$, we are reduced to the word $\prod_{i=1}^3c^{n-n_i}\omega_ic^{n_i-n}$. The element $c^{n-n_i}\omega_ic^{n_i-n}$ represents an element $s_i\in\Omega$, and by Lemma \ref{swt}, the cost to pass from $c^{n-n_i}\omega_ic^{n_i-n}$ to $s_i$ is $\preccurlyeq f(n)$. Hence the area of $w$ is $\preccurlyeq f(n)$.

We can apply Lemma \ref{gentamees} (if necessary, replace $c$ with a large power $c^k$ of it and replace $T$ with $T\cup\{c^{\pm k}\}$ for its hypotheses to be fulfilled. Thanks to the length estimates of Lemma \ref{gentamees}, we can use Gromov's trick (Theorem \ref{gromtri}) to conclude that the Dehn function of $G$ is $\preceq f$.
\end{proof}

We also need a generalization of Theorem \ref{comlest}. We define an admissible family of subgroups as in \ref{lenig} ($A$ being replaced with $N$), and let $T$ be a compact symmetric generating subset of $N$. Since the proof of the following theorem is an immediate adaptation of that of Theorem \ref{comlest}, we omit the proof.

\begin{thm}\label{comlestg}
Fix a generalized standard solvable group $G$, an admissible family of tame subgroups and generating subset as above. There exists a constant $\lambda>0$, a positive integer $\kappa$ and a $\kappa$-tuple $(q_1,\dots,q_\kappa)$ of elements in $\{1,\dots,\nu\}$ such that for every $n$, the $n$-ball of $G$ is contained in the projection $\pi\big(S_{q_1}^{(\lambda n)}\dots S_{q_\kappa}^{(\lambda n)}T^n\big)$.\qed
\end{thm}


\section{Estimates of areas using algebraic presentations}\label{s_ea}

In this section, we provide a new method allowing to obtain upper bounds on the Dehn function of groups with ``enough algebraic structure". 

The basic idea is to consider a group $G$ given by an ``algebraic presentation", that is, it is isomorphic to the quotient of a free product $\Conv_{i=1}^\nu\mathbb{U}_i(\K)$ of algebraic groups $\mathbb{U}_i$ over some locally compact field $\K$ (or, more generally, a finite product of such fields), by the subgroup generated by finitely many ``algebraic families of relators".  An algebraic family of relators is an algebraic subvariety of some $\mathbb{U}_{i_1}\times\dots\times \mathbb{U}_{i_\ell}$, where $(i_1,...,i_\ell)$ is a fixed sequence of integers from ${1,..,\nu}$, which is interpreted as a set of words of length $\ell$ of the free product. 

In order to get estimates, we need to suppose that $G$ itself is of the form $\mathbb{G}(\K)$, and assume that the above presentation holds in a very strong sense. Namely, we need to assume that for every commutative $\K$-algebra $\A$, the group $\mathbb{G}(\A)$ is the quotient of $\Conv_{i=1}^\nu\mathbb{U}_i(\A)$ by the relators (evaluated in $\A$).

Our main result consists in controlling the size and word length of a set of relations that belong to a same algebraic family. More precisely, consider some algebraic family $\mathbb{M}$ of relations: namely a subvariety of some $\mathbb{U}_{j_1}\times\dots\times \mathbb{U}_{j_k}$ such that for every $\K$-algebra $\A$, the $\A$-points of this variety are mapped to the neutral element of $G_{\A}$. Then there exists an integer $N$ such that every element $x=(x_1,\ldots,x_k)$ of $\mathbb{M}$ can be written as a word of length at most $N$ in $\Conv_{i=1}^\nu\mathbb{U}_i(\A)$ consisting of a product of conjugates of relators. Moreover the norms of the letters of this word are polynomially controlled by the norms of the $x_i$'s.

\subsection{Affine norm on varieties over normed fields}\label{affineleng}

\subsubsection{Norms}

In this paragraph, let $(\K,|\cdot|)$ be a normed field. The affine space $\mathbb{A}^d(\K)=\K^d$ is endowed with the sup norm $\|\cdot\|$. 

Let $\mathbb{X}$ be an affine $\K$-variety with basepoint $x_0$ (a fixed element of $\mathbb{X}(\K)$). For every $d$ and every pointed closed $\K$-embedding $\phi:\mathbb{X}\to\mathbb{A}^d$ (pointed means mapping $x_0$ to 0), we can define a ``length function" on $\mathbb{X}(\K)$ by
$$\ell_\phi(x)=\|\phi(x)\|.$$
The following proposition asserts that up to polynomial distortion, this length is unique.

\begin{prop}\label{indepoly}
Let $\phi,\psi$ be two choices of pointed embeddings. Then for some positive constants $C,c>0$,
$$\ell_\psi(x)\le C\max(\ell_\phi(x)^c,1),\quad \forall x\in \mathbb{X}(\K)$$
\end{prop}
\begin{proof}
If the embeddings are into $\mathbb{A}^d$ and $\mathbb{A}^{d'}$, with say $d'\le d$, using that the trivial embedding $\mathbb{A}^{d'}(\K)\subset\mathbb{A}^d$ is isometric, we can enlarge $d'$ to assume that $d=d'$. Now consider the isomorphism $\eta=\psi\circ\phi^{-1}:\phi(\mathbb{X})\to\psi(\mathbb{X})$. Then $\eta$ extends to a regular map $\eta_1:\mathbb{A}^d\to\mathbb{A}^d$: indeed, first view $\eta$ as a regular map $\phi(\mathbb{X})\to\mathbb{A}^d$, which is the same as the data of $d$ regular maps $\phi(\mathbb{X})\to\mathbb{A}^1$ and remind that regular maps on $\phi(\mathbb{X})$ are by definition restrictions of regular maps of the affine space. Then, viewing $\mathbb{A}^d$ embedded into $\mathbb{A}^{2d}$ as the left factor $\mathbb{A}^d\times\{0\}$, we can extend $\eta$ to a $K$-automorphism
\begin{eqnarray*}
\tilde{\eta}: \mathbb{A}^{2d} & \to & \mathbb{A}^{2d} \\
(x,y) & \mapsto & (\eta_1(x)+y,x)
\end{eqnarray*}
Now if $c$ is the total degree of $\tilde{\eta}$, there exist a positive constant $C$ such that for all $u\in\mathbb{A}^{2d}(K)$ we have 
$$\|\tilde{\eta}(u)\|\le C\max(\|u\|^c,1);$$
therefore for all $x\in \mathbb{X}(\K)$ we have 
$$\ell_\psi(x)=\|\psi(x)\|=\|\eta\circ\phi(x)\|\le C\max(\|\phi(x)\|,1)=C\max(\ell_\psi(x)^c,1).\qedhere$$
\end{proof}

Note that the analogue for distances is not true: for instance if we take the disjoint union of two lines $\K\times\{0,1\}$, setting $\phi(x,t)=(x,t)$ (two parallel lines) and $\psi(x,t)=(x,tx^2)$ (a line and a parabola), then considering the sequence of pairs of points $(n,0)$ and $(n,1)$ we see that $\|\psi(x)-\psi(y)\|$ cannot be bounded in terms of $\|\phi(x)-\phi(y)\|$. 

\begin{ex}\label{aflino}
Assume that $\K$ has characteristic zero. Let $U$ be a unipotent group over $\K$. Given an embedding $\psi$ of $U$ as a Zariski-closed subgroup of $\SL_d$, and fixing a norm on the algebra of matrices, we obtain two norms on $U$: the one given by this embedding, and the one given by the embedding $\bar{\psi}$ of the Lie algebra $\mk{u}$ into $\mk{sl}_d$, pulled back to $U$ through the exponential. Each of these two norms on $U$ is polynomially bounded by the other, by Proposition \ref{indepoly}.
\end{ex}

\begin{rem}
We can define the logarithmic length on $\mathbb{X}(\K)$ associated to $\phi$ as
$$L_\phi(x)=\log(1+\ell_\phi(x)).$$
As a consequence of Proposition \ref{indepoly}, for any two choices $\psi$ and $\phi$, we have $L_\psi\simeq L_\phi$, i.e.~there exist positive constants $c\ge 1$ and $C\ge 0$ such that
$$c^{-1}L_\phi(x)-C\le L_\psi(x)\le cL_\phi(x)+C,\quad \forall x\in\mathbb{X}(\K).$$
\end{rem}

\subsubsection{Ring of polynomial growth functions}\label{pgf}

Let $\K$ be a normed field.
Endow $\K$ with the supremum norm and consider the algebra 
$\mathcal{P}_Y=\mathcal{P}_Y(\K)$ of functions from $Y\times\R_{\ge 0}$ to $\K$ that have most polynomial growth with respect to the real variable, 
uniformly in $Y$, that is, satisfying
$$\exists c,\alpha>0,c'\in\R,\;\forall t\ge 0,\;\forall y\in Y,\;\|f(y,t)\|\le ct^\alpha+c'.$$ 

\begin{lem}\label{cabijk}
Let $Y$ be a set. Then for any $\K$-affine variety $\mathbb{X}$, there are canonical bijections
\[\mathbb{X}(\K^Y)\stackrel{\sim}\to\mathbb{X}(\K)^Y,\quad 
\mathbb{X}(\mathcal{P}_Y(\K))\stackrel{\sim}\to\mathcal{P}_Y(\mathbb{X}(\K))\]
\end{lem}
\begin{proof}
The first equality is formal: if $\A=\K[\mathbb{X}]$, then, $\Hom$ being in the category of $\K$-algebras, we have
\[\mathbb{X}(\K^Y)=\Hom(\K[\mathbb{X}],\K^Y)=\Hom(\K[\mathbb{X}],\K)^Y=\mathbb{X}(\K)^Y.\]
Given a $\K$-closed embedding of $\mathbb{X}$ into the $d$-dimensional affine space, it is just given by the function $\mathbb{A}^d(\K^Y)\to \mathbb{A}^d(\K)^Y$
\[(f_1,\dots,f_d)\stackrel{\Phi}\mapsto \Big(y\mapsto (f_1(y),\dots,f_d(y))\Big).\]
Then $\Phi(\mathbb{X}(\K^Y))=\mathbb{X}(\K)^Y$ and $\Phi(\mathbb{A}^d(\mathcal{P}_Y(\K)))=\mathcal{P}_Y(\mathbb{A}^d(\K))$, and hence
\[\mathcal{P}_Y(\mathbb{X}(\K))=\mathbb{X}(\K)^Y\cap \mathcal{P}_Y(\mathbb{A}^d(\K))\qquad\qquad\] \[\qquad\qquad=\Phi(\mathbb{X}(\K^Y)\cap \mathbb{A}^d(\mathcal{P}_Y(\K)))=\Phi(\mathbb{X}(\mathcal{P}_Y(\K))).\qedhere\]
\end{proof}

\subsubsection{Remark}

All the results from this \S\ref{affineleng} immediately to the case of $\mathbb{X}(\K)$ when $\K$ is a finite product of normed fields (endowed with the sup norm). The only difference lies in the language: $\K$-affine variety should be replaced with: affine scheme of finite type over $\K$; if $\K=\prod_{j=1}^k\K_j$, this is just the data of one $\K_j$-affine variety over $\K_j$ for each $j$.

\subsection{Area of words of bounded combinatorial length}\label{s_awbcl}

\subsubsection{The setting: generators (*)}

By $\K$ we mean a finite product of normed fields (in a first reading, we can assume it is a single normed field). We recall the reader familiar with varieties but not schemes that it is not much here: over a single field, ``scheme of finite type" can be thought of as ``variety" and ``affine group scheme of finite type" as ``linear algebraic group". Over a finite product $\prod\K_j$ of fields, the datum of a scheme is just the datum of one scheme over each of the fields $\K_j$. Nevertheless, we stick to the word ``scheme" because it is the only rigorous setting in which the assertions we state are precise.

Let $\mathbb{U}_1,\dots,\mathbb{U}_\nu$ be affine $\K$-group schemes of finite type. Write $U_i=\mathbb{U}_i(\K)$. 

We need to introduce some ``norm function" on each $U_i$. For this, as explained in \S\ref{affineleng}, we fix a closed $\K$-embedding of $\mathbb{U}_i$ into $\SL_q$, so that the norm, written $\|u\|$, of any element $u$ of $U_i$ makes sense (using the operator norm on $q\times q$ matrices, where $\K^q$ is endowed with the sup norm).

\subsubsection{The setting: relators (**)}

We are going to introduce some functors $\mathbb{X}$ from the category of commutative (associative unital) $\K$-algebras to the category of sets, denoted by $\A\mapsto\mathbb{X}[\A]$, we use brackets rather than parentheses to emphasize that these objects are possibly not representable by a scheme.

We need to consider words of some given length whose letters belong to the disjoint union of the $U_i$. To do so, we fix integers $1\le \omega_1,\dots,\omega_{|\omega|}\le \nu$. Hence we can consider the product $\mathbb{U}_\omega=\prod_{\ell=1}^{|\omega|}\mathbb{U}_{\omega_\ell}$. Thus

\[\mathbb{U}_\omega(\A)=\prod_{\ell=1}^{|\omega|}\mathbb{U}_{\omega_\ell}(\A)\]
is an obvious way to parameterize the set of words $u_{\omega_1}\dots u_{\omega_{|\omega|}}$ with $u_{\omega_\ell}\in\mathbb{U}_{\omega_\ell}(\A)$ for all $\ell$.

We are going to consider closed subschemes of $\mathbb{U}_{\omega}$, which can then be used to parameterize sets of words. Formally, this is done as follows. Let $\mathbb{H}[\A]$ be the free product $\Conv_{i=1}^\nu\mathbb{U}_i(\A)$. There is an obvious product map 

\[\pi^\omega_{\A}:\mathbb{U}_{\omega}(\A)\to\mathbb{H}[\A],\quad (u_1,\dots,u_{|\omega|})\mapsto u_1\dots u_{|\omega|}.\]

Now fix finitely many closed subschemes $\mathbb{R}_1,\dots,\mathbb{R}_{\xi}\subset\mathbb{U}_\omega$ (they will play the role of algebraically parameterized relators). For convenience, write $\mathbb{R}[\A]=\mathbb{R}_1(\A)\cup\dots\cup\mathbb{R}_\xi(\A)$ for any $\K$-algebra $\A$ (if $\K$ were a field, it would be representable by a closed subscheme and we could then assume $\xi=1$; anyway this is not an issue); clearly the $\pi^\omega_{\A}$ together define a map $\pi_{\A}:\mathbb{R}[\A]\to\mathbb{H}[\A]$.

 Define $\mathbb{Q}[\A]$ as the quotient of $\mathbb{H}[\A]$ by the normal subgroup generated by $\pi_{\A}(\mathbb{R}[\A])$.
Informally, $\mathbb{Q}$ is a group generated by algebraic generators and algebraic sets of relators. A priori, $\mathbb{Q}$ is not representable by a group scheme over $\K$ (for instance, if $\mathbb{R}$ is empty, $\mathbb{Q}[\A]$ is the free product $\mathbb{H}[\A]$).

\subsubsection{A family of relations (***)}\label{def_lp}

Now given an integer $c\ge 0$ and a $c$-tuple $\wp=(\wp_1,\dots,\wp_c)$ of integers in $\{1,\dots,\nu\}$, so we have the product \[\mathbb{U}_\wp(\A)=\prod_{i=1}^c\mathbb{U}_{\wp_i}(\A)\]
and we define
\[\mathbb{L}_\wp[\A]=\left\{(f_1,\dots,f_c)\in\mathbb{U}_\wp(\A):\;f_1\dots f_c\equiv 1\textnormal{ in }\mathbb{Q}[\A]\right\}.\]
(It should rather be denoted $\mathbb{L}_{\mathbb{R},\wp}[\A])$ but since $\mathbb{R}$ is fixed we omit it in the notation.)
In other words, denoting by $\mathbb{N}[\A]$ the normal subgroup of $\mathbb{H}[\A]$ generated by $\mathbb{R}[\A]$, we have $\mathbb{L}_\wp[\A]=(\pi^\wp_\A)^{-1}(\mathbb{N}[\A])$.

We cannot a priori represent $\mathbb{L}_\wp$ as a $\K$-closed subscheme of $\mathbb{U}_\wp$. We obtain results in case it turns out to be a closed subscheme, or more generally when it contains a $\K$-closed subscheme. Namely, let  
$\mathbb{M}\subset\mathbb{U}_\wp$ be a $\K$-closed subscheme and assume that it is contained in $\mathbb{L}_\wp$, in the sense that for every (reduced) commutative $\K$-algebra $\A$ we have $\mathbb{M}(\A)\subset\mathbb{L}_\wp[\A]$ (equality as subsets of $\mathbb{U}_\wp(\A)$). Also assume that the unit element of $\mathbb{U}_\wp(\K)$ belongs to $\mathbb{M}(\K)$ (this is no restriction since it belongs to $\mathbb{L}_\wp[\K]$).

Define $\mathbb{V}[\A]$ as the union of all $\mathbb{U}_j(\A)$ in $\mathbb{H}[\A]$.

The following lemma roughly says that any element
$x=(x_1,\dots,x_c)\in\mathbb{M}(\K)$ can be written as a product of
boundedly many conjugates of relators whose coefficients are controlled by
polynomials in the size of $x$.

\begin{lem}\label{superlemma}
Fix $\mathbb{U}_1,\dots,\mathbb{U}_{\nu}$ as in (*), $\omega$ and $\mathbb{R}$ as in (**). Given any $\wp$ and any family $\mathbb{M}$ of relations as in (***), there exist positive integers $m,\mu,\alpha$ satisfying the following: for all $x=(x_1,\dots,x_c)\in\mathbb{M}(\K)$, there exist

\begin{itemize}
\item  elements $\rho_{k\ell}(x)\in\mathbb{U}_{\omega_\ell}(\K)$, $1\le k\le m$, $1\le\ell\le|\omega|$, such that setting 
\[\tilde{\rho_k}(x)=(\rho_{k1}(x),\dots,\rho_{k|\omega|}(x))\in \left(\prod_{\ell=1}^{|\omega|}\mathbb{U}_{\omega_\ell}\right)(\K),\] we have $\tilde{\rho_k}(x)\in\mathbb{R}[\K]$ for every $k$;
\item elements $h_{k\ell}(x)$ in $\mathbb{V}[\K]$, $1\le k\le m$, $1\le\ell\le\mu$,
\end{itemize}
satisfying the inequalities
\[\|h_{k\ell}(x)\|,\|\rho_{k\ell}(x)\|\le (2+\|x\|)^\alpha-2,\quad\forall k,\ell;\]
and such that, setting $$\rho_k(x)=\pi_{\K}(\tilde{\rho_k}(x))=\prod_{\ell=1}^{|\omega|} \rho_{k\ell}(x),\quad h_k(x)=\prod_{\ell=1}^\mu h_{k\ell}(x),$$ we have the equality in $\mathbb{H}[\K]$:
\begin{equation}\label{x1xc2}x_1\dots x_c=\prod_{k=1}^mh_k(x)\rho_k(x)h_k(x)^{-1}.\end{equation}
\end{lem}

\begin{rem}
That for every $x$ we can write an equality as in (\ref{x1xc2}) follows directly from the fact that $\mathbb{M}(\K)\subset\mathbb{L}_\wp[\K]$, but this gives such an equality with $m,\mu$ depending on $x$, and with an efficient upper bound on the norm of relators and conjugating elements. Lemma \ref{superlemma}	 is a uniform version of this fact, and relies on the stronger inclusion
$\mathbb{M}(\A)\subset\mathbb{L}_\wp[\A]$ for some well-chosen $\K$-algebra $\A$. 
While a suitable choice of $\A$ would be the algebra $\K^{\exp}$ (as in the sketch of proof of \S\ref{s_skpr}), we find more convenient to work with the algebra $\mathcal{P}_Y(\K)$ introduced in \S\ref{pgf}. The difference is unessential, and especially allows us to avoid to argue by contradiction as we did in \S\ref{s_skpr}.
\end{rem}

\begin{proof}[Proof of Lemma \ref{superlemma}]
Let us fix an abstract set $Y$; we will assume that its cardinal is at least equal to that of $\K$. The ring $\mathcal{P}_Y(\K)$ (\S\ref{pgf}) can be described as the set of functions $Y\times\R_{\ge 0}\to\K$ satisfying
$$\exists\alpha>0,\;\forall t\ge 0,\;\forall y\in Y,\;\|f(y,t)\|\le (2+t)^\alpha-2.$$ 

For any affine $\K$-scheme $X$ of finite type (with a given embedding in an affine space), $X(\K)$ inherits of a norm and the set $\mathcal{P}_Y(X(\K))$ is well-defined. There is an obvious inclusion $\mathbb{L}_\wp[\mathcal{P}_Y(\K)]\subset \mathcal{P}_Y[\mathbb{L}_\wp(\K)]$. It is not necessarily an equality (see Example \ref{lpstri}); however the equality $\mathbb{M}(\mathcal{P}_Y(\K))=\mathcal{P}_Y(\mathbb{M}(\K))$ holds (see Lemma \ref{cabijk}).

For $x\in\mathbb{U}_\wp(\K)$, recall that $\|x\|=\max_{1\le i\le c}\|x_i\|$. Now assume that $Y$ is infinite of cardinal at least equal to that of $\K$. We claim that there exists a function
\[f=(f_1,\dots,f_c):(Y\times\R_{\ge 0})\to \mathbb{M}(\K)\] satisfying 
\begin{enumerate}[(a)]
\item\label{cond1fc} $\|f(y,t)\|\le t$ for all $(y,t)\in Y\times\R_{\ge 0}$;
\item\label{cond2fc} for every $x=(x_1,\dots,x_c)\in \mathbb{M}(\K)$, there exists $y\in Y$ such that $f(y,\|x\|)=x$.
\end{enumerate}
To construct such an $f$, consider an arbitrary injective map $x\mapsto y(x)$ from $\mathbb{U}_{\wp}(\K)$ to $Y$ (it exists because of the cardinality assumption on $Y$); for each $x\in\mathbb{M}(\K)$, define $f(y(x),\|x\|)=x$ and define $f(y,t)=1$ for every $(y,t)$ not of the form $(y(x),\|x\|)$, where 1 here denotes the unit element $(1,\dots,1)$ of $\mathbb{U}_{\wp}(\K)$, which belongs to $\mathbb{M}(\K)$ by assumption. By construction, $f$ takes values in $\mathbb{M}(\K)$. By (\ref{cond1fc}), $f\in\mathcal{P}_Y(\mathbb{M}(\K))$. 

Hence $f\in\mathbb{M}(\mathcal{P}_Y(\K))=\mathbb{M}(\mathsf{P})$, where write $\mathsf{P}=\mathcal{P}_Y(\K)$ as a shorthand. So, by the inclusion $\mathbb{M}(\mathsf{P})\subset \mathbb{L}_\wp(\mathsf{P})$, we have $f\in\mathbb{L}_\wp(\mathsf{P})$. Hence $\pi^{\wp}_{\mathsf{P}}(f)=f_1\dots f_c\in\mathbb{H}[\mathsf{P}]$ belongs to the normal subgroup generated by $\pi^\omega_{\mathsf{P}}(\mathbb{R}[\mathsf{P}])$.
This means that there exist
\begin{itemize}
\item integers $m,\mu$;
\item  $\rho_{k\ell}\in\mathbb{U}_{i\ell}(\mathsf{P})$, $1\le k\le m$, $1\le\ell\le|\omega|$, such that setting $\tilde{\rho_k}=(\rho_{k1},\dots,\rho_{k|\omega|})\in \mathbb{U}_\omega(\mathsf{P})$, we have $\tilde{\rho_k}\in\mathbb{R}[\mathsf{P}]$ for every $k$;
\item elements $h_{k\ell}$ in $\mathbb{V}[\mathsf{P}]$, $1\le k\le m$, $1\le\ell\le\mu$,
\end{itemize}
 such that, setting $$\rho_k=\pi_{\mathsf{P}}(\tilde{\rho_k})=\prod_{\ell=1}^{|\omega|} \rho_{k\ell},\quad h_k=\prod_{\ell=1}^\mu h_{k\ell},$$ we have, in $\mathbb{H}[\mathsf{P}]$, the equality
\begin{equation}f_1\dots f_c=\prod_{k=1}^mh_k\rho_kh_k^{-1}.\label{unifde2}\end{equation}

Hence for all $y\in Y$ and $t\in\R_{\ge 0}$, we have the equality in $\mathbb{H}[\K]$
\[f_1(y,t)\dots f_c(y,t)=\prod_{k=1}^mh_k(y,t)\rho_k(y,t)h_k(y,t)^{-1},\]
where $\rho_k(y,t)=\prod_{\ell=1}^\mu \rho_{k\ell}(y,t)$ and $h_k(y,t)=\prod_{\ell=1}^\mu h_{k\ell}(y,t)$. Since each of the elements $\rho_{k\ell},h_{k\ell}$ belongs to some $\mathbb{U}_i(\mathsf{P})$, there exists $\alpha>0$ such that for all $k,\ell$ and all $y,t$, we have
\[\|\rho_{k\ell}(y,t)\|\le (2+t)^\alpha-2,\quad \|h_{k\ell}(y,t)\|\le (2+t)^\alpha-2.\]

Hence for all $x=(x_1,\dots,x_c)\in\mathbb{M}(\K)$, choosing one $y=y(x)$ satisfying (\ref{cond2fc}) and $t=\|x\|$, and using the shorthands
\[\rho_{k\ell}(x)=\rho_{k\ell}(y(x),\|x\|),\quad h_{k\ell}(x)=h_{k\ell}(y(x),\|x\|),\]
\[\rho_k(x)=\rho_k(y(x),\|x\|)=\prod_{\ell=1}^\mu \rho_{k\ell}(x),\quad h_k(x)=h_k(y(x),\|x\|)=\prod_{\ell=1}^\mu h_{k\ell}(x),\] we precisely obtain that the norms of both $\rho_{k\ell}(x)$ and $h_{k\ell}(x)$ is bounded above by $(2+\|x\|)^\alpha-2$ and that the equality (\ref{x1xc2}) holds in $\mathbb{H}[\K]$.
\end{proof}

\subsubsection{Generating subsets and presentations (****)}

Keep the previous setting (with $\mathbb{U}_i,\mathbb{R}$ given). Assume that we have, in addition, a group $H$, subgroups $H_1,\dots,H_\nu$ generating $H$, and 
homomorphisms $\psi_i:U_i\to H_i$ (in the examples we have in mind, all $\psi_i$ are given as inclusions). We assume that the resulting homomorphism $\mathbb{H}[\K]\to H$ is trivial on $\mathbb{R}[\K]$, or equivalently factors through $\mathbb{Q}[\K]$.

Let $\langle S_i\mid\Pi_i\rangle$ be a presentation of $H_i$.
We suppose that for each $i$ and for every $x\in U_i$, the word length of $\psi_i(x)$ with respect to the generating subset $S_i$ of $H_i$ is $\simeq\log(1+\|x\|)$. For $x\in U_i$, fix a representing word $\overline{x}$ in $S_i$ of $\psi_i(x)$, of size $\simeq\log(1+\|x\|)$. 

Assume in addition the following:
\begin{itemize}
\item
all presentations $\langle S_i\mid \Pi_i\rangle$ have Dehn function bounded above by some superadditive function $\delta_1$; 
\item there is a subset $R\subset \mathbb{R}[\K]$ and a function $\delta_2$ such that for every $r=(r_1,\dots,r_{|\omega|})\in\mathbb{R}[\K]$, the area of $\overline{r_1}\dots\overline{r_{|\omega|}}$ with respect to $\langle \bigsqcup_i S_i\mid\bigsqcup_i\Pi_i\cup R\rangle$ is finite and $\le \delta_2(\max_{\ell=1}^{|\omega|}|\overline{r_\ell}|_{S_i})$.
\end{itemize}

\subsubsection{The fundamental theorem}

\begin{thm}\label{mainsim3}
In the above setting (*),(**), (****), and given a family of relations as in (***), there exist constants $s,s',m>0$ such that for every $x=(x_1,\dots,x_c)\in \mathbb{M}(\K)$, denoting $n=\max_i|x_i|_{S_i}$, the area of $\overline{x_1}\dots\overline{x_c}$ with respect to $\langle \bigsqcup_i S_i\mid\bigsqcup_i \Pi_i\cup R\rangle$ is $\le \delta_1(sn)+m\delta_2(s'n)$.
\end{thm}
\begin{proof}
Fix $x$ as above. As in Lemma \ref{superlemma}, write the equality in $\mathbb{H}[\K]$
\begin{equation}\label{xhrhoh}x_1\dots x_c=\prod_{k=1}^mh_k(x)\rho_k(x)h_k(x)^{-1}\end{equation}
(with all the additional features of the statement of the lemma).

Defining $[x]=\log(2+\|x\|)$ and similarly $[u]=\log(2+\|u\|)$ for any $i$ and $u\in \mathbb{U}_i(\K)$, the upper bounds on $\|h_{k\ell}(x)\|$ and $\|\rho_{k\ell}(x)\|$ provided by the lemma can be rewritten as
\[[h_{k\ell}(x)],[\rho_{k\ell}(x)]\le\alpha [x].\]

Define $S=\bigsqcup S_i$. 
Write $\hat{\rho}_k(x)=\prod_{\ell=1}^{|\omega|} \overline{\rho_{k\ell}(x)}$ and $\hat{h}_k(x)=\prod_{\ell=1}^\mu \overline{h_{k\ell}(x)}$.
Define the word in the free group over $S$

\[
w=(\overline{x_1}\dots\overline{x_c})^{-1}\prod_{k=1}^m
\hat{h}_{k}(x)\hat{\rho}_k(x)
\hat{h}_{k}(x)^{-1}\]

Hence, by (\ref{xhrhoh}), the image of $w$ in $\mathbb{H}[\K]$ is trivial.

Denoting by $|\cdot|$ the word length in $\mathbb{H}[\K]$ with respect to $S$. By assumption, there exist $\gamma,\zeta\ge 1$ such that for every $i$ and $u\in\mathbb{U}_i(\K)$, we have $\zeta^{-1}[u]\le |\overline{u}|_{S_i}\le\gamma [u]$; in particular $|u|\le\gamma [u]$. Also define $n=\max_{i=1}^c|\overline{x_i}|_{S_{\wp_i}}$. Hence
\[n\ge \zeta^{-1}\max_i[x_i]\ge\zeta^{-1}[x]\]
and
\begin{align*}
|w|\le & \sum_{i=1}^c|\overline{x_i}|+\sum_{k=1}^m\left(\sum_{\ell=1}^{|\omega|}|\overline{\rho_{k\ell}(x)}|+2\sum_{\ell=1}^\mu|\overline{h_{k\ell}(x)}|\right)\\
\le & \gamma\left(\sum_{i=1}^c[x_i]+\sum_{k=1}^m\left(\sum_{\ell=1}^{|\omega|}[\rho_{k\ell}(x)]+2\sum_{\ell=1}^\mu[h_{k\ell}(x)]\right)\right)\\
\le & \gamma(c+m\alpha(|\omega|+2\mu))[x]\le sn,
\end{align*}
where $s=\zeta\gamma(c+m\alpha(|\omega|+2\mu))$.

By Lemma \ref{dehnfree} (Dehn function of free products of presentations), the area of $w$ with respect to $\langle\bigsqcup S_i\mid\bigsqcup \Pi_i\rangle$ is $\le\delta_1(sn)$.

We have, by definition of $w$:
\[\overline{x_1}\dots\overline{x_c}=\left(\prod_{k=1}^m
\hat{h}_{k}(x)\hat{\rho}_k(x)
\hat{h}_{k}(x)^{-1} \right)w^{-1};\]
the area of $\hat{\rho}_k(x)$ with respect to the presentation $P=\langle \bigcup S_i\mid\bigcup \Pi_i\cup R\rangle$ is 
\[\le\delta_2\left(\max_{\ell=1}^{|\omega|}|\overline{\rho_{k\ell}(x)}|_{S_i}\right)
\le\delta_2\left(\max_{\ell=1}^{|\omega|}\gamma[\rho_{k\ell}(x)]\right)\le\delta_2(\gamma\alpha [x])\le \delta_2(\gamma\alpha\zeta n)\]
so the area of $\prod_{k=1}^m
\hat{h}(x)\hat{\rho}_k(x)
\hat{h}_{k}(x)^{-1}$ with respect to the presentation $P$ is 
\[\le m\delta_2(\gamma\alpha\zeta n).\]
Therefore the area of $\overline{x_1}\dots\overline{x_c}$ with respect to $P$ is 
\[\le \delta_1(sn)+m\delta_2(\gamma\alpha\zeta n).\qedhere\]
\end{proof}

\subsubsection{Restatement of the theorem involving a group word}\label{bubullet}

Let $F_c$ be the free group on the generators $t_1,\dots,t_c$. 
We say that $w\in F_c$ is {\em essential} if all $t_i$ appear in the reduced form of $w$ (possibly with a negative exponent). We call the word $t_1\dots t_c$, which is essential, the {\em tautological} word; denote it by $\tau$. 

Let $w$ be an essential word in $F_{c'}$, of length $c$ (necessarily $c\ge c'$ if $w$ is essential). 

Let $\wp'$ be a $c'$-tuple of elements of $\{1,\dots,\nu\}$. 

Write $w$ as a reduced word, namely $w=w_1\dots w_{c}$ with $w_i\in\{t_j,t_j^{-1},1\le j\le c'\}$; if $w_i=t_j^{\pm 1}$, we write $W(i)=j$ and $w[i]=\pm 1$.

Let us define a $c$-tuple $\wp=w\bullet\wp'$ of elements in $\{1,\dots,\nu\}$: for $1\le i\le c$, write $\wp_i=j$ if $w_i\in\{t_j,t_j^{-1}\}$.
For instance $\tau\bullet\wp'=\wp'$; another less trivial example is
\[w(t_1,t_2,t_3)=t_2t_1^{-2}t_3t_2^{-1},\quad\wp'=(5,6,3)\quad\Rightarrow\quad w\bullet\wp'=(6,5,5,3,6).\]

Define $(w\bullet\mathbb{U}_{\wp'})(\A)$ as
\[\{(u_1,\dots,u_{c})\in\mathbb{U}_{\wp}(\A):\;\forall j\in\{1,\dots,c'\},\forall k,\ell\in W^{-1}(\{j\}),u_{k}^{w[k]}=u_{\ell}^{w[\ell]}\}.\]
Thus $w\bullet\mathbb{U}_{\wp'}$ is a closed subscheme of $\mathbb{U}_{\wp}$.
In the above example, we have
\[(w\bullet\mathbb{U}_{\wp'})(\A)=\{(u_6,u_5,u_5,u_3,u_6^{-1})\mid (u_3,u_5,u_6)\in(\mathbb{U}_3\times\mathbb{U}_5\times \mathbb{U}_6)(\A)\}.\]

When $w$ is essential, we have an isomorphism $\phi_w$ of schemes $\mathbb{U}_{\wp'}\to w\bullet \mathbb{U}_{\wp'}$, defined on $\mathbb{U}_{\wp'}(\A)$ by
\[\phi_w(u_1,\dots,u_{c'})=(u_{W(1)}^{w[1]},\dots,u_{W(c)}^{w[c]})\]
Note that $\phi_\tau$ is the identity.
If $\mathbb{M}'$ is a closed subscheme of $\mathbb{U}_{\wp'}$, we define $w\bullet\mathbb{M}'=\phi_w(\mathbb{M})$; this is a closed subscheme of $w\bullet \mathbb{U}_{\wp'}$.

The following theorem seems to be a generalization of Theorem \ref{mainsim3} (namely when $w=t_1\dots t_c$), but is actually a simple consequence. It will be useful to refer to it.

Let $\wp$ be a $c$-tuple of elements in $\{1,\dots,\nu\}$. If $w\in F_c$, denote by $\mathbb{L}_\wp^w[\A]$ the set of $(x_1,\dots,x_c)\in\mathbb{U}_\wp(\A)$ such that $w(x_1,\dots,x_c)$ equals 1 in $\mathbb{Q}[\A]$.

\begin{thm}\label{mainsim4}
In the above setting (*),(**), (****), let $c'$ be a positive integer and $\wp'$ a $c'$-tuple of elements of $\{1,\dots,\nu\}$. Let $\mathbb{M}'$ be a closed subscheme of $\mathbb{U}_{\wp'}$ and let $w$ be an essential word in $F_{c'}$, such that $\mathbb{M}'(\A)\subset\mathbb{L}^w_{\wp'}[\A]$ for all $\A$. 

Then there exist constants $s,s',m>0$ such that for every $x=(x_1,\dots,x_{c'})\in\mathbb{M}'(\K)$, denoting $n=\max_i|x_i|_{S_i}$, the area of $w(\overline{x_1},\dots,\overline{x_{c'}})$ with respect to $\langle \bigcup S_i\mid\bigcup \Pi_i\cup R\rangle$ is $\le \delta_1(sn)+m\delta_2(s'n)$.
\end{thm}
\begin{proof}
Define $\wp=w\bullet\wp'$ and $\mathbb{M}=w\bullet\mathbb{M}'$ as above. 

Thus $\pi^{\wp}_\A(\mathbb{M}(\A))$ is precisely the set elements of the form $w(x_1,\dots,x_{c'})\in\mathbb{H}[\A]$ with $(x_1,\dots,x_{c'})\in\mathbb{M}'(\A)$. In particular, $\mathbb{M}(\A)\subset\mathbb{L}^w_\wp[\A]$.
Hence $\mathbb{M}$ satisfies the assumption of (***) and Theorem \ref{mainsim3} can be applied.
\end{proof}

\begin{ex}\label{lpstri}
Here is an example in which the inclusion $\mathbb{L}_\wp[\K^Y]\subset\mathbb{L}_\wp[\K]^Y$ is strict. We choose $\K=\R$, $\nu=1$, and $\mathbb{U}_1(\A)=\A$ for every $\A$ (so $\mathbf{U}_1$ is the 1-dimensional unipotent group. We choose $\mathbb{R}=\mathbb{R}_1$ to be the singleton $\{1\}$ (formally, $\mathbb{R}(\A)=\{1\}\subset\A$). Hence $\mathbb{Q}(\A)$ is equal to $\A/\Z 1_{\A}$, and for $\wp=\{1\}$, we have $\mathbb{L}_\wp[\A]=\Z 1_{\A}$. In particular, the inclusion $\mathbb{L}_\wp[\K^Y]\subset\mathbb{L}_\wp[\K]^Y$ is proper as soon as $Y$ contains two distinct elements. Anyway this example is not really serious because assuming that $\mathbb{R}$ is a subscheme containing 0 (which could easily be arranged), we would have $\mathbb{L}_\wp[\K^Y]=\mathbb{L}_\wp[\K]^Y$ for every finite $Y$, and the serious issue is when $Y$ is infinite. Indeed, now defining $\mathbb{R}=\mathbb{R}_1$ as equal to $\{0,1\}$ (namely $\mathbb{R}=\mathrm{Spec}(\R[t]/(t^2-t))$), we have $\mathbb{R}(\A)$ equal to the set of idempotents in $\A$, and hence $\mathbb{L}_\wp[\A]$ is the additive subgroup of $\A$ generated by idempotents. In particular, $\mathbb{L}_\wp[\R^Y]$ is equal to the set of bounded functions $Y\to\Z$, while $\mathbb{L}_\wp[\R]^Y$ is equal to the set of all functions $Y\to\Z$, which differs from the former if $Y$ is infinite. The inclusion  $\mathbb{L}_\wp[\mathcal{P}_Y(\R)]\subset\mathcal{P}_Y(\mathbb{L}_\wp(\R))$ is also proper, since then $\mathbb{L}_\wp[\mathcal{P}_Y(\R)]$ is equal to the set of bounded functions $Y\to\Z$, while $\mathcal{P}_Y(\mathbb{L}_\wp(\R))$ is equal to the set of all functions $Y\to\Z$ of at most polynomial growth, uniformly in $Y$.
\end{ex}




\section{Central extensions of graded Lie algebras}\label{s_abels}

This section contains results on central extensions of graded Lie algebras, which will be needed in Section \ref{s:am}. Let $Q\subset K$ be fields of characteristic zero (for instance, $Q=\Q$ and $\K$ is a nondiscrete locally compact field). To any graded Lie algebra, we associate a central extension in degree zero, which we call its ``blow-up", whose study will be needed in Section \ref{s:am}.
We are then led to following problem: given a Lie algebra $\g$ over $K$, we need to compare the homologies $H_2^K(\g)$ and $H_2^Q(\g)$ of $\g$ viewed as a Lie algebra over $K$ and as a Lie algebra over $Q$ by restriction of scalars. When $\g$ is defined over $Q$, i.e.\ $\g=K\otimes_Q\mk{l}$, this problem has been tackled in several papers \cite{KL,NW}. Here most of the work is carried out over an arbitrary commutative ring; this generality will be needed as we need to apply the results over suitable rings of functions.


\subsection{Basic conventions}\label{basc}

The following conventions will be used throughout this chapter. The letter $\RR$ denotes an arbitrary commutative ring (associative with unit). Unless explicitly stated, modules, Lie algebras are over the ring~$\RR$ and are {\em not} assumed to be finitely generated. The reader is advised not to read this part linearly but rather refer to it when necessary.

\subsubsection*{Gradings}
We fix an abelian group $\mathcal{W}$, called the weight space.
By graded module, we mean an $\RR$-module $V$ endowed with a {\bf grading}, namely an $\RR$-module decomposition as a direct sum $$V=\bigoplus_{\alpha\in\mathcal{W}}V_\alpha.$$
Elements of $V_\alpha$ are called homogeneous elements of weight $\alpha$. An $\RR$-module homomorphism $f:V\to W$ between graded $\RR$-modules is {\bf graded} if $f(V_\alpha)\subset W_\alpha$ for all $\alpha$. If $V$ is a graded module and $V'$ is a subspace, it is a {\bf graded submodule} if it is generated by homogeneous elements, in which case it is naturally graded and so is the quotient $V/V'$. By the {\bf weights} of $V$ we generally mean the subset $\mathcal{W}_V\subset\mathcal{W}$ consisting of $\alpha\in\mathcal{W}$ such that $V_\alpha\neq\{0\}$. 
We use the notation 
$$V_\td=\bigoplus_{\alpha\neq 
0}V_\alpha.$$

By {\bf graded Lie algebra} we mean a Lie algebra $\g$ endowed with an $\RR$-module grading $\g=\bigoplus\g_\alpha$ such that $[\g_\alpha,\g_\beta]\subset\g_{\alpha+\beta}$ for all $\alpha,\beta\in\mathcal{W}$.

\subsubsection*{Tensor products}

If $V,W$ are modules, the 
tensor product $V\otimes W=V\ot_\RR W$ is defined in the usual way. The symmetric 
product $V\cc V$ is obtained by modding out by the $\RR$-linear span of all 
$v\ot w-w\ot v$ and the exterior product $V\wedge V$ is obtained by 
modding out by the $\RR$-linear span of all $v\ot w+w\ot v$ (or equivalently all $v\ot v$ if 2 is invertible in $\RR$). More generally the $n$th 
exterior product $V\wedge\dots\wedge V$ is obtained by modding out the $n$th 
tensor product $V\otimes\dots\ot V$ by all tensors 
$v_1\otimes\dots\otimes v_n+w_1\otimes\dots\otimes w_n$, whenever for some $1\le i\neq j\le n$, we have $w_i=v_j$, $w_j=v_i$ and $w_k=v_k$ for all $k\neq i,j$.

If $W_1,W_2$ are submodules of $V$, we will sometimes denote by $W_1\wedge W_2$ (resp.~$W_1\cc W_2$) the image of 
$W_1\otimes W_2$ in $V\wedge V$ (resp.~$V\cc V$). In case $W_1=W_2=W$, the latter map factors through a module homomorphism $W\wedge W\to V\wedge V$ (resp.~$W\cc W\to V\cc V$), and this convention is consistent when this homomorphism is injective, for instance when $W$ is a direct factor of $V$.

If $V,W$ are graded then $V\ot W$ is also graded by
$$(V\ot W)_\alpha=\bigoplus_{\{(\beta,\gamma):\beta+\gamma=\alpha\}}V_\beta\ot W_\gamma.$$
When $V=W$, we see that $V\we V$ and $V\cc V$ are quotients of $V\ot V$ by graded submodules and are therefore naturally graded; for instance if $\mathcal{W}$ has no 2-torsion
$$(V\we V)_0=(V_0\we V_0)\oplus\left(\bigoplus_{\alpha\in(\mathcal{W}-\{0\})/\pm}V_\alpha\ot V_{-\alpha}\right).$$

\subsubsection*{Homology of Lie algebras}

Let $\g$ be a Lie algebra (always over the commutative ring $\RR$). We consider the complex of $\RR$-modules
$$\cdots\g\wedge\g\wedge\g\wedge\g\stackrel{d_4}\longrightarrow\g\wedge\g\wedge\g\stackrel{d_3}\longrightarrow\g\wedge\g\stackrel{d_2}\longrightarrow\g\stackrel{d_1}\longrightarrow 0$$


\noindent given by
\begin{align*}
d_2(x_1,x_2) &= -[x_1,x_2]\\
d_3(x_1,x_2,x_3) &= x_1\wedge [x_2,x_3]+ x_2\wedge [x_3,x_1]+x_3\wedge [x_1,x_2] \end{align*}
and more generally the boundary map
$$d_n(x_1,\dots,x_n) = \sum_{1\le i\le j\le 
n}(-1)^{i+j}[x_i,x_j]\wedge x_1\wedge\dots\wedge \widehat{x_i}\wedge\dots\wedge \widehat{x_j}\wedge\dots\wedge x_n;$$
and define the {\bf second homology group} $H_2(\g)=Z_2(\g)/B_2(\g)$, where $Z_2(\g)=\Ker(d_2)$ is the set of {\bf 2-cycles} and $B_2(\g)=\textnormal{Im}(d_3)$ is the set of {\bf 2-boundaries}. (We will focus on $d_i$ for $i\le 3$ although the map $d_4$ will play a minor computational role in the sequel. This is of course part of the more general definition of the $n$th homology module $H_n(\g)=\Ker(d_n)/\textnormal{Im}(d_{n+1})$, which we will not consider.) If $\A\to \RR$ is a homomorphism of commutative rings, then $\g$ is a Lie $\A$-algebra by restriction of scalars, and its 2-homology as a Lie $\A$-algebra is denoted by $H_2^{\A}(\g)$.
If $\g$ is a graded Lie algebra, then the maps $d_i$ are graded as well, so $H_2(\g)$ is naturally a graded $\RR$-module.

\subsubsection*{Iterated brackets}

In a Lie algebra, we define the $n$-fold bracket as the usual bracket for $n=2$ and by induction for $n\ge 3$ as
\[[x_1,\dots,x_n]=[x_1,[x_2,\dots,x_n]].\]

\subsubsection*{Central series and nilpotency}

Define the {\bf lower central series} of the Lie algebra $\g$ by $\g^1=\g$ and $\g^{i+1}=[\g,\g^i]$ for $i\ge 1$. We say that $\g$ is {\bf $s$-nilpotent} if $\g^{s+1}=\{0\}$.

\subsubsection{The Hopf bracket}\label{hopb}

Consider a central extension of Lie algebras 
$$0\to \mk{z}\to \g\stackrel{p}\to \mk{h}\to 0.$$
Since $\mk{z}$ is central, the bracket $\g\we\g\to\g$ factors through an $\RR$-module homomorphism $B:\mk{h}\we\mk{h}\to\g$, called the {\bf Hopf bracket}.
It is unique for the property that $B(p(x)\we p(y))=[x,y]$ for all $x,y\in\g$ (uniqueness immediately follows from surjectivity of $\g\we\g\to\mk{h}\we\mk{h}$).

\begin{lem}\label{eqho}
For all $x,y,z,t\in\mk{h}$ we have $B([x,y]\we[z,t])=[B(x\we y), B(z\we t)]$.
\end{lem}
\begin{proof}
Observe that if $\bar{x},\bar{y}$ are lifts of $x$ and $y$ then $B(x\we y)=[\bar{x},\bar{y}]$, and that $[\bar{x},\bar{y}]$ is a lift of $[x,y]$. In view of this, observe that both terms are equal to $[[\bar{x},\bar{y}],[\bar{z},\bar{t}]]$.
\end{proof}

\subsubsection{1-tameness}\label{def1t}

\begin{defn}
We say that a Lie algebra $\g$ is {\bf 1-tame} if it is generated by $\g_\td=\bigoplus_{\alpha\neq 0}\g_\alpha$.
\end{defn}

\begin{lem}\label{430}
Let $\g$ be a graded Lie algebra. Then $\g$ is 1-tame if and only if we have $\g_0=\sum_{\beta}[\g_\beta,\g_{-\beta}]$, where $\beta$ ranges over nonzero weights.
\end{lem}
\begin{proof}One direction is trivial. Conversely, if $\g$ is 1-tame, then $\g_0$ is generated as an abelian group by elements of the form $x=[x_1,\dots,x_k]$ with $k\ge 2$ and $x_i$ homogeneous of nonzero weight. So $x=[x_1,y]$ with $y=[x_2,\dots,x_k]\in\g_\td$ and $x_1\in\g_\td$.
\end{proof}

\begin{lem}\label{ideal1tame}
Let $\g$ be a graded Lie algebra. Then the ideal generated by $\g_\td$ coincides with the Lie subalgebra generated by $\g_\td$ and in particular is 1-tame.
\end{lem}
\begin{proof}
Let $\mk{h}$ be the subalgebra generated by $\g_\td$; it is enough to check that $\mk{h}$ is an ideal, and it is thus enough to check that $[\g_0,\mk{h}]\subset\mk{h}$. Set $\mk{h}_1=\g_\td$ and $\mk{h}_d=[\g_\td,\mk{h}_{d-1}]$, so that $\mk{h}=\sum_{d\ge 1}\mk{h}_d$. It is therefore enough to check that $[\g_0,\mk{h}_d]\subset\mk{h}_d$ for all $d\ge 1$. This is done by induction. The case $d=1$ is clear. If $d\ge 2$, $x\in\g_0$, $y\in\g_\td$, $z\in\mk{h}_{d-1}$, then using the induction hypothesis$$[x,[y,z]]=[y,[x,z]]-[[y,x],z]\in [\g_\td,\mk{h}_{d-1}]\subset\mk{h}_d$$ and we are done.\end{proof}

We use this to obtain the following result, which will be used in Section \ref{s_cent}.

\begin{lem}\label{g0ni}
Let $\g$ be a graded Lie algebra with lower central series $(\g^i)$, and assume that $\g_0$ is $s$-nilpotent. Then $\g^{s+1}$ is contained in the subalgebra generated by $\g_\td$. In particular, if $\g_0$ is nilpotent and $\g^\infty=\bigcap\g^i$, then $\g^\infty$ is contained in the subalgebra generated by $\g_\td$.
\end{lem}
\begin{proof}
By Lemma \ref{ideal1tame}, it is enough to check that $\g^{s+1}$ is contained in the {\em ideal} $\mk{j}$ generated by $\g_\td$. It is sufficient to show that $\g^{s+1}\cap\g_0\subset\mk{j}$. Each element $x$ of $\g^{s+1}\cap\g_0$ can be written as a sum of nonzero $(s+1)$-fold brackets of homogeneous elements. Each of those brackets involves at least one element of nonzero degree, since otherwise all its entries would be contained in $\g_0$ and it would vanish. So $x$ belongs to the ideal generated by $\g_\td$.
\end{proof}


\subsection{The blow-up}\label{s_bu}

\begin{defn}\label{defbu}Let $\g$ be an arbitrary graded Lie $\RR$-algebra (the grading being in the abelian group $\mathcal{W}$). Define the {\bf blow-up} graded algebra 
$\tilde{\g}$ as follows. As a graded $\RR$-module, 
$\tilde{\g}_\alpha=\g_\alpha$ for all $\alpha\neq 0$, and 
$\tilde{\g}_0=(\g\we\g)_0/d_3(\g\we\g\we\g)_0$.  

Define a graded $\RR$-module homomorphism $\tau:\tilde{\g}\to\g$ by $\tau(x)=x$ if $x\in\tilde{\g}_\td$ and $\tau(x\we y)=[x,y]$ if $x\we y\in\tilde{\g}_0$ (which of course factors through 2-boundaries).

Let us define the Lie algebra structure $[\cdot,\cdot]'$ on $\tilde{\g}$. Suppose that $x\in\tilde{\g}_\alpha,y\in\tilde{\g}_\beta$.
\begin{itemize}
\item if $\alpha+\beta\neq 0$, define $[x,y]'=[\tau(x),\tau(y)]$;
\item if $\alpha+\beta=0$, define $[x,y]'=\tau(x)\we \tau(y)$.
\end{itemize}
\end{defn}

\begin{lem}\label{blowupc}
With the above bracket, $\tilde{\g}$ is a Lie algebra and $\tau$ is a Lie algebra homomorphism, whose kernel is central and naturally isomorphic to $H_2(\g)_0$. Its image is the ideal $\g_\td+[\g,\g]$ of $\g$.
\end{lem}
\begin{proof}
Let us first check that $\tau$ is a homomorphism (of non-associative algebras).
Let $x,y\in\tilde{\g}$ have weight $\alpha$ and $\beta$. In each case, we apply the definition of $[\cdot,\cdot]'$ and then of $\tau$.
If $\alpha+\beta\neq 0$
$$\tau([x,y]')=\tau([\tau(x),\tau(y)])=[\tau(x),\tau(y)];$$
if $\alpha+\beta=0$ then
$$\tau([x,y]')=\tau(\tau(x)\we\tau(y))=[\tau(x),\tau(y)].$$
By linearity, we deduce that $\tau$ is a homomorphism. Since $\tau_\alpha$ is an isomorphism for $\alpha\ne 0$ and $\tau_0=-d_2$, the kernel of $\tau$ is {\it equal} by definition to $H_2(\g)_0$. Moreover, by definition the bracket $[x,y]'$ only depends on $\tau(x)\ot\tau(y)$, and it immediately follows that $\Ker(\tau)$ is central in $\tilde{\g}$.

Let us check that the bracket is a Lie algebra bracket; the antisymmetry being clear, we have to check the Jacobi identity. Take $x\in\tilde{\g}_\alpha$, $y\in\tilde{\g}_\beta$, $z\in\tilde{\g}_\gamma$. From the definition above, we obtain
\begin{itemize}
\item if $\alpha+\beta+\gamma\neq 0$, $[x,[y,z]']'=[\tau(x),[\tau(y),\tau(z)]]$;
\item if $\alpha+\beta+\gamma=0$, $[x,[y,z]']'=\tau(x)\we [\tau(y),\tau(z)]$
\end{itemize}
(the careful verification involves discussing on whether or not $\beta+\gamma$ is zero).
Therefore, the Jacobi identity for $(x,y,z)$ immediately follows from that of $\g$ in the first case, and from the fact we killed 2-boundaries in the second case.

We have $\tau(\tilde{\g}_\td)=\g_\td$ and $\tau(\tilde{\g}_0)=[\g,\g]_0$. Therefore the image of $\tau$ is equal to $\g_\td+[\g,\g]$. The latter is an ideal, since it contains $[\g,\g]$.
\end{proof}

\begin{defn}\label{0per}
We say that a graded Lie algebra $\g$ is {\bf relatively perfect in degree zero} if it satisfies one of the following (obviously) equivalent definitions
\begin{enumerate}[(a)]
\item 0 is not a weight of $\g/[\g,\g]$;
\item $\g_0\subset [\g,\g]$;
\item $\g=\g_\td+[\g,\g]$;
\item $\tilde{\g}\to\g$ is surjective (in view of Lemma \ref{blowupc}).
\item\label{ggenen} $\g$ is generated by $\g_\td+[\g_0,\g_0]$;
\end{enumerate}
\end{defn}

Note that if $\g$ is 1-tame, then it is relatively perfect in degree zero, but the converse is not true, as shows the example of a nontrivial perfect Lie algebra with grading concentrated in degree zero. The interest of this notion is that it is satisfied by a wealth of graded Lie algebras that are very far from perfect (e.g., nilpotent).

\begin{thm}\label{uc0}
Let $\g$ be a graded Lie algebra. If $\g$ is relatively perfect in degree zero (e.g., $\g$ is 1-tame), then the blow-up $\tilde{\g}\stackrel{\tau}\to\g$ is a graded central extension with kernel in degree zero, and is universal among such central extensions. That is, for every surjective graded Lie algebra homomorphism $\mk{h}\stackrel{p}\to\g$ with central kernel $\mk{z}=\mk{z}_0$, there exists a unique graded Lie algebra homomorphism $\phi:\tilde{\g}\to\mk{h}$ so that the composite map $\tilde{\g}\to\mk{h}\to\g$ coincides with the natural projection.
\end{thm}
\begin{proof}
By Lemma \ref{blowupc}, $\tilde{\g}\to\g$ is a central extension with kernel in degree zero. Denote by $B_\mk{h}:\g\we\g\to\mk{h}$ the Hopf bracket associated to $\mk{h}\to\g$ (see \S\ref{basc}).

Let us show uniqueness in the universal property. Clearly, $\phi$ is determined on $\tilde{\g}_\td$. So we have to check that $\phi$ is also determined on $[\tilde{\g},\tilde{\g}]$. Observe that the map $\tilde{\g}\we\tilde{\g}\to\mk{h}$, $x\we y\mapsto \phi([x,y])=[\phi(x),\phi(y)]$ factors through a map $w:\g\we\g\to\mk{h}$, so $w(\tau(x)\we\tau(y))=[\phi(x),\phi(y)]$ for all $x,y\in\tilde{\g}$. Since $p\circ\phi=\tau$ and since $\mk{h}$ is generated by the image of $\phi$ and by its central ideal $\Ker(p)$, we deduce that for all $x,y\in\mk{h}$, we have $w(p(x)\we p(y))=[x,y]$. By the uniqueness property of the Hopf bracket (see \S\ref{basc}), we deduce that $w=B_\mk{h}$. So for all $x,y\in\tilde{\g}$, $\phi([x,y])$ is uniquely determined as $B_\mk{h}(\tau(x)\we \tau(y))$.

Now to prove the existence, define $\phi:\tilde{\g}\to\mk{h}$ to be $p^{-1}$ on $\tilde{\g}_\td$, and $\phi(x\we y)=B_\mk{h}(x\we y)$ if $x\we y\in\tilde{\g}_0$. It is clear that $\phi$ is a graded module homomorphism and that $p\circ \phi=\tau$. Let us show that $\phi$ is a Lie algebra homomorphism, i.e.\ that the graded module homomorphism $\sigma:\tilde{\g}\we\tilde{\g}\to\mk{h}$, $(x\we y)\mapsto \phi([x,y]')-[\phi(x),\phi(y)]$ vanishes. Since $p\circ\phi$ is a 
homomorphism and $p_\td$ is bijective, we have $(p\circ\sigma)_\td=0$, so it is enough to check that $\sigma$ vanishes in degree 0. 
If $x$ and $y$ have nonzero opposite weights, noting that $p\circ\phi$ is the identity on $\g_\td$, 
$$\phi([x,y]')=\phi(x\we y)=B_\mk{h}(x\we y)=[p^{-1}(x),p^{-1}(y)]=[\phi(x),\phi(y)].$$
If $x\we y$ and $z\we t$ belong to $\tilde{\g}_0=(\g\we\g)_0$, then, using Lemma \ref{eqho}, we have
\begin{align*}\phi([x\we y,z\we t]')= & \phi([x,y]\we [z,t])=B_\mk{h}([x,y]\we [z,t])\\
 =& [B_\mk{h}(x\we y),B_\mk{h}(z\we t)] = [\phi(x\we y),\phi(z\we t)].\end{align*}
By linearity, we deduce that $\sigma_0=0$ and therefore $\phi$ is a Lie algebra homomorphism.
\end{proof}

\begin{cor}If $\g$ is relatively perfect in degree zero then $\tilde{\tilde{\g}}=\tilde{\g}$.
\end{cor}
\begin{proof}
Observe that $\tilde{\tilde{\g}}\to\g$ has kernel $\mk{z}$ concentrated in degree zero, so it immediately follows that $[\mk{z},\g_\td]=0$. We need to show that $\mk{z}$ is central in $\tilde{\tilde{\g}}$. Since $\g=[\g,\g]+\g_\td$, it is enough to show that $[\mk{z},[\tilde{\tilde{\g}},\tilde{\tilde{\g}}]]=\{0\}$. By the Jacobi identity, $[\mk{z},[\tilde{\tilde{\g}},\tilde{\tilde{\g}}]]\subset [\tilde{\tilde{\g}},[\mk{z},\tilde{\tilde{\g}}]]$. Now since $\tilde{\g}\to\g$ has a central kernel, $[\mk{z},\tilde{\tilde{\g}}]$ is contained in the kernel of $\tilde{\tilde{\g}}\to\tilde{\g}$, which is central in $\tilde{\tilde{\g}}$. So $[\tilde{\tilde{\g}},[\mk{z},\tilde{\tilde{\g}}]]=\{0\}$. Thus $\mk{z}$ is central in $\tilde{\tilde{\g}}$. The universal property of $\tilde{\g}$ then implies that $\tilde{\tilde{\g}}\to\tilde{\g}$ is an isomorphism.
\end{proof}

\begin{rem}\label{blowgen}
There is an immediate generalization of this blow-up construction, where we replace $\{0\}$ by an arbitrary subset $\mathcal{X}\subset\mathcal{W}$, defining $\tilde{\g}_\alpha$ to be equal to $\g_\alpha$ if $\alpha\notin\mathcal{X}$ and to $(\Lambda^2\g/d_3(\Lambda^3\g))_\alpha$ if $\alpha\in\mathcal{X}$. Then immediate adaptations of Lemma \ref{blowupc}, Theorem \ref{uc0} and its corollary hold, with essentially no change in the proofs. 
(Definition \ref{0per} is valid replacing $\g_0$ with $\g_\mathcal{X}=\sum_{\alpha\in\mathcal{X}}\g_\alpha$ and $\g_\td$ with $\g_{\mathcal{W}\smallsetminus\mathcal{X}}$, 
except that we have to remove (the analogue of) Condition 
(\ref{ggenen}) of Definition \ref{0per}. Nevertheless, under the assumption that that no weight in $\mathcal{X}$ is the sum of a weight $\mathcal{X}$ and of a weight in
$\mathcal{W}\smallsetminus\mathcal{X}$, Condition (\ref{ggenen}) is still equivalent; in particular this works if $\mathcal{X}$ is a subgroup of $\mathcal{W}$ and in particular for $\mathcal{X}=\{0\}$. Otherwise, the reader can find examples with gradings in $\{0,1\}$ and $\mathcal{X}=\{1\}$, where the first four conditions hold but not (\ref{ggenen}).)
\end{rem}

\begin{lem}\label{tameblo}
Let $\g$ be a graded Lie algebra and $\tilde{\g}$ its blow-up. 
If $\g$ is $1$-tame, then so is $\tilde{\g}$.
\end{lem}
\begin{proof}
Suppose that $\g$ is 1-tame. Then by linearity, it is enough to check that 
for every $x\in\g_0$ and $u,v$ of nonzero opposite weights, the element $x\we [u,v]$ belongs to $[\tilde{\g}_\td,\tilde{\g}_\td]$. This is the case since modulo 2-boundaries, this element is equal to $u\we [x,v]+v\we [u,x]\in(\g_\td\we\g_\td)_0$.
\end{proof}

\begin{lem}\label{prodbu}
Let $\g_i$ be finitely many graded Lie algebras (all graded in the same abelian group) and $\tilde{\g}_i$ their blow-up. If $\g=\prod\g_i$ satisfies the assumption that $\g/[\g,\g]$ has no opposite weights, then the natural homomorphism $\widetilde{\prod\g_i}\to\prod\tilde{\g_i}$ is an isomorphism. Equivalently, $H_2(\prod_i \g_i)_0\to\bigoplus H_2(\g_i)_0$ is an isomorphism.
\end{lem}
\begin{proof}
For each $i$, there are homomorphisms $\g_i\to\prod \g_j\to\g_i$, whose composition is the identity, and hence $H_2(\g_i)_0\to H_2(\prod \g_j)_0\to H_2(\g_i)_0$, whose composition is the identity again. So we obtain homomorphisms $$\bigoplus H_2(\g_i)_0\to H_2\left(\prod\g_i\right)_0\to \bigoplus H_2(\g_i)_0,$$ whose composition is the identity. To finish the proof, we have to check that $\bigoplus H_2(\g_i)_0\to H_2(\prod\g_i)_0$ is surjective, or equivalently that in $Z_2(\prod\g_i)_0$, every element $x$ is the sum of an element in $\bigoplus_i Z_2(\g_i)_0$ and a boundary. Now observe that
$$Z_2\left(\prod\g_i\right)_0=\bigoplus Z_2(\g_i)_0\oplus \bigoplus_{i<j}(\g_i\we \g_j)_0.$$
So we have to prove that for $i\neq j$, $\g_i\we\g_j$ consists of boundaries. Given an element $x\we y$ ($x\in\g_{i}$, $y\in\g_{j}$ homogeneous), the assumption on $\g/[\g,\g]$ implies that, for instance, $y$ is a sum $\sum_k [z_k, w_k]$ of commutators. Projecting if necessary into $\g_j$, we can suppose that all $z_k$ and $w_k$ belong to $\g_j$. So 
$$x\we y=x\we \sum_k [z_k,w_k]=\sum_k d_3( x\we z_k\we w_k).\qedhere$$
\end{proof}


\subsection{Homology and restriction of scalars}\label{sec_T}


We now deal with two commutative rings $\A,\B$ coming with a ring homomorphism $\A\to \B$ (we avoid using $\RR$ as the previous results will be used both with $\RR=\A$ and $\RR=\B$). If $\g$ is a graded Lie algebra over $\B$, it can then be viewed as a graded Lie algebra over $\A$ by restriction of scalars. This affects the definition of the blow-up. There is an obvious surjective graded Lie algebra homomorphism $\tilde{\g}^{\A}\to\tilde{\g}^{\B}$.
The purpose of this part is to describe the kernel of this homomorphism (or equivalently of the homomorphism $H_2^{\A}(\g)_0\to H_2^{\B}(\g)_0$), and to characterize, under suitable assumptions, when it is an isomorphism.

Our main object of study is the following kernel.

\begin{defn}
If $\g$ is a Lie algebra over $\B$, we define the {\bf welding module} $W_2^{\A,\B}(\g)$ as the kernel of the natural homomorphism $H_2^{\A}(\g)\to H_2^{\B}(\g)$, or equivalently of the homomorphism $(\g\we_{\A}\g)/B_2^{\A}(\g)\to (\g\we_{\B}\g)/B_2^{\B}(\g)$. If $\g$ is graded, it is a graded module as well, and $W_2^{\A,\B}(\g)_0$ then also coincides with the kernel of $\tilde{\g}^{\A}\to\tilde{\g}^{\B}$. 
\end{defn}

The following module will also play an important role.

\begin{defn}
If $\g$ is a Lie algebra over $\B$, define its {\em Killing module} $\Kill(\g)$ (or $\Kill^{\B}(\g)$ if the base ring need be specified) as the cokernel of the homomorphism
\begin{eqnarray*}\mathcal{T}:\g\ot\g\ot\g & \to & \g\cc\g\\ u\ot v\ot w & \mapsto & u\cc [v,w]+v\cc [u,w].\end{eqnarray*}
If $\g$ is graded, it is graded as well.
\end{defn}

Note that we can also write $\mathcal{T}(u\ot w\ot v)=[u,v]\cc w - u\cc [v,w]$, and thus we see that the set of $\B$-linear homomorphisms from $\Kill(\g)$ to any $\B$-module $M$ is naturally identified to the set of the so-called invariant bilinear maps $\g\times\g\to M$.

\begin{lem}\label{kerli}
Let $\mk{v}$ be a $\B$-module. Then the kernel of the natural surjective $\A$-module homomorphism $\mk{v}\ot_{\A}\mk{v}\to\mk{v}\ot_{\B}\mk{v}$ is generated, as an abelian group, by elements \begin{equation}\lambda x\ot_{\A} y-x\ot_{\A}\lambda y\label{linr2}\end{equation} with $\lambda\in \B$, $x,y\in\mk{v}$. The same holds with $\ot$ replaced by $\cc$ or $\we$.
\end{lem}
\begin{proof} 
Endow $\mk{v}\ot_{\A} \mk{v}$ with a structure of a $\B$-module, using the structure of $\B$-module of the left-hand $\mk{v}$, namely $\lambda(x\ot y)=(\lambda x\ot y)$ if $\lambda\in B$, $x,y\in\mk{v}$.

Let $W$ be the subgroup generated by elements of the form (\ref{linr2}); it is clearly an $\A$-submodule, and is actually a $\B$-submodule as well. The natural $\A$-module surjective homomorphism $\phi:(\mk{v}\ot_{\A}\mk{v})/W\to\mk{v}\ot_{\B}\mk{v}$ is $\B$-linear. To show it is a bijection, we observe that by the universal property of $\mk{v}\ot_{\B}\mk{v}$, we have a $\B$-module homomorphism $\psi:\mk{v}\ot_{\B}\mk{v}\to(\mk{v}\ot_{\A}\mk{v})/W$ mapping $x\ot_{\B} y$ to $x\ot_{\A} y$ modulo $W$. Clearly, $\psi$ and $\phi$ are inverse to each other. 

Let us deal with $\we$, the case of $\cc$ being similar. The group $\mk{v}\we_{\A}\mk{v}$ is defined as the quotient of $\mk{v}\ot_{\A}\mk{v}$ by symmetric tensors (i.e.\ by the subgroup generated by elements of the form $x\ot y+y\ot x$), and the group $\mk{v}\we_{\B}\mk{v}$ is defined the quotient of $\mk{v}\ot_{\B}\mk{v}$ by symmetric tensors. By the case of $\ot$, this means that $\mk{v}\we_{\B}\mk{v}$ is the quotient of $\mk{v}\ot_{\A}\mk{v}$ by the subgroup generated by symmetric tensors and elements (\ref{linr2}). This implies that that $\mk{v}\we_{\B}\mk{v}$ is the quotient of $\mk{v}\we_{\A}\mk{v}$ by the subgroup generated by elements (\ref{linr2}) (with $\ot$ replaced by $\we$)
\end{proof}

\begin{prop}\label{linrr}
For any Lie algebra $\g$ over $\B$, the $\A$-module homomorphism
\begin{eqnarray*}\Phi:\B\otimes_{\A} \g\ot_{\A} \g & \to & W_2^{\A,\B}(\g)\\
\lambda \otimes x\ot y &\mapsto &\lambda x\we y-x\we \lambda y
\end{eqnarray*}
is surjective. If $\g$ is graded and is 1-tame, then $\Phi_0:\B\otimes_{\A} (\g\ot_{\A} \g)_0  \to  W_2^{\A,\B}(\g)_0$ is surjective in restriction to $\B\otimes (\g_\td\ot \g_\td)_0$.
\end{prop}
\begin{proof}
The group $(\g\we_{\A}\g)/B_2^{\A}(\g)$ is defined as the quotient of $\g\we_{\A}\g$ by 2-boundaries, and $(\g\we_{\B}\g)/B_2^{\B}(\g)$ is the quotient of $\g\we_{\B}\g$ by 2-boundaries, or equivalently, by Lemma \ref{kerli}, is the quotient of $\g\we_{\A}\g$ by 2-boundaries and elements of the form \begin{equation}\lambda x\we_{\A} y-x\we_{\A}\lambda y.\label{linr}\end{equation} It follows that the kernel of $(\g\we_{\A}\g)/B_2^{\A}(\g)\to(\g\we_{\B}\g)/B_2^{\B}(\g)$ is generated by elements of the form (\ref{linr}), proving the surjectivity of $\Phi$.

For the additional statement, define $W'=\Phi(\B\otimes (\g_\td\ot\g_\td))_0$. We have to show that any element $\lambda x\we_{\A} y-x\we_{\A}\lambda y$ as in (\ref{linr}) with $x,y$ of zero weight belongs to $W'$. By linearity and Lemma \ref{430}, we can suppose that $y=[z,w]$ with $z,w$ of nonzero opposite weight.
So, modulo boundaries,
\begin{align*}\lambda x\we [w,z]-x\we \lambda [w,z] 
=& -w\we [z,\lambda x]-z\we [\lambda x,w]-x\we \lambda [w,z]\\ 
=& -w\we \lambda [z,x] + \lambda w\we [z,x], 
\end{align*}
which belongs to $W'$.
\end{proof}

\begin{prop}\label{factorki}
Let $\g$ be a Lie algebra over $\B$. Then the map $$\Phi:\B\ot_{\A}\g\ot_{\A}\g\to W_2^{\A,\B}(\g)$$ of Proposition \ref{linrr} factors through the natural projection $$\B\ot_{\A}\g\ot_{\A}\g\to \B\ot_{\A}\Kill^{\A}(\g);$$ moreover in restriction to $\B\ot_{\A}\g\ot_{\A} [\g,\g]$, it factors through $\B\ot_{\A}\Kill^{\B}(\g)$. In particular, if $\g$ is graded and relatively perfect in degree 0 (e.g., 1-tame), then $\Phi_0$ factors through $\B\ot_{\A}\Kill^{\B}(\g)$.
\end{prop}
\begin{proof}
All $\ot$, $\we$, $\cc$ are meant over $\A$.

Write $\Phi^\lambda(x\ot y)=\Phi(\lambda\ot x\ot y)$. It is immediate that $\Phi^\lambda(y\ot x)=\Phi^\lambda(x\ot y)$ (even before modding out by 2-boundaries), and thus $\Phi^\lambda$ factors through $\g\cc\g$. Let us now check that $\Phi^\lambda$ factors through $\Kill^{\A}(\g)$. Modulo 2-boundaries:
\begin{align*}
\Phi^\lambda(x\cc [y,z]-y\cc [z,x]) =& \lambda x\we [y,z]-x\we \lambda [y,z]\\ & -\lambda y\we [z,x] +y\we \lambda [z,x] \\ =& -z\we [\lambda x,y] + z\we [x,\lambda y]=0.
\end{align*}
For the last statement, by Lemma \ref{kerli}, we have to show that $\Phi^\lambda(\mu x\cc [y,z])=\Phi^\lambda(x\cc\mu [y,z])$ for all $x,y,z\in\g$ and $\mu\in \B$. Indeed, using the latter vanishing we get
\[\Phi^\lambda(\mu x\cc [y,z]) =\Phi^\lambda(y\cc [z,\mu x])=\Phi^\lambda(y\cc [\mu z,x])=\Phi^\lambda(x\cc [y,\mu z]).\qedhere\]
\end{proof}

In turn, we obtain, as an immediate consequence:

\begin{cor}\label{noceb}
Let $\g$ be a graded Lie algebra over $\B$; assume that $\g$ is relatively perfect in degree~0. If $\Kill^{\B}(\g)_0=\{0\}$ then $W_2^{\A,\B}(\g)_0=\{0\}$.
\end{cor}
\begin{proof}
By Propositions \ref{linrr} and \ref{factorki}, $\Phi_0$ induces a surjection $\B\ot_{\A}\Kill(\g)^{\B}_0\to W_2^{\A,\B}(\g)_0$.
\end{proof}

Under the additional assumptions that $\A=Q$ is a field of characteristic $\neq 2$ and $\g$ is defined over $Q$, i.e.\ has the form $\mk{l}\otimes_Q\B$ for some Lie algebra $\mk{l}$ over $Q$, Corollary \ref{noceb} easy follows from \cite[Theorem 3.4]{NW}. We are essentially concerned with finite-dimensional Lie algebras $\g$ over a field $K$ of characteristic zero ($K$ playing the role of $\B$) and $\A=\Q$, but nevertheless in general we {\em cannot} assume that $\g$ be defined over $\Q$.

We will also use the more specific application.

\begin{cor}\label{faclamg}
Assume that $\A=\Q$ and $\B=\K=\prod_{j=1}^\tau\K_j$ is a finite product of locally compact fields $\K_j$, each isomorphic to $\R$ or some $\Q_p$. Let $B_{\K_j}$ be the closed unit ball in $\K_j$. Assume that $\g$ is finite-dimensional over $\K$, that is, $\g=\prod_j\g_j$ with $\g_j$ finite-dimensional over $\K_j$. Suppose that $\g$ is 1-tame and $\g/[\g,\g]$ has no opposite weights. For every $j$ and weight $\alpha$, let $V_{j,\alpha}$ be a neighbourhood of 0 in $\g_{j,\alpha}=(\g_j)_\alpha$. Then the welding module $W_2(\g)^{\Q,\K}=\Ker\big(H_2(\g)_0^\Q\to H_2(\g)_0^\K\big)$ is generated by 
elements of the form $\Phi(\lambda\ot x\ot y)$ with $\lambda\in B_{\K_j}$ and $x\in V_{j,\alpha}$, $y\in V_{j,-\alpha}$, where $\alpha$ ranges over nonzero weights and $j=1,\dots,\tau$.
\end{cor}
\begin{proof}
By Proposition \ref{linrr}, $W_2^{\A,\B}(\g)$ is generated by elements of the form $\Phi(\lambda'\ot x'\ot y')$, with $\lambda'\in\K$ and $x',y'\in\g$ are homogeneous of nonzero opposite weight. By linearity and Lemma \ref{prodbu}, these elements for which
$$(\lambda',x',y')\in\bigcup_j\bigcup_{\alpha\neq 0}\K_j\times\g_{j,\alpha}\times\g_{j,-\alpha}$$
are enough. Given such an element $(\lambda',x',y')\in\K_j\times\g_{j,\alpha}\times\g_{j,-\alpha}$, using that $\K_j=\Q+B_{\K_j}$, write $\lambda'=\eps+\lambda$ with $\eps\in\Q$ and $\lambda\in B_\K$. Also, since $\g_{j,\pm\alpha}=\Q V_{j,\pm\alpha}$, write $x'=\mu x$ and $y'=\gamma y$ with $\mu,\gamma\in\Q$, $x\in V_{j,\alpha}$, $y\in V_{j,-\alpha}$. Clearly, by the definition of $\Phi$, we have $\Phi(\eps\ot x'\ot y')=0$, so 
$$\Phi(\lambda'\ot x'\ot y')=\Phi(\lambda\ot \mu x\ot\gamma y)=\mu\gamma\Phi(\lambda\ot x\ot y),$$
and $\Phi(\lambda\ot x\ot y)$ is of the required form.
\end{proof}

\subsubsection{Construction of 2-cycles}

We now turn to a partial converse to Corollary \ref{noceb}.

Define $\HC_1^{\A}(\B)$ (or $\HC_1(\B)$ if $\A$ is implicit) as the quotient of $\B\we_{\A} \B$ by the $\A$-submodule generated by elements of the form $uv\we w+vw\we u+wu\we v$. This is the {\bf first cyclic homology group} of $\B$ (which is usually defined in another manner; we refer to \cite{NW} for the canonical isomorphism between the two).

\begin{lem}\label{parade}
Assume that $K$ is a field of characteristic zero and $Q$ a subfield. Suppose that $K$ contains an element $t$ that is transcendental over $Q$. Assume that either $K$ has characteristic zero, or $K\subset Q(\!(t)\!)$. Then $\HC_1^Q(K)\neq\{0\}$. More precisely, the image of $t\we t^{-1}$ in $\HC_1^Q(K)$ is nonzero.
\end{lem}
\begin{proof}We denote $W$ by the $Q$-linear subspace of $K\we_Q K$ generated by elements of the form $uv\we w+vw\we u+wu\we v$.

Let us begin with the case of $Q(\!(t)\!)$.
Define a $Q$-bilinear map $F:Q(\!(t)\!)^2\to Q$ by \begin{equation}\label{kmo}F\left(\sum x_it^i, \sum y_j t^j\right) = \sum_{k\in \Z} kx_ky_{-k}.\end{equation}
Observe that the latter sum is finitely supported. (In $Q[t,t^{-1}]$, the above map appears in the definition of the defining 2-cocycle of affine Lie algebras, see \cite{Fu}.)
We see that $F$ is alternating by a straightforward computation, and that $F(t,t^{-1})=1$. Setting $f(x\we y)=F(x,y)$, if $x=\sum x_it^i$, etc., we have
\begin{align*}f(xy\we z)= & \sum_{k\in \Z} k(xy)_kz_{-k} \\ = & \sum_{k\in\Z}k\sum_{i+j=k}x_iy_jz_{-k} \\ = & -\sum_{i+j+k=0}kx_iy_jz_k;
\end{align*}
thus 
$$f(xy\we z+yz\we x+zx\we y)=-\sum_{i+j+k=0}(k+i+j)x_iy_jz_k=0.$$
Hence $f$ factors through a $Q$-linear map from $\HC_1^Q(Q(\!(t)\!))\to Q$ mapping $t\we t^{-1}$ to 1. 
The proof is thus complete if $K\subset Q(\!(t)\!)$.

Now assume that $K$ has characteristic zero; let us show that $t\we t^{-1}$ has a nontrivial image in $\HC_1(K)$. Since $K$ has characteristic zero, the above definition (\ref{kmo}) immediately extends to the field $Q(\!(t^{-\infty})\!)=\bigcup_{n>0}Q(\!(t^{1/n})\!)$ of Puiseux series and hence $t\we t^{-1}$ has a nontrivial image in $\HC_1^Q(Q(\!(t^{-\infty})\!))$ for every field $Q$ of characteristic zero.

To prove the general result, let $(u_j)_{j\in J}$ be a transcendence basis of $K$ over $Q(t)$. Replacing $Q$ by $Q(t_j:j\in J)$ if necessary, we can suppose that $K$ is algebraic over $Q(t)$. Let $\hat{Q}$ be an algebraic closure of $Q$. By the Newton-Puiseux Theorem, $\hat{Q}(\!(t^{-\infty})\!)$ is algebraically closed. By the Steinitz Theorem, there exists a $Q(t)$-embedding of $K$ into $L=\hat{Q}(\!(t^{-\infty})\!)$. This induces a $Q$-linear homomorphism $\HC_1^Q(K)\to\HC_1^Q(L)$ mapping the class of $t\we t^{-1}$ in $\HC_1^Q(K)$ to the class of $t\we t^{-1}$ in $\HC_1^Q(L)$; the latter is mapped in turn to the class $t\we t^{-1}$ in $\HC_1^{\widehat{Q}}(L)$, which is nonzero. So the class of $t\we t^{-1}$ in $\HC_1^Q(K)$ is nonzero.
\end{proof}

\begin{thm}\label{nobo}
Let $\g$ be a Lie algebra over $\B$, which is defined over $\A$, i.e.\ $\g\simeq \B\ot_{\A}\g_{\A}$ for some Lie algebra $\g_{\A}$ over $\A$. Consider the homomorphism
\begin{eqnarray*}
\varphi:(\g\we_{\A}\g)/B_2^{\A}(\g) &\to & M=\HC_1^{\A}(\B) \otimes \Kill^{\A}(\g_{\A})\\
(\lambda \ot x)\we (\mu \ot y) &\mapsto & (\lambda\we\mu)\ot(x\cc y).
\end{eqnarray*}
Then $\varphi$ is well-defined and surjective, and $\varphi(W_2^{\A,\B}(\g))=2M$. In particular, if 2 is invertible in $\A$, $\g_{\A}$ is a graded Lie algebra and $M_0=\HC_1(\B) \otimes \Kill(\g_{\A})_0\neq \{0\}$, then $W_2^{\A,\B}(\g)_0\neq \{0\}$.
\end{thm}

The above map $\varphi$ was considered in \cite{NW}, for similar motivations. Assuming that $\A$ is a field of characteristic zero, the methods in \cite{NW} can provide a more precise description (as the cokernel of an explicit homomorphism) of the kernel $W_2^{\A,\B}(\g)=\Ker(H_2^{\A}(\g)\to H_2^{\B}(\g))$. Since we do not need this description and in order not to introduce further notation, we do not include it.



\begin{proof}[Proof of Theorem \ref{nobo}]
Let us first view $\varphi$ as defined on $\g\we_{\A}\g$.
The surjectivity is trivial. By Proposition \ref{linrr} (with grading concentrated in degree 0), we see that $W_2^{\A,\B}(\g)$ is generated by elements of the form $\lambda x\we\mu y-\mu x\we \lambda y$ with $x,y\in\g_{\A}$ (we omit the $\otimes$ signs, which can here be thought of as scalar multiplication); the image by $\varphi$ of such an element is $2(\lambda\we\mu)(x\cc y)$, which belongs to $2M$. Conversely, since $2M$ is generated by elements of the form $2(\lambda\we\mu)(x\cc y)$, we deduce that $\varphi(W_2^{\A,\B}(\g))=2M$. 

Let us check that $\varphi$ vanishes on 2-boundaries (so that it is well-defined on $\g\we_{\A}\g$ modulo 2-boundaries):
\begin{align*}
\varphi(tx\we [uy,vz]) &= (t\we uv)\otimes (x\cc [y,z]);\\
\varphi(uy\we [vz,tx]) & =(u\we vt)\ot(y\cc [z,x])= (u\we vt)\ot(x\cc [y,z]);\\
\varphi(vz\we [tx,uy]) & =(v\we tu)\ot(z\cc[x,y]) = (v\we tu)\ot(x\cc [y,z])                   
\end{align*}
and since $t\we uv+u\we vt+v\we tu=0$ in $\HC_1(\A)$, the sum of these three terms is zero.

The last statement clearly follows.
\end{proof}

\subsubsection{The characterization}

Using all results established in the preceding paragraphs, we obtain
\begin{thm}\label{cano}
Let $\g$ be a finite-dimensional graded Lie algebra over a field $K$ of characteristic zero, assume that $\g/[\g,\g]$ has no opposite weights. Let $Q$ be a subfield of $K$, so that $K$ has infinite transcendence degree over $Q$. We have equivalences
\begin{itemize}
\item $W_2^{Q,K}(\g)_0=\{0\}$ (i.e., $H_2^Q(\g)_0\to H_2^K(\g)_0$ is an isomorphism);
\item $\Kill^K(\g)_0=\{0\}$.
\end{itemize}
\end{thm}

\begin{cor}Under the same assumptions, we have equivalences
\begin{itemize}
\item $H_2^Q(\g)_0=\{0\}$ (i.e.\ the blow-up $\tilde{\g}\to\g$ is an isomorphism);
\item $H_2^K(\g)_0=\Kill^K(\g)_0=\{0\}$.\hfill\qed
\end{itemize}
\end{cor}

The interest is that in both the theorem and the corollary, the first condition is a problem of linear algebra in infinite dimension, while the second is linear algebra in finite dimension (not involving $Q$) and is therefore directly computable in terms of the structure constants of $\g$.

\begin{proof}[Proof of Theorem \ref{cano}]
Suppose that $\Kill^K(\g)_0=0$. By Corollary \ref{noceb}, the induced homomorphism $H_2^Q(\g)_0\to H_2(\g)_0$ is bijective.

Conversely, suppose that $\Kill^K(\g)_0\neq 0$.
Since $\g$ is finite-dimensional over $K$, there exists a subfield $L\subset K$, finitely generated over $Q$, such that $\g$ is 
defined over $L$, i.e.~we can write $\g=\g_L\ot_L K$. Obviously, $\Kill^K(\g)=\Kill^L(\g_L)\otimes_LK$, so $\Kill^L(\g_L)\neq 0$.
Let $(x\cc_L y)$ be the representative of a nonzero element in $\Kill^L(\g_L)$. Let $\lambda$ be an element of $K$, transcendental over $L$.
By Lemma \ref{parade}, the element $\lambda\we\lambda^{-1}$ has a nontrivial image in $\HC_1^L(K)$. By Theorem \ref{nobo} (applied with $(\A,\B)=(L,K)$), we deduce that
$$c_L=\lambda x\we_L \lambda^{-1} y-\lambda^{-1} x\we_L \lambda y$$
is not a 2-boundary, i.e.\ is nonzero in $H_2^L(\g)_0$. In particular the element $c_Q$ (written as $c_L$ with $\we_Q$ instead of $\we_L$) is nonzero in $H_2^K(\g)_0$ since its image in $H_2^L(\g)$ is $c_L$, while its image $c_K$ in $\g\we_K\g$ and hence in $H_2^K(\g)_0$ is obviously zero.
\end{proof}


\subsection{Auxiliary descriptions of $H_2(\g)_0$ and $\Kill(\g)_0$}\label{auxi}
We now prove the result including Theorem \ref{h2killtame} as a particular case. In this subsection, all Lie algebras are over a fixed commutative ring $\RR$.

\begin{defn}
Let $\g$ be a graded Lie algebra. We say that $\g$ is {\bf doubly 1-tame} if for every $\alpha$ we have $\g_0=\sum_{\beta\notin\{0,\alpha,-\alpha\}}[\g_\beta,\g_{-\beta}]$. 
\end{defn}

In view of Lemma \ref{430}, doubly 1-tame implies 1-tame, and the reader can easily find counterexamples to the converse. This definition will be motivated in Section \ref{s:am}, because it is a consequence of 2-tameness (Lemma \ref{431}(\ref{431c})), which will be introduced therein. 

The purpose of this subsection is to provide descriptions of $H_2(\g)_0$ and $\Kill(\g)_0$.

\subsubsection{The tame 2-homology module}\label{t2hm}

Let us begin with the trivial observation that if $\alpha+\beta+\gamma=0$ and $\alpha,\beta,\gamma\neq 0$, then $\alpha+\beta,\beta+\gamma,\gamma+\alpha\neq 0$. It follows that $d_3$ maps $(\g_\td\we\g_\td\we\g_\td)$ into $\g_\td\we\g_\td$. Define the {\bf tame 2-homology module}
$$H_2^\td(\g)_0=(\Ker(d_2)\cap (\g_\td\we\g_\td)_0)/d_3(\g_\td\we\g_\td\we\g_\td)_0.$$

We are going to prove the following result.

\begin{thm}\label{thm:h2tame} Let $\g$ be a graded Lie algebra. The natural 
homomorphism $H_2^\td(\g)_0\to H_2(\g)_0$ 
induced by the inclusion $(\g_\td\we\g_\td)_0\to(\g\we\g)_0$ is surjective if $\g$ is 1-tame, and is an isomorphism if $\g$ is doubly 1-tame.\label{h2simp}
\end{thm}

\begin{rem}\label{A2si}In Abels' second group (\S\ref{abelssecond}), $(\g\we\g)_0$ and $(\g\we\g\we\g)_0$ have dimension 4 and 5, while $(\g_\td\we\g_\td)_0$ and $(\g_\td\we\g_\td\we\g_\td)_0$ have dimension 3 and 2. Thus, we see that the computation of $H_2^\td(\g)_0$ is in practice easier than the computation of $H_2(\g)_0$.
\end{rem}

\begin{lem}\label{zzpzz}
Let $\g$ be any graded Lie algebra. If $\g$ is 1-tame, then
\begin{equation}\g_0\we\g_0\subset\textnormal{Im}(d_3)+(\g_\td\we\g_\td)_0;\tag{1}\label{zpz1}\end{equation}
\begin{equation}\g_0\we\g_0\we\g_0\subset\textnormal{Im}(d_4)+\g_0\we(\g_\td\we\g_\td)_0;\tag{2}\label{zpz2}\end{equation}
\end{lem}
\begin{proof}
Observe that (\ref{zpz1}) is a restatement of Lemma \ref{tameblo}.

The second assertion is similar: if $u,v$ have nonzero opposite weights and $x,y\in\g_0$, then, modulo $\textnormal{Im}(d_4)$, the element $x\we y\we [u,v]$ is equal to
$$y\we u\we [v,x]-[x,y]\we u\we v+x\we v\we [u,y]-y\we v\we [u,x]-x\we u\we [v,y],$$
which belongs to $\g_0\we(\g_\td\we\g_\td)_0$.
\end{proof}

\begin{prop}\label{expsi}
Consider the following $\RR$-module homomorphism
\begin{eqnarray*}
\Phi:\,\g_\td\ot\g_\td\ot\g_\td\ot\g_\td & \to &
(\g_\td\we\g_\td)/d_3(\g_\td\we\g_\td\we\g_\td) \\
u\ot v\ot x\ot y & \mapsto & x\we [y,[u,v]]-y\we [x,[u,v]].
\end{eqnarray*}
If $\g$ is doubly 1-tame, then there exists an $\RR$-module homomorphism $$\Psi:\g_0\ot\g_0\to 
(\g_\td\we\g_\td)_0/d_3(\g_\td\we\g_\td\we\g_\td)_0$$
such that whenever $\alpha,\beta$ are non-collinear weights, we have
$$\Psi([x,y]\ot[u,v])=\Phi(u\ot v\ot x\ot y),\quad\forall 
x\in\g_\alpha,\;y\in\g_{-\alpha},\;u\in\g_\beta,\;v\in\g_{-\beta}.$$
Moreover, $\Psi$ is antisymmetric, i.e.\ 
factors through $\g_0\wedge\g_0$. 
\end{prop}

\begin{proof}
Suppose that $\alpha,\beta$ are non-collinear weights and that $x\in\g_\alpha$, $y\in\g_{-\alpha}$, $u\in\g_\beta$, $v\in\g_{-\beta}$. We have $d_3\circ d_4(x\we y\we u\we v)=0$. If we write this down in $(\g_\td\we\g_\td)_0/d_3(\g_\td\we\g_\td\we\g_\td)_0$, four out of six terms vanish and we get
$$d_3([x,y]\we u\we v)+d_3([u,v]\we x\we y)=0,$$
which expands as 
\begin{equation}\Phi(u\ot v\ot x\ot y)+\Phi(x\ot y\ot u\ot v)=0.\label{phisy}\end{equation}
Define, for $w\in\g_0$, $\Psi_{x,y}(w)=x\we [y,w]-y\we [x,w]$. The 
mapping $$(x,y)\mapsto\Psi_{x,y}\in\Hom(\g_0,(\g_\td\we\g_\td)_0/d_3(\g_\td\we\g_\td\we\g_\td)_0\,)$$ is bilinear and in particular extends to a homomorphism $$\sigma:(\g_\td\ot\g_\td)_0\to \Hom(\g_0,(\g_\td\we\g_\td)_0/d_3(\g_\td\we\g_\td\we\g_\td)_0\,).$$

Since $\g$ is doubly 1-tame, any $w\in\g_0$ can be written as $\sum 
[u_i,v_i]$ with $u_i\in\g_{\beta_i}$, $v_i\in\g_{-\beta_i}$, 
$\beta_i\notin\{0,\pm\alpha\}$, so, using (\ref{phisy})
\begin{align*}\Psi_{x,y}(w)= & \sum_i\Phi(u_i\ot v_i\ot x\ot y) \\
= & -\sum_i\Phi(x\ot y\ot u_i\ot v_i)=\sum_i\Psi_{u_i,v_i}([x,y]).\end{align*}
This shows that $\sigma(x\ot y)$ only depends on $[x,y]$, i.e.\ we can 
write $\sigma(x\ot y)=\sigma'([x,y])$. Define, for $z,w\in\g_0$
$$\Psi(z\ot w)=\sigma'(z)(w).$$
By construction, whenever $z=[x,y]$ and $w=[u,v]$,
with $x\in\g_\alpha$, $y\in\g_{-\alpha}$, $u\in\g_\beta$,
$v\in\g_{-\beta}$ and $\alpha,\beta$ are not collinear, we have
$$\Psi([x,y]\ot [u,v])=\Phi(u\ot v\ot x\ot y);$$
from (\ref{phisy}) we see in particular that $\Psi$ is antisymmetric.
\end{proof}

\begin{proof}[Proof of Theorem~\ref{h2simp}]
If $\g$ is 1-tame, the surjectivity immediately follows from Lemma \ref{zzpzz}(\ref{zpz1}).

Now to show the injectivity of the map of the theorem, suppose that 
$c\in(\g_\td\we\g_\td)_0$ is a 2-boundary and let us show that $c$ 
belongs to $d_3(\g_\td\we\g_\td\we\g_\td)_0$. In view of Lemma 
\ref{zzpzz}(\ref{zpz2}), we already know that $c$ belongs to 
$d_3(\g\we\g_\td\we\g_\td)_0$, and let us work again modulo 
$d_3(\g_\td\we\g_\td\we\g_\td)_0$, so that we can suppose that $c$ 
belongs 
to $d_3(\g_0\we(\g_\td\we\g_\td)_0)$, and we wish to check that $c=0$.
Since $\g$ is doubly 1-tame, we can write 
$$c=\sum d_3([u_i,v_i]\we x_i\we y_i)$$ with $x_i\in\g_{\alpha_i}$, $y_i\in\g_{-\alpha_i}$, $u_i\in\g_{\beta_i}$, $v_i\in\g_{-\beta_i}$, $\alpha_i,\beta_i$ nonzero and $\alpha_i\neq \pm\beta_i$. Write $w_i=[u_i,v_i]$.
Then 
$$c=\left(\sum_i w_i\we [x_i,y_i]\right)+\left(\sum_i(x_i\we[y_i,w_i]+y_i\we [w_i,x_i])\right),$$
the first term belongs to $\g_0\we\g_0$ and the second to 
$(\g_\td\we\g_{\td})_0/d_3(\g_\td\we\g_\td\we\g_\td)_0$; since $c$ is 
assumed to lie in $(\g_\td\we\g_\td)_0$ we 
deduce that 
\begin{equation}\sum_i w_i\we [x_i,y_i]=0\label{wxy}\end{equation}
 in $\g_0\we\g_0$. Therefore 
\[c= \sum_i x_i\we[y_i,w_i]+y_i\we [w_i,x_i]
= \sum_i\Phi(u_i\ot v_i\ot x_i\ot y_i).\]
Now using Proposition \ref{expsi} we get
\[	c = \sum_i\Psi(w_i\we [x_i,y_i])
	 = \Psi\left(\sum_i w_i\we [x_i,y_i]\right)=0\]
again by (\ref{wxy}).
\end{proof}

\subsubsection{The tame Killing module}

\begin{defn}
Let $\g$ be a graded Lie algebra over $\RR$. Consider the $\RR$-module homomorphism
\begin{eqnarray*}
  \mathcal{T}:(\g\cc_\RR \g)\ot_\RR \g  &  \to  & \g\cc_\RR \g \\
 u\cc v\ot w &\mapsto & u\cc [v,w]+v\cc [u,w].
\end{eqnarray*}
By definition, $\Kill(\g)$ is the cokernel of $\mathcal{T}$. We define $\Kill^\td(\g)_0$ as the cokernel of the restriction of $\mathcal{T}$ to $$((\g_\td\cc\g_\td)\ot\g_\td)_0\to(\g_\td\cc\g_\td)_0.$$
\end{defn}

Note that $\mathcal{T}$ satisfies the identities, for all $x,y,z$
$$\mathcal{T}(x\cc y\ot z)+\mathcal{T}(y\cc z\ot x)+\mathcal{T}(z\cc x\ot y)=0.$$

There is an canonical homomorphism $\Kill^\td(\g)_0\to\Kill(\g)_0$.

\begin{thm}\label{killt}
Let $\g$ be a graded Lie algebra. If $\g$ is 1-tame then the homomorphism $\Kill^\td(\g)_0\to\Kill(\g)_0$ is surjective; if $\g$ is doubly 1-tame then it is an isomorphism.
\end{thm}

\begin{lem}\label{tformula}
Let $\g$ be an arbitrary Lie algebra.
Then we have the identity
\begin{align*}\mathcal{T}(w,x,[y,z]) =& \;\;\;\,\mathcal{T}([x,z],w,y)-\mathcal{T}([x,y],w,z)\\&-\mathcal{T}([w,y],x,z)+\mathcal{T}([w,z],x,y);
\end{align*}
\end{lem}
\begin{proof}Use the four equalities
\begin{align*}
\mathcal{T}([x,z],w,y) =& w\cc [[x,z],y]+[x,z]\cc [w,y],\\
\mathcal{T}([y,x],w,z) =& w\cc [[y,x],z]+[y,x]\cc [w,z],\\
\mathcal{T}([w,z],x,y) =& x\cc [[w,z],y]+[w,z]\cc [x,y],\\
\mathcal{T}([y,w],x,z) =& x\cc [[y,w],z]+[y,w]\cc [x,z];
\end{align*}
the sum of the four right-hand terms is, by the Jacobi identity and cancelation of $([\cdot,\cdot]\cc[\cdot,\cdot])$-terms, equal to
$$-w\cc [[z,y],x]-x\cc[[z,y],w]=\mathcal{T}(w,x,[z,y]).\qedhere$$
\end{proof}

\reqnomode

\begin{lem}\label{caf}
Let $\g$ be an arbitrary graded Lie algebra.

\begin{itemize}
\item Let 
$\alpha,\alpha',\beta,\beta'$ be 
nonzero weights, with $\alpha+\beta,\alpha+\beta',\alpha'+\beta,\alpha'+\beta'\neq 0$, and $(x,x',y,y')\in\g_\alpha\times\g_{\alpha'}\times\g_{\beta}\times\g_{\beta'}$. Then, modulo $\mathcal{T}(\g_\td\cc\g_\td\ot\g_\td)$, we have 
\begin{align}\mathcal{T}(x,x',[y,y'])=0.\tag{1}\label{caf1}\end{align}
and
\begin{align}\mathcal{T}([x,x'],y,y')=\mathcal{T}([y,y'],x,x').\tag{2}\label{caf2}\end{align}
\item Let $\alpha_0,\beta,\gamma,\gamma'$ be weights with $\beta,\gamma,\gamma'\neq 0$, $\alpha_0\notin\{-\gamma,-\gamma'\}$.
For all $w_0\in\g_{\alpha_0},x\in\g_{\beta},y\in\g_{\gamma},y'\in\g_{\gamma'}$ we have
\begin{align}\mathcal{T}(w_0,x,[y,y'])=\mathcal{T}([x,y'],w_0,y)-\mathcal{T}([x,y],w_0,y').\tag{3}\label{caf3}\end{align}
\item Let $\alpha,\alpha',\beta,\beta'$ be nonzero weights with $\alpha+\alpha',\beta+\beta'\neq 0$.
For all $x\in\g_\alpha,x'\in\g_{\alpha'},y\in\g_{\beta},y'\in\g_{\beta'}$ we have
\begin{align}\mathcal{T}(x,y,[x',y'])=\mathcal{T}([x,y'],y,x')-\mathcal{T}([y,x'],x,y').\tag{4}\label{caf4}\end{align}
\end{itemize}
\end{lem}
\begin{proof}These are immediate from the formula given by Lemma 
\ref{tformula} applied to $(x,x',y,y')$, resp.\ $(x,y,x',y')$, resp.\ $(w_0,x,y,y')$, resp.\ $(x,y,x',y')$.
\end{proof}

\begin{lem}\label{cad}
Let $\g$ be a doubly 1-tame graded Lie algebra.
\begin{itemize}
\item
Let $\alpha,\beta,\gamma$ be 
weights with $\alpha,\beta\neq 0$ and $\alpha+\beta+\gamma=0$, and
$(x,y,z)\in\g_\alpha\times\g_{\beta}\times\g_{\gamma}$. Then, modulo $\mathcal{T}(\g_\td\cc\g_\td\ot\g_\td)$, we have
\begin{align}\mathcal{T}(x,y,z)=0.\tag{1}\label{cad1}\end{align}
\item Let 
$\alpha,\alpha',\beta,\beta'$ be 
weights, with $\alpha,\beta\neq 0$ and $\alpha+\alpha'+\beta+\beta'=0$, 
and $(x,x',y,y')\in\g_{\alpha}\times\g_{\alpha'}\times\g_{\beta}\times\g_{\beta'}$. Then, modulo $\mathcal{T}(\g_\td\cc\g_\td\ot\g_\td)$, we have
\begin{align}\mathcal{T}([x,y'],y,x')=\mathcal{T}([y,x'],x,y').\tag{2}\label{cad2}\end{align}
\end{itemize}\end{lem}
\begin{proof}
Let us check the first assertion. If $\gamma\neq 0$ this is trivial, so assume 
$\gamma=0$ (so $\alpha=-\beta$). Since $\g$ is doubly 1-tame, we can write $z=\sum 
[u_i,v_i]$ with $u_i,v_i$ of opposite weights, not equal to 
$\pm\alpha$. Then Lemma \ref{caf}(\ref{caf1})
applies.

Let us check the second assertion. We begin with two particular cases
\begin{itemize}
\item $\alpha+\beta'\neq 0$. Then $\alpha'+\beta\neq 0$ and by (\ref{cad1}), modulo $\mathcal{T}(\g_\td\cc\g_\td\ot\g_\td)$, both $\mathcal{T}([x,y'],y,x')$ and $\mathcal{T}([y,x'],x,y')$ are zero.
\item $\alpha+\beta'=0$, $\alpha+\alpha'\neq 0$. Then $\alpha',\beta'$ are nonzero, so we obtain the assertion by applying Lemma \ref{caf}(\ref{caf4}) and then (\ref{cad1}) of the current lemma.
\end{itemize}
Finally, the only remaining case is when $(\alpha,\beta,\alpha',\beta')\neq (\alpha,\alpha,-\alpha,-\alpha)$, as we assume now.

To tackle this last case, let us use this to prove first the following: 
if $\gamma$ is a nonzero weight, $w_0\in\g_0$, $z\in\g_\gamma$ 
,$z'\in\g_{-\gamma}$, then 
$\mathcal{T}(w_0,z,z')+\mathcal{T}(w_0,z',z)=0$. Indeed, since $\g$ is doubly 1-tame, this reduces by linearity to $w_0=[u,u']$ with 
$u\in\g_\delta$, $u'\in\g_{-\delta}$ and $\delta\notin\{0,\pm\gamma\}$. So, 
using twice 
(\ref{cad2}) in one of the cases already proved, we obtain 
\begin{align*}
\mathcal{T}(w_0,z,z') =& \mathcal{T}([u,u'],z,z')\\
	=& \mathcal{T}([z,z'],u,u')\\
	=& -\mathcal{T}([z',z],u,u')\\
	=& -\mathcal{T}([u,u'],z',z)\\
	=& -\mathcal{T}(w_0,z',z)
\end{align*}

Now suppose that $(\alpha,\beta,\alpha',\beta')=(\alpha,\alpha,-\alpha,-\alpha)$.
Then, using the antisymmetry property above and again using one last time one already known case of (\ref{cad2}), we obtain

\begin{align*}
\mathcal{T}([x,y'],y,x') =& -\mathcal{T}([x,y'],x',y)\\
	=& -\mathcal{T}([x',y],x,y')\\
	=& \;\;\mathcal{T}([y,x'],x,y').\qedhere
\end{align*}
\end{proof}

\leqnomode

\begin{lem}\label{redsym}
Let $\g$ be a 1-tame graded Lie algebra. Then
\begin{enumerate}
\item\label{redsym1} $(\g\cc\g)_0=\mathcal{T}(\g\cc\g\ot\g)_0+(\g_\td\cc\g_\td)_0$;
\item\label{redsym2} $\mathcal{T}(\g\cc\g_\td\ot\g_\td)_0=\mathcal{T}(\g\cc\g\ot\g)_0$.
\end{enumerate}
\end{lem}
\begin{proof}
Suppose that $x,y\in\g_0$. To show that $x\cc y$ belongs to the right-hand term in (\ref{redsym1}), it suffices by linearity to deal with the case when $y=[u,v]$ with $u,v$ homogeneous of nonzero opposite weight. Then 
$$x\cc [u,v]=\mathcal{T}(x,u,v)-u\cc [x,v],$$
which is the sum of an element in $((\g\cc\g)\ot\g)_0$ and an element in $(\g_\td\cc\g_\td)_0$. So (\ref{redsym1}) is proved.

Let us prove (\ref{redsym2}). By linearity, it is enough to prove that any element $\mathcal{T}(x,y,[u,v])$, where $x,y$ have weight zero and $u,v$ have nonzero opposite weight, belongs to $\mathcal{T}(\g\ot\g_\td\ot\g_\td)_0$: the formula in Lemma \ref{tformula} expresses $\mathcal{T}(x,y,[u,v])$ as a sum of four terms in $\mathcal{T}(\g\ot\g_\td\ot\g_\td)_0$.
\end{proof}

\begin{prop}\label{tsym}
Let $\g$ be a doubly 1-tame graded Lie algebra. Then
$$\mathcal{T}(\g\cc\g\ot\g)\cap (\g_\td\cc\g_\td)_0\subset\mathcal{T}((\g_\td\cc\g_\td\ot\g_\td)_0).$$
\end{prop}
\begin{proof}
Fix $u,v\in\n_\beta,\n_{-\beta}$ ($\beta\neq 0$) and consider the $\RR$-module homomorphism
\begin{eqnarray*}\Phi_{u,v}:\g_0 &\to& M=(\g\cc\g)_0/\mathcal{T}((\g_\td\cc\g_\td\ot\g_\td)_0)\\
w&\mapsto &u\cc [v,w].
\end{eqnarray*}
The mapping $(u,v)\mapsto\Phi_{u,v}\in\Hom_\RR(\g_0,M)$ is bilinear.
Therefore it extends to a mapping 
$s\mapsto\hat{\Phi}_{s}$ defined for all $s\in(\g_\td\ot\g_\td)_0$. 

If $s=\sum x_i\ot y_i\in\g\we\g$, we write $\langle s\rangle =\sum 
[x_i,y_i]$. 
Now, for $s,s'\in (\g_\td\ot\g_\td)_0$, define $\hat{\Psi}(s\ot 
s')=\hat{\Phi}_s(\langle s'\rangle)\in M$. 
In other words, $$\hat{\Psi}((u\ot v)\ot (x\ot y))=u\cc [v,[x,y]].$$

We have\footnote{Given a map defined on a tensor product such as $\Psi$, we freely view it as a multilinear map when it is convenient.}
\begin{align*}
\hat{\Psi}((u\ot v)\ot (x\ot y)) =& u\cc [v,[x,y]]\\
 =& -\mathcal{T}([x,y],u,v)+[x,y]\cc [u,v]
\end{align*}
and similarly
$$\hat{\Psi}((x\ot y)\ot(u\ot v))=-\mathcal{T}([u,v],x,y)+[u,v]\cc 
[x,y],$$
so
\begin{align*} & \hat{\Psi}((u\ot v)\ot (x\ot y))-\hat{\Psi}((x\ot y)\ot(u\ot v))\\
= & \mathcal{T}([u,v],x,y)-\mathcal{T}([x,y],u,v)=0\end{align*}
by Lemma \ref{cad}(\ref{cad2}).
Thus, $\hat{\Psi}$ is symmetric and we can write $\hat{\Psi}(s\ot s')=\hat{\Psi}(s\cc s')$.
Note that (trivially) $\hat{\Psi}(s\cc s')=0$ whenever $s'$ is a 2-cycle 
(i.e.\ $\langle s'\rangle=0$), so by the symmetry $\hat{\Psi}$ factors 
through a map $\Psi:\g_0\cc\g_0\to M$ such that $$\hat{\Psi}(s\cc 
s')=\Psi(\langle 
s\rangle\cc\langle s'\rangle)$$ for all $s,s'\in(\g_\td\we\g_\td)_0$.
In other words, we can write  $$u\cc [v,[x,y]]=\Psi([u,v]\cc [x,y]).$$

Now consider some element in $\mathcal{T}(\g\cc\g\ot\g)\cap (\g_\td\cc\g_\td)_0$. By Lemma \ref{redsym}, it can be taken in $\mathcal{T}(\g\cc\g_\td\ot\g_\td)$. We write it as 
$$\tau=\sum \mathcal{T}(x_i,y_i,z_i),$$
with $x_i,y_i,z_i$ homogenous and $y_i,z_i$ of nonzero weight. Since we 
work modulo $\mathcal{T}(\g_\td\cc\g_\td\ot\g_\td)_0$, we can suppose $x_i$ is of weight zero for 
all $i$, and we have to prove that $\tau=0$. So 
$$\tau=\left(\sum_i x_i\cc 
[y_i,z_i]\right)+\left(\sum_iy_i\cc[x_i,z_i]\right),$$
the first term belongs to $\g_0\we\g_0$ and the second to the quotient $\Kill^\td(\g)_0$ of $(\g_\td\we\g_{\td})_0$ by $\mathcal{T}(\g_\td\cc\g_\td\ot\g_\td)_0$, so \begin{equation}\sum_i x_i\cc [y_i,z_i]=0.\label{sunus}\end{equation}

Now, writing $x_i=\sum_j [u_{ij},v_{ij}]$ with $u_{ij},v_{ij}$ of nonzero opposite weight
\begin{align*}\tau &= \sum_i y_i\cc[x_i,z_i]\\
 	&= \sum_{i,j}y_i\cc[z_i,[u_{ij},v_{ij}]]\\
	&= \sum_{i,j}\Psi([y_i,z_i]\cc [u_{ij},v_{ij}])\\
	&= \sum_i\Psi([y_i,z_i]\cc x_i)=\Psi\left(\sum_i[y_i,z_i]\cc x_i\right)=0\quad\text{by (\ref{sunus}).}\qedhere
\end{align*}
\end{proof}

\begin{proof}[Proof of Theorem \ref{killt}]
The first statement follows from Lemma \ref{redsym}(\ref{redsym1}) and the second from Proposition \ref{tsym}.
\end{proof}

%
%


\section{Abels' multiamalgam}\label{s:am}

\subsection{2-tameness}\label{gla}

In this section, we deal with real-graded Lie algebras, that is, Lie algebras graded in a real vector space $\mathcal{W}$. As in Section \ref{s_abels}, Lie algebras are, unless explicitly specified, over the ground commutative ring $\RR$.

Let $\g$ be a real-graded Lie algebra.
We say that $\mathcal{P}\subset \mathcal{W}$ is $\g$-principal if $\g$ is generated, as a Lie algebra, by $\g_\mathcal{P}=\sum_{\alpha\in \mathcal{P}}\g_\alpha$ (note that this only depends on the structure of Lie ring, not on the ground ring $\RR$). We say that $\mathcal{P}$ (or $(\g,\mathcal{P})$) is $k$-{\it tame} if whenever $\alpha_1,\dots,\alpha_k\in \mathcal{P}$, there exists an $\R$-linear form $\ell$ on $\mathcal{W}$ such that $\ell(\alpha_i)>0$ for all $i=1,\dots,k$. Note that $\mathcal{P}$ is 1-tame if and only if $0\notin \mathcal{P}$ and is 2-tame\footnote{This paper will not deal with $k$-tameness for $k\ge 3$ but this notion is relevant to the study of higher-dimensional isoperimetry problems.} if and only if for all $\alpha,\beta\in \mathcal{P}$ we have $0\notin [\alpha,\beta]$. Note that $k$-tame trivially implies $(k-1)$-tame.

We say that the graded Lie algebra $\g$ is $k$-{\it tame} if there exists a $\g$-principal $k$-tame subset. Note that for $k=1$ this is compatible with the definition in \S\ref{def1t}.

\begin{ex}As usual, we write $\mathcal{W}_\g=\{\alpha:\g_\alpha\neq\{0\}\}$.\begin{itemize} \item $\g=\mathfrak{sl}_3$ with its standard Cartan grading, 
$\mathcal{W}_\g=\{\alpha_{ij}:1\le i\neq j\le 3\}\cup\{0\}$ (with 
$\alpha_{ij}=e_i-e_j$, $(e_i)$ denoting the canonical basis of $\R^3$); 
then $\{\alpha_{12},\alpha_{23},\alpha_{31}\}$ and 
$\{\alpha_{21},\alpha_{13},\alpha_{32}\}$ are $\g$-principal and 2-tame. 
\item If $\mathcal{P}_1$ is the set of weights of the graded Lie algebra 
$\g/[\g,\g]$, then any $\g$-principal set contains $\mathcal{P}_1$; 
conversely if 
$\g$ is nilpotent then $\mathcal{P}_1$ itself is $\g$-principal. Thus if $\g$ is nilpotent, then $\g$ is $k$-tame if and only if 0 is not in the convex hull of $k$ weights of $\g/[\g,\g]$.
\end{itemize} \end{ex}


\subsection{Lemmas related to 2-tameness}
This subsection gathers a few technical lemmas needed in the study of the multiamalgam in \S\ref{multi} and \S\ref{mulg}. The reader can skip it in a first reading.

The following lemma was proved by Abels under more specific hypotheses ($\g$ nilpotent and finite-dimensional over a $p$-adic field).
As usual, by $[\g_\beta,\g_\gamma]$ we mean the {\it module} generated by such brackets. In a real vector space, we write $\lB\alpha,\beta\rB$ for the segment joining $\alpha$ and $\beta$ (so as to avoid any confusion with Lie brackets).

\begin{lem}\label{431}
Let $\g$ be a real-graded Lie algebra and $\mathcal{P}\subset\mathcal{W}$ a $\g$-principal subset. Suppose that $(\g,\mathcal{P})$ is 2-tame. Then
\begin{enumerate}
\item\label{431c}for any $\omega\in\mathcal{W}$, we have $\g_0=\sum_{\beta}[\g_\beta,\g_{-\beta}]$, with $\beta$ ranging over $\mathcal{W}-\R\omega$;
\item\label{431b}if $\R_+\alpha\cap \mathcal{P}=\emptyset$, then $\g_\alpha=\sum[\g_\beta,\g_\gamma]$, with $(\beta,\gamma)$ ranging over pairs in $\mathcal{W}-\R\alpha$ such that $\beta+\gamma=\alpha$.
\end{enumerate}
\end{lem}


Lemma \ref{431} is a consequence of the more technical Lemma \ref{431ijk} below (with $i=1$). Actually, the proof of Lemma \ref{431} is based on an induction which makes use of the full statement of Lemma \ref{431ijk}. Besides, while Lemma \ref{431} is enough for our purposes in the study of the multiamalgam of Lie algebras in \S\ref{multi}, the statements in Lemma \ref{431ijk} involving the lower central series are needed when studying multiamalgams of nilpotent groups in \S\ref{mulg}.

\begin{lem}\label{431ijk}
Under the assumptions of Lemma \ref{431}, let $(\g^i)$ be the lower central series of $\g$ and $\g^i_\alpha=\g^i\cap\g_\alpha$ (note that $\g^i_0$ thus means $(\g^i)_0$ and can be distinct from $(\g_0)^i$). Then\begin{enumerate}
\item\label{431ijkc}for any $\omega\in\mathcal{W}$, we have $\g^i_{0}=\sum_{j+k=i}\sum_{\beta}[\g^j_{\beta},\g^k_{-\beta}]$, with $\beta$ ranging over $\mathcal{W}\smallsetminus\R\omega$;
\item\label{431ijkb}if $\R_+\alpha\cap \mathcal{P}=\emptyset$, then $\g^i_{\alpha}=\sum_{j+k=i}\sum[\g^j_{\beta},\g^k_{\gamma}]$, with $(\beta,\gamma)$ ranging over pairs in $\mathcal{W}\smallsetminus\R\alpha$ such that $\beta+\gamma=\alpha$; 
\item\label{431ijka}if $\alpha\notin \mathcal{P}$, then $\g^i_{\alpha}=\sum_{j+k=i}\sum [\g^j_{\beta},\g^k_{\gamma}]$, with $(\beta,\gamma)$ ranging over pairs in $\mathcal{W}$ such that $\beta+\gamma=\alpha$ and $0$ does not belong to the segment $\lB\beta,\gamma\rB$. 
\end{enumerate}
\end{lem}

\begin{proof}
Define 
$\g^{[1]}=\sum_{\alpha\in \mathcal{P}}\g_\alpha$, and by induction 
the submodule $\g^{[i]}=[\g^{[1]},\g^{[i-1]}]$ for $i\ge 2$; note that this depends on 
the choice of $\mathcal{P}$. Define 
$\g_\alpha^{[i]}=\g_\alpha\cap\g^{[i]}$. Note that each $\g^{[i]}$ is a 
graded submodule of $\g$.

Let us prove by induction on $i\ge 1$ the 
following statement: in the three cases, 
\begin{equation}\g_\alpha^{[i]}\subset\sum_{j,k\ge 
1,\,j+k=i}\;\sum_{\beta+\gamma=\alpha\dots} 
\left[\g_\beta^{[j]},\g_\gamma^{[k]}\right]\label{ijk}\end{equation} 
where in each case $(\beta,\gamma)$ satisfies the additional requirements of (\ref{431ijkc}), (\ref{431ijkb}), or (\ref{431ijka}) (we encode this in the notation $\sum_{\beta+\gamma=\alpha\dots}$).

Since in all cases, $\alpha\notin \mathcal{P}$, the case $i=1$ is an empty (tautological) statement. Suppose that $i\ge 2$ and the inclusions (\ref{ijk}) are proved up to $i-1$. Consider $x\in\g_\alpha^{[i]}$. By definition, $x$ is a sum of elements of the form $[y,z]$ with $y\in\g^{[1]}_\beta$, $z\in\g^{[i-1]}_\gamma$ with $\beta+\gamma=\alpha$. If $\beta$ and $\gamma$ are linearly independent over $\R$, the additional conditions are satisfied and we are done. Otherwise, since $\beta\in \mathcal{P}$, we have $\beta\neq 0$ and we can write $\gamma=r\beta$ for some $r\in\R$. There are three cases to consider: 
\begin{itemize}
\item $r>0$. Then we are in Case (\ref{431ijka}), and $0\notin \lB\beta,\gamma\rB$, so the additional condition in (\ref{431ijka}) holds.
\item $r=0$. Then $\beta=\alpha\neq 0$ is a principal weight, which is consistent with none of Cases (\ref{431ijkc}), (\ref{431ijkb}) or (\ref{431ijka}).

\item $r<0$. Then since $\beta\in \mathcal{P}$, we have $\R_+\gamma\cap 
\mathcal{P}=\emptyset$, so we can apply the induction hypothesis of Case 
(\ref{431ijkb}) to $z$; by linearity, this reduces to $z=[u,v]$ with $u\in\g^{[j]}_\delta$, $v\in\g^{[k]}_{\eps}$, 
$j+k=i-1$, $\delta+\eps=r\beta$, and $\delta,\eps$ not collinear to 
$\beta$. Note that $\alpha=\beta+\delta+\eps$. By the Jacobi identity, \begin{equation}x=[y,[u,v]]\in 
\left[\g^{[j]}_\delta,\g^{[k+1]}_{\alpha-\delta}\right]+\left[\g^{[k]}_{\eps},\g^{[j+1]}_{\alpha-\eps}\right].\label{ijkijkijk}\end{equation} 

Note that none of $\delta$, $\alpha-\delta$, $\eps$, $\alpha-\eps$ belongs to $\R\alpha$. Hence we are done in both Cases (\ref{431ijkb}) or (\ref{431ijka}). In case (\ref{431ijkc}), if $\omega\in\R\beta$, then since $\delta$ and $\eps$ are not collinear to $\omega$, we see that (\ref{ijkijkijk}) satisfies the 
additional conditions of (\ref{431ijkc}). However, the case 
$\omega\notin\R\beta$ is trivial since then the writing $x=[y,z]$ itself 
already satisfies the additional condition of (\ref{431ijkc}).
\end{itemize}

At this point, (\ref{ijk}) is proved.
By Lemma \ref{glll}, for all $i\ge 1$, we have
$\g^i=\sum_{\ell\ge i}\g^{[\ell]}$. In particular $\g^i_{\alpha}=\sum_{\ell\ge i}\g_\alpha^{[\ell]}$ and we have 
$$\g_\alpha^{[\ell]}\subset\sum_{j+k=\ell}\sum_{\beta+\gamma=\alpha\dots} \left[\g_\beta^{j},\g_\gamma^{k}\right]\subset\sum_{j+k= i}\sum_{\beta+\gamma=\alpha\dots} \left[\g_\beta^{j},\g_\gamma^{k}\right],$$
so
$$\g_\alpha^{i}\subset\sum_{j+k=i}\sum_{\beta+\gamma=\alpha\dots} \left[\g_\beta^{j},\g_\gamma^{k}\right],$$
and since the inclusion $\left[\g_\beta^{j},\g_\gamma^{k}\right]\subset \g_\alpha^{i}$ is clear, we get the desired equality
\[\g_\alpha^{i}=\sum_{j+k=i}\sum_{\beta+\gamma=\alpha\dots} \left[\g_\beta^{j},\g_\gamma^{k}\right].\qedhere\]
\end{proof}


\subsection{Multiamalgams of Lie algebras}\label{multi}

\subsubsection{The definition}

By {\it convex cone} in a real vector space, we mean any subset stable under addition and positive scalar multiplication (such a subset is necessarily convex).
Let $\mathcal{C}$ be the set of convex cones of $\mathcal{W}$ not containing 0. Let $\g$ be a real-graded Lie algebra over the ring $\RR$. If $C\in\mathcal{C}$, define $$\g_C=\bigoplus_{\alpha\in C}\g_\alpha;$$
this is a graded Lie subalgebra of 
$\g$ (if $\mathcal{W}_\g$ is finite, $\g_C$ is nilpotent). Denote by $x\mapsto\bar{x}$ the inclusion of $\g_C$ into $\g$.

\begin{defn}Define $\hat{\g}=\hat{\g}^{\RR}$ as the multiamalgam (or colimit) of all $\g_C$, where $C$ ranges over $\mathcal{C}$.\end{defn}

This is by definition an initial object in the category of Lie algebras $\mk{h}$ endowed with compatible homomorphisms $\g_C\to\mk{h}$.

It can be realized as the quotient of the Lie $\RR$-algebra free product of all $\g_C$, by the ideal generated by elements $x-y$, where $x,y$ range over elements in $\g_C,\g_D$ such that $\bar{x}=\bar{y}$ and $C,D$ range over $\mathcal{C}$. Note that among those relators, we can restrict to homogeneous $x,y$ as the other ones immediately follow. Since the free product as well as the ideal are graded, $\hat{\g}$ is a graded Lie algebra; in particular, $\hat{\g}$ is also the multiamalgam of the $\g_C$ in the category of Lie algebras graded in $\mathcal{W}$. The inclusions $\g_C\to\g$ induce a natural homomorphism $\hat{\g}\to\g$.

\subsubsection{Link with the blow-up}

Let $\tilde{\g}$ be the blow-up introduced in \S\ref{s_bu}. For every $C\in\mathcal{C}$, the structural homomorphism $\tilde{\g}_C\to\g_C$ is an isomorphism, and therefore we obtain compatible homomorphisms $\g_C\to\tilde{\g}_C\subset\g$, inducing, by the universal property, a natural graded homomorphism $\hat{\g}\to\tilde{\g}$.

\begin{thm}\label{kappaisa}
If $\g$ is 1-tame, then the natural Lie algebra homomorphism $\kappa:\hat{\g}^{\RR}\to\tilde{\g}^{\RR}$ is surjective, and if $\g$ is 2-tame, $\kappa$ is an isomorphism. In particular, if $\g$ is 2-tame then the kernel of $\hat{\g}^{\RR}\to\g$ is central in $\hat{\g}$ and is canonically isomorphic (as an $\RR$-module) to $H_2^{\RR}(\g)_0$.
\end{thm}

To prove the second statement, we need the following lemma.

\begin{lem}[Abels]\label{472lem}
If $\g$ is 2-tame, then $\hat{\g}\to\g$ has central kernel, concentrated in degree zero.
\end{lem}

\begin{proof}[Proof of Theorem \ref{kappaisa} from Lemma \ref{472lem}]

If $\g$ is 1-tame, then so is $\tilde{\g}$ by Lemma \ref{tameblo}, and the surjectivity of $\kappa$ follows, proving the first assertion.

If $\g$ is 2-tame, then by Lemma \ref{472lem}, $\hat{\g}\to\g$ has central kernel concentrated in degree zero, so by the universal property of the blow-up (Theorem \ref{uc0}), there is a section $s$ of the natural map $\kappa$.
By uniqueness in the universal properties, $s\circ\kappa$ and $\kappa\circ s$ are both identity and we are done.

The last statement is then an immediate consequence of Lemma \ref{blowupc}.
\end{proof}

\begin{proof}[Proof of Lemma \ref{472lem}]

If $\alpha$ is a nonzero weight of $\g$, then there exists $C\in\mathcal{C}_\g$ such that $\alpha\in C$. By the amalgamation relations, the composite graded homomorphism $\g_\alpha\to\g_C\to\hat{\g}$ does not depend on the choice of $C$. We thus call it $i_\alpha$ and call its image $\mathfrak{m}_\alpha$.

Let us first check that whenever $\alpha,\beta$ are non-zero non-opposite weights, then we have the following inclusion in $\hat{\g}$
\begin{equation}\label{472a}[\mk{m}_\alpha,\mk{m}_\beta]\subset\mk{m}_{\alpha+\beta}.\end{equation}

We begin with the observation that if $\gamma,\delta$ are nonzero weights with $\delta\notin\R_-\gamma$, then \begin{equation}i_{\gamma+\delta}([\g_\gamma,\g_\delta])=[\mk{m}_\gamma,\mk{m}_\delta].\label{ialp}\end{equation} Indeed, in the above definition of $i$, we can choose $C$ to contain both $\gamma$ and $\delta$.
In particular if $\alpha\notin\R_-\beta$, then (\ref{472a}) is clear. Let us prove (\ref{472a}) assuming that $\alpha\in\R_-\beta$. By 2-tameness, we can suppose that $\R_+\beta\cap\mathcal{P}=\emptyset$, so by Lemma \ref{431}(\ref{431b}), $\g_\beta\subset\sum[\g_\gamma,\g_\delta]$, where $(\gamma,\delta)$ ranges over the set $\mathcal{Q}(\beta)$ of pairs of weights such that $\gamma+\delta=\beta$ and $\gamma,\delta$ are not in $\R\beta(=\R\alpha)$. So, applying (\ref{ialp}), we obtain $\mk{m}_\beta\subset\sum_{(\gamma,\delta)\in \mathcal{Q}(\beta)}[\mk{m}_\gamma,\mk{m}_\delta]$. Therefore
$$[\mk{m}_\alpha,\mk{m}_\beta] \subset \sum_{(\gamma,\delta)\in \mathcal{Q}(\beta)} [\mk{m}_\alpha,[\mk{m}_\gamma,\mk{m}_\delta]].$$
Note that $\gamma$ and $\alpha+\delta$ are not collinear, since otherwise we would infer that $\alpha=-\beta$. If we fix such $(\gamma,\delta)\in \mathcal{Q}(\beta)$, we get, by the Jacobi identity and using that $\alpha,\gamma,\delta$ are pairwise non-collinear
\begin{align*}[\mk{m}_\alpha,[\mk{m}_\gamma,\mk{m}_\delta]] \subset & [\mk{m}_\gamma,[\mk{m}_\delta,\mk{m}_\alpha]] + [\mk{m}_\delta,[\mk{m}_\alpha,\mk{m}_\gamma]]\\ \subset & 
[\mk{m}_\gamma,\mk{m}_{\delta+\alpha}] + [\mk{m}_\delta,\mk{m}_{\alpha+\gamma}]\\ \subset & \mk{m}_{\gamma+\delta+\alpha}=\mk{m}_{\alpha+\beta},
\end{align*}
and finally $[\mk{m}_\alpha,\mk{m}_\beta]\subset\mk{m}_{\alpha+\beta}$. So (\ref{472a}) is proved.

Now define $\mk{v}_0$ as the submodule of $\hat{\g}$
\begin{equation}\label{472w}\mk{v}_0=\sum_{\gamma\neq 0}[\mk{m}_\gamma,\mk{m}_{-\gamma}].\end{equation}
We are going to check that for every $\alpha\neq 0$
\begin{equation}[\mk{m}_\alpha,\mk{v}_0]\subset \mk{m}_\alpha\label{472b}\end{equation}
and
\begin{equation}\label{472c}[\mk{v}_0,\mk{v}_0]\subset \mk{v}_0.\end{equation}

Before proving (\ref{472b}), let us first check that for every nonzero $\alpha$, if $\mathcal{X}(\alpha)$ is the set of nonzero weights not collinear to $\alpha$ then 
\begin{equation}
[\mk{m}_\alpha,\mk{m}_{-\alpha}]\subset \sum_{\beta\in \mathcal{X}(\alpha)}[\mk{m}_\beta,\mk{m}_{-\beta}].\label{derout}
\end{equation}
Indeed, by 2-tameness, we can suppose that $\R_+(-\alpha)\cap\mathcal{P}=\emptyset$, and apply Lemma \ref{431}(\ref{431b}), so 
$$\g_{-\alpha}\subset\sum_{(\gamma,\delta)\in \mathcal{Q}(-\alpha)}[\g_\gamma,\g_{\delta}];$$
by (\ref{ialp}) we deduce
$$\mk{m}_{-\alpha}\subset\sum_{(\gamma,\delta)\in \mathcal{Q}(-\alpha)}[\mk{m}_\gamma,\mk{m}_{\delta}];$$
so
$$[\mk{m}_\alpha,\mk{m}_{-\alpha}]\subset\sum_{(\gamma,\delta)\in \mathcal{Q}(-\alpha)}[\mk{m}_\alpha,[\mk{m}_\gamma,\mk{m}_{\delta}]].$$
If $(\gamma,\delta)\in \mathcal{Q}(-\alpha)$, we have, by the Jacobi identity and then (\ref{472a})
\begin{align*}[\mk{m}_\alpha,[\mk{m}_\gamma,\mk{m}_{\delta}]]
\subset & [\mk{m}_\gamma,[\mk{m}_\delta,\mk{m}_{\alpha}]]+[\mk{m}_\delta,[\mk{m}_\alpha,\mk{m}_{\gamma}]]\\
\subset & [\mk{m}_\gamma,\mk{m}_{\delta+\alpha}]+[\mk{m}_\delta,\mk{m}_{\alpha+\gamma}]\\
= & [\mk{m}_\gamma,\mk{m}_{-\gamma}]+[\mk{m}_\delta,\mk{m}_{-\delta}];
\end{align*}
we thus deduce (\ref{derout}).

We can now prove (\ref{472b}), namely if $\alpha,\gamma\neq 0$, then $[\mk{m}_\alpha,[\mk{m}_{\gamma},\mk{m}_{-\gamma}]]\subset\mk{m}_\alpha$. By (\ref{derout}) we can assume that $\alpha$ and $\gamma$ are not collinear, in which case (\ref{472b}) follows from (\ref{472a}) by an immediate application of Jacobi's identity.

To prove (\ref{472c}), we need to prove that $\mk{v}_0$ is a subalgebra, or equivalently that for every $\alpha\neq 0$, we have
$$[\mk{v}_0,[\mk{m}_{\alpha},\mk{m}_{-\alpha}]]\subset\mk{v}_0.$$
Indeed, using the Jacobi identity and then (\ref{472b})
\begin{align*}[\mk{v}_0,[\mk{m}_{\alpha},\mk{m}_{-\alpha}]]\subset &
[\mk{m}_\alpha,[\mk{m}_{-\alpha},\mk{v}_{0}]]+[\mk{m}_{-\alpha},[\mk{v}_0,\mk{m}_\alpha]]\\
\subset & [\mk{m}_\alpha,\mk{m}_{-\alpha}]+[\mk{m}_{-\alpha},\mk{m}_\alpha]\subset\mk{v}_0.
\end{align*}
so (\ref{472c}) is proved.

We can now conclude the proof of the lemma.
By the previous claims (\ref{472a}), (\ref{472b}), and (\ref{472c}), if $\mk{v}_0$ is defined as in (\ref{472w}) the submodule $\left(\bigoplus_{\alpha\neq 
0}\mk{m}_\alpha\right)\oplus\mk{v}_0$ is a Lie subalgebra of 
$\hat{\g}$. Since the $\mk{m}_\alpha$ for $\alpha\neq 0$ 
generate $\hat{\g}$ by definition, this proves that this Lie subalgebra is all of $\hat{\g}$.

Therefore, if $\phi:\hat{\g}\to\g$ is the natural map, we see that $\phi_\alpha$ is the natural isomorphism $\mk{m}_\alpha\to\g_\alpha$. Thus $\Ker(\phi)$ is contained in $\hat{\g}_0$. If $z\in\Ker(\phi)$ and $x\in\mk{m}_\alpha$ for some $\alpha\neq 0$, then $$\phi([z,x])=[\phi(z),\phi(x)]=[0,\phi(x)]=0,$$ so $[z,x]=\phi_\alpha^{-1}(\phi([z,x]))=0$. Thus $z$ centralizes $\mk{m}_\alpha$ for all $\alpha$; since these generate $\hat{\g}$, we deduce that $z$ is central.
\end{proof}

\subsubsection{On the non-2-tame case}

There is a partial converse to Theorem \ref{kappaisa}: if $\RR$ is a field and $\g/[\g,\g]$ is not 2-tame, then $\hat{\g}$ has a surjective homomorphism onto a free Lie algebra on two generators. In particular, if $\g$ is nilpotent (as in all our applications), $\tilde{\g}$ is nilpotent as well but $\hat{\g}$ is not. The argument is straightforward: the assumption implies that there is a graded surjective homomorphism $\g\to\mk{h}$, where $\mk{h}$ is the abelian 2-dimensional algebras with weights $\alpha,\beta$ with $\beta\in\R_-\alpha$, inducing a surjective homomorphism $\hat{\g}\to\hat{\mk{h}}$. Now it follows from the definition that $\mk{h}$ is the free product of the 1-dimensional Lie algebras $\mk{h}_\alpha$ and $\mk{h}_\beta$ and hence is not nilpotent.


\subsection{Multiamalgams of groups}\label{mulg}

In this part, $\g$ is a {\em nilpotent} real-graded Lie 
algebra over a commutative $\Q$-algebra $\RR$ of 
characteristic zero (although the results would work in characteristic 
$p>s+1$, where $s$ is the nilpotency length of the Lie algebra 
involved).

For $C\in\mathcal{C}_\g$, let $G$ and $G_C$ be the groups associated to $\g$ 
and $\g_C$ by Malcev's equivalence of categories between nilpotent Lie algebras over $\Q$ and uniquely divisible nilpotent groups, described in Theorem 
\ref{malcev}. Let the embedding $G_C\to G$ corresponding to 
$\g_C\to\g$ be written as $g\mapsto \bar{g}$. Let $\hat{G}$ be the 
corresponding amalgam, namely the group generated by the free product of 
$G_C$ for $C\in\mathcal{C}_\g$, modded out by the relators $xy^{-1}$ 
whenever $\bar{x}=\bar{y}$.

The following result was proved by Abels \cite[Cor.\ 4.4.14]{A} assuming 
that $\RR=\Q_p$ and 
that $\g$ 
is finite-dimensional. This is one of the most delicate points in 
\cite{A}. We provide a sketch of the (highly technical) proof, in order to indicate how his proof works over our general hypotheses.

\begin{thm}[Abels]Suppose that $\g$ is 2-tame and $s$-nilpotent. Then $\hat{G}$ is $(s+1)$-nilpotent.\label{gn}
\end{thm}

For the purpose of Section \ref{s_spec}, we need a stronger result. Abels asked
\cite[4.7.3]{A} whether $\hat{G}\to G$ is always a central extension. This is answered in the positive by the following theorem.

\begin{thm}\label{oabels}Under the same hypotheses, 
the nilpotent group $\hat{G}$ is uniquely divisible, and its Lie algebra is $\hat{\g}^\Q$ in the natural way. In particular, the homomorphism $\hat{G}\to G$ has a central kernel, naturally isomorphic to $H_2^\Q(\g)_0$ as a $\Q$-linear space.
\end{thm}

\begin{proof}
Our proof is based on Theorem \ref{gn} and some generalities about nilpotent groups, which are gathered in \S\ref{adumn}.

Since by Theorem \ref{gn}, $\hat{G}$ is nilpotent, and since it is generated by divisible subgroups, it is divisible (see Lemma \ref{divsub}).
Therefore by the (standard) Lemma \ref{tfdud}, to check that $\hat{G}$ 
is 
uniquely divisible, it is enough to check that it is torsion-free.
Since $\hat{G}$ is known to be $(s+1)$-nilpotent, we see that $\hat{G}$ is the multiamalgam (=colimit) of the $G_C$ {\it within the category $\mathcal{K}$ of $(s+1)$-nilpotent groups}. Therefore, $\hat{G}$ is the quotient of the free product $W$ in $\mathcal{K}$ of the $G_C$ by amalgamations relations. By Lemma \ref{freenil}, $W$ is a uniquely divisible torsion-free nilpotent group. Since, for $x,y\in W$, whenever $xy^{-1}$ is an amalgamation relation, $x^ry^{-r}$ is an amalgamation relation as well for all $r\in\Q$, Proposition \ref{amaldiv} applies to show that the normal subgroup $N$ generated by amalgamation relations is divisible. Therefore $\hat{G}/N$ is torsion-free, hence uniquely divisible. 

It follows that $\hat{G}$ is also the multiamalgam of the $G_C$ {\it in the category $\mathcal{K}_0$ of uniquely divisible $(s+1)$-nilpotent groups}. By Malcev's Theorem \ref{malcev}, this category is equivalent to the category of $(s+1)$-nilpotent Lie algebras over $\Q$. Therefore, if $H$ is the group associated to $\hat{\g}$ and $H\to G_C$ are the homomorphisms associated to $\g_C\to \hat{\g}$, then $H$ and the family of homomorphisms $G_C\to H$ satisfy the universal property of multiamalgam in the category $\mathcal{K}_0$, and this gives rise to a canonical isomorphism $\hat{G}\to H$. In particular, by Theorem \ref{kappaisa}, the kernel $W$ of $\hat{G}\to G$ is central in $\hat{G}$, and isomorphic to $H_2^\Q(\g)_0$.
\end{proof}

\begin{cor}\label{oabelsh}
Under the same hypotheses, if moreover $H_2^{\RR}(\g)_0=\{0\}$ then the central kernel of $\hat{G}\to G$ is generated, as an abelian group, by elements of the form $$\exp([\lambda x,y])\exp([x,\lambda y])^{-1}, \quad\lambda\in \RR,\;\;x,y\in\bigcup_C\g_C.$$ If $\RR=\prod_{j=1}^\tau \RR_j$ is a finite product of $\Q$-algebras (so that $\g=\prod_j\g_j$ canonically), then those elements $\exp([\lambda x,y])\exp([x,\lambda y])^{-1}$ with $\lambda\in \RR_j$ and $x,y\in\g_j$ are enough.
\end{cor}
\begin{proof}
The additional assumption and Theorem \ref{oabels} imply that the kernel of $\hat{G}\to G$ is naturally isomorphic, as a $\Q$-linear space, to the kernel $W_2^{\Q,\RR}(\g)$ of the natural map $H_2^{\Q}(\g)_0\to H_2^{\RR}(\g)_0$. 

Let $N$ be the normal subgroup of $G$ generated by elements of the form
$$\exp([\lambda x,y])\exp([x,\lambda y])^{-1},\quad x,y\in\bigcup \g_C,\;\lambda\in \RR;$$ clearly $N$ is contained in the kernel $W$ of $\hat{G}\to G$. By Proposition $\ref{amaldiv}$, $N$ is divisible, i.e., $N$ is a $\Q$-linear subspace of $W$. Therefore $\hat{G}/N$ is uniquely divisible and its Lie algebra can be identified with $\hat{\g}/N$.
By definition of $N$, in $\hat{G}/N$ we have 
$\exp([\lambda x,y])=\exp([x,\lambda y])$ for all $x,y\in\bigcup \g_C$. So in the Lie algebra of $\hat{G}/N$, which is equal to $\hat{\g}/N$ we have $[\lambda x,y]=[x,\lambda y]$. Since by Proposition \ref{linrr}, the subgroup $W$ is generated by elements of the 
form $\exp([\lambda x,y]-[x,\lambda y])$ when $x,y$ range over $\bigcup G_C$ and $\lambda$ ranges over $\RR$, it follows that $N=W$.

It remains to prove the last statement. By Lemma \ref{prodbu}, $H_2(\g)_0$ can be identified with the product $\prod_jH_2(\g_j)_0$. In particular, by Proposition \ref{linrr}, it is generated by elements of the form $[\lambda x,y]-[x,\lambda y]$ when $x,y\in\bigcup\g_{j,C}$, $\lambda\in \RR_j$, and $j=1,\dots,\tau$. Then, we can conclude by a straightforward adaptation of the above proof.
\end{proof}

\begin{cor}\label{presam}
Under the same hypotheses, if moreover $H_2^{\RR}(\g)_0=\Kill^{\RR}(\mk{\g})_0=\{0\}$, then $\hat{G}\to G$ is an isomorphism.
\end{cor}
\begin{proof}
As observed in the proof of Corollary \ref{oabelsh}, since $H_2^{\RR}(\g)_0=0$, the kernel of $\hat{G}\to G$ is isomorphic to $W_2^{\Q,\RR}(\g)_0$. Since $\Kill^{\RR}(\mk{\g})_0=0$, Corollary \ref{noceb} implies that $W_2^{\Q,\RR}(\g)_0=0$.
\end{proof}

\begin{proof}[Proof of Theorem \ref{gn} (sketched)]
We only sketch the proof; the proof in \cite{A} (with more restricted hypotheses) is 8 pages long. Here the hypotheses are more general since Abels works with $p$-adic groups. The theorem is stated here assuming that $\RR$ is a $\Q$-algebra, but more generally, it works under the assumption that $\RR$ is a $\Z[1/s!]$-module, and in particular when it is a vector space over a field of characteristic $p>s$.

Fix a 2-tame $\g$-principal subset $\mathcal{P}$. Let $\mathcal{S}$ denote the set of all half-lines $\R_{>0}w$ ($w\neq 0$) in $\mathcal{W}$. Consider the lower central series $(G^i)$. Let $M^{i}_C$ be the image of $G^{i}\cap G_C$ in $\hat{G}$ and $M_C=M^1_C$. Let $A^i$ be the normal subgroup of $\hat{G}$ generated by all $M^{i}_C$, where $C$ ranges over $\mathcal{S}$.

Abels also introduces more complicated subgroups. Let $\mathcal{L}$ denote the set of lines of $\mathcal{W}$ (i.e.\ its projective space). Each line $L\in\mathcal{L}$ contains exactly two half lines: $$L=S_1\cup\{0\}\cup S_2.$$
Define the following subgroups of $\hat{G}$
$$M_{[L]}=M_{[L]}^1=\langle M_{S_1}\cup M_{S_2}\rangle$$
and, by induction
$$M_{[L]}^i=\left\langle M_{S_1}^i\cup M^i_{S_2}\cup\bigcup_{j+k=i}\left(\!\!\left(M_{[L]}^j,M_{[L]}^k\right)\!\!\right)\right\rangle,$$
where $(\!(\cdot,\cdot)\!)$ denote the subgroup generated by group commutators.

Abels proves the 
following lemma 
\cite[4.4.11]{A}: if 
$L\in\mathcal{L}$ 
and $C$ is a open cone in $\mathcal{W}$ such that $L+C\subset C$, then
\begin{equation}\left(\!\left(M_C^j,M_{[L]}^k\right)\!\right)\subset M_C^{j+k}.\label{4411}\end{equation}
The interest is that $M_{[L]}^k$ is a complicated object, while the right-hand term $M_C^{j+k}$ is a reasonable one.

Abels obtains \cite[Prop.\ 4.4.13]{A} the following result, which can appear as a group version of Lemma \ref{431ijk}(\ref{431ijkc}): for any $L_0\in\mathcal{L}$ and any $i$, the $i$th term of the lower central series $\hat{G}^i$ is generated, as a subgroup and modulo $A^i$, by the $M^i_{[L]}$ where $L$ ranges over $\mathcal{L}-\{L_0\}$, or, in symbols,
\begin{equation}\label{abgi}\hat{G}^i=\left\langle\bigcup_{L\neq L_0}M^i_{[L]}\right\rangle A^i.\end{equation}
It is important to mention here that all three items of Lemma \ref{431ijk} are needed in the proof of (\ref{abgi}) (encapsulated in the proof of \cite[Prop.\ 4.4.7]{A}). This is also the step where the Baker-Campbell-Hausdorff formula is used to convey properties from the Lie algebra $\g$ to $G$.

To conclude the proof, suppose that $G$ is $s$-nilpotent. We wish to prove that $(\!(\hat{G},\hat{G}^{s+1})\!)=\{1\}$. Since $\hat{G}$ is easily checked to be generated by $M_S$ for $S$ ranging over $\mathcal{S}$, it is enough to check, for each $S\in\mathcal{S}$
\begin{equation}\left(\!\!\left(M_S,\hat{G}^{s+1}\right)\!\!\right)=\{1\}.\label{comt}\end{equation}
Now since $G$ is $s$-nilpotent, $A^{s+1}=\{1\}$, so we can forget ``modulo $A^{s+1}$" in (\ref{abgi}) (with $i=s+1$), so (\ref{comt}) follows if we can prove 
$$(\!(M_S,M^{s+1}_{[L]})\!)=\{1\}$$
for all $L\in\mathcal{L}$ with maybe one exception $L_0$; namely, we choose $L_0$ to be the line generated by $S$. Then $S+L$ is an open cone, (\ref{4411}) applies and we have
\[(\!(M_S,M^{s+1}_{[L]})\!)\subset (\!(M_{S+L},M^{s+1}_{[L]})\!)\subset M_{S+L}^{s+2}=\{1\}.\qedhere\]
\end{proof}

\begin{rem}
As we observed, Theorem \ref{gn} works when $\RR$ is only assumed to be a $\Z[1/s!]$-module. It follows that Theorem \ref{oabels} works when $\RR$ is assumed to be a $\Z[1/(s+1)!]$-module, and in particular when it is a vector space over a field of characteristic $p>s+1$.
\end{rem}



\section{Presentations of standard solvable groups and their Dehn function}\label{s_main}
 
\subsection{Outline of the section} 
This section, which is the most important of the paper, contains the conclusions of the proofs of Theorems \ref{int_p} and \ref{imainss} and relies on essentially all the preceding sections. 
 
Let $G=U\rtimes A$ be a standard solvable group. 
In \S\ref{amam} we consider the multiamalgam $\hat{G}$ of (standard) tame subgroups (which are finitely many) along their intersections (regardless of any topology), as originally introduced by Abels. This group admits a canonical homomorphism onto $G$. Under the assumption that $G$ is ``2-tame" (which means that $G$ does not satisfy the SOL obstruction, which is one of the hypotheses of Theorem \ref{main}), we use the algebraic results of Section \ref{s:am} to check in \S\ref{su_2t} that the kernel of the canonical homomorphism $\hat{G}\to G$ is central in $G$. Moreover, assuming that $G$ does not satisfy the 2-homological obstruction, the work of \S\ref{amam} provides an explicit generating family of the central kernel, called welding relations. 

This provides us with a first ``algebraic" presentation of $G$, defined as follows. Let $U_1\rtimes A,\dots,U_\nu\rtimes A$ be the standard tame subgroups. We consider the presentation $\mathcal{P}$ whose set of generators is the disjoint union
\[A\sqcup\bigsqcup_{i=1}^\nu U_i,\]
and whose set of relators are the following
\begin{itemize}
\item all tame relations, namely relations of length 3 defining the group law in each of the $U_i\rtimes A$;
\item all amalgamation relations of length 2, namely if some $u\in U_i\cap U_j$, then in this disjoint union it corresponds to two generators $u_i\in U_i$ and $u_j\in U_j$, and the relator is then $u_i^{-1}u_j$;
\item all ``algebraic" welding relations, introduced in \S\ref{amam} (where ``algebraic" means that these are words with letters in the alphabet $\bigsqcup U_i$).
\end{itemize}
Next we obtain a compact presentation $\mathcal{P}_1$ of $G$, informally by replacing elements of the various $U_i$ with efficient representative words, and restricting to generators and relators in a large enough compact subset. It is given by a set of generators
\[S=T\sqcup\bigsqcup_{i=1}^\nu S_i,\]
where $T$ is in bijection with some compact generating subset of $A$ and $S_i$ is in bijection with some compact subset of $U_i$ so that $S_i\cup T$ generates $U_i\rtimes A$. From \S\ref{lengthe}, we have a way to associate to each $i$ and each $u\in U_i$, a word $\overline{u}$ in the generators $S_i\sqcup T$. Replacing each $u$ with $\overline{u}$ is a way to pass from any word in the previous generators to a word in these new generators, which we call its {\bf geometric realization}. The reader should keep in mind that even when we consider algebraic words of bounded length, their geometric realization can be unbounded since the length of $\overline{u}$ with respect to $S$ is unbounded when $u$ varies. The relators of $\mathcal{P}_1$ are:
\begin{itemize}
\item defining relators of each $U_i\rtimes A$ of bounded length;
\item amalgamation relations of length 2, of the form $s_i^{-1}s_j$ for $s_i\in S_i$ and $s_j\in S_j$ representing the same element in $U_i\cap U_j$;
\item geometric realizations of algebraic welding relators whose length is bounded by some suitable number.
\end{itemize}

To obtain an upper bound of the Dehn function of $\mathcal{P}_1$, we proceed in 3 steps
\begin{enumerate}
\item\label{step1} we first bound the area of a special class of relations, namely those geometric realizations of those algebraic relators; namely for the geometric realization of tame and amalgamation relations we obtain a quadratic upper bound (with respect to the word length with respect to $S$), and for the geometric realization of welding relations we obtain in \S\ref{awe} a cubic upper bound.

\item\label{step2} the general method developed in \S\ref{s_awbcl} allows to use these estimates so as to bound a more general class of relations, the so-called ``relations of bounded combinatorial length", that is, when $n_0$ is fixed, we consider all geometric realizations of algebraic relations of length $\le n_0$, and for such relations we show an asymptotic cubic upper bound (the constant depending on $n_0$).

\item We conclude thanks to Gromov's trick, which ensures that it is enough to consider such relations of bounded combinatorial length, using that the efficient words $\overline{u}$ indeed have a bounded combinatorial length. 
\end{enumerate}

The conclusions of these steps are written down in \S\ref{concl0} in the absence of welding relators, thus proving Theorem \ref{imainss}, and in \S\ref{concl} otherwise, ending the proof of Theorem \ref{int_p}.

Let us pinpoint that \S\ref{concl0} also provides information in the case where there are welding relators, in order to get quadratic estimates on the area of relations that follow only from the other types of relators. This is crucial in \S\ref{awe}, where the cubic estimate is obtained by using $\simeq n$ times a homotopy of area $\simeq n^2$.

\subsection{Abels' multiamalgam}\label{amam}

Let $G$ be any group and consider a family $(G_i)$ of subgroups. We define the multiamalgam (or colimit) to be an initial object in the category of groups $H$ endowed with homomorphisms $H\to G$ and $G_i\to H$, such that all composite homomorphisms $G_i\to H\to G$ are equal to the inclusion. Such an object is defined up to unique isomorphism commuting with all homomorphisms. It can be explicitly constructed as follows: consider the free product $H$ of all $G_i$, denote by $\kappa_i:G_i\to H$ the inclusion and mod out by the normal subgroup generated by the $\kappa_i(x)\kappa_j(x)^{-1}$ whenever $x\in G_i\cap G_j$. If the family $(G_i)$ is stable under finite intersections, it is enough to mod out by the elements of the form $\kappa_i(x)\kappa_j(x)^{-1}$ whenever $x\in G_i$ and $G_i\subset G_j$. Clearly, the multiamalgam does not change if we replace the family $(G_i)$ by a larger family $(G'_j)$ such that each $G'_j$ is contained in some $G_i$; in particular it is no restriction to assume that the family is closed under intersections. Also, the image of $\hat{G}\to G$ is obviously equal to the subgroup generated by the $G_i$.

\begin{rem}
Our (and, originally, Abels') motivation in introducing the multiamalgam is to obtain a presentation of $G$. Therefore, in this point of view, the ideal case is when $\hat{G}\to G$ is an isomorphism. In many interesting cases, which will be studied in the sequel, $\hat{G}\to G$ is a central extension.

On the other hand, if the $G_i$ have pairwise trivial intersection, the multiamalgam of the $G_i$ is merely the free product of the $G_i$, and this generally means that the kernel of $\hat{G}\to G$ is ``large".
\end{rem}

\begin{ex}
Consider a group presentation $G=\langle S\mid R\rangle$ in which every relator involves at most two generators. If we consider the family of subgroups generated by 2 elements of $S$, the homomorphism $\hat{G}\to G$ is an isomorphism. Instances of such presentations are presentations of free abelian groups, Coxeter presentations, Artin presentations. 
\end{ex}

\begin{lem}\label{presgh}
In general, let $G$ be any group, $(G_i)$ any family of subgroups. Suppose that each $G_i$ has a presentation $\langle S_i\mid R_i\rangle$ and that $S_i\cap S_j$ generates $G_i\cap G_j$ for all $i,j$. Then the multiamalgam of all $G_i$ has a presentation with generators $\bigsqcup S_i$, relators $\bigsqcup R_i$ and, for all $(i,j)$ the relators of size two identifying an element of $S_i$ and of $S_j$ whenever they are actually equal (if $(G_i)$ is closed under finite intersections, those $(i,j)$ such that $G_i\subset G_j$ are enough). 
\end{lem}
\begin{proof}
This is formal.  
\end{proof}

Following a fundamental idea of Abels, we introduce the following definition.

\begin{defn}Let $G$ be a locally compact group $G$ with a semidirect product decomposition $G=U\rtimes A$ (as in \S\ref{suts}). Consider the family $(G_i=U_i\rtimes A)_{1\le i\le\nu}$ of its tame subgroups. Let $\hat{U}$ be the multiamalgam of the $U_i$ along their intersections; it inherits a natural action of $A$, and we define $\hat{G}=\hat{U}\rtimes A$.
\end{defn}

\begin{rem}
By an easy verification, $\hat{G}$ is the multiamalgam of the $G_i$.
\end{rem}


\subsection{Algebraic and geometric presentations of the multiamalgam}\label{agp}

Let $G=U\rtimes A$ be a standard solvable group and $(G_i)_{1\le i\le\nu}$ the family of its standard tame subgroups, $G_i=U_i\rtimes A$.

Consider the free product $H=\Conv_{i} U_i$. There is a canonical surjective homomorphism $p:H\to\hat{U}$. Besides, there are canonical homomorphisms $j_i:U_i\to H$, so that $p\circ j_i$ is the inclusion $U_i\subset \hat{U}$.

Recall that $\mk{u}_i$ is defined as equal to $\mk{u}_{C(i)}$ for some subset $C(i)$ of the weight space, stable under addition and not containing 0.

\begin{lem}[Algebraic presentation of the multiamalgam]\label{algpma}
The multiamalgam $\hat{U}$ is, through the canonical map $p$, the quotient of $H$ by the normal subgroup generated by pairs $j_{i}(x)j_{i'}(x)^{-1}$ where $(i,i')$ ranges over pairs such that $C(i)\subset C(i')$ and $x$ ranges over $U_i$.
\end{lem}

Our goal is to translate this presentation into a compact presentation of $\hat{G}$. Let $\bar{S_i}$ be disjoint copies of the $S_i$. Define the disjoint union $\bar{S}=\bigsqcup_i\bar{S_i}\sqcup T$ and $F_{\bar{S}}$ the free group over $\bar{S}$; note that there is a canonical surjection $\bar{S}\to S=\bigcup_iS_i\sqcup T$. Denote by $\zeta_i$ the canonical bijection $S_i\to\bar{S_i}$. There is a canonical surjection $\zeta:\bar{S}\to S$ given as the identity on $T$ and as $\zeta_i$ on $S_i$.

There is a canonical surjective homomorphism $$\bar{\pi}:F_{\bar{S}}\to H\rtimes A,$$ and by composition, $p\circ\bar{\pi}$ is a canonical surjection $F_{\bar{S}}\to \hat{U}$. If $\iota_{i}$ is the inclusion of $S_{i}$ into $F_{\bar{S}}$, then $p\circ \bar{\pi}\circ\iota_{i}$ is the inclusion of $S_{i}$ into $\hat{U}$.

Let us introduce some important sets of elements in the kernel of the map $F_{\bar{S}}\to\hat{U}$. We fix a presentation of $A$ over $T$, including all commutation relators between generators. (See \S\ref{dede} for basic conventions about the meaning of {\it relations} and {\it relators}.)

\begin{itemize}
\item $\bar{R}_{\textnormal{tame}}=\bigcup_i\bar{R}_{\textnormal{tame},i}$, where $\bar{R}_{\textnormal{tame},i}$ consists of all elements in $\Ker(\pi)\cap F_{T\sqcup S_i}$. We call these {\bf tame relations}.
\item $\bar{R}_{\textnormal{tame}}^1=\bigcup_i\bar{R}^1_{\textnormal{tame},i}$, where $\bar{R}_{\textnormal{tame},i}^1$ consists of all elements in $\bar{R}_{\textnormal{tame},i}$ that are relators in the presentation of $G_i$ given in Corollary \ref{prestame}, as well as the relators in $\bar{R}_{\textnormal{tame},i}$ of length two; we call these {\bf tame relators}.
\item $R_{\textnormal{amalg}}$ consists of the
elements of the form $\sigma_i(w)^{-1}\sigma_{i'}(w)$, where $w$ ranges over $F_{S_i\sqcup T}$, and where $(i,i')$ ranges over pairs such that $C(i)\subset C(i')$, and $\sigma_{i}:F_{S_i\cup T}\to F_{\bar{S}}$ is the unique homomorphism mapping every $s\in S_i$ to $\zeta_i(s)$ and every $t\in T$ to itself. We call these {\bf amalgamation relations};

\item $R_{\textnormal{amalg}}^1$ consists of those amalgamation relations for which $w\in S_{i}$. These are words of length at most two. We call these {\bf amalgamation relators}.
\end{itemize}

Note that the quotient of $F_{\bar{S}}$ by $R_{\textnormal{amalg}}^1$ is just the free group $F_S$. We define $R_{\textnormal{tame}}$ and $R^1_{\textnormal{tame}}$ as the images of $\bar{R}_{\textnormal{tame}}$ and $\bar{R}^1_{\textnormal{tame}}$ in $F_{\bar{S}}$. 

\begin{prop}[Geometric presentation of the multiamalgam]\label{presmul}
The multiamalgam $\hat{U}$ is, through the canonical map $p\circ\bar{\pi}$, the quotient of $F_{\bar{S}}$ by the normal subgroup generated by $\bar{R}_{\textnormal{tame}}^1\cup R_{\textnormal{amalg}}^1$, and is also the quotient of $F_S$ by the normal subgroup generated by $R_{\textnormal{tame}}^1$.
\end{prop}
\begin{proof}
The quotient of $F_{\bar{S}}$ by the (normal subgroup generated by) tame relations is obviously isomorphic to $\left(\Conv_{i} U_i\right)\rtimes A$. Since by Corollary \ref{prestame}, for each $i$, $\bar{R}_{\textnormal{tame},i}$ is contained in the normal subgroup generated by $\bar{R}^1_{\textnormal{tame},i}$, it follows that $\left(\Conv_{i} U_i\right)\rtimes A$ is also the quotient of $F_{\bar{S}}$ by the tame relators.

Denote by $J_{i'}$ the inclusion of $U_{i'}$ into $\Conv_{i} U_i$. By Lemma \ref{presgh}, the quotient of $\Conv_{i} U_i$ by the relations 
$J_i(u)^{-1}J_{i'}(u)$ for $(i,i')$ ranging over pairs such that 
$C(i)\subset C(i')$, is the multiamalgam $\hat{U}$. It follows that the quotient of $\left(\Conv_{i} U_i\right)\rtimes A$ by the 
$R_{\textnormal{amalg}}$ is $\hat{U}$, and since obviously $R_{\textnormal{amalg}}$ is 
normally generated by $R^1_{\textnormal{amalg}}$, 
we deduce that the quotient of $\left(\Conv_{i} U_i\right)\rtimes A$ by the $R^1_{\textnormal{amalg}}$ is also $\hat{U}$.

Hence $\hat{U}$ is naturally the quotient of $F_{\bar{S}}$ by $\bar{R}_{\textnormal{tame}}^1\cup R_{\textnormal{amalg}}^1$. The second assertion follows.
\end{proof}

\begin{prop}[Quadratic filling of tame and amalgamation relations]\label{quadt}
For the presentation
$$\langle \bar{S}\mid \bar{R}^1_{\textnormal{tame}}\cup R_{\textnormal{amalg}}^1\rangle$$
of $\hat{G}$, the tame relations have an at most quadratic area with respect to their length, and the amalgamation relations $\sigma_i(w)^{-1}\sigma_{i'}(w)$ have an area bounded above by the length of $w$. In the presentation 
$$\langle S\mid R^1_{\textnormal{tame}}\rangle$$
of $\hat{G}$, the tame relations have an at most quadratic area with respect to their length.
\end{prop}
\begin{proof}
The tame relations have an at most quadratic area by Corollary \ref{prestame}.

Let us consider an amalgamation relation. It has the form $\sigma_i(w)^{-1}\sigma_{i'}(w)$ with $w$ of length $n$; then with cost $\le n$ we can replace all letters $\zeta_i(s)$ for $s\in S_i$ with $\zeta_{i'}(s)$; the resulting word is then equal to $\sigma_{i'}(w)^{-1}\sigma_{i'}(w)=1$.

The second assertion immediately follows.
\end{proof}

Proposition \ref{presmul} is a first step towards a compact presentation of $G$.


\subsection{Presentation of the group in the 2-tame case}\label{su_2t}

The following two theorems, which use the notion of 2-tameness introduced in Definition \ref{d_2ta}, are established in Section \ref{s:am} (Theorem \ref{oabels} and Corollary \ref{oabelsh}), relying on Section \ref{s_abels}. In the following statements, we identify $U_i$ with its image in $\hat{U}$. Also, recall that the decomposition $\K=\bigoplus_{j=1}^d\K_j$ induces a decomposition $\mk{u}=\bigoplus_j\mk{u}_{(j)}$, and $\mk{u}_i=\bigoplus_j\mk{u}_{i(j)}$.

\begin{thm}[Presentation of a 2-tame group, weak form]\label{presresum}
Assume that $G=U\rtimes A$ is a {\em 2-tame} standard solvable group. Then the homomorphism $\hat{G}\to G$ has a central kernel. If moreover the degree zero component of the second homology group $H_2(\mk{u})_0$ vanishes, the kernel of $\hat{G}\to G$ is generated by elements of the form 
$$\exp([\lambda x,y])\exp([x,\lambda y])^{-1},\qquad 1\le j\le d,\;\;\; x,y\in\bigcup_i\mk{u}_{i(j)},\;\;\lambda\in\K_j,\;\;. $$ 
\end{thm}

We pinpoint the striking fact that $H_2(\mk{u})_0=\{0\}$ does {\em not} imply that $\hat{G}\to G$ is an isomorphism. This was pointed out by Abels \cite[5.7.4]{A}, and relies on the fact that the space $H_2^\Q(\mk{u})_0$ of 2-homology of $\mk{u}$, where $\mk{u}$ is viewed as a (huge) Lie algebra over the rationals, can be larger than $H_2(\mk{u})_0$. On the other hand, the fact that $\hat{G}\to G$ has a central kernel is a new result, even in Abels' framework ($\mk{u}$ finite-dimensional nilpotent Lie algebra over $\Q_p$); Abels however proved that $\hat{U}$ is $(s+1)$-nilpotent if $U$ is $s$-nilpotent and this is a major step in the proof. 

Actually, in order to use the results of \S\ref{s_awbcl}, we need a (stronger) stable form of Theorem \ref{presresum}, namely holding over an arbitrary commutative $\K$-algebra.
We can view $U$ as the group of $\K$-points of $\mathbb{U}$, where $\mathbb{U}$ is an affine algebraic group over the ring $\K=\prod\K_j$. Thus, for every commutative $\K$-algebra $\A$, $\mathbb{U}(\A)$ is the group associated to the nilpotent Lie $\Q$-algebra $\mk{u}\otimes_\K \A$. Similarly, $\mathbb{U}_i$ is defined so that $\mathbb{U}_i(\A)\subset\mathbb{U}(\A)$ is the exponential of $\mk{u}_i\otimes_\K \A$, and we define $\widehat{\mathbb{U}}[\A]$ as the corresponding multiamalgam of the $\mathbb{U}_i(\A)$.

\begin{thm}[Presentation of a 2-tame group, strong (stable) form]\label{presresums}
Assume that $G=U\rtimes A$ is a 2-tame standard solvable group. Then for every $\K$-algebra $\A$, the homomorphism $\widehat{\mathbb{U}}[\A]\to \mathbb{U}(\A)$ has a central kernel, which remains central in $\widehat{\mathbb{U}}[\A]\rtimes A$. If moreover the zero degree component of the second homology group $H_2(\mk{u})_0$ vanishes, the kernel 
of $\widehat{\mathbb{U}}[\A]\to \mathbb{U}(\A)$ is generated by elements of the 
form 
$$\exp([\lambda x,y])\exp([x,\lambda y])^{-1},\qquad x,y\in\bigcup_{i}\mathbb{U}_{i(j)}(\A),\;\;\lambda\in \A_j,\;\;j=1\dots,d. $$ 
\end{thm}

Note that we could state the theorem without decomposing along the decomposition $\K=\bigoplus_j\K_j$, but we really need this statement when we estimate the area of welding relations in \S\ref{awe}.

Note that Theorem \ref{presresum} is equivalent to the case $\A=\K$ of Theorem \ref{presresums}, given the trivial observation that the kernels of $\hat{G}\to G$ and $\hat{U}\to U$ coincide.


Theorem \ref{presresum} involves the Lie algebra bracket; in order to translate it into a purely group-theoretic setting, we need to use Lazard's formulas from \S\ref{su_lf}. 

We can now restate the second statement of Theorem \ref{presresum} with no reference to the Lie algebra in the conclusion (except taking real powers):

\begin{thm}\label{presalg_ss0}
Let $G=U\rtimes A$ be a 2-tame standard solvable group such that $H_2(\mk{u})_0=\{0\}$. Choose $s$ so that $U$ is $s$-nilpotent. Then the (central) kernel of $\hat{G}\to G$ is generated by elements of the form 
$$B_{s+1}(x^\lambda,y)B_{s+1}(x,y^\lambda)^{-1},\quad x,y\in\bigcup_i U_{i(j)},\;\;\lambda\in\K_j,\;\;j=1\dots,d,$$
where $x^\lambda$ denotes $\exp(\lambda\log(x))$.
\end{thm}

Again we need a stable form of the latter result.

\begin{thm}\label{presalg_ss}
Let $G=U\rtimes A$ be a 2-tame standard solvable group such that $H_2(\mk{u})_0=\{0\}$. Choose $s$ so that $U$ is $s$-nilpotent. Then for every commutative $\K$-algebra $\A$, denoting $\A_j=\A\otimes_\K\K_j$ (so that $\A=\bigoplus_j\A_j$ canonically), the (central) kernel of $\hat{\mathbb{U}}[\A]\to \mathbb{U}(\A)$ is generated by elements of the form 
$$B_{s+1}(x^\lambda,y)B_{s+1}(x,y^\lambda)^{-1},\quad x,y\in\bigcup_i \mathbb{U}_{i(j)}(\A),\;\;\lambda\in\A_j,\;\;j=1\dots,d,$$
where $x^\lambda$ denotes $\exp(\lambda\log(x))$.
\end{thm}

We can view the elements $$B_{s+1}(x^\lambda,y)B_{s+1}(x,y^\lambda)^{-1}$$ as elements of the free product $H=\Conv_i U_i$; now $x,y$ range over the disjoint union $x,y\in\bigsqcup U_{i(j)}$. We call these {\bf welding relations} in the free product $H$. 

By substitution, Theorem \ref{presalg_ss0} gives rises to the set of relations in $F_S$
\begin{equation}\label{hit}R_{\textnormal{weld}}=\left\{B_{s+1}\left(\overline{x^\lambda},\overline{y}\right)B_{s+1}\left(\overline{x},\overline{y^\lambda}\right)^{-1}\;: x,y\in\bigsqcup_i U_{i(j)},\;\;\lambda\in\K_j,\;\;j=1\dots,d\right\}.\end{equation} in the free group $F_S$. We call these {\bf welding relations}. We define the set $R_\textnormal{weld}^1$ of {\bf welding relators} as those welding relations for which $\|x\|',\|y\|', |\lambda|\le 1$, where $\|\cdot\|'$ is the prescribed Lie algebra norm.

It follows from Corollary \ref{faclamg} that $G$ has a presentation with relators those of $\hat{G}$ (given in Lemma \ref{presgh}) along with welding relators. At this point, this already reproves Abels' result.

\begin{cor}[Compact presentation of $G$]\label{comp_pres} Let $G$ be a standard solvable group. Assume that $G$ is 2-tame and $H_2(\mk{u}_\textnormal{na})_0=\{0\}$. Then:
\begin{enumerate}[(a)]
\item\label{aabel} (Abels) $G$ is compactly presented;
\item\label{cpres} if $H_2(\mk{u})_0=\{0\}$, a compact presentation of $G$ is given by 
\begin{equation}\label{prescon}
\langle \bar{S}\mid \bar{R}_\textnormal{tame}^1\cup R_\textnormal{amalg}^1\cup R_\textnormal{weld}^1\rangle.
\end{equation}
\item\label{cpresk} if $H_2(\mk{u})_0=\{0\}$ and $\Kill(\mk{u})_0=\{0\}$, a compact presentation of $G$ is given by 
\[\langle \bar{S}\mid \bar{R}_\textnormal{tame}^1\cup R_\textnormal{amalg}^1\rangle.\]
\end{enumerate}
\end{cor}

\begin{proof}
First assume that $H_2(\mk{u})_0=\{0\}$. The above remarks show that if $\pi'$ is the natural projection $F_S\to\hat{G}$, then $\pi'(R_\textnormal{weld}^1)$ generates normally the kernel of $\hat{G}\to G$. Therefore, by Proposition \ref{presmul}, $G$ admits the compact presentation (\ref{prescon}). This proves (\ref{cpres}).

To obtain (\ref{aabel}), It remains to prove that $G$ is compactly presented only assuming $H_2(\mk{u}_\textnormal{na})_0=\{0\}$, but then $G/G^0$ is compactly presented by the previous case, and it follows that $G$ is compactly presented.

Finally (\ref{cpresk}) is by combining Proposition \ref{presmul}, which says that $\langle \bar{S}\mid R_\textnormal{tame}^1\cup R_\textnormal{amalg}^1\rangle$ is a presentation of $\hat{G}$ (for an arbitrary standard solvable group), and Corollary \ref{presam}, which says that the natural projection $\hat{U}\to U$ is an isomorphism (and hence $\hat{G}\to G$ as well).
\end{proof}

\begin{rem}Let us pinpoint that this does not coincide, at this point, with Abels' proof. Abels did not prove that the kernel of $\hat{G}\to G$ is generated by welding relators. Instead (assuming that $G$ is totally disconnected), he considered the multiamalgamated product $\dot{U}$ of all $U_i$ and a compact open subgroup $\Omega$ of $U$, and $\dot{G}=\dot{U}\rtimes A$, assuming that $\Omega$ is generated by the intersections $\Omega\cap U_i$, so that $\hat{U}\to\dot{U}$ and $\hat{G}\to\dot{G}$ are surjective.

Defining $S_i=U_i\cap\Omega$ and using $S=\bigcup S_i\cup T$ as a generating subset of $G$ (recall that $T$ is a compact generating subset of $A$), we see that $\Omega$ is {\em boundedly} generated by $\bigcup S_i$ (for instance, by the Baire category theorem). It follows that $\dot{G}$ is the quotient of $\hat{G}$ by a normal subgroup normally generated by generators of bounded length. Hence $\dot{G}$ is boundedly presented by $S$.

There is a natural projection $\dot{G}\to G$. Since welding relators are killed by the amalgamation with $\Omega$, we know that the natural projection $\dot{G}\to G$ is an isomorphism. Not having the presentation by welding relators, Abels used instead topological arguments \cite[5.4, 5.6.1]{A} to reach the conclusion that $\dot{G}\to G$ is indeed an isomorphism; thus $G$ is compactly presented.

This approach, with the use of a compact open subgroup is, however, ``unstable", in the sense that it does not yield a presentation of $\mathbb{U}(\A)\rtimes A$ when $\A$ is an arbitrary commutative $\K$-algebra, and the presentation with welding relators will be needed in the sequel in a crucial way when we obtain an upper bound on the Dehn function.
\end{rem}


\subsection{Quadratic estimates and concluding step for standard solvable groups with zero Killing module}\label{concl0}

Let us call the rank of $A$ the unique $d$ such that $A$ admits a copy of $\Z^d$ as a cocompact lattice.

\begin{thm}\label{mainss0}
Let $G=U\rtimes A$ be a standard solvable group. If $G$ is 2-tame and $H_2(\mk{u})_0$ and $\Kill(\mk{u})_0$ both vanish, then the Dehn function of $G$ is quadratic (or linear in case $A$ has rank $\le 1$).
\end{thm}

If $A$ has rank 0 then $G$ is compact; if $A$ has rank 1, then 2-tame implies tame, and in that case, the Dehn function is linear (see Remark \ref{D1}). Otherwise, if $A$ has rank at least 2, the Dehn function is at least quadratic. So the issue is to obtain a quadratic upper bound.

Fix $c$ and a $c$-tuple $(\wp_1,\dots,\wp_c)$ of elements in $\{1,\dots,\nu\}$. Recall that $\mathcal{U}_\wp(r)\subset F_{\bar{S}}$ denotes the set of null-homotopic words of the form $w_1\dots w_c$, with $w_\ell\in S_{\wp_\ell}^{(r)}$. Theorem \ref{grossg} reduces the proof of Theorem \ref{mainss0} to proving that, for each given $\wp$, words in $\mathcal{U}_\wp(r)$ have an at most quadratic area with respect to $r$. We now proceed to prove this. We also prove a statement holding for more specific relations but without the vanishing assumptions, which will be used in the sequel.

\begin{thm}\label{qubo}
Let $G=U\rtimes A$ be a 2-tame standard solvable group. Let $U_1,\dots,U_\nu$ be its standard tame subgroups.
Then
\begin{enumerate}
\item\label{qcb1} Assume that $H_2(\mk{u})_0$ and $\Kill(\mk{u})_0$ both vanish. Fix an integer $c$ and any $c$-tuple $(\wp_1,\dots,\wp_c)$ of elements in $\{1,\dots,\nu\}$. Then for $r\ge 0$, the area of words in $\mathcal{U}_\wp(r)\subset F_{\bar{S}}$ is quadratically bounded in terms of $r$.
\item\label{qnil2} Let $s$ be such that $U$ is $s$-nilpotent (that is, $U^{(s+1)}=\{1\})$.
Then for every group word $w(x_1,\dots,x_c)$ that belongs to the $(s+2)$-th term $F_c^{(s+2)}$ of the lower central series of the free group $F_c$, and for any $u_1,\dots,u_c$ in $\bigsqcup_i U_{i}$, the relation $w(\overline{u_1},\dots,\overline{u_c})$ in $\hat{G}$ has an at most quadratic area with respect to its total length (the constant not depending on $c$).
\end{enumerate}
\end{thm}

\begin{proof}
Theorem \ref{mainss0} makes use of the results of \S\ref{s_awbcl}. With the notation of \S\ref{s_awbcl}, for any $\K$-algebra $\A$ we have $\mathbb{H}[\A]=\Conv_{i=1}^\nu\mathbb{U}_i(\A)$.

Let us encode the amalgamation relations.
For any $1\le i\neq j\le\nu$, write $\mathbb{U}_{ij}(\A)=\mathbb{U}_{i}(\A)\cap \mathbb{U}_{j}(\A)$, and let $s_{ij}$ be the inclusion of $\mathbb{U}_{ij}(\A)$ into $\mathbb{U}_{i}(\A)$. 
Define a closed subscheme $\mathbb{R}_{ij}$ of $\prod_{i=1}^\nu\mathbb{U}_{i}$ by defining
\[\mathbb{R}_{ij}(\A)=\{(u_1,\dots,u_\nu):\;u_k=1\;\forall k\notin\{i,j\},\; u_i,u_j\in\mathbb{U}_{ij}(\A),\text{ and }u_iu_j=1\}.\]
Define $\mathbb{R}[\A]=\bigcup_{i<j}\mathbb{R}_{ij}(\A)$.
Thus, taking the quotient of $\mathbb{H}[\A]$ by $\pi_\A(\mathbb{R}[\A])$ means amalgamating the $\mathbb{U}_i$ along their intersections; we thus denote this quotient as $\widehat{\mathbb{U}}[\A]$ (in \S\ref{s_awbcl} it was denoted as $\mathbb{Q}[\A]$). Define $R=\bar{R}^1_{\textnormal{tame}}\cup R_{\textnormal{amalg}}^1$.

Using this, we now prove (\ref{qcb1}) and (\ref{qnil2}) separately.

We begin with (\ref{qcb1}). Since $\mk{u}$ is 2-tame and $H_2(\mk{u})_0$ and $\Kill(\mk{u})_0$ vanish, for every commutative $\K$-algebra $\A$, we have $H_2^{\A}(\mk{u}\otimes_\K \A)_0=H_2^\K(\mk{u})_0\otimes_\K \A=\{0\}$, and similarly $\Kill^{\A}(\mk{u}\otimes_\K \A)_0=\{0\}$. Thus by Corollary \ref{presam}, $\hat{\mathbb{U}}[\A]\to \mathbb{U}(\A)$ is an isomorphism for every $\A$. It follows that $\mathbb{L}_\wp$ (defined in \S\ref{def_lp}) is a closed subscheme of $\mathbb{U}_\wp$, so we can apply Theorem \ref{mainsim3} with $\mathbb{M}=\mathbb{L}_\wp$. It gives an upper bound of the areas of elements in $\mathbb{M}(\K)$ in terms of upper bounds on the Dehn functions of the $U_i$ and of the areas of the relations of the form $\overline{u_1}\dots\overline{u_\nu}$ with $(u_1,\dots,u_\nu)\in\mathbb{R}[\K]$. In this case, the $U_i$ have at most quadratic Dehn function by Corollary \ref{tameq}, and the amalgamation relations have an at most linear area by Proposition \ref{quadt}. Thus by Theorem \ref{mainsim3}, elements of $\mathbb{M}(\K)$ have an at most quadratic area with respect to their length and $R$.

To prove of (\ref{qnil2}), we argue as follows.
By Theorem \ref{gn}, $\widehat{\mathbb{U}}[\A]$ is $(s+1)$-nilpotent for every $\A$; thus, defining $\mathbb{M}'=\mathbb{U}_\wp$, we have $\mathbb{M}'(\A)\subset\mathbb{L}_\wp^w(\A)$ (with the notation of Theorem \ref{mainsim4}.

By Proposition \ref{quadt}, the tame and amalgamation relations have an at most quadratic area with respect to the presentation $\langle \bar{S}\mid \bar{R}^1_{\textnormal{tame}}\cup R_{\textnormal{amalg}}^1\rangle$. We are thus in position to apply Theorem \ref{mainsim4}, which yields the desired result.
\end{proof}


\subsection{Area of welding relations}\label{awe}

We now bound the area of welding relations, making use of the quadratic bound on the area of general nilpotency relations from \S\ref{concl0}.

\begin{thm}\label{weldcub}
If the standard solvable group $G=U\rtimes A$ is 2-tame, then welding relations in $G$ have an at most cubic area.

More precisely, there exists a constant $K$ such that for all $j$, all $x,y\in\bigcup_i U_{i(j)}$, and all $\lambda\in\K_j$, the welding relation (\ref{hit}) has area at most $Kn^3$, with $n=\log(1+\|x\|+\|y\|+|\lambda|)$.
\end{thm}

Let us assume that the subgroup $U$, in Theorem \ref{weldcub}, is $s$-nilpotent.
In all this subsection, we write $A=A_{s+1}$, $B=B_{s+1}$, $q=q_{s+1}$.
Theorem \ref{qubo}(\ref{qnil2}) applies to the group words corresponding to equalities of Proposition \ref{lazeq}, providing

\begin{thm}\label{adqu}
If $G$ is 2-tame and $U$ is $s$-nilpotent, for all $x\in U_{C_1}$, $y\in U_{C_2}$ and $z\in U_{C_3}$, the relations
$$A(\overline{x},\overline{y})=A(\overline{y},\overline{x});\; B(\overline{x},\overline{y})=B(\overline{y},\overline{x})^{-1}\;$$
$$B(A(\overline{x},\overline{y}),\overline{z})=A(B(\overline{x},\overline{z}),B(\overline{y},\overline{z}));$$
$$A(A(\overline{x},\overline{y}),\overline{z}^q)=A(\overline{x}^q,A(\overline{y},\overline{z}))$$
$$B(\overline{x}^k,\overline{y})=B(\overline{x},\overline{y}^k)=B(\overline{x},\overline{y})^k$$
have an at most quadratic area in $\hat{G}$ for the presentation $\langle\bar{S}\mid\bar{R}_\textnormal{tame}^1\cup R_\textnormal{amalg}^1\rangle$. \qed
\end{thm}
(By ``the relation $w=w'$ has an at most quadratic area, we mean the relation $w^{-1}w'$ has an at most quadratic area.) Note that in the last case, the quadratic upper bound is of the form $c_kn^2$, where $c_k$ may depend on $k$.

Although we only use it as a lemma in the course of proving Theorem \ref{weldcub}, we state Theorem \ref{adqu} as a theorem, because it provides a geometric information which is not encoded in the Dehn function, namely a quadratic area for certain types of loops.

We now proceed to prove Theorem \ref{weldcub}. Using (with quadratic cost) the ``bilinearity" of $B$ (or restricting scalars from the beginning), we can suppose that $\K_j$ is equal to $\R$ or $\Q_p$ (although the forthcoming argument can be adapted with unessential modifications to their finite extensions).
If $\K_j=\Q_p$, define $\pi_j=p$. If $\K_j=\R$, define $\pi_j=2$. 
We consider the following finite subset of $\Q$:
$$\Lambda_j^1(n)=\left\{\lambda=\sum_{i=0}^{n-1}\eps_i\pi_j^i:\; \eps_i\in\{-\pi_j+1,\dots,\pi_j-1\}\right\}\subset\Q\subset\K_j.$$

\begin{lem}\label{l1j}For every $\kappa>0$,
there exists $K>0$ such that for all $n\in\N$, all $j$, all $x,y\in\bigcup U_{i(j)}$ with $\log(1+\|x\|+\|y\|)\le n$ and all $\lambda\in\Lambda^j_1(\kappa n)$, the welding relation
$$
B\left(\overline{x^\lambda},\overline{y}\right)B\left(\overline{x},\overline{y^\lambda}\right)^{-1}$$ has area $\le Kn^3$.
\end{lem}

\begin{proof}
We can work for a given $j$, so we write $\pi=\pi_j$. Write
$$\lambda=\sum_{i=0}^{\kappa n-1}\eps_i\pi^i\quad (\eps_i\in\{-\pi+1,\dots,\pi-1\}).$$
If we set $$\lambda_i=\sum_{k=i}^{\kappa n-1}\eps_k\pi^{k-i},$$ we have $\lambda_0=\lambda$, $\lambda_{\kappa n}=0$, and, for all $i$
$$\lambda_i=\pi\lambda_{i+1}+\eps_i$$
Set $z=q^{-1}y$ and $\sigma_i=\sum_{j=1}^i\eps_i\pi^i$ (so $\sigma_{-1}=0$); consider the word

$$\Phi_i=A\left(B\big(\overline{\lambda_ix},\overline{\pi^iz}\big),B\left(\overline{x},\overline{\sigma_{i-1}z}\right)\right)$$

Here, for readability, we write $\overline{\lambda x}$ (etc.) instead of $\overline{x^\lambda}$. This is natural since $x$ can be identified to its Lie algebra logarithm.
For $i=0$, $\sigma_{i-1}=0$ so, since $\overline{1}=1$ and, formally, $B(x,1)=1$ and $A(x,1)=x^q$ (see Remark \ref{formalab}), we have  

 $$\Phi_0=B\left(\overline{\lambda x},\overline{z}\right)^q.$$

For $i=\kappa n$, $\lambda_n=0$ and $\sigma_{i-1}=\lambda$ so this is (using that formally $B(1,y)=1$ and $A(1,y)=y^q$) $$\Phi_{\kappa n}=B\left(\overline{x},\overline{\lambda z}\right)^q.$$

Let us show that we can pass from $\Phi_i$ to $\Phi_{i+1}$ with quadratic cost.
In the following computation, each $\rs$ means one operation with quadratic cost, i.e., with cost $\le K_0n^2$ for some constant $K_0$ only depending on the group presentation. The tag on the right explains why this quadratic operation is valid, namely:
\begin{itemize}
\item (1) means both the homotopy between loops lies in one tame subgroup.
\item (2) means the operation follows from Theorem \ref{adqu}; to be specific:
 \begin{itemize}
 \item $(2)_{\mbox{distr left}}$ for $$B(A(x,y),z)=A(B(x,z),B(y,z))$$ and similarly
 $(2)_{\mbox{distr right}}$ on the right
 \item $(2)_Q$ for an equality of the type $$B(x,y^k)=B(x^k,y)=B(x,y)^k,$$ with $k$ an integer satisfying $|k|\le\max(q,\pi)$.
 \item $(2)_{\mbox{assoc}}$ for the identity $A(A(x,y),z^q)=A(x^q,A(y,z))$.
 \end{itemize}
\item Or a tag referring to a previous computation, written in brackets.
\end{itemize}

\begin{align}\label{subs1}
     \overline{\lambda_ix}  =\quad & \overline{q(\pi\lambda_{i+1}q^{-1}x+\eps_iq^{-1}x)} & \\
\notag \rs\quad & \left(\overline{\pi\lambda_{i+1}q^{-1}x+\eps_iq^{-1}x}\right)^q & (1)\\
\notag \rs\quad & \left(\overline{\pi\lambda_{i+1}q^{-1}x}+\overline{\eps_iq^{-1}x}\right)^q & (1)\\
\notag \rs\quad & A\left( \overline{\pi\lambda_{i+1}q^{-1}x},\;\overline{\eps_iq^{-1}x}\right) & (1)
  \end{align}

So by substitution we obtain

\begin{align}\label{subs2}
      B\left(\overline{\lambda_ix},\overline{\pi^iz}\right)  \rs & B\left(A\left( \overline{\pi\lambda_{i+1}q^{-1}x},\;\overline{\eps_iq^{-1}x}\right),\overline{\pi^iz}\right) & [\ref{subs1}]\\
\notag \rs & A\left(\;
  B\left(\overline{\pi\lambda_{i+1}q^{-1}x},\overline{\pi^iz}\right)
 \;\;,B\left(\overline{\eps_iq^{-1}x},\overline{\pi^iz}\right)   
 \;   \right) & (2)_{\mbox{distr left}}\\
\notag \rs & A\left(\;
  B\left(\overline{\pi\lambda_{i+1}q^{-1}x},\overline{\pi^iz}\right)
 \;\;,B\left(\overline{q^{-1}x},\overline{\eps_i\pi^iz}\right)   
 \;   \right) & (2)_Q
  \end{align}
  
Independently we have

\begin{align}\label{subs3}
B\left(\overline{x},\overline{\sigma_{i-1}z}\right) \rs  & B\left(\overline{q^{-1}x}^q,\overline{\sigma_{i-1}z}\right) & (1) \\
\notag\rs & B\left(\overline{q^{-1}x},\overline{\sigma_{i-1}z}\right)^q & (2)_Q
\end{align}

Again by substitution, this yields

\begin{align}\label{subs4}
\Phi_i  =   & A\left(B\left(\overline{\lambda_ix},\overline{\pi^iz}\right),B\left(\overline{x},\overline{\sigma_{i-1}z}\right)\right)
  & \\
\notag \rs &
 A\left(
 A\left(\;
   B\left(\overline{\pi\lambda_{i+1}q^{-1}x},\overline{\pi^iz}\right)
 \;\;,
   B\left(\overline{q^{-1}x},\overline{\eps_i\pi^iz}\right)   
  \;\right)
 ,
 B\left(\overline{q^{-1}x},\overline{\sigma_{i-1}z}\right)^q
 \right)
& [\ref{subs2},\ref{subs3}]\\
\notag\rs & 
 A\left(
 B\left(\overline{\pi\lambda_{i+1}q^{-1}x},\overline{\pi^iz}\right)^q
 ,
 A\left(
 B\left(\overline{q^{-1}x},\overline{\eps_i\pi^iz}\right)   
 ,
 B\left(\overline{q^{-1}x},\overline{\sigma_{i-1}z}\right)
 \right)
 \right)
& (2)_{\mbox{assoc}} \\
\label{lasttag}\rs & 
A\left(
   B\left(\overline{\pi\lambda_{i+1}q^{-1}x},\overline{\pi^iz}\right)^q
 ,
   B\left(
      \overline{q^{-1}x}
      ,
      A\left(  
           \overline{\eps_i\pi^iz}   
           ,
           \overline{\sigma_{i-1}z}
      \right)
   \right)
\right)
& (2)_{\mbox{distr right}}
  \end{align}

By similar arguments
$$ B\left(\overline{\pi\lambda_{i+1}q^{-1}x},\overline{\pi^iz}\right)^q
\stackrel{(2)_Q}\rs  B\left(\overline{\lambda_{i+1}x},\overline{\pi^{i+1}z}\right),$$
and 

\begin{align*}
B\left(
      \overline{q^{-1}x}
      ,
      A\left(  
           \overline{\eps_i\pi^iz}   
           ,
           \overline{\sigma_{i-1}z}
      \right)
   \right)
 \rs &
B\left(
      \overline{q^{-1}x}
      ,        
           \overline{q(\eps_i\pi^iz +\sigma_{i-1}z)} 
   \right) & (1) \\
 \rs &
B\left(
      \overline{x}
      ,        
           \overline{\eps_i\pi^iz +\sigma_{i-1}z} 
   \right) & (2)_Q \\
= & B\left(
      \overline{x}
      ,        
           \overline{\sigma_{i}z} 
   \right) & \end{align*}

so substituting from (\ref{lasttag}) we get
$$\Phi_i\rs
A\left(
 B\left(\overline{\lambda_{i+1}x},\overline{\pi^{i+1}z}\right)
 ,
 B\left(
      \overline{x}
      ,        
      \overline{\sigma_{i}z}
  \right)
\right) \quad =\Phi_{i+1}
$$
in quadratic cost, say $\le K_1n^2$ (noting that each $\Phi_i$ has length $\le K_2n$ for some fixed constant $K_2$). Note that the constant $K_1$ only depends on the group presentation, because the above estimates use a quadratic filling only finitely many times, each among finitely many types (note that we used $(2)_Q$ only for $k$ in a bounded interval, only depending on $\K$ and $s$). 

It follows that we can pass from $\Phi_0$ to $\Phi_{\kappa n}$ with cost $\le K_1\kappa n^3$. On the other hand, by substitution of type $(2)_Q$, we can pass with quadratic cost from $B(\overline{\lambda x},\overline{y})$ to $B(\overline{\lambda x},\overline{q^{-1} y})^q=\Phi_0$ and from $\Phi_{\kappa n}=B(\overline{x},\overline{\lambda q^{-1} y})^q$ to $B(\overline{x},\overline{\lambda y})$. So the proof of the lemma is complete.
\end{proof}

\begin{proof}[Conclusion of the proof of Theorem \ref{weldcub}]
Let now $\Lambda_j^2(n)$ be the set of quotients $\lambda'/\lambda''$ with $\lambda',\lambda''\in\Lambda_j^1(n)$, $\lambda''\neq 0$. Lemma \ref{l1j} immediately extends to the case when $\lambda\in\Lambda_j^2$.
It follows from the definition that $\Lambda_j^2(\kappa n)$ contains all elements of the form $\lambda=\sum_{i=-\kappa n}^{\kappa n}\eps_i\pi_j^i$, with $\eps_i\in\{-\pi_j,\dots,\pi_j\}$. 

If $\K_j=\R$, $\pi_j^{\kappa n}\Lambda_j^2(\kappa n)$ contains all integers between $-\pi_j^{-2\kappa n}$ and $\pi_j^{2\kappa n}$. Thus $\Lambda_j^2(\kappa n)$ contains a set which is $|\pi^{-\kappa n}|$-dense in the ball of radius $|\pi^{\kappa n}|$.
If $\K_j=\Q_p$, $\pi^{\kappa n}\Lambda_j^2(\kappa n)$ contains a $|\pi^{2\kappa n}|$-dense subset of $\Z_p$. Thus $\Lambda_j^2(\kappa n)$ contains a $|\pi^{\kappa n}|$-dense subset of the ball of radius $|\pi^{-\kappa n}|$.
In both cases, defining $\varrho_j=\max(|\pi_j|,|\pi_j|^{-1})$, we obtain that $\Lambda_j^2(\kappa n)$ contains a $\varrho_j^{-\kappa n}$-dense subset of the ball of radius $\varrho_j^{\kappa n}$ in $\K_j$.

We now fix $j$ and write $\varrho=\varrho_j$, $\pi=\pi_j$. We pick $\kappa=2/\log(\varrho)$, so that $\varrho^{\kappa n}=e^{2n}$.
We assume that $n\ge\log(1/|q|)$, where $|q|$ is the norm of $q$ in $\K_j$ (if $\K_j=\R$ this is an empty condition). It follows that $\varrho^{-\kappa n}|q|^{-1}e^{n}\le 1$.

Now fix $\lambda\in\K_j$ with $|\lambda|\le e^n$. We need to prove that we can pass from $B(\overline{\lambda x},\overline{y})$ to $B(\overline{x},\overline{\lambda y})$ with cubic cost; clearly it is enough to pass from $B(\overline{\lambda x},\overline{y})^q$ to $B(\overline{x},\overline{\lambda y})^q$ with cubic cost.

Since $|\lambda|\le e^n\le \varrho^{\kappa n}$, we can write $\lambda=\mu+q\eps$ with $\mu\in\Lambda^2(\kappa n)$ and $|q\eps|\le\varrho^{-\kappa n}$. So $|\eps|\le \varrho^{-\kappa n}|q|^{-1}\le e^{-n}$.

Also, assume that $n\ge\log(\varrho)$. So we can find an integer $k$ with $n/\log(\varrho)\le k\le 2n/\log(\varrho)$. Thus, if we define $\eta=\pi^{\pm k}$, with the choice of sign so that $|\eta|>1$; we have $e^n\le |\eta|\le e^{2n}\le e^n|\eps|^{-1}$.

Using a computation as in (\ref{subs1}), we obtain, with quadratic cost

\begin{align*}
B(\overline{\lambda x},\overline{y})= & B\left(\overline{(\mu+q\eps) x},\overline{y}\right) & \\
\rs  & B\left(A(\overline{\mu q^{-1}x},\overline{\eps x}),\overline{y}\right) & [\ref{subs1}]\\
\rs & A\left(B\left(\overline{\mu q^{-1}x},\overline{y}\right),B(\overline{\eps x},\overline{y})\right) & (2)_{\mbox{distr left}}
\end{align*}
and similarly, with quadratic cost.
$$B\left(\overline{x},\overline{\lambda y}\right)\rs A\left(B(\overline{x},\overline{\mu q^{-1}y}),B(\overline{x},\overline{\eps y})\right)$$

By the previous case, with cubic cost we have
$$B\left(\overline{\mu q^{-1}x},\overline{y}\right)\rs B\left(\overline{x},\overline{\mu q^{-1}y}\right)$$

So it remains to check that with cubic cost we have 
\begin{equation}\label{we97}B(\overline{\eps x},\overline{y})\rs B(\overline{x},\overline{\eps  y}).\end{equation}

If $\eta$ is the element introduced above, observe that $\eta\in\Lambda^2_j(\kappa n)$ and $e^n\le |\eta|\le e^n|\eps|^{-1}$. We have, with cubic cost

\begin{equation}\label{we98}B(\overline{\eps x},\overline{y})\rs B(\overline{\eta \eps x},\overline{\eta^{-1} y}); \quad B(\overline{\eta^{-1} x},\overline{\eta\eps y})\rs B(\overline{x},\overline{\eps y}).\end{equation}

Since $\max(|\eta\eps|,|\eta|^{-1})\le e^{-n}$, it follows that all four elements 
$\eta\eps x$, $\eta^{-1} y$, $\eta^{-1}x$, $\eps y$ have norm at most one, and it follows that we can perform
\begin{equation}\label{rseta}B\left(\overline{\eta \eps x},\overline{\eta^{-1} y}\right)\rs B\left(\overline{\eta^{-1} x},\overline{\eta\eps y}\right)\end{equation}
by application of a single welding relator. So (\ref{we97}) follows from (\ref{we98}) and (\ref{rseta}).
\end{proof}


\subsection{Concluding step for standard solvable groups}\label{concl}

\begin{thm}\label{mainss}
Let $G$ be a standard solvable group. If $G$ is 2-tame and $H_2(\mk{u})_0=0$ then its Dehn function is at most cubic. 
\end{thm}

Exactly as in the proof of Theorem \ref{mainss0}, Theorem \ref{grossg}, along with the cubic bound on the area of welding relators provided by Theorem \ref{weldcub}, allows to reduce the proof of Theorem \ref{mainss} to the following:

\begin{thm}
Let $G=U\rtimes A$ be a 2-tame standard solvable group with $H_2(\mk{u})_0=0$; let $U_1,\dots,U_\nu$ be its standard tame subgroups; fix a positive function $n\mapsto g(n)$ such that $n\mapsto g(n)/n^2$ is eventually non-decreasing. Assume that the welding relations in $G$ of length $n$ have an area $\preceq g(n)$.

Fix $c$ and a $c$-tuple $(\wp_1,\dots,\wp_c)$ of elements in $\{1,\dots,\nu\}$. Then words in $\mathcal{U}_\wp(n)$ have an area, for the presentation $\langle \bar{S}\mid \bar{R}_\textnormal{tame}^1\cup R_\textnormal{amalg}^1\cup R_\textnormal{weld}^1\rangle$, asymptotically bounded above by $g(n)$.
\end{thm}
(We state the theorem with an arbitrary function $g(n)$, because we do not know if the cubic bound is optimal; this might depend on $G$ even assuming $\Kill(\mk{u})_0\neq 0$.)

\begin{proof}
The proof follows the same lines as Theorem \ref{qubo}(\ref{qcb1}); let us highlight the differences. We need to include welding relations in the definition of $\mathbb{R}$: namely, defining $\mathbb{R}_{ij}$ 
as in the proof of Theorem \ref{qubo}, we will define $\mathbb{R}[\A]=\mathbb{R}_{\textnormal{weld}}[\A]\cup\mathbb{R}_{ij}(\A)$, where $\mathbb{R}_{\textnormal{weld}}$ is defined as follows.

Let $s$ be such that $U$ is $s$-nilpotent. Define $w\in F_4$ as equal to $w(x,x',y,y')=B_{s+1}(x,y)B_{s+1}(x',y')^{-1}$, where $B_{s+1}$ is a word as given by Theorem \ref{t_laz}.

We wish to define (recall that $\K=\bigoplus_{j=1}^d\K_j$)
\[\mathbb{R}_{\textnormal{weld}}[\A]=\bigcup_{1\le i,i'\le\nu,1\le j\le d}\mathbb{W}_{ii'(j)}(\A),\]
and we have to define $\mathbb{W}_{ii'(j)}\subset\prod_{\ell=1}^\nu\mathbb{U}_i$. 

Define, for $\mathbb{V},\mathbb{V}'$ unipotent $\K$-groups, the following subscheme $\mathbb{M}'=\mathbb{M}'_{\mathbb{V},\mathbb{V}'}$ of $\mathbb{V}\times\mathbb{V}\times\mathbb{V}'\times\mathbb{V}'$ defined by 
\[\mathbb{M}'(\A)=\{(v_1,v_2,v'_1,v'_2)\in (\mathbb{V}\times\mathbb{V}\times\mathbb{V}'\times\mathbb{V}')(\A):\exists\lambda\in\A: v_1=v_2^\lambda,v'_2={v'_1}^\lambda\}. \]
Here $v^\lambda$ means $\exp(\lambda\log(v))$. That it is a subscheme can be seen on the Lie algebra, where it corresponds to the 4-tuples $(v_1,v_2,v'_1,v'_2)$ such that there exists $\lambda$ such that $v_1=\lambda v_2$ and $v'_2=\lambda v'_1$. Then writing in coordinates, this is equivalent to the condition that 
\begin{itemize}
\item $v_{2i}v_{1i'}-v_{1i}v_{2i'}=v'_{2i}v'_{1i'}-v'_{1i}v'_{2i'}=0$  for all $i,i'$,  
\item $v_{2i}v'_{2i}=v'_{1i}v_{1i}$ for all $i$;
\end{itemize}
hence it is indeed a subscheme. 
Using the notation of \S\ref{bubullet}, we have
\[(w\bullet\mathbb{M}'_{\mathbb{V},\mathbb{V}'})(\A)=\{B_{s+1}(x^\lambda,y)B_{s+1}(x,y^\lambda)^{-1}\mid x\in\mathbb{V}(\A),y\in\mathbb{V}'(\A),\lambda\in\K\}.\]

Hence if we define $\mathbb{W}_{ii'(j)}=w\bullet\mathbb{M}'_{\mathbb{U}_{i(j)},\mathbb{U}_{i'(j)}}$, we have .
\[\mathbb{W}_{ii'(j)}(\A)=\{B_{s+1}(x^\lambda,y)B_{s+1}(x,y^\lambda)^{-1}\mid x\in\mathbb{U}_{i(j)}(\A),y\in\mathbb{U}_{i'(j)}(\A),\lambda\in\K_j\}.\]

Now the definition of $\mathbb{R}$ is complete, and by Theorem \ref{presalg_ss}, the quotient of $\mathbb{H}[\A]$ by the normal subgroup generated by $\mathbb{R}[\A]$ is canonically $\mathbb{U}(\A)$, for every commutative $\K$-algebra $\A$. Hence we can, as in the proof of Theorem \ref{qubo}(\ref{qcb1}), define $\mathbb{M}$ as equal to $\mathbb{L}_\wp$.

Fix $c$ and a $c$-tuple $\wp$ of elements in $\{1,\dots,\nu\}$. By Theorem \ref{presalg_ss}, denoting $R=\bar{R}_\textnormal{tame}^1\cup R_\textnormal{amalg}^1\cup R_\textnormal{weld}^1$, we have $R\subset\mathbb{R}[\K]$; by Theorem \ref{comp_pres}(\ref{cpres}), the kernel of the quotient map from $F_{\bar{S}}$ to $G$ is generated as a normal subgroup by $R$. By 
assumption (and Proposition \ref{quadt}), the words defined by the relations in $\mathbb{R}[\K]$ have area $\preceq g(r)$ with respect to their length $r$. Hence by Theorem \ref{mainsim3}, the area of words $\overline{x_1}\dots\overline{x_c}$ with $x_i\in\mathbb{U}_{\wp_i}$ is $\preceq g(r)$, where $r=\max_i|x_i|$.

Since this holds for every given $\wp$, by Theorem \ref{grossg} (which encapsulates Gromov's trick), we deduce that the Dehn function of $G$ is $\preccurlyeq g(n)$.
\end{proof}

\subsection{Dehn function of generalized standard solvable groups}\label{gessg}

We define a generalized standard solvable group as a locally compact group of the form $U\rtimes N$, where the definition is exactly as for standard solvable groups (Definition \ref{d_ssg}), except that $N$ is supposed to be nilpotent instead of abelian. 

Recall from \S\ref{s_gtame} that such a group is {\em generalized tame} if some element $c$ of $N$ acts on $U$ as a compaction. Clearly split triangulable Lie groups, i.e.\ of the form $G=G^\infty\rtimes N$, are special cases of generalized standard solvable groups.

\begin{thm}\label{generalized}
Let $G$ be a generalized standard solvable group not satisfying any of the (SOL or 2-homological) obstructions. Suppose that the Dehn function of $N$ is bounded above by some function $f$ such that $r\mapsto f(r)/r^\alpha$ is non-decreasing for some $\alpha>1$.

Then the Dehn function of $G$ satisfies $\delta_G(n)\preccurlyeq nf(n)$. If moreover $\Kill(\mk{u})_0=\{0\}$, then $\delta_G(n)\preccurlyeq f(n)$.
\end{thm}

Note that in all examples we are aware of, the function $f$ can be chosen to be equivalent to the Dehn function of $N$. 

The proof follows similar steps as in the case of standard solvable groups, so we only sketch it. The first step is an upper bound for the Dehn function of generalized tame groups, which was obtained in \S\ref{s_gtame}.

\begin{proof}[Sketch of proof of Theorem \ref{generalized}]
Theorem \ref{thm:tameN} implies that the generalized standard tame subgroups of $G$ (i.e., the $U_i\rtimes N$, where $U_i$ are the standard tame subgroups of $U$) have their Dehn function $\preceq f(n)$. The remainder of the proof follows the same steps as Theorems \ref{mainss0} and \ref{mainss}, allowing a reduction to giving estimates on the area of amalgamation and welding relations of some given combinatorial length. The amalgamation relations have a linearly bounded area with respect to their length, by the same trivial argument as in Proposition \ref{quadt}. In case $\Kill(\mk{u})_0=\{0\}$, this provides $\delta_N$ as an upper bound for the Dehn function. 

In the context of generalized standard groups, the analogue of Theorem 
\ref{qubo} holds (with the same proof using Gromov's trick and the results of \S\ref{s_awbcl}), with an estimate $\preccurlyeq f(n)$ instead of $\preccurlyeq n^2$. Accordingly, the same holds for its corollary, Theorem \ref{adqu}. The remainder of the proof of the generalized tame analogue of Theorem \ref{weldcub} works in the same way: we perform $\simeq n$ times a homotopy of area $\preccurlyeq f(n)$ instead of $\preccurlyeq n^2$, which yields an area $\preccurlyeq nf(n)$ for welding relators. The concluding step is exactly as in \S\ref{concl} and yields a Dehn function $\preccurlyeq nf(n)$.
\end{proof}


\section{Central and hypercentral extensions and exponential Dehn function}\label{s_cent}

In this section, unless explicitly specified, all Lie algebras are finite-dimensional over a field $K$ of characteristic zero.

\subsection{Introduction of the section}
The purpose of this section is to prove the negative statements of the introduction in presence of 2-homological obstructions (Theorem \ref{thmi_hom}), gathered in the following theorem. 

\begin{thm}\label{ho}
\begin{enumerate}
\item\label{ho1} If $G$ is a standard solvable group (see Definition \ref{d_ssg}) satisfying the non-Archimedean 2-homological obstruction, then $G$ is not compactly presented;
\item\label{ho2} if $G$ is a standard solvable group satisfying the 2-homological obstruction, then it has an least exponential Dehn function (possibly infinite);
\item\label{ho3} if $G$ is a real triangulable group with the 2-homological obstruction, then it has an at least exponential Dehn function.
\end{enumerate}
\end{thm}

(\ref{ho1}) and (\ref{ho2}) are proved, by an elementary argument relying on central extensions, in \S\ref{s_ce}. (\ref{ho3}) is much more involved. The reason is that the exponential radical $\g^\infty$ is not necessarily split in $\g$ and the non-vanishing of $H_2(\g)_0$ does not necessarily yield a central extension of $\g$ in degree zero (an explicit counterexample is given in \S\ref{s_sce}). We then need some significant amount of work to show that it provides, anyway, a {\em hypercentral} extension.

\subsection{FC-Central extensions}\label{s_ce}

We use the following classical definition, which is a slight weakening of the notion of central extension.

\begin{defn}\label{settingzg}
Consider an extension
\begin{equation}\label{fce}1\to Z\stackrel{i}\to \widetilde{G}\to G\to 1.\end{equation}
 We say that it is an {\bf FC-central} extension if $i(Z)$ is FC-central in $\widetilde{G}$, in the sense that every compact subset of $i(Z)$ is contained in a compact subset of $\widetilde{G}$ that is invariant under conjugation.
\end{defn}

This widely used terminology is (lamely) borrowed from the discrete case, in which it stands for ``Finite Conjugacy (class)".

In the following, the reader can assume, in a first reading, that the FC-central extensions (defined below) are central. The greater generality allows to consider, for instance, the case when $Z$ is a nondiscrete locally compact field on which the action by conjugation is given by multiplication by elements of modulus 1.

Consider now an FC-central extension as in (\ref{fce}), and assume that $\widetilde{G}$ generated by a compact subset $S$ (symmetric with 1);
Fix $k$, set $W_k=Z\cap S^k$, and $\check{W_k}=\overline{\bigcup_{g\in\widetilde{G}}gW_kg^{-1}}$, which is a compact subset of $Z$ by assumption. The following easy lemma, is partly a restatement of \cite[Lemma~5]{BaMS} (which deals with finitely generated groups and assumes $Z$ is central).

\begin{lem}\label{distde}
Keep the notation of Definition \ref{settingzg}. Let $\widetilde{\gamma}$ be any path in the Cayley graph of $\widetilde{G}$ with respect to $S$, joining 1 to an element $z$ of $Z$. Let $\gamma$ be the image of $\widetilde{\gamma}$ in the Cayley graph of $G$ (with respect to the image of $S$). If $\widetilde{\gamma}$ can be filled by $m$ $(\le k)$-gons, then $z\in \check{W_k}^m$.
\end{lem}
\begin{proof}
If $\widetilde{\gamma}$ can be filled by $m$ $(\le k)$-gons, then $z$ can be written (in the free group over $S$, hence in $\tilde{G}$) as $z=\prod_{i=1}^mh_ir_ih_i^{-1}$, where $r_i,h_i\in\widetilde{G}$, $r_i\in \Ker(\widetilde{G}\to G)=Z$ having length $\le k$ with respect to $S$, i.e.\ $r_i\in W_k$. Thus $h_ir_ih_i^{-1}\in\check{W_k}$; hence $z\in \check{W_k}^m$. 
\end{proof}

In Definition \ref{settingzg}, the group $Z$ may or not be compactly generated; if it is the case, let $U$ a compact generating set of $Z$, and define the {\bf distortion} of $Z$ in $G$ as
$$d_{G,Z}(n)=\max(n,\sup \{|g|_U:\;g\in Z,|g|_S\le n\});$$
if $Z$ is not compactly generated set $d_{G,Z}(n)=+\infty$. Note that this function actually depends on $S$ and $U$ as well, but its $\sim$-equivalence class only depends on $(G,Z)$.

\begin{prop}\label{mindehn}
Given a FC-central extension as in Definition \ref{settingzg}, if $G$ is compactly presented, then $Z$ is compactly generated and its Dehn function satisfies $\delta_G(n)\succeq d_{G,Z}(n)$.
\end{prop}
\begin{proof}
By Lemma \ref{distde}, if $G$ is presented by $S$ and relators of length $\le k$, then $Z$ is generated by $\check{W_k}$, which is compact.

If $U$ is a compact generating set for $Z$, then $W_k\subset U^\ell$ for some $\ell$. Write $d(n)=d_{G,Z}(n)$ (relative to $S$ and $U$) and $\delta(n)=\delta_G(n)$. Consider $g\in Z$ with $|g|_S\le n$ and $|g|_U=d(n)$. Taking $\widetilde{\gamma}$ to be a path of length $\le n$ in $\widetilde{G}$ joining $1$ and $g$ as in Lemma \ref{distde}, we obtain that the loop $\gamma$ in $G$ has length $\le n$ and area $m$, and Lemma \ref{distde} implies that $d(n)=|g|_U\le m\ell$. So $\delta(n)\ge d(n)/\ell$. 
\end{proof}

\begin{proof}[Proof of (\ref{ho1}) and (\ref{ho2}) in Theorem \ref{ho}]
In the setting of (\ref{ho1}), we assume that $G=U\rtimes A$ satisfies the non-Archimedean 2-homological obstruction, so that for some $j$, $\K_j$ is non-Archimedean
and the condition $Z=H_2(\mk{u}_j)_0\neq 0$ means that the action of $A$ on $U_j$ can be lifted to an action on a certain FC-central extension
$$1\to Z\stackrel{i}\to \tilde{U}_j\to U_j\to 1,$$
so that $i(Z)$ is contained and FC-central in $[\tilde{U}_j,\tilde{U}_j]$; here $\tilde{\mk{u}}$ is the blow-up of the graded Lie algebra $\mk{u}$, as defined in \S\ref{s_bu}, and $Z\simeq H_2(\mk{u}_j)$. Thus it yields an FC-central extension
\begin{equation}1\to Z\stackrel{i}\to \tilde{U}\rtimes A\to U\rtimes A\to 1,\label{zua}\end{equation}
where $\tilde{U}=\tilde{U}_j\times\prod_{j'\neq j}U_{j'}$; note that $\tilde{U}\rtimes A$ is compactly generated. By Proposition \ref{mindehn}, it follows that $U\rtimes A$ is not compactly presented.

In the setting of (\ref{ho2}), the proof is similar, with $\K_j$ being Archimedean; there is a difference however: in (\ref{zua}), $Z$ need not be FC-central in $\tilde{U}\rtimes A$, because of the possible real unipotent part of the $A$-action. Note that $i(Z)$ is central in $\tilde{U}$. We then consider an $A$-irreducible quotient $Z''=Z/Z'$ of $Z$ and consider the FC-central extension
$$1\to Z''\stackrel{i}\to \tilde{U}\rtimes A\to U\rtimes A\to 1.$$
 Since the real group $U_j$ is globally exponentially distorted (in the sense that {\em all} elements of exponential size in $U_j$ have linear size in $U_j\rtimes A$), it follows that $i(Z'')$ is exponentially distorted as well and by Proposition \ref{mindehn}, it follows that $U\rtimes A$ has an at least exponential Dehn function.
\end{proof}

\subsection{Hypercentral extensions}\label{su_hy}

To prove Theorem \ref{ho}(\ref{ho3}), the natural approach seems to start with a real triangulable group $G$ with $H_2(\g^\infty)_0\neq 0$ and find a central extension of $G$ with exponentially distorted center. If the exponential radical of $G$ is split, i.e.\ if $G=G^\infty\rtimes A$ for some nilpotent group $A$, the existence of such a central extension follows by a simple argument similar to that in the proof of Theorem \ref{ho}(\ref{ho2}). Unfortunately, in general such a central extension does not exist; although there are no simple counterexamples, we construct one in \S\ref{s_sce}.

Nevertheless, in order to prove Theorem \ref{ho}(\ref{ho3}), the geometric part of the argument is the following variant of Proposition \ref{mindehn}.

\begin{prop}\label{gexih}
Let $G$ be a connected triangulable Lie group and $G^\infty$ its exponential radical (see Definition \ref{d_er}). Suppose that there exists an extension of connected triangulable Lie groups
$$1\longrightarrow N\longrightarrow H \longrightarrow G\longrightarrow 1$$
with $N$ hypercentral in $H$ (i.e.~the ascending central series of $H$ covers $N$) and with $N\cap H^\infty\neq\{1\}$. Then $G$ has an (at least) exponential Dehn function.
\end{prop}                                                                      
\begin{proof}We fix a compact generating set $S$ in $H$, and its image 
$S'$ in $G$.
Let $h_n\in N$ be an element of linear size $n$ in 
$H$ and exponential size $\simeq e^n$ in $N$. Pick a path of size $n$ joining $1$ to $h$ in $H$, i.e.~represent 
$h$ 
by a word $\gamma_n=x_1\dots x_n$ in $H$ with $x_i\in S$. Push this path 
forward to $G$ to get a loop of size $n$ in $G$; let $a_n$ be its area. So 
in the free group over $S$, we have 
$$\gamma_n=\prod_{j=1}^{a_n} g_{nj}r_{nj}g_{nj}^{-1}$$
where $r_j$ is a relation of $G$ (i.e.\ represents the identity in $G$) of bounded size. By a standard argument using van Kampen 
diagrams (see Lemma \ref{vk}), we can choose the size of $g_{nj}$ to be at most $\le C(a_n+n)$, where $C$ is a positive constant only depending on $(G,S')$.
Push this forward to $H$ to get

$$h_n=\prod_{j=1}^{a_n} g_{nj}r_{nj}g_{nj}^{-1},$$ where $r_{nj}$ here is a bounded 
element of $N$, and $g_{nj}$ has length $\le C(a_n+n)$ in $H$. Since the 
action of $H$ on $N$ by conjugation is unipotent, we deduce that the size in 
$N$ of $g_{nj}r_{nj}g_{nj}^{-1}$ is polynomially bounded with respect to $a_n+n$, say $\preceq (a_n+n)^d$ (uniformly in $j$). Therefore $h_n$ has size $\preceq (a_n+n)^{d+1}$. Since $(h_n)$ has exponential growth in $N$, we deduce that $(a_n)$ also grows exponentially.
\end{proof}

Let us emphasize that at this point, the proof of Theorem \ref{ho}(\ref{ho3}) is not yet complete. Indeed, given a real triangulable group $G$ with $H_2(\g^\infty)_0\neq 0$, we need to check that we can apply Proposition \ref{gexih}. This is the contents of the following theorem. If $V$ is a $G$-module, $V^G$ denotes the set of $G$-fixed points. Also, see Definition \ref{d_gi} for the definition of $\g^\infty$.

\begin{thm}Let $G$ be a triangulable Lie group. Suppose that $H_2(\g^\infty)^G\neq\{0\}$. Then there exists an extension of connected triangulable Lie groups
$$1\longrightarrow N\longrightarrow H \longrightarrow G\longrightarrow 1$$
with $N$ hypercentral in $H$ and with $N\cap H^\infty\neq\{1\}$.\label{existhyp}\end{thm}

The theorem will easily follow from the analogous (more general) result about solvable Lie algebras.

By {\bf epimorphism} of Lie algebras we mean a surjective homomorphism, and we denote it by a two-headed arrow $\mk{h}\tw\g$. Such an epimorphism is $k$-{\bf hypercentral} if its kernel $\mk{z}$ is contained in the $k$th term $\mk{h}^k$ of the ascending central series of $\mk{h}$; when $k=1$, that is, when $\mk{z}$ is central in $\mk{h}$, it is simply called a {\bf central epimorphism}. 
Also, if $\mk{m}$ is a $\g$-module, we write $\mk{m}^\g=\{m\in\mk{m}:\forall g\in\g,gm=0\}$. 

\begin{defn}We say that a hypercentral epimorphism $\mk{h}\epi\mk{g}$ between Lie algebras has 
{\bf polynomial distortion} if the induced epimorphism 
$\mk{h}^\infty\epi\mk{g}^\infty$ is bijective; otherwise we say it has {\bf non-polynomial distortion}.
\end{defn}

The terminology is motivated by the fact that if $G,H$ are triangulable Lie groups and $H\epi G$ is a hypercentral epimorphism, then its kernel is polynomially distorted in $H$ if and only if the corresponding Lie algebra hypercentral epimorphism has polynomial distortion (in the above sense), and otherwise the kernel is exponentially distorted in $H$.

\begin{thm}\label{existhypa}
Let $\g$ be a $K$-triangulable Lie algebra. Assume that $H_2(\g^\infty)^\g\neq \{0\}$. Then there exists a hypercentral epimorphism $\mk{h}\epi\g$ with non-polynomial distortion.
\end{thm}
The condition $H_2(\g^\infty)^\g\neq \{0\}$ can be interpreted as $H_2(\g^\infty)_0\neq \{0\}$, where $\g$ is endowed with a Cartan grading (see \S\ref{su_ca}, especially Lemma \ref{eqh2}). Theorem \ref{existhypa} will be proved in \S\ref{prex}.

\begin{proof}[Proof of Theorem \ref{existhyp} from Theorem \ref{existhypa}]
Theorem \ref{existhypa} provides a hypercentral epimorphism $\mk{h}\epi\g$, with 
kernel denoted by $\mk{z}$, and with non-polynomial distortion, 
i.e.~$\mk{h}^\infty\cap\mk{z}\neq\{0\}$. Since $\mk{z}$ is hypercentral, 
the action of $\g$ on $\mk{z}$ is nilpotent, hence triangulable, and since moreover $\g$ and $\mk{z}$ are triangulable, we deduce that $\mk{h}$ 
is triangulable. Let $H\to G$ be the corresponding surjective homomorphism of triangulable Lie groups and $Z$ its kernel, which is hypercentral. Then $H^\infty\cap Z\neq\{1\}$, because the Lie algebra counterpart holds. So the theorem is proved. 
\end{proof}

\subsection{Proof of Theorem \ref{existhypa}}\label{prex}

\begin{lem}\label{dise}
Let $\mk{g}\epi\mk{h}\epi\mk{l}$ be epimorphisms of Lie algebras such that the composite homomorphism is a hypercentral epimorphism. If $\mk{g}\epi\mk{l}$ has polynomial distortion then so does $\mk{h}\epi\mk{l}$. 
\end{lem}
\begin{proof}
This is trivial.
\end{proof}

\begin{lem}\label{lift_extension_centrale}
Let $\g$ be a Lie algebra and $(T^t)_{t\in K}$ be a one-parameter group of unipotent automorphisms of $\g$. Let $\mathfrak{h}\epi\g$ be a central epimorphism. Then there exists a Lie algebra $\mk{k}$ with an epimorphism $\rho:\mk{k}\epi\mk{h}$ such that the composite homomorphism $\mk{k}\epi\g$ is a central epimorphism, and such that the action of $(T^t)$ lifts to a unipotent action on $\mk{k}$.
\end{lem}

\begin{rem}
The conclusion of Lemma \ref{lift_extension_centrale} cannot be simplified by the requirement that $\mk{k}=\mk{h}$, as we can see, for instance, by taking $\mk{h}$ to be the direct product of the 3-dimensional Heisenberg algebra and a 1-dimensional algebra and a suitable 1-parameter subgroup of unipotent automorphisms of $\g=\mk{h}/[\mk{h},\mk{h}]$.
\end{rem}

\begin{proof}[Proof of Lemma \ref{lift_extension_centrale}]
Set $\mathfrak{z}=\Ker(\mathfrak{h}\epi\g)$ and denote the Hopf bracket (see \S\ref{hopb}) by $$[\cdot,\cdot]':\g\wedge\g\to\mk{h}.$$

Pick a linear projection $\pi:\mk{h}\to\mk{z}$ and define $b(x\wedge y)=\pi([x,y]')$. This gives a linear identification of $\mk{h}$ with $\g\oplus\mk{z}$, for which the law is given as 
$$[\,\langle x_1,z_1\rangle \,,\,\langle x_2,z_2\rangle \,]=\langle \,[x_1,x_2]\,,\,b(x_1\wedge x_2)\,\rangle $$
(in this proof, we write pairs $\langle x,z\rangle$ rather than $(x,z)$ for the sake of readability).
Observe that $T^t$ naturally acts on $\g\wedge\g$, preserving $Z_2(\g)$ and $B_2(\g)$. If $c\in\g\wedge\g$, define the function 
\begin{eqnarray*}
\alpha_{c}: K & \to & \mk{z}\\
u & \mapsto & b(T^uc).
\end{eqnarray*}

If $d$ is the dimension of $\g$, let $W$ denote the space of $K$-polynomial mappings of degree $<2d$ from $K$ to $\mathfrak{z}$; the dimension of $W$ is $2d\dim(\mk{z})$. 
Now $t\mapsto T^t$ is a polynomial of degree $<d$ valued in the space of endomorphisms of $\g$, so is also polynomial of degree $<2d$ valued in the space of endomorphisms of $\g\wedge\g$. So $\alpha_c$ is a polynomial of degree $<2d$, from $K$ to $\mk{z}$.

If $c\in\g\wedge\g$ is a boundary then $\alpha_c=0$. So $\alpha$ defines a central extension $\mk{k}=\g\oplus W$ (as a vector space) of $\g$ with kernel $W$, with law
$$[\langle x_1,\zeta_1\rangle ,\langle x_2,\zeta_2\rangle ]=\langle [x_1,x_2],\alpha_{x_1\wedge x_2}\rangle ,\quad \langle x_1,\zeta_1\rangle ,\langle x_2,\zeta_2\rangle \in \g\oplus W.$$
From now on, since elements of $W$ are functions, it will be convenient 
to write elements of $\mk{k}$ as $\langle x,\zeta(u)\rangle $, where $u$ is thought of 
as an indeterminate. 
For $t\in K$, the automorphism $T^t$ lifts to an automorphism of $\mk{k}$ given by
$$T^t(\langle x,\zeta(u)\rangle )=\langle T^tx,\zeta(u+t)\rangle .$$
This is obviously a one-parameter subgroup of linear automorphisms; let 
us check that these are Lie algebra automorphisms (in the computation, 
for readability we write the brackets as $[\cdot ;\cdot]$, with 
semicolons instead of commas).
\begin{align*}
[\,T^t(\langle x_1,\zeta_1(u)\rangle )\,;\,T^t(\langle x_2,\zeta_2(u)\rangle )\,] & = [\,\langle T^tx_1,\zeta_1(u+t)\rangle \,;\,\langle T^tx_2,\zeta_2(u+t)\rangle \,] \\
 &= \langle \,[T^tx_1;T^tx_2]\,,\,\alpha_{T^tx_1\wedge T^tx_2}(u)\,\rangle  \\
 &= \langle \,[T^tx_1;T^tx_2]\,,\,b(T^u(T^tx_1\wedge T^tx_2))\,\rangle  \\
 &= \langle \,T^t[x_1;x_2]\,,\,b(T^{t+u}x_1\wedge T^{t+u}x_2)\,\rangle  \\
 &= \langle \,T^t[x_1;x_2]\,,\,\alpha_{x_1\wedge x_2}(t+u)\,\rangle  \\
 &= T^t(\langle \,[x_1;x_2]\,,\,\alpha_{x_1\wedge x_2}(u)\rangle \,) \\
 &= T^t(\;[\,\langle x_1,\zeta_1(u)\rangle \,;\,\langle x_2,\zeta_2(u)\rangle \,]\;),
\end{align*}
so these are Lie algebra automorphisms.
Now the mapping
\begin{eqnarray*}
\mk{k} & \to & \mk{h}\\
\rho:\langle x,\zeta(u)\rangle &  \mapsto & \langle x,\zeta(0)\rangle \end{eqnarray*} is clearly a surjective linear map; it is also
a Lie algebra homomorphism: indeed
\begin{align*}
 [\,\rho(\langle x_1,\zeta_1(u)\rangle )\,;\,\rho(\langle x_2,\zeta_2(u)\rangle )\,] 
& = [\,\langle x_1,\zeta_1(0)\rangle \,;\,\langle x_2,\zeta_2(0)\rangle \,] \\
 &= \langle \,[x_1;x_2]\,,\,b(x_1\wedge x_2)\,\rangle  \\
 &= \langle \,[x_1;x_2]\,,\,\alpha_{x_1\wedge x_2}(0)\,\rangle  \\
 &=\rho(\langle \,[x_1;x_2]\,,\,\alpha_{x_1\wedge x_2}(u)\rangle \,) \\
 &=\rho([\,\langle x_1,\zeta_1(u)\rangle \,;\,\langle x_2,\zeta_2(u)\rangle \,])\qedhere 
 \end{align*}
\end{proof}

\begin{lem}\label{li_ext_hyp}
The statement of Lemma \ref{lift_extension_centrale} holds true if we replace, in both the hypotheses and the conclusion, central by $k$-hypercentral.
\end{lem}
\begin{proof} The case $k=1$ was done in Lemma 
\ref{lift_extension_centrale}.
Decompose $\mk{h}\twoheadrightarrow\g$ as 
$\mk{h}\twoheadrightarrow\mk{h}_1\twoheadrightarrow\g$, with 
$\mk{h}\twoheadrightarrow\mk{h}_1$ central (with kernel $\mk{z}$) and 
$\mk{h}_1\twoheadrightarrow\g$ $(k-1)$-hypercentral. By induction 
hypothesis, there exists $\mk{k}$ with $\mk{k}\tw\mk{h}_1$ such that the composite epimorphism $\mk{k}\epi\g$ is $(k-1)$-hypercentral and such that $(T^t)$ lifts to $\mk{k}$. Consider the fibered product $\mk{h}\times_{\mk{h}_1}\mk{k}$ of the two epimorphisms $\mk{h}\tw\mk{h}_1$ and $\mk{k}\tw\mk{h}_1$, so that the two lines in the diagram below are central extensions and both squares commute.
 $$\xymatrix{   &  & & &\g\\
 0 \ar[r]& \mk{z} \ar@{=}[d]\ar[r]  & \mk{h}\ar@{->>}[r] & \mk{h}_1\ar[r]\ar@{->>}[ru] & 0\\
    0 \ar@{-}[r]   & \mk{z} \ar@{-}[r]  & \mk{h}\times_{\mk{h}_1}\mk{k}\ar@{->>}[u] \ar[r] &\mk{k}\ar@{->>}[u]\ar[r] & 0\\
      &  & \mk{m}\ar@{->>}[u]\ar@{->>}[ru] & &
  }$$

Applying Lemma \ref{lift_extension_centrale} again to $\mk{h}\times_{\mk{h}_1}\mk{k}\tw\mk{k}$, we obtain $\mk{m}\tw\mk{h}\times_{\mk{h}_1}\mk{l}$ so that the composite epimorphism $\mk{m}\epi\mk{k}$ is central and so that $(T^t)$ lifts to $\mk{m}$. So the composite map $\mk{m}\tw\mk{h}$ is the desired homomorphism.
\end{proof}

We say that a Lie algebra $\g$ is {\bf spread} if it can be written as $\g=\mk{n}\rtimes\mk{s}$ where $\mk{n}$ is nilpotent, $\mk{s}$ is reductive and acts reductively on $\mk{n}$. It is {\bf spreadable} if there exists such a decomposition.

When $\g$ is solvable, $\mk{s}$ is abelian and a Cartan subalgebra of $\g$ is given by the centralizer $\mk{h}=C_\g(\mk{s})=C_\mk{n}(\mk{s})\times\mk{s}$. In particular, the $\mk{s}$-characteristic decomposition of $\g$ coincides with the $\mk{h}$-characteristic decomposition, and the associated Cartan gradings are the same.

This remark is useful when we have to deal with a homomorphism $\mk{n}_1\rtimes\mk{s}\to\mk{n}_2\rtimes\mk{s}$ which is the identity on $\mk{s}$: indeed such a homomorphism is graded for the Cartan gradings.

\begin{proof}[Proof of Theorem \ref{existhypa}]
We first prove the result when $\g$ is spread, so
$\g=\mk{u}\rtimes\mk{d}$. Let $k$ be the dimension of $\mk{u}/\g^\infty$. We argue by induction on $k$.

Suppose that $k=0$. 
We have a grading of $\g^\infty$, valued in 
$\mk{d}^\vee$. Consider the blow-up construction (Lemma \ref{blowupc}): 
it gives a graded Lie algebra $\widetilde{\g^\infty}$ and a graded 
surjective map $\widetilde{\g^\infty}\tw\g^\infty$ with central kernel 
concentrated in degree zero and isomorphic to $H_2(\g^\infty)_0$. This 
grading, 
valued in $\mk{d}^\vee$, defines a natural action of $\mk{d}$ on 
$\widetilde{\g^\infty}$ and the epimorphism 
$\widetilde{\g^\infty}\rtimes\mk{d}\tw\g^\infty\rtimes\mk{d}$ is 
central. By construction, the kernel $H_2(\g^\infty)_0$ is contained in 
$(\widetilde{\g^\infty}\rtimes\mk{d})^\infty$ so this (hyper)central 
epimorphism has non-polynomial distortion.

Now suppose that $k\ge 1$. Let $\mk{n}$ be a codimension 1 ideal of 
$\g$ containing $\g^\infty\rtimes\mk{d}$. Since $\mk{n}$ contains 
$\g_\td\oplus\mk{d}$, the intersection of $\mk{n}$ with $\mk{u}_0$ is a 
hyperplane in $\mk{u}_0$. So there exists a one-dimensional subspace 
$\mk{l}\subset\mk{u}_0$, such that $\g=\mk{n}\rtimes\mk{l}$. Note that 
the grading of $\g$, valued in $\mk{d}^\vee$, extends that of $\mk{n}$ and 
$\mk{n}^\infty=\g^\infty$ and in particular, 
$H_2(\mk{n}^\infty)_0\neq\{0\}$.

By induction hypothesis, there exists a hypercentral epimorphism $\mk{h}\epi\mk{n}$, with non-polynomial distortion. By Lemma \ref{li_ext_hyp}, there exists a hypercentral epimorphism $\mk{m}\epi\mk{h}$ (with kernel $\mk{z}$) so that the action of $e^{\mk{ad}(\mk{l})}$ on $\mk{n}$ lifts to a unipotent action on $\mk{m}$. This corresponds to a nilpotent action of $\mk{l}$ on $\mk{m}$.
Let $\mk{z}_i$ be the intersection of the $i$th term of the ascending 
central series of $\mk{m}$ with $\mk{z}$, so $\mk{z}_\ell=\mk{z}$ for some $\ell$. On each $\mk{z}_{i+1}/\mk{z}_i$, the action of $\mk{l}$ is nilpotent and the adjoint action of $\mk{m}$ is trivial. So the action of $\mk{m}\rtimes\mk{l}$ on each $\mk{z}_{i+1}/\mk{z}_i$, hence on $\mk{z}$, is nilpotent. That is, $\mk{z}$ is hypercentral in $\mk{m}\rtimes\mk{l}$ (this is where the argument would fail with ``hypercentral" replaced by ``central"). So $\mk{m}\rtimes\mk{l}\epi\g$ is the desired hypercentral epimorphism: by Lemma \ref{dise}, $\mk{m}\epi\mk{n}$ has non-polynomial distortion and therefore so does $\mk{m}\rtimes\mk{l}\epi\g$.

Now the result is proved when $\g$ is spread. In general, fix a faithful 
linear representation $\g\to\mk{gl}_n$, and let 
$\mk{h}=\mk{u}\rtimes\mk{d}$ 
be 
the splittable hull of $\g$ in $\mk{gl}_n$ (that is, the subalgebra generated by semisimple and nilpotent parts of elements of $\g$ for the additive Jordan decomposition, see \cite[Chap.~VII, \S 5]{Bou}). If $\mk{n}$ is any Cartan 
subalgebra of 
$\mk{h}$, then $\mk{n'}=\mk{n}\cap\g$ is a Cartan subalgebra of $\g$ 
\cite[Chap.~VII, \S5, Ex.~8]{Bou}. Now 
$\mk{n}=(\mk{u}\cap\mk{n})+\mk{n}'$, but every $\mk{n}$-weight of 
$\mk{h}$ vanishes on $\mk{u}\cap\mk{n}$. So the $\mk{n}$-grading of 
$\mk{h}$ extends the $\mk{n}'$-grading of $\mk{g}$ (in other words, the 
embedding $\g\subset\mk{h}$ is a graded map). In particular, since 
$\g^\infty=\mk{h}^\infty$, this equality is an isomorphism of graded 
algebras and we deduce that $H_2(\mk{h}^\infty)_0\neq\{0\}$. So we 
obtain a 
hypercentral extension $\mk{m}\epi\mk{h}$, with non-polynomial distortion 
because $\g^\infty=\mk{h}^\infty$. By 
taking the inverse image of $\g$ in $\mk{m}$, we obtain the desired hypercentral 
extension of $\g$.
\end{proof}

\subsection{An example without central extensions}\label{s_sce}

We prove here that in the conclusion of Theorem \ref{existhyp}, it is not always possible to replace, in the conclusion, hypercentral by central. We begin with the following useful general criterion.

\begin{prop}\label{excentr}
Let $\g$ be a finite-dimensional solvable Lie algebra with its Cartan 
grading. 
Then $\g$ has no central extension with non-polynomial 
distortion if and only if the image of 
$(\Ker(d_2)\cap(\g_\td\we\g_\td))_0$ in 
$H_2(\g)_0$ is zero (i.e.\ it is contained in $\textnormal{Im}(d_3)$). 
\end{prop}
\begin{proof}
Suppose that the image of the above map is nonzero. By the blow-up 
construction (Lemma \ref{blowupc}), we obtain a central extension 
$\breve{\g}\tw\g$ with kernel $H_2(\g)_0$ concentrated in degree zero. 
By the 
assumption, there exist $x_i,y_i$ in $\g$, of nonzero opposite weights 
$\pm \alpha_i$, such that $\sum x_i\we y_i$ is a 2-cycle and is 
nonzero in $H_2(\g)_0$. 
This means that in $\breve{\g}$, the element $z=\sum [x_i,y_i]$ is a 
nonzero 
element of the central kernel $H_2(\g)_0$. So $z\in\breve{\g}^\infty$ and 
the central epimorphism $\breve{\g}\tw\g$ does not have polynomial 
distortion.

Conversely, suppose that there exists a central epimorphism 
$\breve{\g}\tw\g$ 
with 
non-polynomial distortion, with kernel $\mk{z}$. Note that the Cartan grading lifts to $\breve{\g}$, so that the kernel $\mk{z}$ is concentrated in degree zero. By assumption, 
$\mk{z}\cap\breve{\g}^\infty$ contains a nonzero element $z$. By 
Lemmas \ref{430} and 
\ref{ideal1tame}, we can write, in $\breve{\g}$, $z=\sum [x_i,y_i]$ with 
$x_i,y_i$ of nonzero opposite weights. Thus in $\g$, $\sum x_i\wedge 
y_i$ is a nonzero element of $H_2(\g)_0$.
\end{proof}

Let $\tilde{G}$ be the 15-dimensional $\R$-group of $6\times 6$ upper triangular matrices of the form

\begin{equation}\label{sixsix}\begin{pmatrix}
1 & x_{12} & u_{13} & u_{14} & u_{15} & u_{16}\\
0 & 1 & 0 & 0 & x_{25} & x_{26} \\
0 & 0 & t_3 & u_{34} & u_{35} & u_{36} \\
0 & 0 & 0 & t_4 & u_{45} & u_{46} \\
0 & 0 & 0 & 0 & 1 & x_{56}\\
0 & 0 & 0 & 0 & 0 & 1 
\end{pmatrix},
\end{equation}

\noindent where $t_3,t_4$ are nonzero. Its unipotent radical $\tilde{U}$ consists 
of elements of the form (\ref{sixsix}) with $t_3=t_4=1$ and its exponential radical $\tilde{E}$ consists of elements in $\tilde{U}$ for which $x_{12}=x_{25}=x_{26}=x_{56}=0$. If $D$ denotes the (two-dimensional) diagonal subgroup in $\tilde{G}$, the quotient $\tilde{G}/\tilde{E}$ is isomorphic to the direct product of $D$ with a 4-dimensional unipotent group (corresponding to coefficients $x_{12}$, $x_{25}$, $x_{26}$, $x_{56}$). Note that the extension $1\to\tilde{E}\to\tilde{U}\to \tilde{U}/\tilde{E}\to 1$ is not split.

Let $Z$ the 2-dimensional subgroup of $\tilde{U}$ consisting of matrices with all entries zero except $u_{16}$ and $x_{26}$. Note that $Z$ is hypercentral and has non-trivial intersection with the exponential radical of $\tilde{G}$.

Define $G=\tilde{G}/Z$; it is 13-dimensional. The weights of $E=\tilde{E}/Z$ are arranged as follows (the principal weights are in boldface)

\begin{equation}\label{poids6}
 \xymatrix{ & 14 \ar@{-}[d]  & \mathbf{34}\\
    \mathbf{13} \ar@{-}[r]   & 15\ar@{-}[ru] \ar@{-}[r] \ar@{-}[d] & 35\;36 \\
      & \mathbf{45}\;\mathbf{46} & 
  }
\end{equation}
and the other basis elements of weight zero in $\g$ are 33, 44, 12, 25, 56.

We see that $\mathfrak{e}$ is 2-tame. Besides, $H_2(\mathfrak{e})_0\neq 
0$, as $13\we 36$ is a 2-cycle in degree 0 that is not a 2-boundary, 
as follows from the observation that
$\mathfrak{e}$ has an obvious nontrivial central extension in degree zero, given by at the level of groups by
$$1\to Z/Z'\to \tilde{E}/Z'\to E\to 1,$$
where $Z'$ is the one-dimensional subgroup at position $26$ (that is, the subgroup of $Z$ consisting of matrices with $u_{16}=0$), which is normalized by $\tilde{E}$ but not by $\tilde{U}$.

Since $H_2(\g^\infty)_0\neq\{0\}$, by Theorem \ref{existhypa} there 
exists a hypercentral epimorphism $\mk{h}\to\mk{g}$ with non-polynomial 
distortion. By contrast, every central epimorphism $\mk{h}\to\mk{g}$ has 
polynomial distortion. This follows from Proposition \ref{excentr} and 
the 
following proposition.

\begin{prop}\label{six}Let $\g$ be the above 13-dimensional triangulable Lie algebra. Then in $H_2(\mathfrak{g})_0$, the image of 
$(\Ker(d_2)\cap(\g_\td\we\g_\td))_0$ is zero.
\end{prop}

\begin{proof}To streamline the notation, we denote by $ij$ the elementary matrix usually denoted by $E_{ij}$,  with 1 at position $(i,j)$ and zero everywhere else (including the diagonal).
By considering each pair of nonzero opposite weights in (\ref{poids6}),
we can describe the map $d_2$ on a basis of the 4-dimensional space $(\mk{u}_\td\we \mk{u}_\td)_0$.
\begin{align*}13\wedge 35 & \stackrel{d_2}\longmapsto -15  & 13\wedge 36 & \longmapsto 0,\\                                                                   
14\wedge 45 & \longmapsto -15,      & 14\wedge 46 & \longmapsto 0;
\end{align*}
accordingly a basis of $(\Ker(d_2)\cap(\g_\td\we\g_\td))_0$ is given by
$$13\wedge 35 - 14\wedge 45, \quad  13\wedge 36, \quad 14\wedge 46;$$ 
we have to check that these are all boundaries; let us snatch them one 
by one:
\begin{align*}
12\wedge 25\wedge 56 & \stackrel{d_3}\longmapsto  56\wedge 15.  \\
13\wedge 34\wedge 45 & \longmapsto 13\wedge 45 -14\wedge 45  \\  
13\wedge 35\wedge 56 & \longmapsto 56\wedge 15 + 13\wedge 36  \\               
14\wedge 45\wedge 56 & \longmapsto 56\wedge 15 + 14\wedge 46.\qedhere
\end{align*}
\end{proof}

Combining with Proposition \ref{excentr}, we get:

\begin{cor}\label{nocentral}
We have $H_2(\g^\infty)_0\neq \{0\}$, but there is no central extension of Lie groups
$$1\longrightarrow \R\stackrel{j}\longrightarrow\breve{G}\longrightarrow G\longrightarrow 1$$
with $j(\R)$ exponentially distorted in $\breve{G}$.\qed
\end{cor}


\section{$G$ not 2-tame}\label{s_not}

Here we prove that any group satisfying the SOL obstruction has an at least exponential Dehn function. The method also provides the result that any group satisfying the non-Archimedean SOL obstruction is not compactly presented.

\subsection{Combinatorial Stokes formula}\label{cst}

\begin{defn}
Let $X$ be a set and let $\RR$ be any commutative ring. We call a {\it closed path} a sequence $\mathbf{c}=(c_0,\dots,c_n)$ of points in $X$ with $c_0=c_n$ (so we can view it as indexed by $\Z/n\Z$). If $\alpha,\beta$ are functions $X\to \RR$, we define 
\begin{align*}\int_\mathbf{c} \beta d\alpha=&\sum_{i\in\Z/n\Z} \beta(c_i)(\alpha(c_{i+1})-\alpha(c_{i-1}))\\
= & \sum_{i\in\Z/n\Z} \beta(c_i)\alpha(c_{i+1})-\beta(c_{i+1})\alpha(c_{i}).
\end{align*}
\end{defn}

Clearly, this is invariant if we shift indices.
The following properties are immediate consequences of the definition.

\begin{itemize}
\item(Antisymmetry) We have
$$\int_\mathbf{c}  \beta d\alpha=-\int_\mathbf{c}  \alpha d \beta.$$

\item(Concatenation)
If $c_0=c_i=c_n$ and we write $\mathbf{c}'=(c_0,\dots,c_i)$ and $\mathbf{c}''=(c_i,\dots,c_n)$,
$$\int_\mathbf{c}  \beta d\alpha=\int_{\mathbf{c}'}  \beta d\alpha+\int_{\mathbf{c}''} \beta d\alpha.$$

\item(Filiform vanishing) If $n=2$ then the integral vanishes. More generally, the integral vanishes when $\mathbf{c}$ is filiform, i.e., $n$ is even and $\mathbf{c}_i=\mathbf{c}_{n-i}$ for all $i$.
\end{itemize}
Indeed, the difference $\int_\mathbf{c}  \beta d\alpha-\int_{\mathbf{c}'} 
 \beta d\alpha-\int_{\mathbf{c}''} \beta d\alpha$ is equal to
$$\beta(c_0)[(\alpha(c_{i+1})-\alpha(c_{i-1}))+(\alpha(c_{1})-\alpha(c_{n-1}))$$ $$-(\alpha(c_{1})-\alpha(c_{i-1}))-(\alpha(c_{i+1})-\alpha(c_{n-1}))]=0.$$
The filiform vanishing is immediate for $n=2$ and follows in general by an induction based on the concatenation formula.

Now let us deal with a Cayley graph of a group $G$ with a generating set $S$, and we consider  paths in the graph, that is sequences of vertices linked by edges. Thus any closed path based at 1 can be encoded by a unique element of the free group $F_S$, which is a relation (i.e.\ an element of the kernel of $F_S\to G$), and conversely, if $r$ is a relation, we denote by $[r]$ the corresponding closed path based at 1. Note that $G$ acts by left translations on the set of closed paths. The above properties imply the following

\begin{itemize}
\item (Product of relations) If $r,r'$ are relations, we have
$$\int_{[rr']} \beta d\alpha=\int_{[r]} \beta d\alpha+\int_{[r']} \beta d\alpha.$$

\item (Conjugate of relations) If $r$ is a relation and $\gamma\in F_S$,
$$\int_{[\gamma r\gamma^{-1}]} \beta d\alpha=\int_{\gamma\cdot [r]} \beta d\alpha.$$
\item (Combinatorial Stokes formula) 
 Suppose that a relation $r$ is written as a product $r=\prod_{i=1}\gamma_ir_i\gamma_i^{-1}$ of conjugates of relations. Then
$$\int_{[r]} \beta d\alpha=\sum_{i=1}^k\int_{\gamma_i\cdot [r_i]} \beta d\alpha.$$
\end{itemize}

The formula for products follows from concatenation if there is no simplification in the product $rr'$, and follows by also using the filiform vanishing otherwise. The formula for conjugates also follows using the filiform vanishing. The Stokes formula follows from the two previous by an immediate induction.

\begin{rem}
The above combinatorial Stokes formula is indeed analogous to the classical Stokes formula on a disc: here the left-hand term is thought of as an integral along the boundary, while the right-hand term is a discretized integral over the surface. 
\end{rem}

\subsection{Loops in groups of SOL type}\label{lgst}

Let $\K_1$ and $\K_2$ be two nondiscrete locally compact normed fields. Consider the group
$$G=(\K_1\times\K_2)\rtimes_{(\ell_1,\ell_2^{-1})}\Z,$$
where $|\ell_2|_{\K_2}\ge |\ell_1|_{\K_1}>1$, with group law written so that the product depends affinely on the right term:
$$(x,y,n)(x',y',n')=(x+\ell_1^{n}x',y+\ell_2^{-n}y',n+n').$$
Write $|\ell_1|_{\K_1}=|\ell_2|_{\K_2}^\mu$ with $0<\mu\le 1$.

Now we can also view $x$ and $y$ as the projections to the coordinates in the above description.
In the next lemmas, we consider a normed ring $\K$ (whose norm is submultiplicative, not necessarily multiplicative), and functions $A:\K_1\to\K$, and $B:\K_2\to \K$, yielding 
functions $\alpha,\beta:G\to \K$ defined by $\alpha=A\circ x$ and $\beta=B\circ y$. 

\begin{lem}
 Suppose that $A$ is $1$-Lipschitz and that $B$ satisfies the H\"older-like 
condition $$|B(s)-B(s')|\le |s-s'|^{\mu},\quad\forall s,s'\in\K_2.$$
Then $\int_{\mathbf{c}} \beta d\alpha$ is bounded on triangles of bounded diameter, that is, when $\mathbf{c}$ ranges over triples $(c_i)_{i\in\Z/3\Z}$ with $c_i\in G$ such all $c_i^{-1}c_j$ belong to some given compact subset of $G$.\label{mst}
\end{lem}
\begin{proof} Let us consider a triangle $T$ of bounded diameter (viewed 
as a closed path of length three), i.e.\ three points $(g_0,g_0h,g_0h')$, where $h$ and $h'$ are bounded (but not $g_0$!). Note that as a consequence of the antisymmetry relation,  $\int_T \beta d\alpha$ does not change if we add constants to both $\alpha$ and $\beta$. We can therefore assume that $\alpha(g_0)=\beta(g_0)=0$. Hence $$\int_T \beta d\alpha= \beta(g_0h)\alpha(g_0h')-\beta(g_0h')\alpha(g_0h).$$ 
In coordinates, suppose that $g_0=(x_0,y_0,n_0)$, $h=(x,y,n)$ and $h'=(x',y',n')$. Then $g_0h=(x_0+\ell_1^{n_0}x,y_0+\ell_2^{-n_0}y,n_0+n)$, and $g_0h'=(x_0+\ell_1^{n_0}x',y_0+\ell_2^{-n_0}y',n_0+n').$ Since $A(x_0)=B(x_0)=0$, we have
$$\int_T\beta d\alpha  = B(y_0+\ell_2^{-n_0}y)A(x_0+\ell_1^{n_0}x')-B(y_0+\ell_2^{-n_0}y')A(x_0+\ell_1^{n_0}x).$$
By our assumptions on $A$ and $B$, we have
\begin{align*}\left|\int_T\beta d\alpha\right| & \le |\ell_2^{-n_0}y|^{\mu}|\ell_1^{n_0}x'|+|\ell_2^{-n_0}y'|^{\mu}|\ell_1^{n_0}x|\\
& =|y|^{\mu}|x'|+|y'|^{\mu}|x|,\end{align*}
which is duly bounded when $h,h'$ are bounded.
\end{proof}

Fix $n\ge 1$.
We consider the relation $$\gamma_{1,n}=t^nxt^{-n}yt^nx^{-1}t^{-n}y^{-1};$$
this defines a closed path of length $4n+4$, where $x$, $y$ and $t$ denote (by abuse of notation) the elements $(1,0,0)$ , $(0,1,0)$ and $(0,0,1)$ of $G=\K_1\times\K_2\rtimes\Z$.

\begin{lem}\label{aireg1}Consider $A,B,\alpha,\beta$ as introduced before Lemma \ref{mst}.
Suppose that $A(0)=B(1)=0$.
Then we have $$\int_{\gamma_{1,n}} \beta d\alpha=2B(0)A(\ell_1^n).$$
\end{lem}
\begin{proof}[Proof of Lemma \ref{aireg1}]
To simplify the notation, let us denote by $\mathbf{c}$ the closed path of length $4n+4$ defined by $\gamma_{1,n}.$
In the integral $\int_{\mathbf{c}}\beta d\alpha$, only those points $c_i$ for 
which ``$\beta d\alpha$" is nonzero, i.e.\ both $\alpha (c_{i+1})\neq \alpha(c_{i-1})$ and $\beta(c_i)\neq 0$, do contribute. In this example as well as the forthcoming ones, this will make most terms be equal to zero.
The closed path  $c$ can be decomposed as
$$c_0=(0,0,0),(0,0,1),\dots,(0,0,n-1),(0,0,n)=c_n,$$
$$c_{n+1}=(\ell_1^n,0,n),(\ell_1^n,0,n-1),\dots,(\ell_1^n,0,1),(\ell_1^n,0,0)=c_{2n+1},$$
$$c_{2n+2}=(\ell_1^n,1,0),(\ell_1^n,1,1),\dots,(\ell_1^n,1,n-1),(\ell_1^n,1,n)=c_{3n+2},$$
$$c_{3n+3}=(0,1,n),(0,1,n-1),\dots,(0,1,1),(0,1,0)=c_{4n+3}.$$

We see that $x(c_{i+1})\neq x(c_{i-1})$ only for $i=n,n+1,3n+2,3n+3$. Moreover, for $i=3n+2,3n+3$, $y(c_i)=1$, so $B(y(c_i))=0$. Thus 
$$\int \beta d\alpha=\sum_{i=n}^{n+1}B(y(c_i))(A(x(c_{i+1}))-A(x(c_{i-1})))=2B(0)(A(\ell_1^n)-A(0));$$
whence the result.
\end{proof}

To illustrate the interest of the notions developed here, note that this is enough to obtain the following result.
\begin{prop}\label{solex}
Under the assumptions above, if $2\neq 0$ in $\K_1$, the group $$G=(\K_1\times\K_2)\rtimes_{(\ell_1,\ell_2^{-1})}\Z$$ has at least exponential Dehn function, and if both $\K_1$ and $\K_2$ are ultrametric then $G$ is not compactly presented.\label{krd}
\end{prop}

\begin{rem}
Actually the lower bound in Proposition \ref{solex} is optimal: if either $\K_1$ or $\K_2$ is Archimedean, then $G$ has an exponential Dehn function. Indeed, the upper bound follows from Corollary \ref{corex}.
\end{rem}

We use the following convenient language: in a locally compact group $G$ with a compact system of generators $S$, we say that a sequence of null-homotopic words $(w_n)$ in $F_S$ has {\em asymptotically infinite area} if for every $R$, there exists $N(R)$ such that no $w_n$ for $n\ge N(R)$ is contained in the normal subgroup of $F_S$ generated by null-homotopic words of length $\le R$. By definition, the non-existence of such a sequence is equivalent to $G$ being compactly presented.

\begin{proof}[Proof of Proposition \ref{krd}]
Set $I=\int_{\gamma_{1,n}}\beta d\alpha$.

Define $\K=\K_1\times\K_2$. We need to use suitable functions $A:\K_1\to\K$ and $B:\K_2\to\K$ satisfying the hypotheses of Lemmas \ref{mst} and \ref{aireg1}. For $x\in\K_2$, define $o(x)\in\K_1$ to be equal to $1$ if $|x|<1$ and to 0 if $|x|\ge 1$; also define $\Lambda(x)=\max(1-|x|,0)\in\R$; also keep in mind that any Archimedean locally compact field naturally contains $\R$ as a closed subfield. We define $A$ and $B$ according to whether $\K_1$ and $\K_2$ are Archimedean:
\begin{itemize}
\item $\K_2$ non-Archimedean ($\K_1$ arbitrary): $A(x_1)=(x_1,0)$ and $B(x_2)=(o(x_2),0)$;
\item $\K_1$ and $\K_2$ both Archimedean: $A(x_1)=(x_1,0)$ and $B(x_2)=(\Lambda(x_2),0)$;
\item $\K_2$ Archimedean, $\K_1$ non-Archimedean: $A(x_1)=(0,|x_1|)$ and $B(x_2)=(0,\Lambda(x_2))$.
\end{itemize}
(Note that the cases $\K_1$ Archimedean and not $\K_2$, and vice versa, cannot be treated simultaneously because of the dissymmetry resulting from the condition $\mu\le 1$.)

Since $A(0)=B(1)=(0,0)$, Lemma \ref{aireg1} implies that $I=2B(0)A(\ell_1^n)$, and thus $I=(2\ell_1^n,0)$ in the first two cases, and $I=(0,2|\ell_1|^n)$ in the last case.

It is clear that in each case, $A$ is 1-Lipschitz. We have to check the H\"older-like condition for $B$, which is clear when $\K_2$ is non-Archimedean. If $\K_2$ is Archimedean, we need to check the H\"older-like condition for $f(x_2)=\max(0,1-|x_2|)\in\R$. Indeed, if $s,s'\in\K_2$, if both $|s|,|s'|\ge 1$, then $f(s)=f(s')$; if $|s|\le 1\le |s'|$, then $|f(s)-f(s')|=1-|s|$; then since $1\le |s'|\le |s'-s|+|s|$, we deduce $1-|s|\le |s-s'|$. If $|s-s'|\le 1$, we deduce that $|f(s)-f(s')|\le |s-s'|\le |s-s'|^\mu$ If $|s-s'|\ge 1$, we directly see $|f(s)-f(s')|=1-|s|\le 1\le |s-s'|^\mu$. Finally if both $|s|,|s'|\le 1$, we have $|f(s)-f(s')|=|s|-|s'|\le 1$. Thus the inequality is again clear if $|s-s'|\ge 1$, and otherwise $|f(s)-f(s')|=|s|-|s'|\le |s-s'|\le |s-s'|^\mu$. Therefore by Lemma \ref{mst}, for each $R$ there is a bound $C(R)$ on the norm of $\int\beta d\alpha$ over any triangle of diameter $\le R$.
We now again discuss:

\begin{itemize}
\item $\K_1$ and $\K_2$ both non-Archimedean, $I=(2\ell_1^n,0)$: by ultrametricity of $\K_1$, and the combinatorial Stokes theorem, $C(R)$ is a bound for the norm of $\int\beta d\alpha$ over an arbitrary loop that can be decomposed into triangles of diameter $\le R$. Since $(|2|_{\K_1}|\ell_1|_{\K_1}^n)$ goes to infinity, this shows that for every $R$ there exists $n_0$ such that for every $n\ge n_0$, the loop $\gamma_{1,n}$ cannot be decomposed into triangles of diameter $\le R$. Thus the sequence $(\gamma_{1,n})$ has asymptotically infinite area. This shows that $G$ is not compactly presented.
\item $\K_1$ Archimedean, $\K_2$ arbitrary, $I=(2\ell_1^n,0)$. Fix $R$ so that every combinatorial loop in $G$ can be decomposed into triangles of diameter $\le R$. 
Suppose that $\gamma_{1,n}$ can be decomposed into $j_n$ triangles of 
diameter $\le R$. By the combinatorial Stokes formula (see \S\ref{cst}) and Lemma \ref{mst}, $|I|\le C(R)j_n$. It follows that $j_n\ge 2|\ell_1|^n_{\K_1}/C(R)$. Hence the area of $\gamma_{1,n}$ grows at least exponentially, so the Dehn function of $G$ grows at least exponentially.
\item $\K_2$ Archimedean, $\K_1$ non-Archimedean: $I=(0,2|\ell_1|^n)$. The argument is exactly as in the previous case.
\end{itemize}
\end{proof}

\begin{rem}
The assumption that $\K_1$ does not have characteristic two can be removed, but in that case we need to redefine $\int\beta d\alpha$ as $\sum_i\beta(c_i)(\alpha(c_{i+1})-\alpha(c_{i}))$. The drawback of this definition is that the integral is not invariant under conjugation. However, with the help of Lemma \ref{vk}, it is possible to conclude. Since we are not concerned with characteristic two here, we leave the details to the reader.
\end{rem}

However Proposition \ref{solex} is not enough for our purposes, because we do not only wish to bound below the Dehn function of the group $G$, but also of various groups $H$ mapping onto $G$. In general, the loop $\gamma_{1,n}$ does not lift to a loop in those groups, so we consider more complicated loops $\gamma_{k,n}$ in $G$, which eventually lift to the groups we have in mind. However, to estimate the area, we will go on working in $G$, because we know how to compute therein, and because obviously the area of a loop in $H$ is bounded below by the area of its image in $G$.

Define by induction $$\gamma_{k,n}=\gamma_{k-1,n}g_k\gamma_{k-1,n}^{-1}g_k^{-1}.$$
Here, $g_k$ denotes the element $(0,y_k,0)$ in the group $G$, where the sequence $(y_i)$ in $\K_2$ satisfies the following property:
$y_1=1$ and for any non-empty finite subset $I$ of integers, $$\left|\sum_{i\in I}y_i\right|\ge 1.$$ For instance, if $\K_2$ is ultrametric, this is satisfied by $y_i=\ell_2^i$; if $\K_2=\R$, the constant sequence $y_i=1$ works. The sequence $(y_i)$ will be fixed once and for all.

Fix $n$. We wish to compute, more generally, $\int_{\gamma_{k,n}}\beta d\alpha$. Write the path $\gamma_{k,n}$ as $(c_i)$. Note that for given $n$, $c_i$ does not depend on $k$ (because $\gamma_{k,n}$ is an initial segment of $\gamma_{k+1,n}$). Write the combinatorial length of $\gamma_{k,n}$ as $\lambda_{k,n}$ ($\lambda_{1,n}=4n+4$, $\lambda_{k+1,n}=2\lambda_{k,n}+2$).

\begin{lem}The number $n$ being fixed, we have
\begin{itemize}
\item[(1)] There exists a sequence finite subsets $F_i$ of the set of positive integers, such that for all $i$, we have $y(c_i)=\sum_{j\in F_i}y_j$, and satisfying in addition: for all $i<\lambda_{k,n}$ and all $k\ge 1$, we have $F_i\subset\{1,\dots,k\}$. Moreover $y(c_i)\neq 0$ (and thus $F_i\neq\emptyset$), unless either
\begin{itemize}
\item $i\le 2n+2$, or
\item $i=\lambda_{j,n}$ for some $j$. 
\end{itemize}
\item[(2)] Assume that $1\le i\le n-1$, or $n+2\le i\le 2n+2$, or $i=\lambda_{j,n}$ for some $j$. Then $x(c_{i-1})=x(c_{i+1})$.
\end{itemize}\label{nulpartout}
\end{lem}
\begin{proof}~
\begin{itemize}
\item[(1)] The sequence $(F_i)$ is constructed for $i<\lambda_{k,n}$, by induction on $k$. For $k=1$, we set $F_i=\emptyset$ if $i\le 2n-1$ and $F_i=\{1\}$ if $2n+2\le i\le 4n+3=\lambda_{1,n}-1$, and it satisfies the equality for $y(c_i)$ (see the proof of Lemma \ref{aireg1}, where $c_i$ is made explicit for all $i\le \lambda_{1,n}=4n+4$).

Now assume that $k\ge 2$ and that $F_i$ is constructed for $i<\lambda=\lambda_{k-1,n}$ with the required properties. We set $F_{\lambda}=\emptyset$; since $c_\lambda=(0,0,0)$, the condition holds for $i=\lambda$. It remains to deal with $i$ when $\lambda<i<\lambda_{k,n}$; in this case $c_i=g_kc_{2\lambda-i}$, so $y(c_i)=y_k+y(c_i)$. Thus if we set $F_i=\{k\}\cup F_{2\lambda-i}$, remembering by induction that $F_{2\lambda-i}\subset\{1,\dots,k-1\}$, we deduce that $y(c_i)=\sum_{j\in F_i}y_j$; clearly $F_i\subset\{1,\dots,k\}$.

\item[(2)]This was already checked for $i\le 2n+2$ (see the proof of Lemma \ref{aireg1}). In the case $i=\lambda_{j,n}$, we have $c_{i-1}=g_{j}$ and $c_{i+1}=g_{j+1}$, so $x(c_{i-1})=x(c_{i+1})$.\qedhere
\end{itemize}
\end{proof}

\begin{lem}
Under the assumptions of Lemma \ref{aireg1}, we have 
$$\int_{\gamma_{k,n}}\beta d\alpha=|2\ell_1|_{\K_1}^n.$$
\label{airegk}
\end{lem}
\begin{proof}
By Lemma \ref{nulpartout}, if both $|y(c_i)|<1$ and $x(c_{i-1})\neq x(c_{i+1})$, then $i=n$ or $i=n+1$. It follows that the desired integral on $\gamma_{k,n}$ is the same as the integral on $\gamma_{1,n}$ computed in Lemma \ref{aireg1}.
\end{proof}

\subsection{Groups with the SOL obstruction}

\begin{thm}\label{expar}
Let $G_1$ be a locally compact, compactly generated group, and suppose there is a continuous surjective homomorphism $$G_1\to G=(\K_1\times\K_2)\rtimes_{(\ell_1,\ell_2^{-1})}\Z,$$
where $0\neq 2$ in $\K_1$. Suppose that $G_1$ has a nilpotent normal subgroup $H$ whose image in $G$ contains $\K_1\times\K_2$. 
Then 
\begin{itemize}
\item the Dehn function of $G$ is at least exponential. 
\item if both $\K_1$ and $\K_2$ are ultrametric, then $G$ is not compactly presented.
\end{itemize}

\end{thm}
\begin{proof}
Let $k_0$ be the nilpotency length of $H$ and fix $k\ge k_0$. Using Lemma \ref{airegk} and arguing as in the proof of Proposition \ref{krd} (using Lemma \ref{airegk} instead if Lemma \ref{aireg1}), we obtain that the loops $\gamma_{k,n}$, which have linear length with respect to $n$, have at least exponential area, and asymptotically infinite area in case $\K_1$ and $\K_2$ are both ultrametric.

Lift $x$, $y$, and $g_k$ to elements $\tilde{x},\tilde{y},\tilde{g_k}$ in $H$ and $t$ to an element $\tilde{t}$ in $G$; set $\tilde{X}_n=\tilde{t}^n\tilde{x}\tilde{t}^{-n}$; since $H$ is normal, $\tilde{X}_n\in H$. This lifts $\gamma_{k,n}$ to a path $\widetilde{\gamma_{k,n}}$ based at 1; let $v_{k,n}$ be its value at $\lambda_{k,n}$, so $v_{1,n}=\tilde{X}_n\tilde{y}\tilde{X}_n^{-1}\tilde{y}^{-1}$ and $v_{k+1,n}=v_{k,n}\tilde{g_k}v_{k,n}^{-1}\tilde{g_k}^{-1}$. We see by an immediate induction that $v_{k,n}$ belongs to the $(k+1)$th term in the lower central series of $H$. Since $k\ge k_0$, we see that $\widetilde{\gamma_{k,n}}$ is a loop of $G$, of linear length with respect to $n$, mapping to $\gamma_{k,n}$. In particular, its area is at least the area of $\gamma_{k,n}$. So we deduce that $\widetilde{\gamma_{k,n}}$ has at least exponential area with respect to $n$, and has asymptotically infinite area in case $\K_1$ and $\K_2$ are both ultrametric.
\end{proof}

We will also need the following variant, in the real case.
\begin{thm}\label{exparr}
Let $G_1$ be a locally compact, compactly generated group, and suppose there is a continuous homomorphism with dense image $$G_1\to G=(\R\times\R)\rtimes\R,$$
so that the element $t\in\R$ acts by the diagonal matrix $(\ell_1^t,\ell_2^{-t})$ ($\ell_2\ge \ell_1>0$).
Suppose that $G_1$ has a nilpotent normal subgroup $H$ whose image in $G$ contains $\R\times\R$. 
Then the Dehn function of $G$ is at least exponential. 
\end{thm}
\begin{proof}
Since the homomorphism has dense image containing $\R\times\R$, there exists some element $\tilde{t}$ mapping to an element $t$ of the form $(0,0,\tau)$ with $\tau>0$. Changing the parameterization of $G$ if necessary (replacing $\ell_i$ by $\ell_i^\tau$ for $i=1,2$), we can suppose that $\tau=1$. Then lift $x$ and $y$ and pursue the proof exactly as in the proof of Theorem \ref{expar}.
\end{proof}

\begin{rem}
To summarize the proof, the lower exponential bound is obtained by finding two functions $\alpha,\beta$ on $G_1$ such that the integral $\int\beta d\alpha$ is bounded on triangles of bounded diameter, and a sequence $(\gamma_n)$ of combinatorial loops of linear diameter such that $\int_{\gamma_n}\beta d\alpha$ grows exponentially.

This approach actually also provides a lower bound on the homological Dehn function \cite{Ger92,BaMS,Ger} as well. Let us recall the definition. Let $G$ be a locally compact group with a generating set $S$ and a subset $R$ of the kernel of $F_S\to G$ consisting of relations of bounded length, yielding a polygonal complex structure with oriented edges and 2-faces.  
Let $A$ denote either $\Z$ or $\R$. For $i=0,1,2$, let $C_i(G,A)$ be the $A$-module freely spanned by the set of vertices, resp.\ oriented edges, resp.\ oriented 2-faces. Endow each $C_i(G,A)$ with the $\ell^1$ norm. There are usual boundary operators
$$C_2(G,A)\stackrel{\partial_2}\to C_1(G,A)\stackrel{\partial_1}\to C_0(G,A).$$
satisfying $\partial_1\circ\partial_2=0$. If $Z_1(G,A)$ is the kernel of $\partial_1$, then it is easy to extend, by linearity, the definition of $\int_\mathbf{c}\beta d\alpha$ (from \S\ref{cst}) to $\mathbf{c}\in Z_1(G,A)$.

Following \cite{Ger}, define, for $\mathbf{c}\in Z_1(G,A)$ $$\textnormal{HFill}^A_{G,S,R}(\mathbf{c})=\inf\{\|P\|_1:\;P\in C_2(G,A),\partial_2(P)=\mathbf{c}\}.$$
and $$\textnormal{H}\delta^A_{G,S,R}(n)=\sup\{\textnormal{HFill}(z)^A_{G,S,R}:\;z\in Z_1(G,\Z),\;\|z\|_1\le n\}.$$

Clearly, if $\mathbf{c}$ is a basis element (so that its area makes sense)
$$\textnormal{HFill}(\mathbf{c})^\R_{G,S,R}\le \textnormal{HFill}(\mathbf{c})^\Z_{G,S,R}\le \textnormal{area}^{G,S,R}(\mathbf{c})\le\infty;$$
it follows that
$$\textnormal{H}\delta^\R_{G,S,R}(n)\le\textnormal{H}\delta^\Z_{G,S,R}(n)\le\widehat{\delta_{G,S,R}}(n),$$
where $\widehat{\delta_{G,S,R}}$ is the smallest superadditive function greater or equal to $\delta_{G,S,R}$. (Examples of non-superadditive Dehn functions of finite presentations of groups are given in \cite{GS}; however the question, raised in \cite{GS}, whether any Dehn function of a finitely presented group is asymptotically equivalent to a superadditive function, is still open.)

The function $\textnormal{HFill}(\mathbf{c})^A_{G,S,R}$ is called the $A$-homological Dehn function of $(G,R,S)$ (the function $\textnormal{HFill}(\mathbf{c})^\Z_{G,S,R}$ is called abelianized isoperimetric function in \cite{BaMS}). If finite, it can be shown by routine arguments that its $\approx$-asymptotic behavior only depends on $G$. Some Bestvina-Brady groups \cite{BeBr} provide examples of finitely generated groups with finite homological Dehn function but infinite Dehn function. Until recently, no example of a compactly presented group was known for which the integral (or even real) homological Dehn function is not equivalent to the integral homological Dehn function; the issue was raised, for finitely presented groups, both in \cite[p.~536]{BaMS} and \cite[p.~1]{Ger}; the first examples have finally been obtained by Abrams, Brady, Dani and Young in \cite{ABDY}. 

Let us turn back to $G_1$ (a group satisfying the hypotheses of Theorem \ref{expar} or \ref{exparr}): for this example, since $R$ consists of relations of bounded length, it follows that the integral of $\beta d\alpha$ over the boundary of any polygon is bounded. Since $\int_{\gamma_n}\beta d\alpha$ grows exponentially, it readily follows that $\textnormal{HFill}^\R_{G_1}(\gamma_n)$ grows at least exponentially and hence $\textnormal{H}\delta^\R_{G_1}(n)$ (and thus $\textnormal{H}\delta^\Z_{G_1}(n)$) grows at least exponentially.
\end{rem}



\end{document}